\documentclass[graybox,envcountchap]{svmono} 

\usepackage{type1cm}         

\usepackage{makeidx}         
\usepackage{subcaption}
\usepackage{graphicx}        
\usepackage{multicol}        
\usepackage[bottom]{footmisc}

\usepackage{newtxtext}       %
\usepackage[varvw]{newtxmath}       

\usepackage{dsfont}
\usepackage{hyperref}
\usepackage{enumerate}
\usepackage{cleveref}

\allowdisplaybreaks

\newtheorem{thm}{Theorem}[chapter]
\newtheorem{basrepthm}[thm]{Basic Representation Theorem}
\newtheorem{bernthm}[thm]{Bernstein's Theorem}
\newtheorem{boasthm}[thm]{Boas' Theorem}
\newtheorem{carathm}[thm]{Carath\'eodory's Theorem}
\newtheorem{cmsthm}[thm]{Curtis--Mairhuber--Sieklucki Theorem}

\newtheorem{danthm}[thm]{Daniell's Representation Theorem}
\newtheorem{dansigthm}[thm]{Signed Daniell's Representation Theorem}
\newtheorem{havithm}[thm]{Haviland's Theorem}
\newtheorem{hausthm}[thm]{Hausdorff's Theorem}
\newtheorem{habathm}[thm]{Hahn--Banach Theorem}
\newtheorem{hamthm}[thm]{Hamburger's Theorem}
\newtheorem{karthm}[thm]{Karlin's Theorem}
\newtheorem{karposthm}[thm]{Karlin's Positivstellensatz}
\newtheorem{karnegthm}[thm]{Karlin's Nichtnegativstellensatz}
\newtheorem{lumathm}[thm]{Luk\'acs--Markov Theorem}
\newtheorem{rithm}[thm]{Riesz' Representation Theorem}
\newtheorem{risigthm}[thm]{Signed Riesz' Representation Theorem}
\newtheorem{snakethm}[thm]{Snake Theorem}
\newtheorem{swthm}[thm]{Stone--Weierstrass Theorem}
\newtheorem{stielthm}[thm]{Stieltjes' Theorem}

\newtheorem{svethm}[thm]{{\v{S}}venco's Theorem}
\newtheorem{richthm}[thm]{Richter's Theorem}
\newtheorem{lem}[thm]{Lemma}
\newtheorem{rideclem}[thm]{Riesz Decomposition Lemma}
\newtheorem{urylem}[thm]{Urysohn's Lemma}

\newtheorem{cor}[thm]{Corollary}

\theoremstyle{definition}
\newtheorem{dfn}[thm]{Definition}
\newtheorem{exm}[thm]{Example}
\newtheorem{exms}[thm]{Examples}

\theoremstyle{remark}
\newtheorem{rem}[thm]{Remark}

\newcommand{\exmsymbol}{\hfill$\circ$}

\newcommand{\cset}{\mathds{C}}

\newcommand{\nset}{\mathds{N}}

\newcommand{\qset}{\mathds{Q}}
\newcommand{\rset}{\mathds{R}}

\newcommand{\tset}{\mathds{T}}
\newcommand{\zset}{\mathds{Z}}

\newcommand{\cone}{\mathrm{cone}\,}
\newcommand{\conv}{\mathrm{conv}\,}

\newcommand{\diff}{\mathrm{d}}

\newcommand{\lin}{\mathrm{lin}\,}

\newcommand{\pos}{\mathrm{Pos}}

\newcommand{\sign}{\mathrm{sgn}}

\newcommand{\supp}{\mathrm{supp}\,}

\newcommand{\folgt}{\ensuremath{\Rightarrow}}
\newcommand{\gdw}{\ensuremath{\Leftrightarrow}}
\newcommand{\inter}{\mathrm{int}\,}

\newcommand{\cA}{\mathcal{A}}

\newcommand{\cat}{\mathcal{C}}

\newcommand{\cF}{\mathcal{F}}
\newcommand{\cG}{\mathcal{G}}
\newcommand{\cH}{\mathcal{H}}

\newcommand{\cL}{\mathcal{L}}
\newcommand{\cN}{\mathcal{N}}
\newcommand{\cM}{\mathcal{M}}
\newcommand{\cP}{\mathcal{P}}
\newcommand{\cS}{\mathcal{S}}
\newcommand{\cT}{\mathcal{T}}
\newcommand{\cX}{\mathcal{X}}

\newcommand{\cV}{\mathcal{V}}
\newcommand{\cW}{\mathcal{W}}
\newcommand{\cY}{\mathcal{Y}}
\newcommand{\cZ}{\mathcal{Z}}

\newcommand{\fA}{\mathfrak{A}}
\newcommand{\fB}{\mathfrak{B}}

\ExplSyntaxOn
\NewDocumentCommand{\advanced}{mO{#3}m}
 {
  \cs_set_eq:Nc \__advanced_save: { the \cs_to_str:N #1 }
  \cs_set:cpn { the \cs_to_str:N #1 } { \__advanced_save: \advancedmarker }
  \bool_set_true:N \__advanced_killwidth_bool
  #1[#2]{#3}
  \bool_set_false:N \__advanced_killwidth_bool
  \cs_set_eq:cN { the \cs_to_str:N #1 } \__advanced_save:
 }
\NewDocumentCommand{\advancedmarker}{}
 {
  \bool_if:NTF \__advanced_killwidth_bool { \makebox[5pt][l]{*} } { * }
 }

\bool_new:N \__advanced_killwidth_bool

\AddToHook{cmd/@starttoc/begin}{\bool_set_true:N \__advanced_killwidth_bool}
\ExplSyntaxOff

\makeindex             

\date{\today}

\begin{document}

\author{Philipp J.\ di\,Dio}
\title{An Introduction to T-Systems}
\subtitle{with a special Emphasis on Sparse Moment Problems, Sparse Positivstellensätze, and Sparse Nichtnegativstellensätze}
\maketitle

\frontmatter

\begin{dedication}
Samuel Karlin\quad {\small(June 8, 1924 -- December 18, 2007)}\\ \medskip
He solved almost unnoticed an\\ important algebraic question.
\end{dedication}





\preface
\addcontentsline{toc}{chapter}{Preface}

These are the lecture notes based on \cite{didio23sparse} for the (upcoming) lecture
\emph{T-systems with a special emphasis on sparse moment problems and sparse Positivstellensätze}
in the summer semester 2024 at the University of Konstanz.

The main purpose of this lecture is to prove the sparse Positiv- and Nichtnegativstellensätze of Samuel Karlin\index{Karlin, S.} (1963) and to apply them to the algebraic setting.
That means given finitely many monomials, e.g.
\[1, x^2, x^3, x^6, x^7, x^9,\]
how do all linear combinations of these look like which are strictly positive or non-negative on some interval $[a,b]$ or $[0,\infty)$, e.g.\ describe and even write down all
\[f(x) = a_0 + a_1 x^2 + a_2 x^3 + a_3 x^6 + a_4 x^7 + a_5 x^9\]
with $f(x)>0$ or $f(x)\geq 0$ on $[a,b]$ or $[0,\infty)$, respectively.

To do this we introduce the theoretical framework in which this question can be answered:  T-systems.
We study these T-systems to arrive at Karlin's Positiv- and Nichtnegativstellensatz but we also do not hide the limitations of the T-systems approach.

The main limitation is the Curtis--Mairhuber--Sieklucki Theorem which essentially states that every T-system is only one-dimensional and hence we can only apply these results to the \emph{univariate} polynomial case.
This can also be understood as a lesson or even a warning that this approach has been investigated and found to fail, i.e., learning about these results and limitations shall save students and researchers from following old footpaths which lead to a dead end.

We took great care finding the correct historical references where the results appeared first but are perfectly aware that like people before we not always succeed.

\vspace{\baselineskip}
\begin{flushright}\noindent
Konstanz,\hfill \, 
\\
February 2024
\hfill \textit{Philipp J.\ di\,Dio}
\end{flushright}

\extrachap{Acknowledgements}
\addcontentsline{toc}{chapter}{Acknowledgements}

The author was supported by the Deutsche Forschungs\-gemein\-schaft DFG with the grant DI-2780/2-1 and his research fellowship at the Zukunfs\-kolleg of the University of Konstanz, funded as part of the Excellence Strategy of the German Federal and State Government.

The author thanks Konrad Schmüdgen for reading large parts of the manuscript and providing valuable remarks.

\chapter*{Contents}
\contentsline {chapter}{Preface}{vii}{chapter*.1}%
\contentsline {chapter}{Acknowledgements}{ix}{chapter*.2}%
\contentsline {chapter}{Contents}{xi}{chapter*.3}%
\contentsline {chapter}{\numberline {0}Preliminaries}{1}{chapter.0}%
\contentsline {section}{\numberline {0.1}Sets, Relations, and Orders}{1}{section.0.1}%
\contentsline {section}{\numberline {0.2}Topology}{2}{section.0.2}%
\contentsline {section}{\numberline {0.3}Stone--Weierstrass Theorem}{3}{section.0.3}%
\contentsline {section}{\numberline {0.4}Convex Geometry}{3}{section.0.4}%
\contentsline {section}{\numberline {0.5}Linear Algebra}{3}{section.0.5}%
\contentsline {section}{\numberline {0.6}Measures}{4}{section.0.6}%
\contentsline {section}{\numberline {0.7\advancedmarker }Daniell's Representation Theorem}{6}{section.0.7}%
\contentsline {section}{\numberline {0.8}Riesz' Representation Theorem}{11}{section.0.8}%
\contentsline {section}{\numberline {0.9\advancedmarker }Riesz Decomposition}{12}{section.0.9}%
\contentsline {part}{Part\ I\hspace {\betweenumberspace }Introduction to Moments}{15}{part.1}%
\contentsline {chapter}{\numberline {1}Moments and Moment Functionals}{17}{chapter.1}%
\contentsline {section}{\numberline {1.1}Moments and Moment Functionals}{17}{section.1.1}%
\contentsline {section}{\numberline {1.2}Determinacy and Indeterminacy}{19}{section.1.2}%
\contentsline {section}{Problems}{20}{section*.6}%
\contentsline {chapter}{\numberline {2}Choquet's Theory and Adapted Spaces}{21}{chapter.2}%
\contentsline {section}{\numberline {2.1}Extensions of Linear Functionals preserving Positivity}{21}{section.2.1}%
\contentsline {section}{\numberline {2.2}Adapted Spaces of Continuous Functions}{23}{section.2.2}%
\contentsline {section}{\numberline {2.3}Existence of Integral Representations}{24}{section.2.3}%
\contentsline {section}{\numberline {2.4\advancedmarker }Adapted Cones}{25}{section.2.4}%
\contentsline {section}{\numberline {2.5\advancedmarker }Continuity of Positive Linear Functionals}{27}{section.2.5}%
\contentsline {section}{Problems}{28}{section*.7}%
\contentsline {chapter}{\numberline {3}The Classical Moment Problems}{29}{chapter.3}%
\contentsline {section}{\numberline {3.1}Classical Results}{29}{section.3.1}%
\contentsline {section}{\numberline {3.2}Early Results with Gaps}{35}{section.3.2}%
\contentsline {section}{\numberline {3.3}Finitely Atomic Representing Measures: Richter's Theorem}{35}{section.3.3}%
\contentsline {section}{\numberline {3.4}Signed Representing Measures: Boas' Theorem}{38}{section.3.4}%
\contentsline {section}{\numberline {3.5}Solving all Truncated Moment Problems solves the Moment Problem}{39}{section.3.5}%
\contentsline {section}{Problems}{40}{section*.8}%
\contentsline {part}{Part\ II\hspace {\betweenumberspace }Tchebycheff Systems}{41}{part.2}%
\contentsline {chapter}{\numberline {4}T-Systems}{43}{chapter.4}%
\contentsline {section}{\numberline {4.1}The Early History of T-Systems}{43}{section.4.1}%
\contentsline {section}{\numberline {4.2}Definition and Basic Properties}{44}{section.4.2}%
\contentsline {section}{\numberline {4.3}The Curtis--Mairhuber--Sieklucki Theorem}{47}{section.4.3}%
\contentsline {section}{\numberline {4.4}Examples of T-Systems}{48}{section.4.4}%
\contentsline {section}{\numberline {4.5}Representation as a Determinant, Zeros, and Non-Negativity}{49}{section.4.5}%
\contentsline {section}{Problems}{55}{section*.9}%
\contentsline {chapter}{\numberline {5}ET- and ECT-Systems}{57}{chapter.5}%
\contentsline {section}{\numberline {5.1}Definitions and Basic Properties}{57}{section.5.1}%
\contentsline {section}{\numberline {5.2}Wronskian Determinant}{60}{section.5.2}%
\contentsline {section}{\numberline {5.3}Characterizations of ECT-Systems}{63}{section.5.3}%
\contentsline {section}{\numberline {5.4}Examples of ET- and ECT-Systems}{67}{section.5.4}%
\contentsline {section}{\numberline {5.5}Representation as a Determinant, Zeros, and Non-Negativity}{68}{section.5.5}%
\contentsline {section}{Problems}{70}{section*.10}%
\contentsline {chapter}{\numberline {6}Generating ET-Systems from T-Systems by Using Kernels}{71}{chapter.6}%
\contentsline {section}{\numberline {6.1}Kernels}{71}{section.6.1}%
\contentsline {section}{\numberline {6.2}The Basic Composition Formulas}{73}{section.6.2}%
\contentsline {section}{\numberline {6.3}Smoothing T-Systems into ET-Systems}{74}{section.6.3}%
\contentsline {section}{Problems}{74}{section*.11}%
\contentsline {part}{Part\ III\hspace {\betweenumberspace }Karlin's Positivstellensätze and Nichtnegativstellensätze}{75}{part.3}%
\contentsline {chapter}{\numberline {7}Karlin's Positivstellensatz and Nichtnegativstellensatz on $[a,b]$}{77}{chapter.7}%
\contentsline {section}{\numberline {7.1}Karlin's Positivstellensatz for T-Systems on $[a,b]$}{77}{section.7.1}%
\contentsline {section}{\numberline {7.2}The Snake Theorem: An Interlacing Theorem}{82}{section.7.2}%
\contentsline {section}{\numberline {7.3}Karlin's Nichtnegativstellensatz for ET-Systems on $[a,b]$}{84}{section.7.3}%
\contentsline {section}{Problems}{88}{section*.14}%
\contentsline {chapter}{\numberline {8}Karlin's Positivstellensätze and Nichtnegativstellensätze on $[0,\infty )$ and $\mathds {R}$}{89}{chapter.8}%
\contentsline {section}{\numberline {8.1}Karlin's Positivstellensatz for T-Systems on $[0,\infty )$}{89}{section.8.1}%
\contentsline {section}{\numberline {8.2}Karlin's Nichtnegativstellensatz for ET-Systems on $[0,\infty )$}{91}{section.8.2}%
\contentsline {section}{\numberline {8.3}Karlin's Positivstellensatz for T-Systems on $\mathds {R}$}{92}{section.8.3}%
\contentsline {section}{\numberline {8.4}Karlin's Nichtnegativstellensatz for ET-Systems on $\mathds {R}$}{92}{section.8.4}%
\contentsline {section}{Problems}{93}{section*.15}%
\contentsline {part}{Part\ IV\hspace {\betweenumberspace }Non-Negative Algebraic Polynomials on $[a,b]$, $[0,\infty )$, and $\mathds {R}$}{95}{part.4}%
\contentsline {chapter}{\numberline {9}Non-Negative Algebraic Polynomials on $[a,b]$}{97}{chapter.9}%
\contentsline {section}{\numberline {9.1}Sparse Algebraic Positivstellensatz on $[a,b]$}{97}{section.9.1}%
\contentsline {section}{\numberline {9.2}Sparse Hausdorff Moment Problem}{101}{section.9.2}%
\contentsline {section}{\numberline {9.3}Sparse Algebraic Nichtnegativstellensatz on $[a,b]$}{103}{section.9.3}%
\contentsline {section}{Problems}{104}{section*.24}%
\contentsline {chapter}{\numberline {10}Non-Negative Algebraic Polynomials on $[0,\infty )$ and on $\mathds {R}$}{105}{chapter.10}%
\contentsline {section}{\numberline {10.1}Sparse Algebraic Positivstellensatz on $[0,\infty )$}{105}{section.10.1}%
\contentsline {section}{\numberline {10.2}Sparse Stieltjes Moment Problem}{108}{section.10.2}%
\contentsline {section}{\numberline {10.3}Sparse Algebraic Nichtnegativstellensatz on $[0,\infty )$}{109}{section.10.3}%
\contentsline {section}{\numberline {10.4}Algebraic Positiv- and Nichtnegativstellensatz on $\mathds {R}$}{110}{section.10.4}%
\contentsline {section}{Problems}{110}{section*.25}%
\contentsline {part}{Part\ V\hspace {\betweenumberspace }Applications of T-Systems}{111}{part.5}%
\contentsline {chapter}{\numberline {11}Moment Problems for continuous T-Systems on $[a,b]$}{113}{chapter.11}%
\contentsline {section}{\numberline {11.1}General Moment Problems for continuous T-Systems on $[a,b]$}{113}{section.11.1}%
\contentsline {section}{\numberline {11.2}A Non-Polynomial Example}{114}{section.11.2}%
\contentsline {chapter}{\numberline {12}Polynomials of Best Approximation and Optimization over Linear Functionals}{117}{chapter.12}%
\contentsline {section}{\numberline {12.1}Polynomials of Best Approximation}{117}{section.12.1}%
\contentsline {section}{\numberline {12.2}Optimization over Linear Functionals}{121}{section.12.2}%
\contentsline {part}{Appendices}{125}{part*.34}%
\contentsline {chapter}{Solutions}{125}{appendix*.35}%
\contentsline {section}{Problems of \Cref {ch:measures}}{125}{section*.36}%
\contentsline {section}{Problems of \Cref {ch:choquet}}{126}{section*.37}%
\contentsline {section}{Problems of \Cref {ch:classical}}{127}{section*.38}%
\contentsline {section}{Problems of \Cref {ch:tsystems}}{129}{section*.39}%
\contentsline {section}{Problems of \Cref {ch:etsystems}}{130}{section*.40}%
\contentsline {section}{Problems of \Cref {ch:ETfromT}}{135}{section*.42}%
\contentsline {section}{Problems of \Cref {ch:karlinPosab}}{135}{section*.43}%
\contentsline {section}{Problems of \Cref {ch:karlinPos0inftyR}}{135}{section*.44}%
\contentsline {section}{Problems of \Cref {ch:nonNegAlgPolab}}{136}{section*.45}%
\contentsline {section}{Problems of \Cref {ch:nonNegAlgPol0infty}}{137}{section*.46}%
\contentsline {chapter}{References}{139}{appendix*.47}%
\contentsline {chapter}{List of Symbols}{145}{appendix*.48}%
\contentsline {chapter}{Index}{147}{appendix*.52}%

\mainmatter



\motto{Pure mathematics is, in its way, the poetry of logical ideas.\\ \medskip
\ \hspace{1cm} \normalfont{Albert Einstein \cite{einstein35noether}}\index{Einstein, A.}}

\setcounter{chapter}{-1}
\chapter{Preliminaries}
\label{ch:prel}

The purpose of this preliminary chapter is not to establish and prove results but to clarify notation and to give the reader a survey of what will be assumed as known.

For the representation theorems of linear functionals of Daniell (\Cref{thm:daniellSigned}) and Riesz (\Cref{thm:rieszSigned}) more care is invested since these are the essential representation theorems in the theory of moments in the following chapters, i.e., we include the proofs.

\section{Sets, Relations, and Orders}
\label{sec:orderCone}

We let $\nset := \{1,2,3,\dots\}$ be the natural numbers, $\nset_0 := \{0,1,2,\dots,\}$ be the natural numbers including zero, and as usual $\zset$, $\qset$, $\rset$, and $\cset$.
The unit circle is denoted by $\tset := \{(x,y)\in\rset^2 \,|\, x^2+y^2 = 1\}$.\label{numbers}

For inclusions we use $\subseteq$ and $\subsetneq$.
To avoid any confusion we avoid the use of $\subset$ since $\subset$ is used in the literature by different authors either as $\subseteq$ or $\subsetneq$.

For a set $\cX$ we denote by $\cP(\cX)$\label{eq:PX} the set of all subsets of $\cX$.

A \emph{partial order}\index{order!partial}\index{partial order} on a set $\cX$ is a relation $R\subseteq\cX\times\cX$, usually denoted by $\leq$,\label{eq:leq} such that
\begin{enumerate}[(i)]
\item $x = y \quad\Leftrightarrow\quad x\leq y$ and $y\leq x$,
\item $x\leq y$ and $y\leq z \quad\Rightarrow\quad x\leq z$.
\end{enumerate}
A relation $\leq$ is a \emph{total order}\index{order!total}\index{total order} if for all $x,y\in\cX$ we have either $x\leq y$ or $y\leq x$.
A vector space $E$ with a partial order $\leq$ such that
\begin{enumerate}[(i)]
\item $x\leq y$ and $z\in\cX \quad\Rightarrow\quad x+z\leq y+z$,
\item $x\leq y$ and $a\in [0,\infty)\quad\Rightarrow\quad ax\leq ay$
\end{enumerate}
is called an \emph{ordered vector space}.\index{order!vector space}\index{vector space!ordered}
If $E$ is an ordered vector space then $E_+ := \{x\in E \,|\, 0\leq x\}$\label{eq:E+} denotes the \emph{positive cone}\index{cone!positive}\index{positive cone} and $E_- := \{x\in E \,|\, x\leq 0\}$ denoted the \emph{negative cone}.\index{cone!negative}\index{negative cone}
Let $C\subseteq E$ be a cone in a vector space $E$. Then $E$ with $x\leq y$ if and only if $y-x\in C$ is a (partially) ordered vector space.

For a vector space $E$ a (linear) function $f:E\to\rset$ is called (\emph{linear}) \emph{functional}.
For a vector space $E$ the (\emph{algebraic}) \emph{dual}\index{dual} $E^*$\label{eq:Estar} is the set of all linear functionals $f:E\to\rset$.
A functional $f:E\to\rset$ is called \emph{sublinear}\index{sublinear!functional}\index{functional!sublinear} if $f(\rho x) \leq \rho f(x)$ and $f(x+y) \leq f(x) + f(y)$ hold for all $\rho\geq 0$ and $x,y\in E$.
It is called \emph{superlinear}\index{superlinear!functional}\index{functional!superlinear} if $-f$ is sublinear.

\begin{habathm}\label{thm:hahnBanach}\index{Theorem!Hahn--Banach}\index{Hahn--Banach Theorem}
Let $\cX$ be a real vector space, let $p:\cX\to\rset$ be a sublinear function, $\cV\subseteq\cX$ be a subspace, and $f:\cV\to\rset$ be a linear functional such that $f(x)\leq p(x)$ for all $x\in\cV$.
Then there exists a linear functional $F:\cX\to\rset$ such that
\begin{enumerate}[(i)]
\item $f(x) = F(x)$ for all $x\in\cV$, and

\item $F(x)\leq p(x)$ for all $x\in\cX$.
\end{enumerate}
\end{habathm}

The \Cref{thm:hahnBanach} was proved by H.\ Hahn\index{Hahn, H.} \cite{hahn27} and S.\ Banach\index{Banach, S.} \cite{banach29a,banach29b}. 
A previous version is due to \mbox{E.\ Helly} \cite{helly12}.\index{Helly, E.}
For more see e.g.\ \cite{pietsch07} or standard functional analysis textbooks like \cite{yosida68,wernerFA}.

\section{Topology}

A \emph{topology}\index{topology} $\cT$ on a set $\cX$ is a set $\cT\subseteq\cP(\cX)$ of subsets of $\cX$ which is closed under finite intersections and arbitrary unions, i.e., especially $\emptyset,\cX\in\cT$.
$(\cX,\cT)$ is called a \emph{topological space}\index{space!topological} and sets $A\in\cT$ are called \emph{open}.
A set $A\subseteq\cX$ is called \emph{closed} if $\cX\setminus A$ is open.
The \emph{interior}\index{interior!of a set} $\inter A$\label{eq:intA} of a set $A\subseteq\cX$ is the union of all open sets $O\subseteq A$.
A subset $U$ of a topological space $(\cX,\cT)$ is called a \emph{neighborhood}\index{neighborhood} of $x$ if $x\in\inter U$.

A function $f:\cX\to\cY$ between two topological spaces $\cX$ and $\cY$ is called \emph{continuous}\index{continuous} at $x\in\cX$ if for each neighborhood $V$ of $y = f(x)$ the set $f^{-1}(y)$ is a neighborhood of $x$.
The function $f$ is called \emph{continuous} if it is continuous at every $x\in\cX$.
The set of continuous functions $f:\cX\to\cY$ is denoted by $\cat(\cX,\cY)$.\label{eq:catXY}
A set $K\subseteq\cX$ is called \emph{compact}\index{compact!set} if every open cover $K\subseteq \bigcup_{i\in I} U_i$, $U_i\in\cT$, has a finite subcover $K\subseteq\bigcup_{k=1}^n U_{i_k}$.
For a function $f:\cX\to\rset$ we have the support $\supp f := \{x\in\cX \,|\, f(x)\neq 0\}$.
The set of all continuous functions with compact support are denoted by $\cat_c(\cX,\rset)$.\label{eq:catcXR}

A topological space $\cX$ is called \emph{Hausdorff space}\index{Hausdorff!space}\index{space!Hausdorff} if each pair of distinct points $x,y\in\cX$ have disjoint neighborhoods.
A Hausdorff space $\cX$ is called \emph{locally compact}\index{compact!locally!Hausdorff space}\index{space!Hausdorff!locally compact}\index{locally compact} if every point $x\in\cX$ has a compact neighborhood.
On Hausdorff spaces we have the following important topological result.

\begin{urylem}[see \cite{urysohn25}]\label{lem:ury}
Let $\cX$ be a Hausdorff space.
The following are equivalent:
\begin{enumerate}[(i)]
\item For every pair of disjoint closed sets $A,B\subseteq\cX$ there exist a neighborhood $U$ of $A$ and a neighborhood $V$ of $B$ such that $U\cap V = \emptyset$.

\item For each pair $A,B\subseteq\cX$ of disjoint closed sets there exists a continuous function $f:\cX\to [0,1]$ such that $f(x) = 1$ for all $x\in A$ and $f(y) = 0$ for all $y\in B$.
\end{enumerate}
\end{urylem}

\section{Stone--Weierstrass Theorem}

\begin{swthm}[\cite{weierstrass85} and {\cite[pp.\ 467--468]{stone37}} or e.g.\ {\cite[p.\ 9]{yosida68}}]\label{thm:stoneWeierstrass}\index{Stone--Weierstrass Theorem}\index{Theorem!Stone--Weierstrass}
Let $\cX$ be a compact set and let $B\subseteq\cat(X,\rset)$ be such that
\begin{enumerate}[(i)]
\item $fg, \alpha f + \beta g\in B$ for all $f,g\in B$ and $\alpha,\beta\in\rset$,
\item there exists a $f\in B$ with $f>0$ on $\cX$, and
\item for all $x,y\in\cX$ with $x\neq y$ there is a $f\in B$ such that $f(x)\neq f(y)$
\end{enumerate}
then for any $f\in\cat(\cX,\rset)$ there exists $\{f_n\}_{n\in\nset_0}\subseteq B$ such that
\[\| f - f_n\|_\infty \xrightarrow{n\to\infty} 0.\]
\end{swthm}

Especially $\rset[x_1,\dots,x_n]$ on any compact $K\subseteq\rset^n$, $n\in\nset$, is dense in $\cat(K,\rset)$ in the $\sup$-norm.

For more on the history of the \Cref{thm:stoneWeierstrass} see e.g.\ \cite[§4.5.6--§4.5.8]{pietsch07}.

\section{Convex Geometry}

A set $\cX$ is \emph{convex}\index{convex} if $\lambda x + (1-\lambda)y\in\cX$ for all $x,y\in\cX$ and $\lambda\in [0,1]$.
A set $\cX$ is a \emph{cone}\index{cone} if $\lambda x\in\cX$ for all $x\in\cX$ and $\lambda\in [0,\infty)$.
For a set $A\subseteq\rset^n$ we denote by $\conv A$\label{eq:convA} the \emph{convex hull}\index{hull!convex} of $A$.

\begin{carathm}[see \cite{carath11}]\label{thm:cara}\index{Theorem!Carath\'eodory}\index{Carath\'eodory!Theorem}
Let $n\in\nset$ and let $S\subseteq\rset^n$ be a set.
If $x\in\conv A$ then there is a $k\leq n+1$, points $x_1,\dots,x_k\in A$, and $\lambda_1,\dots,\lambda_k > 0$ with
\[x = \lambda_1 x_1 + \dots + \lambda_k x_k \qquad\text{and}\qquad \lambda_1 + \dots + \lambda_k = 1.\]
\end{carathm}

For more on convex geometry we recommend \cite{rock72} and \cite{schne14}.

\section{Linear Algebra}

A matrix $M = (a_{i,j})_{i,j=1}^n$ with $a_{i,j} = a_{k,l}$ if $i+j = k+l$ is called \emph{Hankel matrix}.\index{Hankel matrix}
For a sequence $s = (s_\alpha)_{\alpha\in\nset_0:|\alpha|\leq 2n}$ with $n\in\nset_0$ we denote by
\begin{equation}\label{eq:hankel}
\cH(s) := (s_{\alpha+\beta})_{\alpha,\beta\in\nset_0:|\alpha|,|\beta|\leq n}
\end{equation}
the \emph{Hankel matrix of $s$}.\index{Hankel matrix!of a sequence}

\section{Measures}

For a set $\cX$ an \emph{algebra}\index{algebra} $\fA$ is a set $\fA\subseteq\cP(\cX)$ such that $\emptyset,\cX\in\fA$ and for all $A,B\in\fA$ we have $A\cap B, A\cup B, A\setminus B\in\fA$.
If additionally $\bigcup_{n=1}^\infty A_n\in\fA$ for all $A_n\in\fA$ then $\fA$ called a \emph{$\sigma$-algebra}\index{sigma-algebra@$\sigma$-algebra} and $(\cX,\fA)$ is called a \emph{measurable space}.\index{space!measurable}\index{measurable!space}
By $\fB(\rset^n)$\label{eq:borelSigmaAlg} we denote the \emph{Borel $\sigma$-algebra}.\index{Borel!sigma-algebra@$\sigma$-algebra}\index{sigma-algebra@$\sigma$-algebra!Borel}
For $A\subseteq\cP(\cX)$ we denote by $\sigma(A)$\label{sigmaA} the smallest $\sigma$-algebra containing $A$.
A function $f:(\cX,\fA)\to(\cY,\fB)$ between two measurable spaces is called \emph{measurable}\index{function!measurable}\index{measurable!function} if $f^{-1}(B)\in\fA$ for all $B\in\fB$.

A \emph{measure}\footnote{For us all measures are non-negative unless stated otherwise. In \cite{bogachevMeasureTheory} the theory is developed in greater generality.}\index{measure} $\mu$ is a function $\mu:\fA\to[0,\infty]$ on an algebra $\fA$ such that $\mu$ is \emph{countably additive}, i.e.,
\[\mu\left( \bigcup_{n=1}^\infty A_n \right) = \sum_{n=1}^\infty \mu(A_n)\]
for all pairwise disjoint sets $A_n\in\fA$.
A measure $\mu$ on $\fB(\rset^n)$ is called \emph{Borel measure}.\index{measure!Borel}\index{Borel!measure}
A Borel measure $\mu$ is called a \emph{Radon measure}\index{Radon!measure}\index{measure!Radon} if for every $A\in\fB(\rset^n)$ and $\varepsilon>0$ there exists a compact set $K_\varepsilon\subseteq A$ such that $\mu(A\setminus K_\varepsilon) < \varepsilon$.
We denote by $\cM(\cX)_+$\label{eq:MX+} the set of all Borel measures on $(\cX,\fA)$.
By $(\cX,\fA,\mu)$ we denote a \emph{measure space}.\index{space!measure}\index{measure!space}
A measurable function $f:(\cX,\fA)\to\rset$ is called \emph{$\mu$-integrable}\index{function!mu-integrable@$\mu$-integrable}\index{mu-integrable@$\mu$-integrable} if
\[\int_\cX |f(x)|~\diff\mu(x) < \infty.\]
For any $p\geq 1$ we denote by $\cL^p(\cX,\mu)$\label{eq:Lp} all $\mu$-integrable functions on $\cX$.
For $p=\infty$, i.e., $\cL^\infty(\cX,\mu)$, the essential supremum is bounded.

Since we are proving the (signed) Daniell's Theorem and the (signed) Riesz' Representation Theorem we will give a more detailed background on measures.
For more on measure theory we recommend \cite{bogachevMeasureTheory} and \cite{federerGeomMeasTheo}.

\begin{dfn}\label{dfn:outerMeasure}
Let $\cX$ be a set. A function $\mu: \cP(\cX)\to [0,\infty]$ with
\begin{enumerate}[(i)]
\item $\mu(\emptyset) = 0$,
\item $\mu(A) \leq \mu(B)$ for all $A\subseteq B\subseteq\cX$, and
\item $\mu\left(\bigcup_{i=1}^\infty A_i\right) \leq \sum_{i=1}^\infty \mu(A_i)$ for all $A_i\in\cX$
\end{enumerate}
is called a (\emph{Carath\'eodory)} \emph{outer measure}.\index{measure!outer}\index{measure!Carath\'eodory outer}
\end{dfn}

\begin{dfn}\label{dfn:CaraMeasurable}
For an outer measure $\mu$ on $\cX$ a set $A\subseteq\cX$ is called (\emph{Carath\'eodory}) \emph{$\mu$-measurable}\index{mu-measurable@$\mu$-measurable} if for every $E\subseteq\cX$ we have
$\mu(E) = \mu(E\cap A) + \mu(E\setminus A)$.
\end{dfn}

\begin{rem}\label{rem:CaraMeasurable}
Since by \Cref{dfn:outerMeasure} (iii) we always have
\[\mu(E) = \mu((E\cap A)\cup (E\setminus A)) \leq \mu(E\cap A) + \mu(E\setminus A)\]
it is sufficient for $\mu$-measurability to test
\begin{equation}\label{eq:measurableTest}
\mu(E) \geq \mu(E\cap A) + \mu(E\setminus A).
\end{equation}
\end{rem}

An outer measure is in fact a measure on all its measurable sets.

\begin{thm}\label{thm:outerMeasure}
Let $\mu$ be an outer measure on a set $\cX$ and $\cA_\mu\subseteq\cP(\cX)$ be the set of all $\mu$-measurable sets. Then $\cA_\mu$ is a $\sigma$-algebra of $\cX$ and $\mu$ is a measure on $(\cX,\cA_\mu)$.
\end{thm}
\begin{proof}
See e.g.\ \cite[Thm.\ 1.11.4 (iii)]{bogachevMeasureTheory}.
\end{proof}

Outer measures give another characterization of measurable functions.

\begin{lem}\label{lem:equivMeasurable}
Let $\mu$ be an outer measure on $\cX$ and $f:\cX\to [-\infty,\infty]$ be a function. Then $f$ is $\mu$-measurable if and only if
\[\mu(A) \geq \mu(\{x\in A \,|\, f(x)\leq a\}) + \mu(\{x\in A \,|\, f(x)\geq b\})\]
for all $A\subseteq\cX$ and $-\infty< a < b < \infty$.
\end{lem}
\begin{proof}
See e.g.\ \cite[§2.3.2(7), pp.\ 74--75]{federerGeomMeasTheo}.
\end{proof}

\begin{dfn}\label{dfn:regularmeasure}
An outer measure $\mu$ is called \emph{regular}\index{regular!measure}\index{measure!regular} if for each set $A\subseteq\cX$ there exists a $\mu$-measurable set $B\subseteq\cX$ with $A\subseteq B$ and $\mu(A) = \mu(B)$.
\end{dfn}

\begin{dfn}
Let $f,g:(\cX,\cA)\to\rset$ be two functions. Then we define $\inf(f,g)$ by
\[\inf(f,g)(x) := \inf(f(x),g(x))\]
for all $x\in\cX$ and similarly $\sup(f,g)$. Additionally, $f\leq g$ iff $f(x)\leq g(x)$ for all $x\in\cX$.
We have $f_+ := \sup(f,0)$, $f_- := f-f_+$, and $|f| = f_+ - f_-$.\label{eq:f+-||}
\end{dfn}

\begin{dfn}\label{dfn:lattice}
Let $\cX$ be a set. We call a set $\cF$ of functions $f:\cX\to\rset$ a \emph{lattice} (\emph{of functions})\index{lattice!of functions} if the following holds:
\begin{enumerate}[(i)]
\item $c\cdot f\in\cF$ for all $c\geq 0$ and $f\in\cF$,
\item $f+g\in\cF$ for all $f,g\in\cF$,
\item $\inf(f,g)\in\cF$ for all $f,g\in\cF$,
\item $\inf(f,c)\in\cF$ for all $c\geq 0$ and $f\in\cF$, and
\item $g-f\in\cF$ for all $f,g\in\cF$ with $f\leq g$.
\end{enumerate}
\end{dfn}

Some authors require that a lattice of functions is a vector space (\emph{lattice space}).\index{lattice!space}\index{space!lattice} But for proving \Cref{thm:daniell} it is only necessary that a lattice is a convex cone as in \Cref{dfn:lattice}.

\begin{exm}
Let $\cX$ be a locally compact Hausdorff space. Then $\cat_c(\cX,\rset)$ is a lattice of functions and even a lattice space.\exmsymbol
\end{exm}

Given a lattice $\cF$ we get another lattice $\cF_+$ by taking only the non-negative functions.

\begin{lem}[see e.g.\ {\cite[§2.5.1, p.\ 91]{federerGeomMeasTheo}}]
Let $\cF$ be a non-empty lattice on a set $\cX$ and define
\[\cF_+ := \{f\in\cF \,|\, f\geq 0\}.\]
Then
\begin{enumerate}[(i)]
\item $f_+, f_-,|f|\in\cF_+$ for all $f\in\cF$ and

\item $\cF_+$ is a non-empty lattice on $\cX$.
\end{enumerate}
\end{lem}
\begin{proof}
(i): Since $\inf(f,0)\in\cF$ and $\inf(f,0)\leq f$ we have $f_+ = \sup(f,0) = f - \inf(f,0)\in\cF_+$ for all $f\in\cF$.
Since $f \leq f_+ = \sup(f,0)\in\cF$ we have $f_- = f_+ - f\in\cF_+$ for all $f\in\cF$.
It follows that $|f| = f_+ + f_- \in\cF_+$ for all $f\in\cF$.

(ii): Since $\cF$ is non-empty there is a $f\in\cF$ and by (ii) we have $|f|\in\cF$ and hence $|f|\in\cF_+$.
$\cF_+$ is a lattice by directly checking the \Cref{dfn:lattice}.
\end{proof}

\advanced
\section{Daniell's Representation Theorem}

The question when a linear functional acting on (not necessarily measurable) functions is represented by a measure was already fully answered by P.\ J.\ Daniell\index{Daniell, P.\ J.} in 1918 \cite{daniell18}, see also \cite{daniell20}.

Nowadays only the \Cref{thm:rieszRepr} is given in standard texts for the moment problem.
We therefore take the time to present also Daniell's approach which is more general and has some interesting features the standard \Cref{thm:rieszRepr} does not have.

Note, that $h_n\nearrow g$\label{nearrow} denotes a sequence $(h_n)_{n\in\nset}$ with $h_1 \leq h_2 \leq ... \leq g$, i.e., point-wise non-decreasing, with $\lim_{n\to\infty} h_n(x) = g(x)$ for all $x\in\cX$. Equivalently, $h_n\searrow 0$ denotes a point-wise non-increasing sequence with $\lim_{n\to\infty} h_n(x) = 0$ for all $x\in\cX$.

\begin{danthm}[\cite{daniell18}, see also \cite{daniell20} or {\cite[Thm.\ 2.5.2]{federerGeomMeasTheo}}]\index{Daniell's Representation Theorem}\index{Theorem!Daniell's Representation}\index{representation!Theorem!Daniell}
\label{thm:daniell}
Let $\cF$ be a lattice of functions on a set $\cX$ and let $L:\cF\to\rset$ be such that
\begin{enumerate}[(i)]
\item $L(f+g) = L(f) + L(g)$ for all $f,g\in\cF$,
\item $L(c\cdot f) = c\cdot L(f)$ for all $c\geq 0$ and $f\in\cF$,
\item $L(f) \leq L(g)$ for all $f,g\in\cF$ with $f\leq g$,
\item $L(f_n)\nearrow L(g)$ as $n\to\infty$ for all $g\in\cF$ and $f_n\in\cF$ with $f_n\nearrow g$.
\end{enumerate}
Then there exists a measure $\mu$ on $(\cX,\cA)$ with
\begin{equation}\label{eq:algebraDaniell}
\cA := \sigma(\{f^{-1}((-\infty,a]) \,|\, a\in\rset,\ f\in\cF\})
\end{equation}
such that
\[L(f) = \int_\cX f(x)~\diff\mu(x)\]
for all $f\in\cF$.
\end{danthm}

We follow the proof in \cite[Thm.\ 2.5.2, pp.\ 92--93]{federerGeomMeasTheo}.

\begin{proof}
By assumption (iii) we have $L(f) \geq L(0\cdot f) = 0$ for all $f\in\cF_+$.

For any $A\subseteq\cX$ we say a sequence $(f_n)_{n\in\nset}$ \emph{suits}\index{suits} $A$ if and only if $f_n\in\cF_+$ and $f_n \leq f_{n+1}$ for all $n\in\nset$ and
\[\lim_{n\to\infty} f_n(x) \geq 1 \qquad\text{for all}\ x\in A.\]
Note, that we can even assume equality by replacing the $f_n$'s by $\tilde{f}_n := \inf(f_n,1)\in\cF_+$. Then we define
\begin{equation}\label{eq:measureDaniellDfn}
\mu(A) := \inf \left\{\lim_{n\to\infty} L(f_n) \;\middle|\; (f_n)_{n\in\nset}\ \text{suits}\ A\right\}
\end{equation}
which is $\infty$ if there is no sequence $(f_n)_{n\in\nset}$ that suits $A$.

We prove that $\mu$ is an outer measure, see \Cref{dfn:outerMeasure}.
By assumption (iii) $L(f_n)$ is a non-negative increasing sequence and therefore $\lim_{n\to\infty} L(f_n)$ exists and is in $[0,\infty]$. Hence, $\mu:\cP(\cX)\to [0,\infty]$.
For $A=\emptyset$ the zero sequence $f_n=0\in\cF_+$ is suited and therefore $\mu(\emptyset) = 0$.
Let $A\subseteq B\subseteq\cX$, then a suited sequence $(f_n)_{n\in\nset}$ of $B$ is also a suited sequence for $A$ and therefore $\mu(A)\leq\mu(B)$.
Let $A_i\subseteq\cX$, $i\in\nset$, and set $A := \bigcup_{i=1}^\infty A_i$.
Any suited sequence for $A$ is a suited sequences for all $A_i$.
Assume there is an $A_i$ which has no suited sequence, then $A$ has no suited sequence and $\mu(A) = \infty \leq \sum_{i=1}^\infty \mu(A_i) = \infty$.
So assume all $A_i$ have suited sequences, say $(f_{i,n})_{n\in\nset}$ suits $A_i$, $i\in\nset$.
Then $f_n := \sum_{i=1}^n f_{i,n}$ suits $A$ and
\[\mu(A) \leq \lim_{n\to\infty} L(f_n) = \lim_{n\to\infty} \sum_{i=1}^n L(f_{i,n}) \leq \sum_{i=1}^\infty \lim_{m\to\infty} L(f_{i,m}).\]
Taking the infimum on the right side for all $A_i$'s retains the inequality and gives
\[\mu\left(\bigcup_{i=1}^\infty A_i\right) = \mu(A) \leq \sum_{i=1}^\infty \mu(A_i).\]
Hence, all conditions in \Cref{dfn:outerMeasure} are fulfilled and $\mu$ is an outer measure.

Since $\mu$ is an outer measure on $\cX$ by \Cref{thm:outerMeasure} the set $\tilde{\cA}$ of all $\mu$-measurable sets of $\cX$ is a $\sigma$-algebra and $\mu$ is a measure on $(\cX,\tilde{\cA})$.

It remains to show that all $f\in\cF$ are $\mu$-measurable, $\mu$ is a measure on $(\cX,\cA)$ with $\cA = \sigma(\{f^{-1}((-\infty,a]) \,|\, a\in\rset,\ f\in\cF\})$, and $L(f) = \int_\cX f(x)~\diff\mu(x)$ for all $f\in\cF$.

Since $f = f_+ - f_-$ with $f_+,f_-\in\cF_+$ it is sufficient to show that every function in $\cF_+$ is $\mu$-measurable.
So let $f\in\cF_+$.
To show that $f$ is $\mu$-measurable it is sufficient to show that $A := f^{-1}((-\infty,a]) = \{x\in\cX \,|\, f(x)\leq a\}\in\cA$ for all $a\in\rset$, i.e., $A$ is $\mu$-measurable by \Cref{dfn:CaraMeasurable} resp.\ \Cref{rem:CaraMeasurable} if (\ref{eq:measurableTest}) holds for all $E\subseteq\cX$.
From $E\setminus A = E\cap (\cX\setminus A) = E\cap\{x\in\cX \,|\, f(x)> a\}$ we have to verify
\[\mu(E) \geq \mu\big(\{x\in E \,|\, f(x)\leq a\}\big)  + \mu\big(\{x\in E \,|\, f(x)> a\}\big)\]
and by \Cref{lem:equivMeasurable} this is equivalent to
\begin{equation}\label{eq:danProof1}
\mu(E) \geq \mu\big(\underbrace{\{x\in E \,|\, f(x)\leq a\}}_{=:E_a}\big) + \mu\big(\underbrace{\{x\in E \,|\, f(x)\geq b\}}_{=: E_b}\big)
\end{equation}
for all $a<b$. For $a<0$ or $\mu(E)=\infty$ (\ref{eq:danProof1}) is trivial, so assume $a\geq 0$ and $\mu(E)<\infty$.

Let $(g_n)_{n\in\nset}$ be a sequence that suits $E$ and set
\[h := (b-a)^{-1}\cdot [\inf(f,b)-\inf(f,a)]\in\cF_+ \qquad\text{and}\qquad k_n := \inf(g_n,h)\in\cF_+.\]
Then we have $0\leq k_{n+1} - k_n \leq g_{n+1} - g_n$,
\begin{align*}
h(x)&=1  \qquad\text{for all}\ x\in\cX\ \text{with}\ f(x)\geq b,
\intertext{and}
h(x)&=0 \qquad\text{for all}\ x\in\cX\ \text{with}\ f(x)\leq a.
\end{align*}
It follows that $(k_n)_{n\in\nset}$ suits $E_b$ and $(g_n-k_n)_{n\in\nset}$ suits $E_a$. Therefore,
\[\lim_{n\to\infty} L(g_n) = \lim_{n\to\infty} [L(g_n-k_n) + L(k_n)] \geq \mu(E_a) + \mu(E_b)\]
and taking the infimum on the left side retains the inequality and proves (\ref{eq:danProof1}). Hence, all $f\in\cF_+$ and therefore all $f\in\cF$ are $\mu$-measurable.

Let us show that $\mu$ remains a measure on $(\cX,\cA)$.
Since all $f\in\cF$ are $\mu$- and $\cA$-measurable we have
\[f^{-1}((-\infty,a])\in\tilde{\cA}\]
for all $a\in\rset$ and $f\in\cF$.
Therefore,
\[\cA = \sigma(\{f^{-1}((-\infty,a]) \,|\, a\in\rset,\ f\in\cF\}) \subseteq\tilde{\cA}\]
is a $\sigma$-algebra and we can restrict $\mu$ resp.\ $\tilde{\cA}$ to $\cA$.
$\mu$ is a measure on $(\cX,\cA)$.

We show that $L(f) = \int_\cX f(x)~\diff\mu(x)$ holds for all $f\in\cF_+$.
Let $f\in\cF_+$ and set
\[f_t := \inf(f,t)\]
for $t\geq 0$.
If $\varepsilon>0$ and $k\in\nset$ then
\begin{align*}
0 \leq f_{k\varepsilon}(x) - f_{(k-1)\varepsilon}(x) &\leq \varepsilon \quad\text{for all}\ x\in\cX,\\
f_{k\varepsilon}(x) - f_{(k-1)\varepsilon}(x) &= \varepsilon \quad\text{for all}\ x\in\cX\ \text{with}\ f(x)\geq k\varepsilon,
\intertext{and}
f_{k\varepsilon}(x) - f_{(k-1)\varepsilon}(x) &= 0 \quad\text{for all}\ x\in\cX\ \text{with}\ f(x)\leq (k-1)\varepsilon.
\end{align*}
The constant sequence $(\varepsilon^{-1}\cdot(f_{k\varepsilon} - f_{(k-1)\varepsilon}))_{n\in\nset}$ suits $\{x\in\cX \,|\, f(x)\geq k\varepsilon\}$ and consequently
\begin{alignat*}{2}
L(f_{k\varepsilon} - f_{(k-1)\varepsilon}) &\geq \varepsilon\cdot\mu(\{x\in\cX \,|\, f(x)\geq k\varepsilon\})\\
&\geq \int_\cX f_{(k+1)\varepsilon}(x) - f_{k\varepsilon}(x)~\diff\mu(x)\\
&\geq \varepsilon\cdot \mu(\{x\in\cX \,|\, f(x)\geq (k+1)\varepsilon\})
&&\geq L(f_{(k+2)\varepsilon} - f_{(k+1)\varepsilon}).
\intertext{Summing with respect to $k$ from $1$ to $n$ we find}
L(f_{n\varepsilon}) &\geq \;\;\int_\cX f_{(n+1)\varepsilon}(x) - f_\varepsilon(x)~\diff\mu(x) &&\geq L(f_{(n+2)\varepsilon} - f_{2\varepsilon})
\intertext{and since $f_{n\varepsilon}\nearrow f$ as $n\to\infty$ we get from assumption (iv) for $n\to\infty$}
L(f) &\geq \qquad\int_\cX f(x) - f_\varepsilon(x)~\diff\mu(x) &&\geq L(f - f_{2\varepsilon})
\intertext{which gives again from assumption (iv) for $\varepsilon\searrow 0$}
L(f) & \geq \qquad\quad\;\int_\cX f(x)~\diff\mu(x) &&\geq L(f).
\end{alignat*}
Hence, $L(f) = \int_\cX f(x)~\diff\mu(x)$ for all $f\in\cF_+$.

Finally, for all $f\in\cF$ we have $f = f_+ - f_-$ with $f_+,f_-\in\cF_+$ which implies
\[\int_\cX f(x)~\diff\mu(x) = \int_\cX f_+(x)~\diff\mu(x) - \int_\cX f_-(x)~\diff\mu(x) = L(f_+) - L(f_-) = L(f)\]
where the last equality follows from $f_+ = f + f_-$ and assumption (i).
\end{proof}

The most impressive part is that the functional $L:\cF\to\rset$ lives only on a lattice $\cF$ of functions $f:\cX\to\rset$ where $\cX$ is a set without any structure. \Cref{thm:daniell} provides a representing measure $\mu$ by (\ref{eq:measureDaniellDfn}) including the $\sigma$-algebra $\cA$ of the measurable space $(\cX,\cA)$ by (\ref{eq:algebraDaniell}).

\begin{rem}\label{rem:DaniellCondition}
In \Cref{thm:daniell} the assumption (iv) is equivalent to
\begin{enumerate}[\itshape (i')]\setcounter{enumi}{3}
\item \textit{$L(h_n)\searrow 0$ as $n\to\infty$ for all $h_n\in\cF$ with $h_n\searrow 0$ as $n\to\infty$}
\end{enumerate}
since $f_n\nearrow g$ implies $f_n\leq g$ and $0\leq h_n=g-f_n\in\cF$:
\[L(g) = L(g-f_n+f_n) = L(g-f_n) + \underbrace{L(f_n)}_{\nearrow L(g)} = \underbrace{L(h_n)}_{\searrow 0}\ +\ L(f_n).\tag*{$\circ$}\]
\end{rem}

The representing measure $\mu$ in \Cref{thm:daniell} is not unique.
But the representing measure $\mu$ constructed in (\ref{eq:measureDaniellDfn}) has further properties, see e.g.\ \cite[§2.5.3]{federerGeomMeasTheo}.

\Cref{thm:daniell} also has a signed version.

\begin{dansigthm}[\cite{daniell18}, see also {\cite[Thm.\ 2.5.5]{federerGeomMeasTheo}}]\index{Daniell's Signed Representation Theorem}\index{Theorem!Daniell's Representation!Signed}\index{representation!Theorem!Daniell, signed}
\label{thm:daniellSigned}
Let $\cF$ be a lattice of functions on some set $\cX$ and let $L:\cF\to\rset$ be such that for all $f,g,h_1, h_2, h_3,{\dots}\in\cF$ we have
\begin{enumerate}[\;(a)]
\item $L(f+g) = L(f) + L(g)$,
\item $L(c\cdot f) = c\cdot L(f)$ for all $c\geq 0$,
\item $\sup L\big(\{ k\in\cF \,|\, 0 \leq k \leq f\}\big) < \infty$,
\item $h_n \nearrow g$ as $n\to\infty$ implies $L(h_n)\to L(g)$ as $n\to\infty$.
\end{enumerate}
Let $L_+$ and $L_-$ be the functionals on $\cF_+$ defined by
\[L_+(f) := \sup L\big(\{k\in\cF \,|\, 0\leq k\leq f\}\big)\]
and
\[L_-(f) := -\inf L\big(\{k\in\cF \,|\, 0\leq k\leq f\}\big)\]
for all $f\in\cF_+$.
Then there exist $\cF_+$ regular measures $\mu_+$ and $\mu_-$ on $\cX$ such that
\begin{enumerate}[(i)]
\item $L_+(f) = \int_{\cX} f(x)~\diff\mu_+(x)$ for all $f\in\cF_+$,
\item $L_-(f) = \int_{\cX} f(x)~\diff\mu_-(x)$ for all $f\in\cF_+$, and
\item $L(f) = L_+(f) - L_-(f)$ for all $f\in\cF$.
\end{enumerate}
\end{dansigthm}

The proof is taken from \cite[pp.\ 96--97]{federerGeomMeasTheo} and uses \Cref{thm:daniell}.

\begin{proof}
Let $f_+\in\cF_+$. Then $f\geq g\in\cF_+$ implies $f\geq f-g\in\cF_+$ and
\[L(g) - L_-(f) \leq L(g) + L(f-g) \leq L(g) + L_+(f).\]
Hence,
\[L_+(f) - L_-(f) \leq L(f) \leq -L_-(f) + L_+(f)\]
so that
\[L(f) = L_+(f) - L_-(f).\]

Now let $f,g\in\cF_+$.
If $f+g\geq h\in\cF_+$ then
\[f\geq k := \inf(f,h)\in\cF_+ \quad\text{and}\quad g\geq h-k\in\cF_+\]
and hence
\[L_+(f) + L_+(g) \geq L(k) + L(h-k) = L(h).\]
Therefore, $L_+(f) + L_+(g)\geq L_+(f+g)$.
Since the opposite inequality is clear, we have that $L_+$ is additive on $\cF_+$.
Additionally, $L_+$ is positively homogeneous and monotone.

We now show that $L_+$ preserves increasing convergence.
Suppose $h_n\nearrow g$ as $n\nearrow\infty$ with $g,h_n\in\cF_+$.
If $g\geq k\in\cF_+$ then $f_n := \inf(h_n,k)\nearrow k$ as $n\nearrow\infty$, i.e.,
\[L(k) = \lim_{n\to\infty} L(f_n) \leq \lim_{n\to\infty} L_+(h_n).\]
Hence, $L_+(h_n)\nearrow L_+(g)$ as $n\nearrow\infty$.
By \Cref{thm:daniell} we have that there is a $\cF_+$ regular measure $\mu_+$ on $\cX$ such that $L_+(f) = \int f(x)~\diff\mu_+(x)$ for all $f\in\cF_+$.

Similarly, we have $L_-(f) = \int f(x)~\diff\mu_-(x)$ for some measure $\mu_-$ on $\cX$.
\end{proof}

\section{Riesz' Representation Theorem}

The \Cref{thm:rieszRepr} was developed in several stages.
A first version for continuous functions on the unit interval $[0,1]$ is due to F.\ Riesz\index{Riesz, F.} \cite{riesz09}.
It was extended by Markov\index{Markov, A.\ A.} to some non-compact spaces \cite{markov38} and then by Kakutani\index{Kakutani, S.} to locally compact Hausdorff spaces \cite{kakutani41}.
It is therefore sometimes also called the \emph{Riesz--Markov--Kakutani Representation Theorem}.\index{Riesz--Markov--Kakutani Representation Theorem}\index{Theorem!Riesz--Markov--Kakutani Representation}\index{representation!Theorem!Riesz--Markov--Kakutani}

However, we will see now that the general version already follows from the \Cref{thm:daniellSigned} and \Cref{thm:daniell} from 1918 \cite{daniell18} combined with \Cref{lem:ury} from 1925 \cite{urysohn25}, see also \cite[Sect.\ 2.5]{federerGeomMeasTheo}.
\Cref{lem:ury} is used to ensure that $\cat_c(\cX,\rset)$ is large enough.

At first let us give the signed version.

\begin{risigthm}[see e.g.\ {\cite[Thm.\ 2.5.13]{federerGeomMeasTheo}}]\index{Riesz' Representation Theorem!Signed}\index{Theorem!Riesz' Representation!Signed}\index{representation!Theorem!Riesz, signed}
\label{thm:rieszSigned}
Let $\cX$ be a locally compact Hausdorff space.
If $L:\cat_c(\cX,\rset)\to\rset$ is a linear functional such that
\begin{equation}\label{eq:rieszBounded}
\sup L(\{g\in\cat_c(\cX,\rset) \,|\, 0\leq g\leq f\}) < \infty
\end{equation}
for all $f\in\cat_c(\cX,\rset)_+$ then there exist $\cat_c(\cX,\rset)$ regular measures $\mu_+$ and $\mu_-$ such that
\[L(f) = \int_\cX f(x)~\diff\mu_+(x) - \int_\cX f(x)~\diff\mu_-(x)\]
for all $f\in\cat_c(\cX,\rset)$.
\end{risigthm}

The following proof is taken from \cite[Thm.\ 2.5.13, pp.\ 106--107]{federerGeomMeasTheo}.

\begin{proof}
It is sufficient to verify condition (d) in the \Cref{thm:daniellSigned}.

Let $g,h_1,h_2,{\dots}\in\cat_c(\cX,\rset)_+$ be such that $h_n\nearrow g$ as $n\to\infty$.
By \Cref{lem:ury} there exists a $f\in\cat_c(\cX,\rset)_+$ such that $f(x) = 1$ for all $x\in\supp g$.
Then
\[c := \sup \left\{ |L(k)| \,\middle|\, k\in\cat_c(\cX,\rset)\ \text{and}\ 0\leq k\leq f\right\} < \infty.\]

For each $\varepsilon>0$ the intersection of all compact sets
\[S_n := \{x\in\cX \,|\, g(x) \geq h_n(x) + \varepsilon\}\]
is empty.
Since $S_{n+1}\subset S_n$ for all $n\in\nset$ it follows that $S_n = \emptyset$ when $n$ is sufficiently large.
But $S_n = \emptyset$ implies $0 \leq g - h_n \leq \varepsilon f$ and $|L(g-h_n)|\leq \varepsilon c$ which proves condition (d).
\end{proof}

\begin{cor}[see e.g.\ {\cite[§2.5.14]{federerGeomMeasTheo}}]
If in the \Cref{thm:rieszSigned} we additionally have that the topology of $\cX$ has a countable base then $\mu_+$ and $\mu_-$ are Radon measures.
\end{cor}

Since positivity of $L$ on $\cat_c(\cX,\rset)_+$ implies (\ref{eq:rieszBounded}) by
\[0\leq g\leq f \;\;\Rightarrow\;\; 0\leq f-g \;\;\Rightarrow\;\; 0\leq L(f-g) \;\;\Rightarrow\;\; 0 \leq L(g) \leq L(f)<\infty\]
we have as an immediate consequence of the \Cref{thm:rieszSigned} the non-negative version.

\begin{rithm}\index{Riesz' Representation Theorem}\index{Theorem!Riesz' Representation}\index{representation!Theorem!Riesz}
\label{thm:rieszRepr}
Let $\cX$ be a locally compact Hausdorff space and $L:\cat_c(\cX,\rset)\to\rset$ be a non-negative linear functional on $\cat_c(\cX,\rset)_+$.
Then there exists a measure $\mu$ on $\cX$ such that
\[L(f) = \int_\cX f(x)~\diff\mu(x)\]
for all $f\in\cat_c(\cX,\rset)$.

If additionally $\cX$ as a topological space has a countable base then $\mu$ can be chosen to be a Radon measure.
\end{rithm}

From a topological point of view measures can also be introduced abstractly as linear functionals over certain spaces, see e.g.\ \cite[p.\ 216]{treves67}.
The Riesz representation theorem is then used to show the equivalence of the measure theoretic approach and the topological approach.

\advanced
\section{Riesz Decomposition}

The results in this section about the Riesz decomposition will be used only in \Cref{thm:coneExtension2} (ii) about adapted cones and extensions of linear functionals on these.
\Cref{thm:coneExtension2} is not used for the T-systems and can be omitted on first reading.

In \Cref{dfn:lattice} we introduced lattices.
Lattice spaces fulfill the following.

\begin{rideclem}[see e.g.\ {\cite[Lem.\ 10.5]{choquet69}}] \label{lem:rieszdecomp}
Let $\cF$ be a lattice space and $x,y_1,y_2\geq 0$ with $x \leq y_1 + y_2$.
Then there exist $x_1,x_2\geq 0$ such that
\[x = x_1 + x_2,\quad x_1\leq y_1,\quad \text{and}\quad x_2\leq y_2\]
hold.
\end{rideclem}

While the previous results holds for lattice spaces, also other spaces have this property.

\begin{dfn}\label{dfn:rieszDecompProp}
Let $F$ be an ordered vector space.
We say $F$ has the \emph{Riesz decomposition property}\index{Riesz!decomposition property}\index{decomposition!Riesz property} if
\begin{equation}\label{eq:rieszDecompProp}
x,y_1,y_2\in F_+: x\leq y_1+y_2 \quad\Rightarrow\quad \exists x_1,x_2\in F_+: x = x_1 + x_2,\ x_1\leq y_1,\ x_2\leq y_2.
\end{equation}
\end{dfn}

We have the following corollary.

\begin{cor}[see e.g.\ {\cite[Cor.\ 10.6]{choquet69}}]\label{cor:rieszdecomp}
Let $F$ be an ordered vector space with the Riesz decomposition property, let $x_1,\dots,x_n\in F_+$, and let $y_1,\dots,y_m\in F_+$ with
\[\sum_{i=1}^n x_i = \sum_{j=1}^m y_j.\]
Then for all $i=1,\dots,n$ and $j=1,\dots,m$ there exist $z_{i,j}\in F_+$ such that
\[x_i = \sum_{j=1}^m z_{i,j}\qquad \text{and}\qquad y_j = \sum_{i=1}^n z_{i,j}.\]
\end{cor}

\part{Introduction to Moments}

\motto{Extremes in nature equal ends produce;\\
In man they join to some mysterious use.\\ \medskip
\ \hspace{1cm} \normalfont{Alexander Pope: Essay on Man, Epistle II}\index{Pope, A.}} 

\chapter{Moments and Moment Functionals}
\label{ch:measures}

In this chapter we deal with the basics of moments and moment functionals.
More on moments and moment functionals can be found e.g.\ in \cite{schmudMomentBook,lauren09,marshallPosPoly} and the classical literature \cite{shohat43,ahiezer62,kreinMarkovMomentProblem}.

\section{Moments and Moment Functionals}

\begin{dfn}
Let $(\cX,\fA,\mu)$ be a measure space and let $f:\cX\to\rset$ be a $\mu$-integrable function.
The real number
\[\int_\cX f(x)~\diff\mu(x)\]
is called the \emph{$f$-moment of $\mu$}.\index{moment!$f$-moment of $\mu$}
\end{dfn}

The name \emph{moment} comes from the most famous example of moments: $\cX = \rset^3$ and $f(x,y,z) = f_\alpha(x,y,z) = x^{\alpha_1}\cdot y^{\alpha_2}\cdot z^{\alpha_3}$.
Then
\[\int_{\rset^3} (x^2 + y^2)\cdot\rho(x,y,z)~\diff x~\diff y~\diff z\]
is the $z$-rotational moment of a body with mass distribution $\rho$ in $\rset^3$.

In the modern theory of moments the investigation is about moment \emph{functionals}.

\begin{dfn}\label{dfn:momentFunctional}
Let $(\cX,\fA)$ be a measurable space and let $\cV$ be a vector space of real-valued measurable functions on $(\cX,\fA)$.
A linear functional $L:\cV\to\rset$ is called a \emph{moment functional}\index{moment!functional}\index{functional!moment} if there exists a measure $\mu$ such that
\begin{equation}\label{eq:momFunctDfn}
L(f) = \int_\cX f(x)~\diff\mu(x)
\end{equation}
for all $f\in\cV$.
Any measure $\mu$ such that (\ref{eq:momFunctDfn}) holds is called a \emph{representing measure of $L$}.\index{measure!representing}\index{representing!measure}
We denote by $\cM(L)$ the set of all representing measures of $L$.
\end{dfn}

\begin{cor}\label{cor:convexML}
Let $(\cX,\fA)$ be a measurable space, $\cV$ be a space of measurable functions $f:\cX\to\rset$, and let $L:\cV\to\rset$ be a moment functional.
Then $\cM(L)$ is convex.
\end{cor}
\begin{proof}
See Problem \ref{prob:convexML}.
\end{proof}

While a moment functional comes from a measure, conversely a measure $\mu$ gives a moment functional on $\mu$-integrable functions.

\begin{dfn}\label{dfn:Lmu}
Let $(\cX,\fA)$ be a measurable space and let $\cV$ be a vector space of measurable functions on $(\cX,\fA)$.
Given a measure $\mu$ such that all $f\in\cV$ are $\mu$-integrable then\label{eq:Lmu}
\[L_\mu:\cV\to\rset,\quad f\mapsto L_\mu(f):= \int_\cX f(x)~\diff\mu(x)\]
is the \emph{moment functional generated by $\mu$}.\index{moment!functional!generated by $\mu$}
\end{dfn}

We did not give any restrictions to the possible representing measures $\mu$ of a moment functional $L$.
In practice and hence also in theory restrictions can and even must be made, e.g., $\supp\mu\subseteq K$ for some $K\in\fA$.

\begin{dfn}
Let $(\cX,\fA)$ be a measurable space, $K\in\fA$ be a measurable set, let $\cV$ be a vector space of measurable functions $f:\cX\to\rset$, and let $L:\cV\to\rset$ be a linear functional.
We call $L$ to be a \emph{$K$-moment functional}\index{functional!K-moment@$K$-moment}\index{K-moment functional@$K$-moment functional} if there exists a measure $\mu$ on $\cX$ such that
\[L(f) = \int_\cX f(x)~\diff\mu(x)\]
for all $f\in\cV$ and $\supp\mu\subseteq K$.
\end{dfn}

A linear functional $L:\cV\to\rset$ can also be described by the numbers $s_i := L(f_i)$ for a basis $\{f_i\}_{i\in I}$ of $\cV$.

\begin{dfn}\label{dfn:rieszFunctional}
Let $(\cX,\fA)$ be a measurable space, let $\cV$ be a space of measurable functions $f:\cX\to\rset$ with basis $\{f_i\}_{i\in I}$ for some index set $I$.
Given any real sequence $s = (s_i)_{i\in I}$ the linear functional $L_s:\cV\to\rset$ defined by
\[L_s(f_i) := s_i\]
for all $i\in I$ is called the \emph{Riesz functional of $s$}.\index{functional!Riesz}\index{Riesz functional}
The sequence $s$ is called a \emph{moment sequence}\index{moment!sequence}\index{sequence!moment} if $L_s:\cV\to\rset$ is a moment functional.
\end{dfn}

\begin{exm}
Let $n\in\nset$, $\cX = \rset^n$ with $\fA = \fB(\rset^n)$ the Borel $\sigma$-algebra, and let $\cV = \rset[x_1,\dots,x_n]$ be the ring of polynomials.
Then a real sequence $s = (s_\alpha)_{\alpha\in\nset_0^n}$ gives a linear functional $L_s:\rset[x_1,\dots,x_n]\to\rset$ by $L_s(x^\alpha) := s_\alpha$ for all $\alpha\in\nset_0^n$.
The matrix $\cH(s) = (s_{\alpha+\beta})_{\alpha,\beta\in\nset_0^n}$ is the \emph{Hankel matrix} of the sequence $s$ (resp.\ the linear functional $L_s$).\index{Hankel matrix!of a sequence}
\exmsymbol
\end{exm}

In practice and hence also in theory we have the special case that $\cV$ is finite dimensional.

\begin{dfn}
Let  $(\cX,\fA)$ be a measurable space, let $\cV$ be a vector space of measurable functions $f:\cX\to\rset$, and $L:\cV\to\rset$ be a moment functional. Then $L$ is called a \emph{truncated} moment functional\index{moment!functional!truncated}\index{truncated!moment functional} if $\cV$ is finite dimensional.
\end{dfn}

\section{Determinacy and Indeterminacy}

We introduced the set of all representing measures $\cM(L)$ of a moment functional in \Cref{dfn:momentFunctional}.
We have the special and important case when $\cM(L)$ is a singleton, i.e., the moment functional $L$ has a unique representing measure.

\begin{dfn}
Let $(\cX,\fA)$ be a measurable space, $\cV$ a real vector space of measurable functions $f:\cX\to\rset$, and let $L:\cV\to\rset$ be a moment functional.
If $\cM(L)$ is a singleton, i.e., $L$ has a unique representing measure, then $L$ is called \emph{determinate}.\index{determinate!moment functional}\index{moment!functional!determinate}
Otherwise it is call \emph{indeterminate}.\index{indeterminate!moment!functional}\index{moment!functional!indeterminate}
\end{dfn}

\begin{cor}\label{cor:indeter}
Let $(\cX,\fA)$ be a measurable space, $\cV$ a real vector space of measurable functions $f:\cX\to\rset$, and let $L:\cV\to\rset$ be an indeterminate moment functional.
Then $L$ has infinitely many representing measures.
\end{cor}
\begin{proof}
See Problem \ref{prob:indeter}.
\end{proof}

The first example of an indeterminate moment functional/sequence was given by T.\ J.\ Stieltjes\index{Stieltjes, T.\ J.} \cite{stielt94}.
In {\cite[p.\ J.105, §55]{stielt94}} he states that all
\[s_k = \int_0^\infty x^k\cdot \left(1+c\cdot\sin(\sqrt[4]{x})\right)\cdot e^{-\sqrt[4]{x}}~\diff x \tag{$k\in\nset_0$}\]
are independent on $c\in [-1,1]$.

The first explicit example then follows in {\cite[pp.\ J.106--J.107, §56]{stielt94}}.

\begin{exm}[see {\cite[pp.\ J.106--J.107, §56]{stielt94}}]\index{indeterminacy!Stieltjes example}\index{Stieltjes!example!indeterminacy}
Let $c\in [-1,1]$ and
\[f(x) = \frac{1}{\sqrt{\pi}}\cdot \exp\left( -\frac{1}{2} (\ln x)^2\right)\]
for all $x\in [0,\infty)$. Then the measure $\mu_c\in\cM(\rset)$ defined by
\[\diff\mu_c(x) := [1 + c\cdot\sin(2\pi \ln x)]\cdot f(x)~\diff x\]
has the moments
\[s_k = \int_0^\infty x^k~\diff\mu_c(x) = e^{\frac{1}{4} (k+1)^2}\]
for all $k\in\nset_0$, i.e., independent on $c\in [-1,1]$.\exmsymbol
\end{exm}

Criteria for determinacy and indeterminacy are well-studied, see e.g.\ \cite{schmudMomentBook} and reference therein.

\section*{Problems}
\addcontentsline{toc}{section}{Problems}

\begin{prob}\label{prob:determinacy}
Let $n\in\nset$ and let $L:\rset[x_1,\dots,x_n]\to\rset$ be a moment functional with a representing measure $\mu$ such that $\supp\mu\subseteq K$ for some compact $K\subset\rset^n$.
Show that $L$ is determinate, i.e., show that $\mu$ is the unique representing measure of $L$.

\emph{Hint}: Use the \Cref{thm:stoneWeierstrass}.
\end{prob}

\begin{prob}\label{prob:convexML}
Prove \Cref{cor:convexML}.
\end{prob}

\begin{prob}\label{prob:indeter}
Prove \Cref{cor:indeter}.
\end{prob}

\motto{Progress imposes not only new possibilities for the future\\
but new restrictions.\\ \medskip
\ \hspace{1cm} \normalfont{Norbert Wiener {\cite[p.\ 46]{wiener88}}}\index{Wiener, N.}}

\chapter{Choquet's Theory and Adapted Spaces}
\label{ch:choquet}

This chapter is devoted to the theory of Choquet and the concept of adapted spaces.
The results can also be found in e.g.\ \cite{choquet69,phelpsLectChoquetTheorem,schmudMomentBook}.

\section{Extensions of Linear Functionals preserving Positivity}

We remind the reader that a convex cone $C\subseteq F$ in a real vector space $F$ induces an order $\leq$ on $F$, i.e., for any $x,y\in F$ we have $x\leq y$ iff $y-x\in C$, see \Cref{sec:orderCone}.

\begin{lem}[see e.g.\ {\cite[Prop.\ 34.1]{choquet69}}]\label{lem:linearCone}
Let $F$ be a real vector space, $E\subseteq F$ be a linear subspace, and let $C\subseteq F$ be a convex cone which induces the order $\leq$ on $F$.
Then the following are equivalent:
\begin{enumerate}[(i)]
\item $F+C$ is a vector space.

\item $F+C = F-C$.

\item Any $x\in (F+C)\cup (F-C)$ is majorized\index{majorized} by some $z\in F$, i.e., $x\leq z$, and is minorized\index{minorized} by some $y\in F$, i.e., $y\leq x$.
\end{enumerate}
\end{lem}
\begin{proof}
See Problem \ref{prob:linearCone}.
\end{proof}

\begin{dfn}
Let $F$ be a real vector space and $C\subseteq F$ be a convex cone. A linear functional $L:F\to\rset$ is called \emph{$C$-positive}\index{cone!positive!linear functional}\index{functional!linear!cone positive} if $L(f)\geq 0$ holds for all $f\in C$. $L$ is called \emph{strictly} $C$-positive if $L(f)>0$ holds for all $f\in C\setminus\{0\}$.
\end{dfn}

\begin{thm}[see e.g.\ {\cite[Thm.\ 34.2]{choquet69}}]\label{thm:coneExtension}
Let $F$ be a real vector space, $E\subseteq F$ be a linear subspace, and $C\subseteq F$ be a convex cone with $F = E+C$.
Then any $(C\cap E)$-positive linear functional $L:E\to\rset$ can be extended to a $C$-positive linear functional $\tilde{L}:F\to\rset$.

The extension $\tilde{L}$ is unique if and only if for all $x\in E$ we have
\begin{equation}\label{eq:uniquenessCriteriaConeExtension}
\sup\{L(y) \,|\, y\leq x,\ y\in F\} = \inf\{ L(y) \,|\, x\leq y,\ y\in F\}.
\end{equation}
\end{thm}

The proof is taken from \cite[vol.\ 2, p.\ 270--271]{choquet69}. It adapts the idea behind the proof of the \Cref{thm:hahnBanach}.

\begin{proof}
Let $\cH := \{(H,h)\}_{H\ \text{subspace}:\ E\subseteq H\subseteq F}$ where $h:H\to\rset$ extends $L$.
The family $\cH$ has a natural order by the extension property, i.e., we have $(H_1,h_1)\leq (H_2,h_2)$ if $h_2:H_2\to\rset$ is an extension of $h_1:H_1\to\rset$.
By Zorn's Lemma $\cH$ has a maximal element $(G,g)$.
We have to show $G = F$.
For that it is sufficient that $E$ is a hyperplane in $F$ and $L$ can be extended to $F$.

Let $x_0\in F\setminus E$.
By \Cref{lem:linearCone} (iii) there exist $y,z\in E$ with $y\leq x_0\leq z$.
We define
\[\alpha := \sup\{L(y) \,|\, y\leq x_0\ \text{and}\ y\in E\}\]
and
\[\beta := \inf\{L(z) \,|\, x_0\leq z\ \text{and}\ z\in E\}.\]
Since $L$ is $C$-positive we have $\alpha\leq\beta$ and any extension $\tilde{L}$ must satisfy $\alpha\leq \tilde{L}(x_0)\leq \beta$.

We show that for each $\gamma\in [\alpha,\beta]$ there exists an extension $\tilde{L}$ with $\tilde{L}(x_0) = \gamma$.
Each point $u\in F$ can be uniquely written as $u = y - \lambda x_0$ with $y\in E$ and $\lambda\in\rset$.
Define $\tilde{L}(u) := L(y) - \lambda\gamma$.
Then $\tilde{L}$ is a linear extension of $L$ and we have to show that $\tilde{L}$ is $C$-positive.
Let $u\in C$, i.e., $y\geq \lambda x_0$.
If $\lambda>0$ then $x_0 \leq y/\lambda$ and $\beta\leq L(y/\lambda)$.
Hence, $L(y) \geq \lambda\beta\geq \lambda\gamma$ and so $\tilde{L}(u)\geq 0$.
If on the other hand $\lambda<0$ then $x_0\geq y/\lambda$ and $\alpha\geq L(y/\lambda)$ which implies $L(y) \geq \lambda\alpha \geq \lambda\gamma$ and $\tilde{L}(u)\geq 0$.
At last, if $\lambda=0$ then $\tilde{L}(u) = \tilde{L}(y) \geq 0$.
In summary, we proved that $\tilde{L}$ is $C$-positive.

For the uniqueness it is sufficient to note that if (\ref{eq:uniquenessCriteriaConeExtension}) holds for all $x\in E$ then $\tilde{L}$ is uniquely determined since every extension $\tilde{L}$ arises from this construction.
If on the other hand $\alpha<\beta$, i.e., (\ref{eq:uniquenessCriteriaConeExtension}) does not hold, then some extension $(H,h)\in \cH$ is not unique for $H$ and consequently $\tilde{L}$ is not a unique extension of $L$.
\end{proof}

From the previous proof we see that by redoing the proof of the Hahn--Banach Theorem the uniqueness criteria (\ref{eq:uniquenessCriteriaConeExtension}) can be incorporated.
A second proof using the Hahn--Banach Theorem is much shorter but loses the uniqueness condition (\ref{eq:uniquenessCriteriaConeExtension}), see e.g.\ \cite[Prop.\ 1.7]{schmudMomentBook}.

A third proof of \Cref{thm:coneExtension} follows from the following lemma.

\begin{lem}[see e.g.\ {\cite[Prop.\ 34.3]{choquet69}}]\label{lem:linearMapInBetween}
Let $E$ be a real vector space, let $g:E\to\rset$ be superlinear and let $h:E\to\rset$ be sublinear.
Then there exists a linear map $f:E\to\rset$ such that $g \leq f \leq h$.
\end{lem}
\begin{proof}
Equip $E$ with the topology of all semi-norms.
Then $p(x) := \sup\{h(x),h(-x)\}$ is a semi-norm and $h\leq p$.
Since $p$ is continuous and $h$ is convex we have that $h$ is continuous.
Thus $g$ and $h$ can be separated by a closed hyperplane.
\end{proof}

\Cref{lem:linearMapInBetween} not only gives a third proof of \Cref{thm:coneExtension} but also has a generalization which is known as \emph{Strassen's Theorem} \cite{strassen65}.\index{Theorem!Strassen}\index{Strassen's Theorem}

Strassen's Theorem states that if $(\cY,\mu)$ is a measure space, $\{h_y:E\to\rset\}_{y\in\cY}$ is a family of sublinear maps, and let $l:E\to\rset$ be a linear map with
\[l \leq \int_\cY h_y~\diff\mu(y).\]
Then there exists a family $\{l_y:E\to\rset\}_{y\in\cY}$ of linear maps $l_y$ with $l_y\leq h_y$ such that
\[l = \int_\cY l_y~\diff\mu(y).\]
For more on Strassen's Theorem see e.g.\ \cite{edwards78,skala93,lindvall99} and references therein.

\section{Adapted Spaces of Continuous Functions}

We now come to the adapted spaces.
To define them we need the following.

\begin{dfn}\label{dfn:dominate}
Let $\cX$ be a locally compact Hausdorff space and $f,g\in\cat(\cX,\rset)_+$.
We say $f$ \emph{dominates}\index{dominate} $g$ if for any $\varepsilon>0$ there is an $h_\varepsilon\in\cat_c(\cX,\rset)$ such that $g \leq \varepsilon f + h_\varepsilon$.
\end{dfn}

Equivalent expressions are the following.

\begin{lem}[see e.g.\ {\cite[Lem.\ 1.4]{schmudMomentBook}}]\label{lem:adapted}
Let $\cX$ be a locally compact Hausdorff space and let $f,g\in\cat(\cX,\rset)_+$.
Then the following are equivalent:
\begin{enumerate}[(i)]
\item $f$ dominates $g$.

\item For every $\varepsilon>0$ there exists a compact set $K_\varepsilon\subseteq\cX$ such that $g(x) \leq \varepsilon\cdot f(x)$ holds for all $x\in\cX\setminus K_\varepsilon$.

\item For every $\varepsilon>0$ there exists an $\eta_\varepsilon\in\cat_c(\cX,\rset)$ with $0 \leq \eta_\varepsilon\leq 1$ such that $g \leq \varepsilon\cdot f + \eta_\varepsilon\cdot g$.
\end{enumerate}
\end{lem}
\begin{proof}
See Problem \ref{prob:adaptedLem}.
\end{proof}

The main definition of this chapter is the following.

\begin{dfn}\label{dfn:adaptedSpace}
Let $\cX$ be a locally compact Hausdorff space and let $E\subseteq\cat(\cX,\rset)$ be a vector space. Then $E$ is called an \emph{adapted space}\index{space!adapted}\index{adapted!space} if the following conditions hold:
\begin{enumerate}[(i)]
\item $E = E_+ - E_+$,

\item for all $x\in\cX$ there is a $f\in E_+$ such that $f(x) > 0$, and

\item every $g\in E_+$ is dominated by some $f\in E_+$.
\end{enumerate}
\end{dfn}

The space $\cat_c(\cX,\rset)_+$ is of special interest because of the \Cref{thm:rieszRepr}.
The following result shows that any $g\in\cat_c(\cX,\rset)_+$ is dominated (and even bounded) by some $f\in E_+$ for any given adapted space $E\subseteq\cat(\cX,\rset)$.

\begin{lem}\label{lem:adaptedCompact}
Let $\cX$ be a locally compact Hausdorff space, $g\in\cat_c(\cX,\rset)_+$, and let $E\subseteq\cat(\cX,\rset)$ be an adapted space. Then there exists a $f\in E_+$ such that $f>g$.
\end{lem}
\begin{proof}
See Problem \ref{prob:adaptedCompact}.
\end{proof}

\section{Existence of Integral Representations}

One important reason adapted spaces have been introduced is to get the following representation theorem.
It is a general version of \Cref{thm:haviland} and will be used to solve most moment problems in an efficient way.

\begin{basrepthm}[see e.g.\ {\cite[Thm.\ 34.6]{choquet69}}]\label{thm:basicrepresentation}\index{Theorem!Basic Representation}\index{Basic Representation Theorem}\index{representation!Theorem!Basic}
Let $\cX$ be a locally compact Hausdorff space, $E\subseteq \cat(\cX,\rset)$ be an adapted subspace, and let $L:E\to\rset$ be a linear functional.
The following are equivalent:
\begin{enumerate}[(i)]
\item The functional $L$ is $E_+$-positive.

\item $L$ is a moment functional, i.e., there exists a (Radon) measure $\mu$ on $\cX$ such that
\begin{enumerate}[(a)]
\item all $f\in E$ are $\mu$-integrable and

\item $L(f) = \int_\cX f(x)~\diff\mu(x)$ holds for all $f\in E$.
\end{enumerate}
\end{enumerate}
\end{basrepthm}

The following proof is adapted from \cite[vol.\ 2, p.\ 276--277]{choquet69}.

\begin{proof}
The direction (ii) $\Rightarrow$ (i) is clear. It is therefore sufficient to prove (i) $\Rightarrow$ (ii).

Define
\begin{equation}\label{eq:extensionF}
F := \{f\in\cat(\cX,\rset) \,|\, |f|\leq g\ \text{for some}\ g\in E_+\}.
\end{equation}
Then $F_+$ is a convex cone.
We have $F = E + F_+$.
To see this let $f\in F$ and write $f = -g + (f+g)$ where $|f|\leq g$ for some $g\in E_+$, i.e., $f\in E + F_+$ and hence $F\subseteq E + F_+$.
The inclusion $E+F_+\subseteq F$ is clear and we therefore have $F = E + F_+$.

By \Cref{thm:coneExtension} we can extend $L$ to a $F_+$-positive linear functional $\tilde{L}:F\to\rset$.
By \Cref{lem:adaptedCompact} we have $\cat_c(\cX,\rset)\subseteq F$ and hence by the \Cref{thm:rieszRepr} there exists a representing Radon measure $\mu$ on $\cX$ of $\tilde{L}|_{\cat_c(\cX,\rset)}$.

We need to show that $\mu$ is also a representing measure of $L$.
Let $f\in E_+$.
Since $\mu$ is Radon we have
\begin{equation}\label{eq:inequalityBasic}
\int_\cX f(x)~\diff\mu(x) = \sup\left\{\int_\cX \varphi(x)~\diff\mu(x) \,\middle|\, \varphi
\in\cat_c(\cX,\rset),\ \varphi\leq f\right\} \leq \tilde{L}(f) = L(f)
\end{equation}
and hence $f$ is $\mu$-integrable.
Since $E = E_+ - E_+$ we have that all $f\in E$ are $\mu$-integrable.

Then
\begin{equation}\label{eq:linfuncK}
K(f) := \tilde{L}(f) - \int_\cX f(x)~\diff\mu(x)
\end{equation}
for all $f\in F$ defines a $F_+$-positive linear functional on $F$ which vanishes on $\cat_c(\cX,\rset)$.
For every $g\in E_+$ there is an $f\in E_+$ dominating $g$.
Let $\varepsilon>0$ and $h_\varepsilon\in\cat_c(\cX,\rset)$ be such that $g \leq \varepsilon f + h_\varepsilon$.
Then $0\leq K(g) \leq \varepsilon\cdot K(f) \xrightarrow{\varepsilon\to 0} 0$, i.e., $K=0$ on $E_+$ and hence on $E$ which proves that $\mu$ is a representing measure of $L$.
\end{proof}

We actually proved that $L$ can be extended to $\tilde{L}$ on $F$ in (\ref{eq:extensionF}) and that $\mu$ is a representing measure for $\tilde{L}$.
This is included in (ii-b).

For the uniqueness of the representing measure $\mu$ of $L$ we have the following.

\begin{cor}[see e.g.\ {\cite[Cor.\ 34.7]{choquet69}}]
Let $\cX$ be a locally compact Hausdorff space, $E\subseteq\cat(\cX,\rset)$ be an adapted space, and let $L:E\to\rset$ be a $E_+$-positive linear functional.
Then the following are equivalent:
\begin{enumerate}[(i)]
\item The representing measure $\mu$ of $L$ from the \Cref{thm:basicrepresentation} is unique.
\item For any $f\in\cat_c(\cX,\rset)$ and $\varepsilon>0$ there are $f_1,f_2\in E$ with $f_1 \leq f\leq f_2$ and $0\leq T(f_2 - f_1) \leq \varepsilon$.
\end{enumerate}
\end{cor}
\begin{proof}
Reformulating (i) we get that the measure $\mu$ must be uniquely defined by the extension of $L:E\to\rset$ to $\tilde{L}:E + \cat_c(\cX,\rset)\to\rset$.
By \Cref{thm:coneExtension} eq.\ (\ref{eq:uniquenessCriteriaConeExtension}) this is equivalent to
\[\sup\{L(\varphi) \,|\, \varphi\leq f,\ \varphi\in E\} = \inf\{ L(\varphi) \,|\, f\leq\varphi,\ \varphi\in E\}.\]
But this is equivalent to our condition (ii), i.e., we showed (i) $\Leftrightarrow$ (ii).
\end{proof}

\advanced
\section{Adapted Cones}

A generalization of adapted spaces is to go to adapted cones, i.e., dropping the vector space property.
This is presented in \cite{choquet69} but not included in \cite{schmudMomentBook} and we want to show it to the reader for the sake (or at least a glimpse) of completeness.
It is not used in the T-systems and can be omitted on first reading.

\begin{dfn}\label{dfn:dominate2}
Let $F$ be an ordered vector space and let $E\subseteq F$ be a convex cone.
For $x,y\in F$ with $x,y\geq 0$ we say that $y$ \emph{dominates} $x$ (\emph{relative to $E$})\index{dominate!cone}\index{cone!adapted} if for any $\varepsilon>0$ there exists a $z_\varepsilon\in E$ such that $x\leq \varepsilon y + z_\varepsilon$.

For two convex cones $C,E\subseteq F_+$ we say that $(C,E)$ are \emph{adapted} (\emph{cones})\index{adapted!cones}\index{cone!adapted} if every $x\in C$ is dominated by some $x'\in C$ (relative to $E$) and for each $g\in E$ there is an $f\in C$ so that $g\leq f$.
\end{dfn}

The previous definition is a generalization of \Cref{dfn:dominate}.
The convex cone $C$ has the role of $\cat_c(\cX,\rset)_+$, $F$ has the role of $\cat(\cX,\rset)$, and $E$ is the adapted space.

\begin{lem}[see e.g.\ {\cite[Prop.\ 35.3]{choquet69}}]\label{lem:CEextention}
Let $F$ be an ordered vector space, let $(C,E)$ be adapted cones, and let $L:E\to\rset$ be a positive linear functional.
Then
\[L|_E = 0 \quad\Rightarrow\quad L|_C = 0.\]
\end{lem}
\begin{proof}
Let $x\in C$.
Since $(C,E)$ are adapted cones there exists a $x'\in C$ such that for any $\varepsilon>0$ there is a $z_\varepsilon\in E$ with
\[0 \leq x \leq \varepsilon x' + z_\varepsilon.\]
Since $L\geq 0$ on $E$ we have
\[0\leq L(x) \leq \varepsilon L(x') \xrightarrow{\varepsilon\to 0} 0\]
which proves $L|_C = 0$.
\end{proof}

\begin{thm}[see e.g.\ {\cite[Thm.\ 35.4]{choquet69}}]\label{thm:coneExtension2}\index{Theorem!Conic Extension}\index{Conic Extension Theorem}
Let $F$ be an ordered vector space.
\begin{enumerate}[(i)]
\item Let $C\subseteq F_+$ be a convex cone and let $L:C\to [0,\infty)$ be a positive linear functional.
Define
\[\hat{C} := \{g\in F_+ \,|\, g\leq f\ \text{for some}\ x\in C\}.\]
Then $L$ has an extension to a positive linear functional $\hat{L}:\hat{C}\to [0,\infty)$.

\item Let $(C,E)$ be adapted cones such that $E\subseteq\hat{C}$ and $\hat{C}$ has the Riesz decomposition property (\ref{eq:rieszDecompProp}).
Then for each $f\in\hat{C}$ we have
\[\hat{L}(f) = \sup\left\{ \hat{L}(g) \,\middle|\, g\in E\ \text{with}\ g\leq f\right\}.\]
\end{enumerate}
\end{thm}
\begin{proof}
(i): First, extend $L$ by linearity to the vector space $C-C$.
Let $F_0 := \hat{C}-\hat{C}$.
Then $F_0 = C-C+\hat{C} = -C+\hat{C}$.
By \Cref{thm:coneExtension} $L$ extends to a $\hat{C}$-positive linear functional on $F_0$.

(ii): Define $L_0: \hat{C} \to\rset$ by
\[L_0(f) := \sup\{L(g) \,|\, g\in E\ \text{with}\ g\leq f\}.\]
Hence, $0\leq L_0(f) \leq \hat{L}(f)$ for all $f\in\hat{C}$.
Clearly, $L_0(\lambda f) = \lambda L_0(f)$ holds for all $\lambda\geq 0$ and $f\in\hat{C}$. Additionally,
\begin{align*}
L_0(f_1+f_2) &= \sup\left\{\hat{L}(g) \,\middle|\, g\in E,\ g\leq f_1 + f_2\right\}
\intertext{which is by the Riesz decomposition property (\ref{eq:rieszDecompProp})}
&= \sup\left\{\hat{L}(g_1 + g_2) \,\middle|\, g_1,g_2\in E,\ g_1\leq f_1,\ g_2\leq f_2\right\}\\
&= L_0(f_1) + L_0(f_2)
\end{align*}
for all $f_1,f_2\in \hat{C}$ and hence by linearity extension $L_0$ is linear on $F_0$.

We now show at last that $L - L_0 = 0$ on $\hat{C}$.
Since $(C,E)$ are adapted cones we have that $(\hat{C},E)$ are adapted cones.
We have $L(f) - L_0(f) = 0$ for all $f\in E$ and hence by \Cref{lem:CEextention} we have $L = L_0$ on $\hat{C}$ which proves (ii).
\end{proof}

\Cref{thm:coneExtension2} (ii) is the analogue of extending a Radon measure on $\cat_c(\cX,\rset)$ to continuous integrable functions.

\begin{exm}[see e.g.\ {\cite[Exm.\ 35.5]{choquet69}}]\label{exm:adaptedCones}
Let $(\cX,\fA,\mu)$ be a measure space.
Let $C = (\cL^1(\cX,\mu))_+$ and $E = \cL^\infty(\cX,\mu)\cap (\cL^1(\cX,\mu))_+$.
Then $(C,E)$ are adapted cones.
Hence, every positive linear functional is uniquely determined by its values on $\cL^\infty\cap \cL^1$.\exmsymbol
\end{exm}

\advanced
\section{Continuity of Positive Linear Functionals}

At the end of this chapter we want to point out some continuity results.
But we will leave out the proofs since these results will not be used for our T-system treatment.

\begin{thm}[see e.g.\ {\cite[Thm.\ 36.1]{choquet69}}]
Let $E$ be an ordered Hausdorff topological vector space such that $E = E_+ - E_+$ and let either
\begin{enumerate}[(i)]
\item $\inter E_+ \neq \emptyset$
\end{enumerate}
or
\begin{enumerate}[(i)]\setcounter{enumi}{1}
\item $E$ is complete, metrizable, and $E_+$ is closed.
\end{enumerate}
Then any positive linear functional $L:E\to\rset$ is continuous.
\end{thm}

The previous results holds for general convex pointed cones in $E$.

\begin{cor}[see e.g.\ {\cite[Cor.\ 36.1]{choquet69}}]
Let $E$ be a Hausdorff topological vector space and $P\subset E$ be a convex pointed cone.
The following hold:
\begin{enumerate}[(i)]
\item If $\inter P \neq \emptyset$ then any linear $P$-positive functional $T:E\to\rset$ is continuous.

\item If $E$ is complete, metrizable, $P$ is closed, and $E = P - P$, then any linear $P$-positive functional $T:E\to\rset$ is continuous.
\end{enumerate}
\end{cor}

Further conditions for continuity can be found e.g.\ in \cite[Ch.\ 36]{choquet69} or \cite{schaefer99}.
\cite[Ch.\ 36]{choquet69} also gives results for positive linear functionals on C$^*$-algebras, the Schwartz space $\cS(\rset^n,\rset)$, Lipschitz functions, and on general vector lattices.

Another direction is more operator theoretic and deals with linear functionals over algebras.
An \emph{algebra}\index{algebra} $\cA$ is a (complex) vector space with a multiplication $\cdot\,:\cA\times\cA\to\cA$, $(a,b)\mapsto ab$ such that
\begin{enumerate}[(i)]
\item $a(bc) = (ab)c$,

\item $(a+b)c = ac+bc$, and

\item $\alpha (ab) = (\alpha a)b = a(\alpha b)$
\end{enumerate}
for all $a,b,c\in\cA$ and $\alpha\in\cset$.
An element $1\in\cA$ is called \emph{unit element}\index{unit element} if $1a = a = a1$ for all $a\in\cA$.
A \emph{$*$-algebra}\index{algebra!$*$-} is an algebra with an involution $^*:\cA\to\cA$, $a\mapsto a^*$ that also satisfies $(ab)^* = b^* a^*$ and $(\alpha a)^* = \overline{\alpha} a^*$.
A \emph{linear functional}\index{functional!linear} $L:\cA\to\cset$ is called \emph{non-negative}\index{functional!linear!non-negative} if $L(a^* a) \geq 0$ for all $a\in\cA$.
A \emph{topological $*$-algebra} is a $*$-algebra with a topology $\cT$ such that the multiplication and involution are continuous.
A \emph{Fr\'echet topological $*$-algebra}\index{algebra!$*$-!Fr\'echet topological} is a topological algebra which is a Fr\'echet space\index{Fr\'echet space}, i.e., a complete metrizable locally convex space.
An example is $\cset[x_1,\dots,x_n]$.

We have the following.

\begin{thm}[\cite{xia59} and {\cite{ng72}}; or e.g.\ {\cite[Thm.\ 3.6.1]{schmud90UnboundedOperatorAlgebras}}]
Let $\cA$ be a Fr\'echet topological $*$-algebra with unit element and let $L:\cA\to\cset$ be a linear functional.
If $L$ is non-negative then it is continuous.
\end{thm}

A more general statement is \cite[Thm.\ 1]{ng72}.
For more see e.g.\ \cite[Ch.\ 3.6]{schmud90UnboundedOperatorAlgebras} and references therein.

\section*{Problems}
\addcontentsline{toc}{section}{Problems}

\begin{prob}\label{prob:linearCone}
Prove \Cref{lem:linearCone}.
\end{prob}



%





\begin{prob}\label{prob:adaptedLem}
Prove \Cref{lem:adapted}.
\end{prob}

\begin{prob}\label{prob:compactAdapted}
Let $\cX$ be a compact topological Hausdorff space and let $E\subseteq\cat(\cX,\rset)$ be a subspace such that there exists an $e\in E$ such that $e(x)>0$ for all $x\in\cX$. Show that $E$ is an adapted space.
\end{prob}

\begin{prob}\label{prob:adaptedPolynomials}
Let $n\in\nset$ and $\cX\subseteq\rset^n$ be closed. Show that $\rset[x_1,\dots,x_n]$ on $\cX$ is an adapted space.
\end{prob}

\begin{prob}\label{prob:adaptedPolynomials2}
Let $n\in\nset$, $\cX\subseteq\rset^n$ be closed, and let $E\subseteq\rset[x_1,\dots,x_n]$ be an adapted space. Show that if $E$ is finite dimensional then $\cX$ is compact.
\end{prob}

\begin{prob}\label{prob:adaptedCompact}
Prove \Cref{lem:adaptedCompact}.
\end{prob}

\motto{Those who cannot remember the past are condemned to repeat it.\\ \medskip
\ \hspace{1cm} \normalfont{George Santayana \cite{santa05}}\index{Santayana, G.}}

\chapter{The Classical Moment Problems}
\label{ch:classical}

In this chapter we give several classical solutions of moment problems: the Stieltjes, Hamburger, and Hausdorff moment problem.
Additionally, we collect other classical results such as Haviland's Theorem, Richter's Theorem on the existence of finitely atomic representing measures for truncated moment functionals, and Boas' Theorem on the existence of signed representing measures for any linear functional \mbox{$L:\rset[x_1,\dots,x_n]\to\rset$.}

\section{Classical Results}
\label{sec:classicalResults}

In this section we give a chronological list of the early moment problems which have been solved.
We will explicitly discuss the historical (first) proofs of these results.
Our modern proofs here will be based on the Choquet's theory from \Cref{ch:choquet} and for a modern operator theoretic approach see e.g.\ \cite{schmudMomentBook}.

The first moment problem was solved by T.\ J.\ Stieltjes\index{Stieltjes, T.\ J.} \cite{stielt94}.
He was the first who fully stated the moment problem, solved the first one, and by doing that also introduced the integral theory named after him: the Stieltjes integral.

\begin{stielthm}\label{thm:stieltjesMP}\index{Theorem!Stieltjes}\index{Stieltjes' Theorem}\index{moment!problem!Stieltjes|see{Theorem, Stieltjes}}
Let $s = (s_i)_{i\in\nset_0}$ be a real sequence. The following are equivalent:
\begin{enumerate}[(i)]
\item $s$ is a $[0,\infty)$-moment sequence (Stieltjes moment sequence).

\item $L_s(p)\geq 0$ for all $p\in\pos([0,\infty))$.

\item $L_s(p^2)\geq 0$ and $L_{Xs}(p^2) = L_s(x\cdot p^2)\geq 0$ for all $p\in\rset[x]$.

\item $s$ and $Xs = (s_{i+1})_{i\in\nset_0}$ are positive semidefinite.

\item $\cH(s)\succeq 0$ and $\cH(Xs)\succeq 0$ for all $d\in\nset_0$.
\end{enumerate}
\end{stielthm}
\begin{proof}
See Problem \ref{prob:stieltjes}.
\end{proof}

In the original proof of \Cref{thm:stieltjesMP} Stieltjes \cite{stielt94} does not use non-negative polynomials.
Instead he uses continued fractions and introduces new sequences which we (nowadays) denote by $s$ and $Xs$.

Stieltjes only proves (i) $\Leftrightarrow$ (iv).
The implication (i) $\Leftrightarrow$ (ii) is \Cref{thm:haviland}, (ii) $\Leftrightarrow$ (iii) is the description of $\pos([0,\infty))$, and (iv) $\Leftrightarrow$ (v) is a reformulation of $s$ and $Xs$ being positive semi-definite.

The next moment problem was solved by H.\ L.\ Hamburger\index{Hamburger, H.\ L.} \cite[Satz X and Existenztheorem (§8, p.\ 289)]{hamburger20}.

\begin{hamthm}\label{thm:hamburgerMP}\index{Hamburger's Theorem}\index{Theorem!Hamburger}\index{moment!problem!Hamburger|see{Theorem, Hamburger}}
Let $s = (s_i)_{i\in\nset_0}$ be a real sequence. The following are equivalent:
\begin{enumerate}[(i)]
\item $s$ is a $\rset$-moment sequence (Hamburger moment sequence or short moment sequence).

\item $L_s(p)\geq 0$ for all $p\in\pos(\rset)$.

\item $L_s(p^2)\geq 0$ for all $p\in\rset[x]$.

\item $s$ is positive semidefinite.

\item $\cH(s)\succeq 0$.
\end{enumerate}
\end{hamthm}
\begin{proof}
See Problem \ref{prob:hamburger}.
\end{proof}

Hamburger proved similar to Stieltjes the equivalence (i) $\Leftrightarrow$ (iv) via continued fractions.
In \cite[Satz XIII]{hamburger20} Hamburger solves the full moment problem by approximation with truncated moment problems.
This was later in a slightly more general framework proved in \cite{stochel01}, see also \Cref{sec:stochel}.
Hamburger needed to assume that the sequence of measures $\mu_k$ (which he called ``Belegungen'' and denoted by $\diff\Phi^{(k)}(u)$) to converge to some measure $\mu$ (condition 2 of \cite[Satz XIII]{hamburger20}).
Hamburgers additional condition 2 is nowadays replaced by the vague convergence and the fact that the solution set of representing measures is vaguely compact \cite[Thm.\ 1.19]{schmudMomentBook}, i.e., it assures the existence of a $\mu$ as required by Hamburger in the additional condition 2.

Shortly after Hamburger the moment problem on $[0,1]$ was solved by F.\ Hausdorff\index{Hausdorff, F.} \cite[Satz II and III]{hausdo21}.

\begin{hausthm}\label{thm:hausdorffMP}\index{Hausdorff's Theorem}\index{Theorem!Hausdorff}\index{moment!problem!Hausdorff|see{Theorem, Hausdorff}}
Let $s = (s_i)_{i\in\nset_0}$ be a real sequence. The following are equivalent:
\begin{enumerate}[(i)]
\item $s$ is a $[0,1]$-moment sequence (Hausdorff moment sequence).

\item $L_s(p)\geq 0$ for all $p\in\pos([0,1])$.

\item $L_s(p^2)\geq 0$, $L_{Xs}(p^2)\geq 0$, and $L_{(1-X)s}(p^2)\geq 0$ for all $p\in\rset[x]$.

\item $s$, $Xs$, and $(1-X)s$ are positive semidefinite.

\item $\cH(s)\succeq 0$, $\cH(Xs)\succeq 0$, and $\cH((1-X)s)\succeq 0$.
\end{enumerate}
\end{hausthm}
\begin{proof}
See Problem \ref{prob:hausdorff}.
\end{proof}

Hausdorff proved the equivalence (i) $\Leftrightarrow$ (iii) via so called C-sequences.
In \cite{toeplitz11} Toeplitz treats general linear averaging methods.
In \cite{hausdo21} Hausdorff uses these.
Let the infinite dimensional matrix $\lambda = (\lambda_{i,j})_{i,j\in\nset_0}$ be row-finite, i.e., for every row $i$ only finitely many $\lambda_{i,j}$ are non-zero.
Then the averaging method
\[A_i = \sum_{j\in\nset_0} \lambda_{i,j} a_j\]
shall be consistent: If $a_j \to \alpha$ converges then $A_i\to\alpha$ converges to the same limit.
Toeplitz proved a necessary and sufficient condition on $\lambda$ for this property.
Hausdorff uses only part of this property. He calls a matrix $(\lambda_{i,j})_{i,j\in\nset_0}$ with the property that a convergent sequence $(a_j)_{j\in\nset_0}$ is mapped to a convergent sequence $(A_j)_{j\in\nset_0}$ (the limit does not need to be preserved) a C-matrix (convergence preserving matrix).
Hausdorff gives the characterization of C-matrices \cite[p.\ 75, conditions (A) -- (C)]{hausdo21}.
Additionally, if $\lambda$ is a C-matrix and a diagonal matrix with diagonal entries $\lambda_{i,i} = s_i$ then $s = (s_i)_{i\in\nset_0}$ is called a C-sequence.
The equivalence (i) $\Leftrightarrow$ (iii) is then shown by Hausdorff in the result that a sequence is a $[0,1]$-moment sequence if and only if it is a C-sequence \cite[p.\ 102]{hausdo21}.

A much simpler approach to solve the $K$-moment problem for any closed $K\subseteq\rset^n$, $n\in\nset$, was presented by E.\ K.\ Haviland\index{Haviland, E.\ K.} in \cite[Theorem]{havila36}, see also \cite[Theorem]{havila35} for the earlier case $K=\rset^n$.
He no longer used continued fractions but employed the \Cref{thm:rieszRepr}, i.e., representing a linear functional by integration, and connected the existence of a representing measure to the non-negativity of the linear functional on
\begin{equation}\label{eq:posKdfn}
\pos(K) := \{f\in\rset[x_1,\dots,x_n] \,|\, f\geq 0\ \text{on}\ K\}.
\end{equation}

\begin{havithm}\label{thm:haviland}\index{Haviland's Theorem}\index{Theorem!Haviland}
Let $n\in\nset$, $K\subseteq\rset^n$ be closed, and $s = (s_\alpha)_{\alpha\in\nset_0^n}$ be a real sequence. The following are equivalent:
\begin{enumerate}[(i)]
\item $s$ is a $K$-moment sequence.

\item $L_s(p)\geq 0$ for all $p\in\pos(K)$.
\end{enumerate}
\end{havithm}
\begin{proof}
See Problem \ref{prob:haviland}.
\end{proof}

As noted before, in \cite[Theorem]{havila35} Haviland proves ``only'' the case $K=\rset^n$ with the extension method by M.\ Riesz.\index{Riesz, M.}
In \cite[Theorem]{havila36} this is extended to any closed $K\subseteq\rset^n$.
The idea to do so is attributed by Haviland to A.\ Wintner\index{Wintner, A.} \cite[p.\ 164]{havila36}:
\begin{quote}
A.\ Wintner has subsequently suggested that it should be possible to extend this result [\cite[Theorem]{havila35}] by requiring that the distribution function [measure] solving the problem have a spectrum [support] contained in a preassigned set, a result which would show the well-known criteria for the various standard special momentum problems (Stieltjes, Herglotz [trigonometric], Hamburger, Hausdorff in one or more dimensions) to be put particular cases of the general $n$-dimensional momentum problem mentioned above.
The purpose of this note [\cite{havila36}] is to carry out this extension.
\end{quote}

In \cite{havila36} after the general Theorem \ref{thm:haviland} Haviland then goes through all the classical results (Theorems \ref{thm:stieltjesMP} to \ref{thm:hausdorffMP}, and the Herglotz\index{Herglotz!moment problem}\index{moment!problem!Herglotz} (trigonometric) moment problem\index{trigonometric!moment problem}\index{moment!problem!trigonometric} on the unit circle $\tset$ which we did not included here) and shows how all these results (i.e., conditions on the sequences) are recovered from the at this point known representations of non-negative polynomials.

For the \emph{Hamburger moment problem}\index{Hamburger moment problem|see{Theorem, Hamburger}} (\Cref{thm:hamburgerMP}) Haviland uses
\begin{equation}\label{eq:posR}
\pos(\rset) = \left\{ f^2 + g^2 \,\middle|\, f,g\in\rset[x]\right\}
\end{equation}
which was already known to D.\ Hilbert\index{Hilbert, D.} \cite{hilbert88}.
We prove a stronger version of (\ref{eq:posR}) in \Cref{cor:nonnegRRx}.
For the \emph{Stieltjes moment problem}\index{Stieltjes moment problem|see{Theorem, Stieltjes}} (\Cref{thm:stieltjesMP}) he uses
\begin{equation}\label{eq:pos0infty1}
\pos([0,\infty)) = \left\{ f_1^2 + f_2^2 + x\cdot (g_1^2 + g_2^2) \,\middle|\, f_1,f_2,g_1,g_2\in\rset[x]\right\}
\end{equation}
with the reference to G.\ P\'olya\index{P\'olya, G.} and G.\ Szeg{\"o}\index{Szeg\"o, G.} (previous editions of \cite{polya64,polya70}).
In \cite[p.\ 82, ex.\ 45]{polya64} the representation (\ref{eq:pos0infty1}) is still included while it was already known before, see \cite[p.\ 6, footnote]{shohat43}, that
\begin{equation}\label{eq:pos0infty2}
\pos([0,\infty)) = \left\{ f^2 + x\cdot g^2 \,\middle|\, f,g\in\rset[x]\right\}
\end{equation}
is sufficient.
Also in \cite[Prop.\ 3.2]{schmudMomentBook} the representation (\ref{eq:pos0infty1}) is used, not the simpler representation (\ref{eq:pos0infty2}).
We prove a stronger version of (\ref{eq:pos0infty2}) in \Cref{cor:nonneg0inftyRx}.

For the $[-1,1]$-moment problem Haviland uses
\begin{equation}\label{eq:pos-11}
\pos([-1,1]) = \left\{ f^2 + (1-x^2)\cdot g^2 \,\middle|\, f,g\in\rset[x]\right\}.
\end{equation}
For the \emph{Hausdorff moment problem}\index{Hausdorff moment problem|see{Theorem, Hausdorff}} (\Cref{thm:hausdorffMP}) he uses that any strictly positive polynomial on $[0,1]$ is a linear combination of
\begin{equation}\label{eq:lukacs}
x^m\cdot (1-x)^{p}
\end{equation}
with $m,p\in\nset_0$, $p\geq m$, and with non-negative coefficients.

Haviland gives this with the references to a previous edition of \cite{polya70}.
This result is actually due to S.\ N.\ Bernstein\index{Bernstein, S.\ N.} \cite{bernstein12,bernstein15}.

\begin{bernthm}[\cite{bernstein12} for (i), \cite{bernstein15} for (ii); or see e.g.\ {\cite[p.\ 30]{achieser56}} or {\cite[Prop.\ 3.4]{schmudMomentBook}}]\label{thm:bernstein}\index{Theorem!Bernstein}\index{Bernstein!Theorem}
Let $f\in\cat([0,1],\rset)$ and let
\begin{equation}\label{eq:bernPoly}
B_{f,d}(x) := \sum_{k=0}^d \binom{d}{k}\cdot x^k\cdot (1-x)^{d-k}\cdot f\left(\frac{k}{d}\right)
\end{equation}
be the \emph{Bernstein polynomials}\index{polynomial!Bernstein}\index{Bernstein!polynomial} of $f$ with $d\in\nset$.
Then the following hold:
\begin{enumerate}[(i)]
\item The polynomials $B_{f,d}$ converge uniformly on $[0,1]$ to $f$, i.e.,
\[\|f-B_{f,d}\|_\infty \xrightarrow{d\to\infty}\infty.\]

\item If additionally $f\in\rset[x]$ with $f>0$ on $[0,1]$ then there exist a constant $D=D(f)\in\nset$ and constants $c_{k,l}\geq 0$ for all $k,l=0,\dots,D$ such that
\[f(x) = \sum_{k,l=0}^D c_{k,l}\cdot x^k\cdot (1-x)^l.\]

\item The statements (i) and (ii) also hold on $[0,1]^n$ for any $n\in\nset$. Especially every $f\in\rset[x_1,\dots,x_n]$ with $f>0$ on $[0,1]^n$ is of the form
\[f(x) = \sum_{\alpha_1,\dots,\beta_n=0}^D c_{\alpha_1,\dots,\beta_n}\cdot x_1^{\alpha_1}\cdots x_n^{\alpha_n}\cdot (1-x_1)^{\beta_1}\cdots (1-x_n)^{\beta_n}\]
for some $D\in\nset$ and $c_{\alpha_1,\dots,\beta_n}\geq 0$.
\end{enumerate}
\end{bernthm}

The multidimensional statement (iii) follows from the classical one-dimensional cases (i) and (ii).
For this and more on Bernstein polynomials see e.g.\ \cite{lorentz86}.

\Cref{thm:bernstein} only holds for $f>0$.
Allowing zeros at the interval end points is possible and gives the following ``if and only if''-statement.

\begin{cor}\label{cor:bernstein}
Let $f\in\rset[x]\setminus\{0\}$. The following are equivalent:
\begin{enumerate}[(i)]
\item $f>0$ on $(0,1)$.

\item $\displaystyle f(x) = \sum_{i=0}^D c_{k,l}\cdot x^l\cdot (1-x)^k$ for some $D\in\nset$, $c_{k,l}\geq 0$ for all $k,l=0,\dots,D$, and $c_{k',l'}>0$ at least once.
\end{enumerate}
\end{cor}
\begin{proof}
See Problem \ref{prob:bernstein}.
\end{proof}

On $[-1,1]$ a strengthened version of \Cref{thm:bernstein} (ii) is attributed to F.\ Luk\'acs\index{Luk\'acs, F.} \cite{lukacs18} (\emph{Luk\'acs Theorem}).\index{Theorem!Luk\'acs|see{Luk\'acs--Markov}}\index{Luk\'acs Theorem|see{Luk\'acs--Markov Theorem}}
Note that Luk\'acs in \cite{lukacs18} reproves several results/formulas which already appeared in a work by M.\ R.\ Radau\index{Radau, M.\ R.} \cite{radau80}, as pointed out by L.\ Brickman\index{Brickman, L.} \cite[p.\ 196]{brickman59}.
Additionally, in \cite[p.\ 61, footnote 4]{kreinMarkovMomentProblem} M.\ G.\ Krein\index{Krein, M.\ G.} and A.\ A.\ Nudel'man\index{Nudel'man, A.\ A.} state that A.\ A.\ Markov\index{Markov, A.\ A.} proved a more precise version of Luk\'acs Theorem already in 1906 \cite{markov06},\footnote{We do not have access to \cite{markov06} and can therefore neither confirm nor decline this statement.} see also \cite{markov95}.
Krein and Nudel'man call it \emph{Markov's Theorem}.\index{Markov's Theorem|see{Luk\'acs--Markov Theorem}}\index{Theorem!Markov|see{Luk\'acs--Markov}}
It is the following.

\begin{lumathm}[\cite{markov06} or e.g.\ \cite{lukacs18}, {\cite[p.\ 61, Thm.\ 2.2]{kreinMarkovMomentProblem}}]\index{Luk\'acs--Markov Theorem}\index{Theorem!Luk\'acs--Markov}\label{thm:luma}
Let $-\infty < a < b < \infty$ and let $p\in\rset[x]$ be with $\deg p = n$ and $p\geq 0$ on $[a,b]$.
The following hold:
\begin{enumerate}[(i)]
\item If $\deg p = 2m$ for some $m\in\nset_0$ then $p$ is of the form
\[p(x) = f(x)^2 + (x-a)(b-x)\cdot g(x)^2\]
for some $f,g\in\rset[x]$ with $\deg f = m$ and $\deg g = m-1$.

\item If $\deg p = 2m+1$ for some $m\in\nset_0$ then $p$ is of the form
\[p(x) = (x-a)\cdot f(x)^2 + (b-x)\cdot g(x)^2\]
for some $f,g\in\rset[x]$ with $\deg f = \deg g = m$.
\end{enumerate}
\end{lumathm}

For case (i) note that the relation
\begin{equation}\label{eq:abProduct}
(x-a)(b-x) = \frac{1}{b-a}\left[ (x-a)^2(b-x) + (x-a)(b-x)^2 \right]
\end{equation}
implies
\begin{equation}\label{eq:posab}
\pos([a,b]) = \left\{f(x)^2 + (x-a)\cdot g(x)^2 + (b-x)\cdot h(x)^2 \,\middle|\, f,g,h\in\rset[x]\right\}.
\end{equation}
The special part about the \Cref{thm:luma} are the degree bounds on the polynomials $f$ and $g$.
Equation (\ref{eq:abProduct}) destroyes these degree bounds since we have to go one degree higher.

In the \Cref{thm:luma2} we will see how from \Cref{thm:karlinPosab} an even stronger version follows which describes the polynomials $a$ and $b$ more precisely and up to a certain point uniquely.
In \cite[p.\ 61 Thm.\ 2.2 and p.\ 373 Thm.\ 6.4]{kreinMarkovMomentProblem} the \Cref{thm:luma} is called \emph{Markov--Luk\'acs Theorem}\index{Markov--Luk\'acs Theorem|see{Luk\'acs--Markov Theorem}}\index{Theorem!Markov--Luk\'acs|see{Luk\'acs--Markov}} since Markov gave the more precise version much earlier than Luk\'acs.
In \cite{havila36} Haviland uses this result without any reference or attribution to either Luk\'acs or Markov.

For the two-dimensional Hausdorff moment problem Haviland uses with a reference to \cite{hildeb33} that any polynomial $f\in\rset[x,y]$ which is strictly positive on $[0,1]^2$ is a linear combination of $x^m\cdot y^n\cdot (1-x)^{p}\cdot (1-y)^{q}$, $n,m,q,p\in\nset_0$, with non-negative coefficients.
This is actually \Cref{thm:bernstein} (iii).

T.\ H.\ Hildebrandt\index{Hildebrandt, T.\ H.} and I.\ J.\ Schoenberg\index{Schoenberg, I.\ J.} \cite{hildeb33} already solved the moment problem on $[0,1]^2$ (and even on $[0,1]^n$ for all $n\in\nset$) getting the same result as Haviland.
The idea of using $\pos(K)$-descriptions to solve the moment problem was therefore already used by Hildebrandt and Schoenberg in 1933 \cite{hildeb33} before Haviland uses this in \cite{havila35} and generalized this in \cite{havila36} as suggested to him by Wintner.

With these broader historical remarks we see that of course more people are connected to Theorem \ref{thm:haviland}.
It might also be appropriate to call Theorem \ref{thm:haviland} the \emph{Haviland--Wintner}\index{Theorem!Haviland--Wintner}\index{Haviland--Wintner Theorem} or \emph{Haviland--Hildebrandt--Schoenberg--Wintner Theorem}.\index{Theorem!Haviland--Hildebrandt--Schoenberg--Wintner}\index{Haviland--Hildebrandt--Schoenberg--Wintner Theorem}
But as so often, the list of contributors is long (and maybe even longer) and hence the main contribution (the general proof) is rewarded by calling it just Haviland's Theorem.

The last classical moment problem which we want to mention on the long list was solved by K.\ I.\ {\v{S}}venco \cite{svenco39}.\index{Svenco, K.\ I.@{\v{S}}venco, K.\ I.}

\begin{svethm}\index{Svenco's Theorem@{\v{S}}venco's Theorem}\index{Theorem!Svenco@{\v{S}}venco}\index{moment!problem!Svenco@{\v{S}}venco|see{Theorem, {\v{S}}venco}}\label{thm:svenco}
Let $s=(s_i)_{i\in\nset_0}$ be a real sequence. The following are equivalent:
\begin{enumerate}[(i)]
\item $s$ is a $(-\infty,0]\cup [1,\infty)$-moment sequence.

\item $L_s(p)\geq 0$ for all $p\in\pos((-\infty,0]\cup[1,\infty))$.

\item $L_s(p^2)\geq 0$, $L_{(X^2-X)s}(p^2)\geq 0$ for all $p\in\rset[x]$.

\item $s$ and $(X^2-X)s$ are positive semi-definite.

\item $\cH(s)\succeq 0$ and $\cH((X^2-X)s)\succeq 0$.
\end{enumerate}
\end{svethm}

The general case of \Cref{thm:svenco} on
\begin{equation}\label{eq:zaun}
\rset\setminus \bigcup_{i=1}^n (a_i,b_i)
\end{equation}
for any $n\in\nset$ and $a_1 < b_1 < \dots < a_n < b_n$ was proved by V.\ A.\ Fil'\v{s}tinski\v{\i}\index{Fil'\v{s}tinski\v{\i}, V.\ A.} \cite{filsti64}.
All non-negative polynomials on (\ref{eq:zaun}) can be explicitly written down.
More precisely, all moment problems on closed and semi-algebraic sets $K\subseteq\rset$ follow nowadays easily from \Cref{thm:haviland} resp.\ the \Cref{thm:basicrepresentation} and some well established results from real algebraic geometry, e.g.\ \cite[Prop.\ 2.7.3]{marshallPosPoly}.

\Cref{thm:haviland} was important to give the solutions of the classical moment problem, i.e., mostly one-dimensional cases.
After that is was no longer used and only became important again when descriptions of strictly positive and non-negative polynomials on $K\subseteq\rset^n$ with $n\geq 2$ be came available.
This process was started with and real algebraic geometry was revived by \cite{schmud91}.

\section{Early Results with Gaps}
\label{sec:earlyGaps}

The early history of moment problems with gaps is very thin. We discuss only \cite{hausdo21a} and \cite{boas39}.

Hausdorff\index{Hausdorff, F.} just solved \Cref{thm:hausdorffMP} in \cite{hausdo21}\footnote{Submitted: February 11, 1920.} and in \cite{hausdo21a}\footnote{Submitted: September 8, 1920.} he treats
\[s_n = \int_0^1 x^{k_n}~\diff\mu(x)\]
with
\[k_0 = 0 < k_1 < k_2 < \dots < k_n < \dots \]
for a sequence of real numbers $k_i$, i.e., not necessarily in $\nset_0$.
See also \cite[p.\ 104]{shohat43}.
Since Hausdorff in \cite{hausdo21a} did not have access to \Cref{thm:haviland} \cite{havila36} or the description of all non-negative linear combinations of $1,x^{k_1}, \dots, x^{k_n},\dots$ the results in \cite{hausdo21a} need complicated formulations and are not very strong.
Only with the description of non-negative linear combinations by Karlin \cite{karlin63} an easy formulation of the result is possible.
We will therefore postpone the exact formulation to \Cref{thm:sparseTruncHausd} and \Cref{thm:generalSparseHausd} where we present easy proofs using also the theory of adapted spaces from \Cref{ch:choquet}, especially the \Cref{thm:basicrepresentation}.

In \cite{boas39} Boas\index{Boas, R.\ P.} investigates the Stieltjes moment problem ($K=[0,\infty)$) with gaps.
Similar to \cite{hausdo21a} the results are difficult to read and they are unfortunately incomplete since Boas (like Hausdorff) did not have access to the description of all non-negative or strictly positive polynomials with gaps (or more general exponents).
We will give the complete solution of the $[0,\infty)$-moment problem with gaps and more general exponents in \Cref{thm:sparseStieltjesMP}.

\section{Finitely Atomic Representing Measures: Richter's Theorem}

When working with a truncated moment sequence resp.\ functionals it is often useful in theory and applications to find a representing measure with finitely many atoms.
That this is always possible for truncated moment functionals was first proved in full generality by H.\ Richter \cite[Satz 4]{richte57}.\index{Richter, H.}

Its proof proceeds by induction via the dimension of the moment cone.
To do that we need to look at the boundary of the moment cone.
We need that when part of the boundary of the moment cone is cut out by a supporting hyperplane then this intersection is again a moment cone of strictly smaller dimension.
That is the content of the following lemma.

\begin{lem}\label{lem:boundaryCone}
Let $n\in\nset$, $(\cX,\fA)$ be a measurable space, $\cF=\{f_i\}_{i=1}^n$ be a family of measurable functions $f_i:\cX\to\rset$, $\cS_\cF$\label{dfn:SF} be the moment cone spanned by $\cF$, and let $H$ be a supporting hyperplane of $\cS_\cF$. Then $\cS_\cF\cap H$ is a moment cone of dimension $m = \dim (\cS_\cF\cap H) < n$ spanned by a family $\cG\subset\lin\cF$ on a measurable space $(\cY,\fA|_\cY)$ with $\cY\subseteq\cX$.
\end{lem}
\begin{proof}
See Problem \ref{prob:boundaryCone}.
\end{proof}

With the previous lemma we can now prove Richter's Theorem.

\begin{richthm}[{\cite[Satz 4]{richte57}}; or see e.g.\ {\cite[Thm.\ 1]{kemper68}}, {\cite[p.\ 198, Thm.\ 1]{floudas01}}]\label{thm:richter}\index{Richter's Theorem}\index{Theorem!Richter}
Let $n\in\nset$, let $(\cX,\fA)$ be a measurable space, and let $\{f_i\}_{i=1}^n$ be a family of real linearly independent measurable functions $f_i:\cX\to\rset$. Then for every measure $\mu$ on $\cX$ such that all $f_i$ are $\mu$-integrable, i.e.,
\[s_i := \int_\cX f_i(x)~\diff\mu(x) \quad \in \rset\]
for all $i=1,\dots,n$, there exist a $k\in\nset_0$ with $k\leq n$, points $x_1,\dots,x_k\in\cX$ pairwise different, and $c_1,\dots,c_k\in (0,\infty)$ such that
\[s_i = \sum_{j=1}^k c_j\cdot f_i(x_j) = \int_\cX f_i(x)~\diff\nu(x) \quad\text{with}\quad \nu = \sum_{j=1}^k c_j\cdot\delta_{x_j}\]
holds for all $i=1,\dots,n$.
\end{richthm}
\begin{proof}
We show that every truncated moment sequence $s=(s_1,\dots,s_n)$ has a finitely atomic representing measure with at most $n$ atoms in $\cX$.
We prove this statement by induction on $n$.

$n=1$: We have
\[s_1 = \int_\cX f_1(x)~\diff\mu(x).\]
If $s_1 = 0$ then take $\nu = 0$ which proves the statement. Let us assume $s_1\neq 0$. Since $\mu\geq 0$ on $\cX$ there exists a point $x_1\in\cX$ such that $\sign\, f_1(x_1) = \sign\, s_1$. Hence, we have $\frac{s_1}{f_1(x_1)} =: c_1 > 0$ and
\[s_1 = \frac{s_1}{f_1(x_1)}\cdot f_1(x_1) = \int_\cX f_1(x)~\diff (c_1\cdot\delta_{x_1})(x)\]
which proves the statement.

$n\geq 2$: Let $\cS_\cF\subseteq\rset^n$ be the moment cone generated from $\cF$.
We make the distinction of the two cases
\begin{enumerate}[\qquad (a)]
\item $s = (s_1,\dots,s_n)\in\inter\cS_\cF$ and
\item $s \in \partial\cS_\cF\cap\cS_\cF$.
\end{enumerate}
For (a) let $\cS := \cone \{(f_1(x),\dots,f_n(x))^T \,|\, x\in\cX\}$ be the cone generated by all point evaluations $(f_1(x),\dots,f_n(x))^T$.
By \Cref{thm:cara} every $s\in\cS$ is a moment sequences with a $k$-atomic representing measure with $k\leq n$.
Additionally, we have that $\inter\cS$ is non-empty since $\cS$ is full dimensional.

Assume $\inter\cS\neq\inter\cS_\cF$ then $\inter(\cS_\cF\setminus\cS)\neq\emptyset$. Let $s\in\inter(\cS_\cF\setminus\cS)$ with a representing measure $\mu$.
Then there exists a separating linear functional $l$, i.e., $l(s)<0$ and $l(t)>0$ for all $t\in\cS$.
Since $(f_1(x),\dots,f_n(x))^T\in\cS$ we have that $f(x) := l((f_1(x),\dots,f_n(x)) > 0$ for all $x\in\cX$ but
\[\int_\cX f(x)~\diff\mu(x) = l(s) < 0\]
with is a contradiction to $\mu\geq 0$.
Hence, $\inter\cS = \inter\cS_\cF$ and every $s\in\inter\cS_\cF$ has a $k$-atomic representing measure with $k\leq n$.

For (b) assume $s\in\partial\cS_\cF\cap\cS_\cF$.
Since $\cS_\cF$ is a convex cone there exists a supporting hyperplane $H$ of $\cS_\cF$ at $s$.
But then $\cS_\cF\cap H$ is by \Cref{lem:boundaryCone} a moment cone of dimension at most $n-1$ and here the theorem holds by induction.
\end{proof}

The previous proof is the original proof by Richter and only the mathematical language is updated.
The following historical overview about \Cref{thm:richter} first appeared in \cite{didioCone22}.

Replacing integration by finitely many point evaluations was already used and investigated by C.\ F.\ Gau{\ss} \cite{gauss15}.\index{Gauss@Gau\ss, C.\ F.}
The $k$-atomic representing measures from \Cref{thm:richter} are therefore also called (\emph{Gaussian}) \emph{cubature formulas}.\index{cubature formula!Gaussian}

The history of \Cref{thm:richter} is confusing and the literature is often misleading.
We therefore list in chronological order previous versions or versions which appeared almost at the same time.
The conditions of these versions (including Richter) are the following:
\begin{enumerate}[\bfseries\;\;(A)]
\item A.\ Wald 1939\footnote{Received: February 25, 1939. Published: September 1939.} \cite[Prop.\ 13]{wald39}:\index{Wald, A.}\index{Theorem!Wald} $\cX = \rset$ and $f_i(x) = |x-x_0|^{d_i}$ with $d_i\in\nset_0$, $0\leq d_1 < d_2 < \dots < d_n$, and $x_0\in\cX$.\medskip

\item P.\ C.\ Rosenbloom 1952 \cite[Cor.\ 38e]{rosenb52}:\index{Rosenbloom, P.\ C.}\index{Theorem!Rosenbloom} $(\cX,\fA)$ a measurable space and $f_i$ bounded measurable functions.\medskip

\item H.\ Richter 1957\footnote{Received: December 27, 1956. Published: April, 1957.} \cite[Satz 4]{richte57}: $(\cX,\fA)$ measurable space and $f_i$ measurable functions.\medskip

\item M.\ V.\ Tchakaloff 1957\footnote{Published: July-September, 1957} \cite[Thm.\ II]{tchaka57}: $\cX\subset\rset^n$ compact and $f_i$ monomials of degree at most $d$.\label{item:tchakaloff}\index{Tchakaloff, M.\ V.}\index{Theorem!Tchakaloff}\medskip

\item W.\ W.\ Rogosinski 1958\footnote{Received: August 22, 1957. Published: May 6, 1958.} \cite[Thm.\ 1]{rogosi58}: $(\cX,\fA)$ measurable space and $f_i$ measurable functions.\label{item:rosenbloom}\index{Rogosinski, W.\ W.}\index{Theorem!Rogosinski}
\end{enumerate}
From this list we see that Tchakaloff's result (\ref{item:tchakaloff}) from 1957 is a special case of Rosenbloom's result (\ref{item:rosenbloom}) from 1952 and that the general case was proved by Richter and Rogosinski almost about at the same time, see the exact dates in the footnotes.
If one reads  Richter's paper, one might think at first glance that he treats only the one-dimensional case, but a closer look reveals that his Proposition (Satz) 4 covers actually the  general case of measurable functions.
Rogosinski treats the one-dimensional case, but states at the end of the introduction of \cite{rogosi58}:
\begin{quote}
Lastly, the restrictions in this paper to moment problems of dimension one is hardly essential.
Much of our geometrical arguments carries through, with obvious modifications, to any finite number of dimensions, and even to certain more general measure spaces.
\end{quote}
The above proof of \Cref{thm:richter}, and likewise the one in \cite[Theorem 1.24]{schmudMomentBook}, are nothing but modern formulations of the proofs of Richter and Rogosinski without additional arguments.
Note that Rogosinki's paper \cite{rogosi58} was submitted about a half year after the appearance of Richter's \cite{richte57}.

It might be of interest that the general results of Richter and Rogosinski from 1957/58 can be derived from Rosenbloom's Theorem from 1952, see Problem \ref{prob:richterFromRosen}.
With that wider historical perspective in mind it might be justified to call \Cref{thm:richter} also the \emph{Richter--Rogosinski--Rosenbloom Theorem}.\index{Richter--Rogosinski--Rosenbloom Theorem}\index{Theorem!Richter--Rogosinski--Rosenbloom}

\Cref{thm:richter} was overlooked in the modern literature on truncated polynomial moment problems.
The problem probably arose around 1997/98 when it was stated as an open problem in a published paper.\footnote{We do not give the references for this and subsequent papers who reproved \Cref{thm:richter}.}
The paper \cite{richte57} and numerous works of J.\ H.\ B.\ Kemperman were not included back then.
Especially \cite[Thm.\ 1]{kemper68} where Kemperman fully states the general theorem (\Cref{thm:richter}) and attributed it therein to Richter and Rogosinski is missing.
Later on, this missing piece was not added in several other works.
The error continued in the literature for several years and \Cref{thm:richter} was reproved in several papers in weaker forms.
Even nowadays papers appear not aware of \Cref{thm:richter} or of the content of \cite{richte57}.

\section{Signed Representing Measures: Boas' Theorem}

In the theory of moments almost exclusively the representation by non-negative measures is treated.
The reason is the following result due to R.\ P.\ Boas\index{Boas, R.\ P.} from 1939.

\begin{boasthm}[\cite{boas39a} or e.g.\ {\cite[p.\ 103, Thm.\ 3.11]{shohat43}}]\label{thm:boas}\index{Boas' Theorem}\index{Theorem!Boas}
Let $s = (s_i)_{i\in\nset_0}$ be a real sequence. Then there exist infinitely many signed measures $\mu$ on $\rset$ and infinitely many signed measures $\nu$ on $[0,\infty)$ such that
\[s_i = \int_\rset x^{i}~\diff\mu(x) = \int_0^\infty x^{i}~\diff\nu(x)\]
holds for all $i\in\nset_0$.
\end{boasthm}

The proof follows the arguments in \cite[pp.\ 103--104]{shohat43}.

\begin{proof}
We prove the case on $[0,\infty)$.
The case on $\rset$ is then only a special case.

By induction we write $s = v - w$ such that $v$ and $w$ are positive definite sequences where we can apply the \Cref{thm:basicrepresentation}.

$i=0$: We can chose $v_0,w_0\gg 1$ with $s_0 = v_0 - w_0$, i.e., $L_v(p), L_w(p)\geq 0$ for all $p\in\pos([0,\infty))_{\leq 0} = [0,\infty)$.

$i\to i+1$: Assume we found $(v_j)_{j=0}^i$ and $(w_j)_{j=0}^i$ such that $L_v(p),L_w(p)\geq 0$ for all $p\in\pos([0,\infty))_{\leq i}$.
Since for $i+1$ the term $x^{i+1}$ appears additionally to $1,x,x^2,\dots,x^i$, the convex cone $\pos([0,\infty))_{\leq i+1}$ has compact base, and $L$ is continuous on $\rset[x]_{\leq i+1}$ we find $v_{i+1},w_{i+1}\gg 1$ with $s_{i+1} = v_{i+1} - w_{i+1}$ such that $L_v(p),L_w(p)\geq 0$ for all $p\in\pos([0,\infty))_{\leq i+1}$.

Hence, we found sequences $v,w$ with $s = v-w$ and $L_v(p),L_w(p)\geq 0$ for all $p\in\pos([0,\infty))$.
By the \Cref{thm:basicrepresentation} $L_v$ is represented by some non-negative $\mu_+$ and $L_w$ is represented by some non-negative $\mu_-$ both with support in $[0,\infty)$, i.e., $L_s$ is represented by $\mu = \mu_+ - \mu_-$ supported on $[0,\infty)$.
\end{proof}

T. Sherman\index{Sherman, T.} showed that \Cref{thm:boas} (even when $L$ is a complex linear functional) also holds in the $n$-dimensional case on $\rset^n$ and $[0,\infty)^n$ for any $n\in\nset$ \cite[Thm.\ 1]{sherman64}.
Similar results are proved for linear functionals on the universal enveloping algebra $\mathcal{E}(G)$ of a Lie group $G$ by K.\ Schmüdgen\index{Schmüdgen, K.} \cite{schmud78}.
If the Lie group $G$ is $\rset^n$ then this again gives Sherman's result.
G.\ P\'olya\index{P\'olya, G.} \cite{polya38} (see also \cite[p.\ 104]{shohat43}) showed an extension which kinds of measures can be chosen.
On $\rset^n$ it is even possible to find a Schwartz function $f\in\cS(\rset^n)$ such that
\[s_\alpha = \int_{\rset^n} x^\alpha\cdot f(x)~\diff x\]
for all $\alpha\in\nset_0^n$.
Use e.g.\ \cite{curtoHeat22}.

\Cref{thm:boas} also covers the case with gaps.
If any gaps in the real sequence $s$ are present then fill them with any real number you like.

\section{Solving all Truncated Moment Problems solves the Moment Problem}
\label{sec:stochel}

The following result was already indicated by Hamburger\index{Hamburger, H.\ L.} in \cite{hamburger20} and formalized by J.\ Stochel\index{Stochel, J.} in \cite{stochel01}.
We have the following.

\begin{thm}\label{thm:stochel}
Let $n\in\nset$, $K\subseteq\rset^n$ be closed, $\cV\subseteq\rset[x_1,\dots,x_n]$ be an adapted space on $K$, and let $L:\cV\to\rset$ be a linear functional on $\cV$.
The following are equivalent:
\begin{enumerate}[(i)]
\item $L:\cV\to\rset$ is a $K$-moment functional.

\item $L_k := L|_{\cV\cap\rset[x_1,\dots,x_n]_{\leq k}}$ are truncated $K$-moment functionals for all $k\in\nset_0$.
\end{enumerate}
\end{thm}
\begin{proof}
While ``(i) $\Rightarrow$ (ii)'' is clear it is sufficient to prove the reverse direction.

Let $L_k$ be a truncated $K$-moment functionals for all $k\in\nset_0$.
Since $\cV\subseteq\rset[x_1,\dots,x_n]$ for any $p\in\cV$ we have that $L:\cV\to\rset$ is well-defined by $L(p) := L_{\deg p}(p)$.
Let $p\in\cV$ with $p\geq 0$ on $K$ then $L(p) = L_{\deg p}(p) \geq 0$, i.e., by the \Cref{thm:basicrepresentation} we have that $L$ is a $K$-moment functional. 
\end{proof}

Note, $\cV$ can also be finite dimensional when $K$ is compact.
Then the result is trivial.
For unbounded $K$ the adapted space $\cV$ is always infinite dimensional.

A more general version of \Cref{thm:stochel} can e.g.\ be found in \cite[Thm.\ 1.20]{schmudMomentBook}.

\section*{Problems}
\addcontentsline{toc}{section}{Problems}

\begin{prob}\label{prob:stieltjes}
Prove \Cref{thm:stieltjesMP} with the \Cref{thm:basicrepresentation} and the representation (\ref{eq:pos0infty2}).
\end{prob}

\begin{prob}\label{prob:hamburger}
Prove \Cref{thm:hamburgerMP} with the \Cref{thm:basicrepresentation} and the representation (\ref{eq:posR}).
\end{prob}

\begin{prob}\label{prob:hausdorff}
Prove \Cref{thm:hausdorffMP} with the \Cref{thm:basicrepresentation} and the \Cref{thm:luma}, resp.\ $\pos([a,b])$ in (\ref{eq:posab}).
\end{prob}

\begin{prob}\label{prob:haviland}
Prove \Cref{thm:haviland} with the \Cref{thm:basicrepresentation}.
\end{prob}

\begin{prob}\label{prob:bernstein}
Use \Cref{thm:bernstein} (ii) to prove \Cref{cor:bernstein}.
\end{prob}

\begin{prob}\label{prob:boundaryCone}
Prove \Cref{lem:boundaryCone}.
\end{prob}

\begin{prob}\label{prob:richterFromRosen}
Show that \Cref{thm:richter} follows from Rosenbloom's Theorem, i.e., show that the additional assumption that all $f_i$ are bounded on the measurable space $(\cX,\fA)$ can be removed.
\end{prob}

\part{Tchebycheff Systems}
\label{part:tSystems}

\motto{There is nothing more practical 
than a good theory.\\ \medskip
\ \hspace{1cm} \normalfont{Kurt Lewin \cite{lewin43}}\index{Lewin, K.}}

\chapter{T-Systems}
\label{ch:tsystems}

In this chapter we introduce the Tchebycheff systems or short T-systems.
We give basic examples and properties.

\section{The Early History of T-Systems}

In our presentation we mostly limit ourselves to the works \cite{krein51,karlin63,karlinStuddenTSystemsBook,kreinMarkovMomentProblem}.
However, the concept of T-system was introduces much earlier.
It goes back to its name giver: P.\ L.\ Tchebycheff \cite{tcheby74}.\index{Tchebycheff, P.\ L.}
See especially \cite{krein51} for a good overview of the history of the development of T-systems and also \cite{goncha00}.

In \cite{tcheby74} Tchebycheff states the following open problem:
\begin{quote}
Let
\[a < \xi < \eta < b\]
be real numbers and let the numbers
\begin{equation}\label{eq:momentconstr}
s_k = \int_a^b x^k f(x)~\diff x
\end{equation}
for $k=0,1,\dots,n-1$ for some $n\in\nset_0$ be given.
Find the bounds on the integral
\begin{equation}\label{eq:optim}
\int_\xi^\eta f(x)~\diff x
\end{equation}
under the conditions that $f\geq 0$ on $[a,b]$ and (\ref{eq:momentconstr}) holds.
\end{quote}
From this investigation Tchebycheff arrives at the method of continued fractions, which was used in the early results in the moment problems, see \Cref{sec:classicalResults}.
Tchebycheff gives without proof the inequalities (upper and lower bounds) of (\ref{eq:optim}).
The proof was independently found by others, see \cite[pp.\ 3--4]{krein51}.
The key here is to work over a finitely dimensional space spanned by $f_0,\dots,f_n$.

A well-known and guiding example are the functions $1,x,\dots,x^n$.

\begin{exm}\label{exm:vandermonde}
Let $n\in\nset$ and $\cX \subseteq \rset$ with $|\cX|\geq n+1$. Then the family $\cF = \{x^i\}_{i=0}^n$ is a T-system, see \Cref{dfn:tSystem} below. This follows immediately from the Vandermonde determinant\index{determinant!Vandermonde}\index{Vandermonde!determinant}
%
\[\det \big(x_i^j\big)_{i,j=0}^n = \prod_{0\leq i<j\leq n} (x_j-x_i)\]
%
for any $x_0,\dots,x_n\in\cX$.\exmsymbol
\end{exm}

Krein states that he developed ``the connection between ideas of Markov and functional-geometric ideas'' which made it possible to remove the Wronskian approach (\Cref{dfn:wronski}) and replacing it with continuity and the condition
\begin{quote}
The curve $\Gamma$ of the $(n+1)$-dimensional space $\rset^{n+1}$:
\[y_0 = f_0(x),\quad y_1 = f_1(x),\quad \dots, y_n = f_n(x)\]
does not intersect itself and no hyperplane through the origin intersects it in more than $n$ points.
\end{quote}
which is equivalent to
\begin{quote}
No linear combination
\[\sum_{i=0}^n a_i f_i \quad\text{with}\quad \sum_{i=0}^n a_i^2 > 0\]
vanishes more than $n$ times in the closed interval $[a,b]$.
\end{quote}
see \cite[pp.\ 19--20]{krein51}.
The later is then generalized to leave out continuity and replacing $[a,b]$ with any set $\cX$, see \Cref{dfn:tSystem}.
For a family $\{f_i\}_{i=0}^n$ with this property S.\ N.\ Bernstein\index{Bernstein, S.\ N.} \cite{bernstein37} introduced the name \emph{Tchebycheff system}\index{system!Tchebycheff|see{T-}} and Krein \cite[p.\ 20]{krein51} and Archieser \cite[p.\ 73, §47]{achieser56} continued using this terminology.

For more on the history see e.g.\ \cite{krein51}.
We especially recommend the very nice survey article \cite{goncha00} with the references therein for more on the works, the contributions, and the impact of Tchebycheff's work.

\section{Definition and Basic Properties}

\begin{dfn}\label{dfn:tSystem}
Let $n\in\nset_0$, $\cX$ be a set with $|\cX|\geq n+1$, and $\cF = \{f_i\}_{i=0}^n$ be a family of real functions $f_i:\cX\to\rset$. We call a linear combination
\begin{equation}\label{eq:linF}
f = \sum_{i=0}^n a_i\cdot f_i \quad\in\lin\cF := \{a_0 f_0 + \dots + a_n f_n \,|\, a_0,\dots,a_n\in\rset\}
\end{equation}
a \emph{polynomial}.\index{polynomial}
The family $\cF$ on $\cX$ is called a \emph{Tchebycheff system}\index{Tchebycheff system|see{T-system}} (or short \emph{T-system})\index{T-system}\index{system!T-} \emph{of order $n$ on $\cX$} if every polynomial $f\in\lin\cF$ with $\sum_{i=0}^n a_i^2 > 0$ has at most $n$ zeros in $\cX$.

If additionally $\cX$ is a topological space and $\cF$ is a family of continuous functions we call $\cF$ a \emph{continuous T-system}.\index{T-system!continuous}
If additionally $\cX$ is the unit circle $\tset$ then we call $\cF$ a \emph{periodic T-system}.\index{T-system!periodic}
\end{dfn}

The following immediate consequence shows that we can restrict the domain $\cX$ of the T-system $\cF$ to some $\cY\subseteq\cX$ and as long as $|\cY|\geq n+1$ the restricted T-system remains a T-system. In applications and examples we therefore only need to prove the T-system property on some larger set $\cX$.

\begin{cor}\label{cor:restriction}
Let $n\in\nset_0$ and let $\cF=\{f_i\}_{i=0}^n$ be a T-system of order $n$ on some set $\cX$ with $|\cX|\geq n+1$. Let $\cY\subseteq\cX$ with $|\cY|\geq n+1$. Then $\cG := \{f_i|_\cY\}_{i=0}^n$ is a T-system of order $n$ on $\cY$.
\end{cor}
\begin{proof}
See Problem \ref{prob:restriction}.
\end{proof}

The set $\cX$ does not require any structure or property except $|\cX|\geq n+1$.

In the theory of T-systems we often deal with one special matrix.
We use the following abbreviation.

\begin{dfn}\label{dfn:kreinMatrix}
Let $n\in\nset_0$, $\cF=\{f_i\}_{i=0}^n$ be a family of real functions on a set $\cX$ with $|\cX|\geq n+1$. We define the matrix
\begin{equation}\label{eq:kreinMatrix}
\begin{pmatrix}
f_0 & f_1 & \dots & f_n\\ x_0 & x_1 & \dots & x_n
\end{pmatrix}
:=
\begin{pmatrix}
f_0(x_0) & f_1(x_0) & \dots & f_n(x_0)\\
f_0(x_1) & f_1(x_1) & \dots & f_n(x_1)\\
\vdots & \vdots & & \vdots\\
f_0(x_n) & f_1(x_n) & \dots & f_n(x_n)
\end{pmatrix} = (f_i(x_j))_{i,j=0}^n
\end{equation}
for any $x_0,\dots,x_n\in\cX$.
\end{dfn}

\begin{lem}[see e.g.\ {\cite[p.\ 31]{kreinMarkovMomentProblem}}]\label{lem:determinant}
Let $n\in\nset_0$, $\cX$ be a set with $|\cX|\geq n+1$, and $\cF=\{f_i\}_{i=0}^n$ be a family of real functions $f_i:\cX\to\rset$. The following are equivalent:
\begin{enumerate}[(i)]
\item $\cF$ is a T-system of order $n$ on $\cX$.

\item The determinant
\[\det\begin{pmatrix}
f_0 & f_1 & \dots & f_n\\ x_0 & x_1 & \dots & x_n
\end{pmatrix}\]
does not vanish for any pairwise distinct points $x_0,\dots,x_n\in\cX$.
\end{enumerate}
\end{lem}
\begin{proof}
(i) $\Rightarrow$ (ii):
Let $x_0,\dots,x_n\in\cX$ be pairwise distinct.
Since $\cF$ is a T-system we have that any non-trivial polynomial $f$ has at most $n$ zeros, i.e., the matrix
\[\begin{pmatrix}
f_0 & f_1 & \dots & f_n\\ x_0 & x_1 & \dots & x_n
\end{pmatrix}\]
has trivial kernel and hence its determinant is non-zero. Since $x_0,\dots,x_n\in\cX$ are arbitrary pairwise distinct we have (ii).

(ii) $\Rightarrow$ (i): Assume there is a polynomial $f$ with $\sum_{i=0}^n a_i^2 >0$ which has the $n+1$ pairwise distinct zeros $z_0,\dots,z_n\in\cX$. Then the matrix
\[Z=\begin{pmatrix}
f_0 & f_1 & \dots & f_n\\ z_0 & z_1 & \dots & z_n
\end{pmatrix}\]
has non-trivial kernel since $0\neq (a_0,a_1,\dots,a_n)^T\in\ker Z$ and hence $\det Z = 0$ in contradiction to (ii).
\end{proof}

\Cref{lem:determinant} is used in \cite[p.\ 3, Dfn.\ 2.1]{karlinStuddenTSystemsBook} as the definition of a continuous T-system where it is called a weak T-system.
In \cite[p.\ 22, Thm.\ 4.1]{karlinStuddenTSystemsBook} then the equivalence to \Cref{dfn:tSystem} is shown.

\begin{rem}
\Cref{lem:determinant} implies that for any $x\in\cX$ there is a $f\in\lin\cF$ such that $f(x)\neq 0$, i.e., the $f_0,\dots,f_n$ do not have common zeros.\exmsymbol
\end{rem}

\begin{rem}\label{rem:signDet}
After adjusting the sign of $f_n$ in a continuous T-system $\cF=\{f_i\}_{i=0}^n$ on $[a,b]$ we can assume that
\[\det\begin{pmatrix}
f_0 & f_1 & \dots & f_n\\ x_0 & x_1 & \dots & x_n
\end{pmatrix} > 0\]
holds for all $a\leq x_1 < x_2 < \dots < x_n \leq b$.\exmsymbol
\end{rem}

The previous lemma implies the following.

\begin{cor}[see e.g.\ {\cite[p.\ 33]{kreinMarkovMomentProblem}}]\label{cor:transf}
Let $n\in\nset_0$, and $\cF = \{f_i\}_{i=0}^n$ be a T-system of order $n$ on some set $\cX$ with $|\cX|\geq n+1$.
Let $\cW$ be a set with $n+1\leq |\cW|\leq |\cX|$ and let $g:\cW\to\cX$ be injective.
Then $\cG = \{g_i\}_{i=0}^n$ with $g_i:=f_i\circ g$ is a T-system of order $n$ on $\cW$.
\end{cor}
\begin{proof}
See Problem \ref{prob:transf}.
\end{proof}

\begin{cor}[see e.g.\ {\cite[p.\ 10]{karlinStuddenTSystemsBook}} or {\cite[p.\ 33]{kreinMarkovMomentProblem}}]\label{cor:scaling}
Let $n\in\nset_0$, and $\cF = \{f_i\}_{i=0}^n$ be a T-system of order $n$ on some set $\cX$ with $|\cX|\geq n+1$.
Let $g:\cX\to\rset$ be such that $g>0$ on $\cX$.
Then $\cG = \{g_i\}_{i=0}^n$ with $g_i := g\cdot f_i$ is a T-system of order $n$ on $\cX$.
\end{cor}
\begin{proof}
See Problem \ref{prob:scaling}.
\end{proof}

\begin{cor}[see e.g.\ {\cite[p.\ 33]{kreinMarkovMomentProblem}}]\label{cor:uniqueDeter}
Let $n\in\nset_0$, and $\cF = \{f_i\}_{i=0}^n$ be a T-system of order $n$ on some set $\cX$ with $|\cX|\geq n+1$. The following hold:
\begin{enumerate}[(i)]
\item The functions $f_0,\dots,f_n$ are linearly independent over $\cX$.

\item For any $f = \sum_{i=0}^n a_i\cdot f_i\in\lin\cF$ the coefficients $a_0,\dots,a_n\in\rset$ are unique.
\end{enumerate}
\end{cor}
\begin{proof}
See Problem \ref{prob:4.1}.
\end{proof}

The previous corollary extends to the following result.

\begin{thm}[see e.g.\ {\cite[p.\ 33]{kreinMarkovMomentProblem}}]
Let $n\in\nset_0$, $\cF$ be a T-system on some set $\cX$ with $|\cX|\geq n+1$, and let $x_0,\dots,x_n\in\cX$ be $n+1$ pairwise different points.
The following hold:
\begin{enumerate}[(i)]
\item Every $f\in\lin\cF$ is uniquely determined by its values $f(x_0),\dots, f(x_n)$.

\item For any $y_0,\dots,y_n\in\rset$ there exists a unique $f\in\lin\cF$ such that $f(x_i)=y_i$ holds for all $i=0,\dots,n$.
\end{enumerate}
\end{thm}
\begin{proof}
(i):
Since $f\in\lin\cF$ we have $f = \sum_{i=0}^n a_i\cdot f_i$.
Let $x_1,\dots,x_n\in\cX$ be pairwise distinct points. Then by \Cref{lem:determinant} (i) $\Rightarrow$ (ii) we have that
\[\begin{pmatrix}
f(x_0)\\ \vdots\\ f(x_n)
\end{pmatrix} = \begin{pmatrix}
f_0 & f_1 & \dots & f_n\\ x_0 & x_1 & \dots & x_n
\end{pmatrix}\cdot \begin{pmatrix}
\alpha_0\\ \vdots\\ \alpha_n
\end{pmatrix} \]
has the unique solution $\alpha_0 = a_0$, \dots, $\alpha_n = a_n$.

(ii):
By the same argument as in (i) the system
\[\begin{pmatrix}
y_0\\ \vdots\\ y_n
\end{pmatrix} = \begin{pmatrix}
f_0 & f_1 & \dots & f_n\\ x_0 & x_1 & \dots & x_n
\end{pmatrix}\cdot \begin{pmatrix}
\alpha_0\\ \vdots\\ \alpha_n
\end{pmatrix} \]
has the unique solution $\alpha_0 = a_0$, \dots, $\alpha_n = a_n$.
\end{proof}

\section{The Curtis--Mairhuber--Sieklucki Theorem}

So far we imposed no structure on the set $\cX$.
We now get a structure of $\cX$.
The following structural result was proved in \cite[Thm.\ 2]{mairhu56} for compact subsets $\cX$ of $\rset^n$ and for arbitrary compact sets $\cX$ in \cite{sieklu58} and \cite[Thm.\ 8 and Cor.]{curtis59}.

\begin{cmsthm}\label{thm:cms}\index{Curtis--Mairhuber--Sieklucki Theorem}\index{Theorem!Curtis--Mairhuber--Sieklucki}
Let $n\in\nset_0$ and $\cF$ be a continuous T-system of order $n$ on a topological space $\cX$. If $\cX$ is a compact metrizable space then $\cX$ can be homeomorphically embedded in the unit circle $\{(x,y)\in\rset^2 \,|\, x^2 + y^2 = 1\}$.
\end{cmsthm}

The proof is not difficult but technical and too lengthy for our purposes.
We therefore refer the reader to \cite[Thm.\ 8]{curtis59}.

An immediate consequence of the \Cref{thm:cms} is that every T-system is up to homomorphisms one-dimensional, i.e., in algebraic applications of the theory of T-systems we can only deal with the univariate case.
Additionally, we have the following result.

\begin{cor}[see e.g.\ {\cite[Cor.\ after Thm.\ 8]{curtis59}}]
The order $n$ of a periodic T-system is even.
\end{cor}
\begin{proof}
Let $\varphi:[0,2\pi]\to S=\{(x,y)\in\rset^2 \,|\, x^2 + y^2\}$ with $\varphi(\alpha) = (\sin\alpha, \cos\alpha)$ and $\cF=\{f_i\}_{i=0}^n$ be a periodic T-system. Then the $f_i$ are continuous and hence also
\[\det\begin{pmatrix}
f_0 & f_1 & \dots & f_n\\ t_0 & t_1 & \dots & t_n
\end{pmatrix}\]
is continuous in $t_0,\dots,t_n\in S$. If $\cF$ is a T-system we have that
\[d(\alpha) := \det\begin{pmatrix}
f_0 & f_1 & \dots & f_n\\ \varphi(\alpha) & \varphi(\alpha+2\pi/(n+1)) & \dots & \varphi(\alpha + 2n\pi/(n+1))
\end{pmatrix}\]
in non-zero for all $\alpha\in [0,2\pi]$ and never changes singes. If $n$ is odd then $d(0) = -d(2\pi/(n+1))$ which is a contradiction. Hence, $n$ must be even.
\end{proof}

\section{Examples of T-Systems}

\begin{exm}[\Cref{exm:vandermonde} continued]\label{exm:vandermonde2}
Let $n\in\nset_0$ and $\cX = \rset$.
Then the family $\cF = \{x^i\}_{i=0}^n$ of monomials is a T-system.
To see this let $x_0 < x_1 < \dots < x_n$ be $n+1$ points in $\rset$.
We then have by the Vandermonde determinant\index{Vandermonde!determinant}\index{determinant!Vandermonde}
\begin{equation}\label{eq:vandermonde}
\det \begin{pmatrix}
1 & x & \dots & x^n\\ x_0 & x_1 & \dots & x_n
\end{pmatrix} = \prod_{0\leq i<j\leq n} (x_j - x_i)
\end{equation}
which is always non-zero and hence $\cF$ is a T-system of order $n$ on $\rset$ by \Cref{lem:determinant}. Additionally, by \Cref{cor:restriction} we have that $\cF$ is a T-system of order $n$ on any $\cY\subseteq\rset$ with $|\cY|\geq n+1$.\exmsymbol
\end{exm}

Note, that in (\ref{eq:vandermonde}) the functions $f_i$ should be written more precisely as
\[f_i:\rset\to\rset,\ x\mapsto x^i\]
and not just as $x^i$. 
However, we then would have the notation
\[\begin{pmatrix} \cdot^0 & \cdot^1 & \dots & \cdot^n\\ x_0 & x_1 & \dots & x_n\end{pmatrix}
\quad\text{or more general}\quad
\begin{pmatrix}\cdot^{\alpha_0} & \cdot^{\alpha_1} & \dots & \cdot^{\alpha_n}\\
x_0 & x_1 & \dots & x_n\end{pmatrix}\]
for $\alpha_i$ with $-\infty < \alpha_0 < \alpha_1 < \dots < \alpha_n <\infty$ which seems to be hard to read.
We will therefore abuse the notation and use $x^i$, $x^{\alpha_i}$, and (\ref{eq:vandermonde}).

\Cref{exm:vandermonde2} can be generalized to non-negative real exponents.

\begin{exm}[see e.g.\ {\cite[p.\ 9, Exm.\ 1]{karlinStuddenTSystemsBook}} or {\cite[p.\ 38, §2(d)]{kreinMarkovMomentProblem}}]\label{exm:tsysAlpha}
Let $n\in\nset_0$ and let $0 = \alpha_0 < \alpha_1 < \dots < \alpha_n$ be non-negative reals.
Then
\[\cF = \{x^{\alpha_0},x^{\alpha_1},\dots,x^{\alpha_n}\}\]
is a T-system of order $n$ on any $\cX\subseteq [0,\infty)$ with $|\cX|\geq n+1$.\exmsymbol
\end{exm}

If we restrict $\cX$ to $\cX\subseteq (0,\infty)$ then we can allow arbitrary real exponents $\alpha_i$.

\begin{exm}\label{exm:tsysAlphaReal}
Let $n\in\nset$ and $\alpha_0 < \alpha_1 < \dots < \alpha_n$ be reals. Then
\[\cF = \left\{x^{\alpha_0}, x^{\alpha_1},\dots, x^{\alpha_n}\right\}\]
is a T-system on any $\cX\subseteq (0,\infty)$ with $|\cX|\geq n+1$.\exmsymbol
\end{exm}

By using $\exp:\rset\to (0,\infty)$ we find that the previous example is by \Cref{cor:transf} equivalent to the following.

\begin{exm}[see e.g.\ {\cite[p.\ 38]{kreinMarkovMomentProblem}}]\label{exm:tsysExp}
Let $n\in\nset$ and $\alpha_0 < \alpha_1 < \dots < \alpha_n$ be reals. Then
\[\cG = \left\{e^{\alpha_0 x}, e^{\alpha_1 x},\dots, e^{\alpha_n x}\right\}\]
is a T-system on any $\cY\subseteq\rset$ with $|\cY|\geq n+1$.\exmsymbol
\end{exm}

That the equivalent Examples \ref{exm:tsysAlphaReal} and \ref{exm:tsysExp} are T-systems will be postponed to \Cref{exm:etSystems}.
The reason is that with the introduction of ET-systems in \Cref{ch:etsystems} and especially \Cref{thm:generalETsystem} we generate plenty of examples of ET- and T-systems.

\begin{exm}[see e.g.\ {\cite[p.\ 41, no.\ 26]{polya64}} or {\cite[p.\ 37-38]{kreinMarkovMomentProblem}}]\label{exm:fraction}
Let $n\in\nset$ and $\alpha_0 < \alpha_1 < \dots < \alpha_n$ be reals. Then
\[\cF=\left\{\frac{1}{x+\alpha_0},\frac{1}{x+\alpha_1},\dots,\frac{1}{x+\alpha_n}\right\}\]
is a continuous T-system on any $[a,b]$ or $[a,\infty)$ with $-\alpha_0 < a < b$.\exmsymbol
\end{exm}
\begin{proof}
See Problem \ref{prob:fraction}.
\end{proof}

\begin{exm}[see e.g.\ {\cite[p.\ 38]{kreinMarkovMomentProblem}}]\label{exm:xf}
Let $n\in\nset$ and let $f\in \cat^n(\cX,\rset)$ with $\cX=[a,b]$, $a<b$, and $f^{(n)}>0$ on $\cX$. Then
\[\cF = \{1,x,x^2,\dots, x^{n-1},f\}\]
is a continuous T-system of order $n$ on $\cX = [a,b]$. We can also allow $\cX = (a,b)$, $[a,\infty)$, $(-\infty,b)$, \dots .\exmsymbol
\end{exm}

With the techniques developed in \Cref{ch:etsystems} it will be easy to show that \Cref{exm:xf} is not only a T-system but in fact also an ET- and ECT-system.
We will therefore postpone its proof to Problem \ref{prob:xf}.

\section{Representation as a Determinant, Zeros, and Non-Negativity}

The following result shows that when enough zeros of a polynomial $f\in\lin\cF$ are known then $f$ has the following representation as a determinant.\index{representation!as a determinant}\index{determinant!representation as a}

\begin{thm}[see e.g.\ {\cite[p.\ 33]{kreinMarkovMomentProblem}}]\label{thm:detRepr}
Let $n\in\nset$, $\cF = \{f_i\}_{i=0}^n$ be a T-system on some set $\cX$ with $|\cX|\geq n+1$, $x_1,\dots,x_n\in\cX$ be $n$ pairwise distinct points, and let $f\in\lin\cF$. The following are equivalent:
\begin{enumerate}[(i)]
\item $f(x_i)=0$ holds for all $i=1,\dots,n$.

\item There exists a constant $c\in\rset$ such that
\begin{equation}\label{eq:fDeterminant}
f(x) = c\cdot\det\begin{pmatrix}
f_0 & f_1 & \dots & f_n\\ x & x_1 & \dots & x_n
\end{pmatrix}.
\end{equation}
\end{enumerate}
\end{thm}
\begin{proof}
(ii) $\Rightarrow$ (i): Clear.

(i) $\Rightarrow$ (ii):
If $f=0$ then $c=0$ so the assertion holds.
If $f\neq 0$ then there exists a point $x_0\in\cX\setminus\{x_1,\dots,x_n\}$ such that $f(x_0)\neq 0$ since $\cF$ is a T-system.
Then also the determinant in (ii) is non-zero and we can choose $c$ such that both $f$ and the scaled determinant coincide also in $x_0$.
By \Cref{cor:uniqueDeter} a polynomial $f$ is uniquely determined by its values $f(x_i)$ at $x_0,\dots,x_n$.
This shows that (\ref{eq:fDeterminant}) is the only polynomial which fulfills (i).
\end{proof}

So far we treated general T-systems.
For further properties we go to continuous T-systems.
By the \Cref{thm:cms} we can assume $\cX\subseteq\rset$.

\begin{dfn}
Let $n\in\nset_0$, $\cF$ be a continuous T-system on $\cX\subseteq\rset$ an interval, $f\in\lin\cF$, and let $x_0$ be a zero of $f$. Then $x_0\in\inter\cX$ is called a \emph{non-nodal} zero\index{zero!non-nodal} if $f$ does not change sign at $x_0$. Otherwise the zero $x_0$ is called \emph{nodal},\index{zero!nodal} i.e., either $f$ changes signs at $x_0$ or $x_0$ is a boundary point of $\cX$.
\end{dfn}

The following result bounds the number of nodal and non-nodal zeros.

\begin{thm}[see {\cite[Lem.\ 3.1]{krein51}} or e.g.\ {\cite[p.\ 34, Thm.\ 1.1]{kreinMarkovMomentProblem}}]\label{thm:zeros1}
Let $n\in\nset_0$, $\cF$ be a continuous T-system of order $n$ on $\cX = [a,b]$ with $-\infty<a<b<\infty$.
If $f\in\lin\cF$ has $k\in\nset_0$ non-nodal zeros and $l\in\nset_0$ nodal zeros in $\cX$ then $2k + l\leq n$.
\end{thm}

The proof is adapted from {\cite[pp.\ 34, Thm.\ 1.1]{kreinMarkovMomentProblem}}.

\begin{proof}
We make two case distinctions, one for $k=0$ and one for $k\geq 1$.

$k=0$: If $f\in\lin\cF$ has $l$ zeros then $l \leq n$ by \Cref{dfn:tSystem}.

$k\geq 1$: Let $x_1,\dots,x_p\in\inter\cX$ with $p\leq k+l$ be the zeros of $f$ in $\inter\cX$. Set
\[M_i := \max_{x_{i-1} \leq x\leq x_i} |f(x)|\]
for all $i=1,\dots,p+1$ with $x_0 = a$ and $x_{p+1}=b$. Additionally, set
\[m := \frac{1}{2}\min_{i=1,\dots,p+1} M_i,\]
i.e., $m>0$.

We construct a polynomial $g_1\in\lin\cF$ such that
\begin{enumerate}[\qquad (a)]
\item $g_1$ has the value $g(x_i)=m$ at the non-nodal zeros $x_i$ of $f$ with $f\geq 0$ in a neighborhood of $x_i$,

\item $g_1$ has the values $g(x_i)=-m$ at the non-nodal zeros $x_i$ of $f$ with $f\leq 0$ in a neighborhood of $x_i$, and

\item $g_1$ vanishes at all nodal zeros $x_i$, i.e., $g(x_i)=0$.
\end{enumerate}
After renumbering the zeros $x_i$ we can assume $x_1,\dots,x_{k_1}$ fulfill (a), $x_{k_1+1},\dots,x_{k_1+k_2}$ fulfill (b), and $x_{k_1+k_2+1},\dots,x_{k_1+k_2+l}$ fulfill (c) with $k_1+k_2=k$. By \Cref{dfn:tSystem} we have $k+l\leq n$ and hence by \Cref{lem:determinant} we have that
\begin{equation}\label{eq:bigGleichungssystem}
\begin{pmatrix}
m\\ \vdots\\ m\\ -m\\ \vdots\\ -m\\ 0\\ \vdots\\ 0
\end{pmatrix} = \begin{pmatrix}
f_0(x_1) & \dots & f_n(x_1)\\
\vdots & & \vdots\\
f_0(x_{k_1}) & \dots & f_n(x_{k_1})\\
f_0(x_{k_1+1}) & \dots & f_n(x_{k_1+1})\\
\vdots & & \vdots\\
f_0(x_k) & \dots & f_n(x_k)\\
f_0(x_{k+1}) & \dots & f_n(x_{k+1})\\
\vdots & & \vdots\\
f_0(x_{k+l}) & \dots & f_n(x_{k+l})
\end{pmatrix}\cdot \begin{pmatrix}
\beta_0\\ \vdots\\ \beta_n
\end{pmatrix}
\end{equation}
has at least one solution, say $\beta_0 = b_0$, \dots, $\beta_n = b_n$. Then $g_1 = \sum_{i=0}^n b_i\cdot f_i\in\lin\cF $ fulfills (a) to (c).

Set
\[\rho := \frac{m}{2\cdot \|g_1\|_\infty}\]
and define $g_2 := f - g_1$.

We show that to each non-nodal zero $x_i$ of $f$ there correspond two zeros of $g_2$.
Let $x_i$ be a non-nodal zero of $f$ with $f\geq 0$ in a neighborhood of $x_i$.
We can find a point $y_i\in (x_{i-1},x_i)$ and a point $y_{i+1}\in (x_i,x_{i+1})$ such that
\[f(y_i) = M_i > m \qquad\text{and}\qquad f(y_{i+1}) = M_{i+1} > m.\]
Hence, $g_2(y_i) > 0$ and $g_2(y_{i+1}) > 0$. Since $g_2(x_i) = -\rho\cdot m < 0$ it follows that $g_2$ has a zero both in $(y_i,x_i)$ and in $(x_i,y_{i+1})$.

Additionally, $g_2$ also vanishes at all nodal zeros of $f$ and therefore has at least $2k+l$ distinct zeros. By \Cref{dfn:tSystem} we have $2k+l\leq n$.
\end{proof}

The previous result holds for more general sets $\cX$.

\begin{cor}\label{cor:zeros1}
\Cref{thm:zeros1} holds for sets $\cX\subseteq\rset$ of the form
\begin{enumerate}[(i)]
\item $\cX = (a,b)$, $[a,b)$, $(a,b]$ with $-\infty < a < b < \infty$,

\item $\cX = (a,\infty)$, $[a,\infty)$, $(-\infty,b)$, $(-\infty,b]$ with $-\infty < a, b < \infty$,

\item $\cX = \{x_1,\dots,x_k\}\subseteq\rset$ with $k\geq n+1$ and $x_1 < \dots < x_k$, and

\item countable unions of (i) to (iii).
\end{enumerate}
\end{cor}
\begin{proof}
$\cX=[0,\infty)$:
Let $0\leq x_1 < \dots < x_k$ be the zeros of $f$ in $[0,\infty)$.
Since every T-system on $[0,\infty)$ is also a T-system on $[0,b]$ for any $b>0$ by \Cref{cor:restriction} the assertion follows from \Cref{thm:zeros1} with $b=x_k + 1$.

For the other assertions adapt (if necessary) the proof of \Cref{thm:zeros1}.
\end{proof}

That non-nodal points are always inner points and have a weight of (at least) $2$ in counting with multiplicities as well as that boundary points are always non-nodal and are counted (at least) once in counting the multiplicities is generalized in the following.

\begin{dfn}\label{dfn:index}
Let $x\in [a,b]$ with $a\leq b$. We define the \emph{index}\index{index} $\varepsilon(x)$ by
\begin{equation}\label{eq:index}
\varepsilon(x) := \begin{cases}
2 &\text{if}\ x\in (a,b),\\
1 &\text{if}\ x=a\ \text{or}\ b.
\end{cases}
\end{equation}
The same definition holds for sets as in \Cref{cor:zeros1}.

Let $\cX\subseteq\rset$ be a set. We define the \emph{index $\varepsilon(\cX)$ of the set $\cX$}\index{index!of a set} by
\begin{equation}\label{eq:indexSet}
\varepsilon(\cX) := \sum_{x\in\cX} \varepsilon(x).
\end{equation}
\end{dfn}

We now want to show that for each T-system $\cF$ not only non-negative polynomials $f\in\lin\cF$ exists but we can even specify the zeros.
We need the following definition.

\begin{dfn}\label{dfn:linFsubsets}
Let $n\in\nset_0$ and $\cF$ be a T-system of order $n$ on some set $\cX$. We define
\begin{align*}
(\lin\cF)^e &:= \left\{\; \sum_{i=0}^n a_i\cdot f_i \;\middle|\; \sum_{i=0}^n a_i^2 = 1 \;\right\},\\
(\lin\cF)_+ &:= \left\{ f\in\lin\cF \,\middle|\, f\geq 0\ \text{on}\ \cX\right\},
\intertext{and}
(\lin\cF)_+^e &:= (\lin\cF)^e \cap (\lin\cF)_+.
\end{align*}
\end{dfn}

With these definitions we can prove the following existence criteria for non-negative polynomials in a T-systems on $[a,b]$.

\begin{thm}[see {\cite[Lem.\ 3.2]{krein51}} or e.g.\ {\cite[p.\ 35, Thm.\ 1.2]{kreinMarkovMomentProblem}}]\label{thm:zeros2}
Let $n\in\nset_0$, $\cF$ be a continuous T-system on $\cX=[a,b]$, and let $x_1,\dots,x_m\in\cX$ be $m$ distinct points for some $m\in\nset$.
The following are equivalent:
\begin{enumerate}[(i)]
\item The points $x_1,\dots,x_m$ are zeros of a non-negative polynomial $f\in\lin\cF$.

\item $\displaystyle\sum_{i=1}^m \varepsilon(x_i) \leq n$.
\end{enumerate}
\end{thm}

The proof is adapted from {\cite[pp.\ 35, Thm.\ 1.2]{kreinMarkovMomentProblem}}.

\begin{proof}
``(i) $\Rightarrow$ (ii)'' is \Cref{thm:zeros1} and we therefore only have to prove ``(ii) $\Rightarrow$ (i)''.

\textit{Case I}: At first assume that $a < x_1 < \dots < x_m < b$ and $\sum_{i=0}^m \varepsilon(x_i) = 2m = n$.
If $2m < n$ then add $k$ additional points $x_{m+1},\dots,x_{m+k}$ such that $2m + 2k = n$ and $x_m < x_{m+1} < \dots < x_{m+k} < b$.

Select a sequence of points $(x_1^{(j)},\dots,x_m^{(j)})\in\rset^m$, $j\in\nset$, such that
\[a < x_1 < x_1^{(j)} < \dots < x_m < x_m^{(j)} < b\]
for all $j\in\nset$ and $\lim_{j\to\infty} x_i^{(j)} = x_i$ for all $i=1,\dots,m$. Set
\begin{equation}\label{eq:gjdfns}
g_j(x) := c_j\cdot \det\begin{pmatrix}
f_0 & f_1 & f_2 & \dots & f_{2m-1} & f_{2m}\\
x & x_1 & x_1^{(j)} & \dots & x_m & x_m^{(j)}
\end{pmatrix} \quad\in (\lin\cF)^e
\end{equation}
for some $c_j>0$.
Since $(\lin\cF)^e$ is compact we can assume that $g_j$ converges to some $g_0\in (\lin\cF)^e$.
Then $g_0$ has $x_1,\dots,x_m$ as zeros with $\varepsilon(x_i)=2$ and $g_0$ is non-negative since $g_j>0$ on $[a,x_1)$, $(x_1^{(j)},x_2)$, \dots, $(x_{m-1}^{(j)}, x_m)$, and $(x_m^{(j)},b]$ as well as $g_j<0$ on $(x_1,x_1^{(j)})$, $(x_2,x_2^{(j)})$, \dots, $(x_m,x_m^{(j)})$.

\textit{Case II}:
If $a = x_1 < x_2 < \dots < x_m < b$ with $\sum_{i=1}^m \varepsilon(x_i) = 2m-1 = n$ the only modification required in case I is to replace (\ref{eq:gjdfns}) by
\[g_j(x) := -c_j\cdot \det \begin{pmatrix}
f_0 & f_1 & f_2 & f_3 & \dots & f_{2m-2} & f_{2m-1}\\
x & a & x_2 & x_2^{(j)} & \dots & x_m & x_m^{(j)}
\end{pmatrix} \quad\in (\lin\cF)^e\]
with some normalizing factor $c_j > 0$.

\textit{Case III}:
The procedure is similar if $x_m = b$ and $\sum_{i=1}^m \varepsilon(x_i) = n$.
\end{proof}

\begin{rem}\label{rem:kreinError}
\Cref{thm:zeros2} appears in {\cite[p.\ 35, Thm.\ 1.2]{kreinMarkovMomentProblem}} in a stronger version, see also \cite[Lem.\ 3.4]{krein51}.

In {\cite[p.\ 35, Thm.\ 1.2]{kreinMarkovMomentProblem}} and \cite[Lem.\ 3.4]{krein51} Krein claims that the $x_1,\dots,x_m$ are the \emph{only} zeros of some non-negative $f\in\lin\cF$.
This holds when $n=2m + 2p$ for some $p\geq 0$ and $x_1,\dots,x_m\in\inter\cX$.
To see this add to $x_1,\dots,x_m$ in (\ref{eq:gjdfns}) points $x_{m+1},\dots,x_{m+p}\in\inter\cX\setminus\{x_1,\dots,x_m\}$ and get $g_0$.
Hence, $g_0\geq 0$ has exactly the zeros $x_1,\dots,x_{m+p}$.
Then construct in a similar way $\tilde{g}_0$ with the zeros $x_1,\dots,x_m,\tilde{x}_{m+1},\dots,\tilde{x}_{m+p}$ with $\tilde{x}_{m+1},\dots,\tilde{x}_{m+p}\in\inter\cX\setminus\{x_1,\dots,x_{m+p}\}$.
Hence, $g_0 + \tilde{g}_0\geq 0$ has only the zeros $x_1,\dots,x_m$.

A similar construction works for $n=2m+1$ with or without end points $a$ or $b$.
If $x_1,\dots,x_m$ contains no end point, i.e., all $x_i\in\inter\cX$, then construct a $g_0$ with an zero in $a$ (and therefore $g_0(b)>0$ since the index is odd) and a $\tilde{g}_0$ with zero in $b$ (and therefore $\tilde{g}_0(a)>0$).
Then $g_0 + \tilde{g}_0$ has no end point as a zero.

However, Krein misses that for $n=2m + 2p$ with $p\geq 0$ and when one end point is contained in $x_1,\dots,x_m$ then it might happen that also the other end point must appear.
In \cite[p.\ 28, Thm.\ 5.1]{karlinStuddenTSystemsBook} additional conditions are given which ensure that $x_1,\dots,x_m$ are the only zeros of some $f\geq 0$.

For example if also $\{f_i\}_{i=0}^{n-1}$ is a T-system then it can be ensured that $x_1,\dots,x_m$ are the only zeros of some non-negative polynomial $f\in\lin\cF$, see \cite[p.\ 28, Thm.\ 5.1 (b-i)]{karlinStuddenTSystemsBook}, see Problem \ref{prob:kreinError}.
For our main example(s), the algebraic polynomials with gaps, this holds.

The same problem appears in \cite[p.\ 36, Thm.\ 1.3]{kreinMarkovMomentProblem}.
A weaker but correct version is given in \Cref{thm:zeros3} below.

\Cref{thm:zeros1} with the condition that $\cF$ is an ET-system \cite[p.\ 28, Thm.\ 5.1]{karlinStuddenTSystemsBook} is given below in \Cref{thm:zeros4}.
\exmsymbol
\end{rem}

\begin{rem}\label{rem:doublezeros}
Assume that in \Cref{thm:zeros2} we have additionally that $f_0,\dots,f_n\in \cat^1([a,b],\rset)$. Then in (\ref{eq:gjdfns}) we can set $x_i^{(j)} = x_i + j^{-1}$ for all $i=0,\dots,m$ and $j\gg 1$. For $j\to\infty$ with $c_j := j^m$ we then get
\begin{align}
g_0(x) &= \lim_{j\to\infty} j^m\cdot\det \begin{pmatrix}
f_0 & f_1 & f_2 & \dots & f_{2m-1} & f_{2m}\\
x & x_1 & x_1 + j^{-1} & \dots & x_m & x_m + j^{-1}
\end{pmatrix}\notag\\
&= \lim_{j\to\infty} j^m\cdot\det \begin{pmatrix}
f_0(x) & \dots & f_{2m}(x)\\
f_0(x_1) & \dots & f_{2m}(x_1)\\
f_0(x_1+j^{-1}) & \dots & f_{2m}(x_1+j^{-1})\\
\vdots & & \vdots\\
f_0(x_m) & \dots & f_{2m}(x_m)\\
f_0(x_m+j^{-1}) & \dots & f_{2m}(x_m+j^{-1})\\
\end{pmatrix}\notag\\
&= \lim_{j\to\infty} \det \begin{pmatrix}
f_0(x) & \dots & f_{2m}(x)\\
f_0(x_1) & \dots & f_{2m}(x_1)\\
\frac{f_0(x_1+j^{-1})-f_0(x_1)}{j^{-1}} & \dots & \frac{f_{2m}(x_1+j^{-1})-f_{2m}(x_1)}{j^{-1}}\\
\vdots & & \vdots\\
f_0(x_m) & \dots & f_{2m}(x_m)\\
\frac{f_0(x_m+j^{-1})-f_0(x_m)}{j^{-1}} & \dots & \frac{f_{2m}(x_m+j^{-1})-f_{2m}(x_m)}{j^{-1}}
\end{pmatrix}\label{eq:generalg0construction}\\
&= \det\begin{pmatrix}
f_0(x) & \dots & f_{2m}(x)\\
f_0(x_1) & \dots & f_{2m}(x_1)\\
f_0'(x_1) & \dots & f_{2m}'(x_1)\\
\vdots & & \vdots\\
f_0(x_m) & \dots & f_{2m}(x_m)\\
f_0'(x_m) & \dots & f_{2m}'(x_m)
\end{pmatrix},\notag
\end{align}
i.e., a double zero at $x_j$ is included by including the values $f_i'(x_j)$, $i=0,\dots,n$.
We will define that procedure and need these definitions for ET-systems in \Cref{ch:etsystems}.
\exmsymbol
\end{rem}

\begin{cor}\label{cor:zeros2}
\Cref{thm:zeros2} also holds for intervals $\cX\subseteq\rset$, i.e.,
\begin{equation}\label{eq:intervals}
\cX = (a,b),\ (a,b],\ [a,b),\ [a,b],\ (a,\infty),\ [a,\infty),\ (-\infty,b),\ (-\infty,b],\ \text{and}\ \rset
\end{equation}
with $a<b$.
\end{cor}
\begin{proof}
We have that ``(i) $\Rightarrow$ (ii)'' follows from \Cref{cor:zeros1}.
For ``(ii) $\Rightarrow$ (i)'' we apply \Cref{thm:zeros2} on $[\min_i x_i,\max_i x_i]$.
\end{proof}

We will now give a sharper version of \Cref{thm:zeros1}, see also \Cref{rem:kreinError}.

\begin{thm}[see e.g.\ {\cite[p.\ 30, Thm.\ 5.2]{karlinStuddenTSystemsBook}}]\label{thm:zeros3}
Let $n\in\nset$ and $\cF$ be a continuous T-system on $\cX=[a,b]$.
Additionally, let $x_1,\dots,x_k\in\cX$ and $y_1,\dots,y_l\in\cX$ be pairwise distinct points.
The following are equivalent:
\begin{enumerate}[(i)]
\item There exists a polynomial $f\in\lin\cF$ such that
\begin{enumerate}[(a)]
\item $x_1,\dots,x_k$ are the non-nodal zeros of $f$ and
\item $y_1,\dots,y_l$ are the nodal zeros of $f$.
\end{enumerate}

\item $2k+l \leq n$.
\end{enumerate}
\end{thm}
\begin{proof}
(i) $\Rightarrow$ (ii): That is \Cref{thm:zeros1}.

(ii) $\Rightarrow$ (i): Adapt the proof and especially the $g_j$'s in (\ref{eq:gjdfns}) of \Cref{thm:zeros2} accordingly.
Let $z_1 < \dots < z_{k+l}$ be the $x_i$'s and $y_i$'s together ordered by size.
Then in $g_j$ treat every nodal $z_i$ like the endpoint $a$ or $b$, i.e., include it only once in the determinant, and insert for every non-nodal point $z_i$ the point $z_i$ and the sequence $z_i^{(j)}\in (z_i,z_{i+1})$ with $\lim_{j\to\infty} z_i^{(j)} = z_i$.
\end{proof}

\begin{cor}
\Cref{thm:zeros3} also holds for sets $\cX\subseteq\rset$ of the form
\begin{enumerate}[(i)]
\item $\cX = (a,b)$, $[a,b)$, $(a,b]$ with $a<b$,

\item $\cX = (a,\infty)$, $[a,\infty)$, $(-\infty,b)$, $(-\infty,b]$,

\item $\cX = \{x_1,\dots,x_k\}\subseteq\rset$ with $k\geq n+1$ and $x_1 < \dots < x_k$, and

\item finitely many unions of (i) to (iii).
\end{enumerate}
\end{cor}
\begin{proof}
In the adapted proof and the $g_j$'s in (\ref{eq:gjdfns}) of \Cref{thm:zeros2} we do not need to have non-negativity, i.e., in the $g_j$'s sign changes at the $y_i$'s are allowed (and even required).
\end{proof}

\section*{Problems}
\addcontentsline{toc}{section}{Problems}

\begin{prob}\label{prob:restriction}
Prove \Cref{cor:restriction}.
\end{prob}

\begin{prob}\label{prob:transf}
Prove \Cref{cor:transf}.
\end{prob}

\begin{prob}\label{prob:scaling}
Prove \Cref{cor:scaling}.
\end{prob}

\begin{prob}\label{prob:4.1}
Prove \Cref{cor:uniqueDeter}.
\end{prob}

\begin{prob}\label{prob:fraction}
Prove \Cref{exm:fraction}.
\end{prob}

\begin{prob}\label{prob:4.2}
Why does (\ref{eq:bigGleichungssystem}) have at least one solution?
\end{prob}

\begin{prob}\label{prob:kreinError}
Assume in \Cref{thm:zeros2} we not only have that $\cF = \{f_i\}_{i=0}^n$ is a T-system of order $n$, but additionally that $\{f_i\}_{i=0}^{n-1}$ is T-systems of order $n-1$. 
Then show that the following are equivalent:
\begin{enumerate}[(i)]
\item The distinct points $x_1,\dots,x_p\in [a,b]$ are the \emph{only} zeros of some non-negative polynomial $f\in\lin\cF$.

\item $\sum_{i=1}^p \varepsilon(x_i) \leq n$.
\end{enumerate}
\end{prob}

\motto{Curiouser and curiouser!\\ \medskip
\ \hspace{1cm} \normalfont{Lewis Carroll: Alice's Adventures in Wonderland}\index{Carroll, L.}}

\chapter{ET- and ECT-Systems}
\label{ch:etsystems}

In this chapter we introduce the concept of ET- and ECT-systems, i.e., \emph{extended} and \emph{extended complete} Tchebycheff systems.
The sparse algebraic polynomial systems on $(0,\infty)$ are the main examples.
Being an ET-system is required for certain Positiv- and Nichtnegativstellensätze in later chapters.

\section{Definitions and Basic Properties}

We remind the reader that a function $f\in\cat^n(\rset,\rset)$ has a \emph{zero} at $x_0\in\rset$ \emph{of multiplicity} (at least) $m$\index{zero!multiplicity} if
\begin{equation}\label{eq:zeroMultipl}
f^{(k)}(x_0) = 0 \qquad\text{for all}\ k=0,1,\dots,m-1.
\end{equation}
For univariate polynomials $f\in\rset[x]$ this translates into a factorization
\begin{equation}\label{eq:zeroFactor}
f(x) = (x-x_0)^m\cdot g(x) \qquad\text{for some}\ g\in\rset[x].
\end{equation}
While the concept of T-systems comes from the univariate polynomials, a relation like (\ref{eq:zeroFactor}) is in general not accessible for T-systems.
Hence, we rely on the more general (analytic) notion (\ref{eq:zeroMultipl}) of multiplicity but still call it \emph{algebraic multiplicity}.
At endpoints of intervals $[a,b]$ we use of course the one-sided derivatives.

\begin{dfn}\label{dfn:etSystem}
Let $n\in\nset$ and let $\cF = \{f_i\}_{i=0}^n\subseteq \cat^n([a,b],\rset)$ be a T-system of order $n$ on $[a,b]$ with $a<b$. $\cF$ is called an \emph{extended Tchebycheff system}\index{T-system!extended} (short \emph{ET-system})\index{ET-system}\index{system!ET-} \emph{on $[a,b]$} if any polynomial $f\in\lin\cF\setminus\{0\}$ has at most $n$ zeros in $[a,b]$ counting algebraic multiplicities.
\end{dfn}

\begin{rem}
It is clear that every ET-system is also a T-system by only allowing multiplicity one for each zero.\exmsymbol
\end{rem}

In \Cref{rem:doublezeros} eq.\ (\ref{eq:generalg0construction}) we showed how double zeros can be included in the determinantal representation.
Whenever we have $\cat^1$-functions in $\cF = \{f_i\}_{i=0}^n$ and
\[ x_0 < \dots < x_i = x_{i+1} < \dots < x_n\]
we define
\begin{equation}\label{eq:doublezeroDfn}
\begin{pmatrix}
f_0 & \dots & f_{i-1} & f_i & f_{i+1} & f_{i+2} & \dots & f_n\\
x_0 & \dots & x_{i-1} & (x_i & x_i) & x_{i+2} & \dots & x_n
\end{pmatrix} := \begin{pmatrix}
f_0(x_0) & \dots & f_n(x_0)\\
\vdots & & \vdots\\
f_0(x_{i-1}) & \dots & f_n(x_{i-1})\\
f_0(x_i) & \dots & f_n(x_i)\\
f_0'(x_i) & \dots & f_n'(x_i)\\
f_0(x_{i+2}) & \dots & f_n(x_{i+2})\\
\vdots & & \vdots\\
f_0(x_n) & \dots & f_n(x_n)
\end{pmatrix}
\end{equation}
and equivalently when $x_j = x_{j+1}$, $x_k = x_{k+1}$, \dots\ for additional entries.

We use the additional brackets ``$($'' and ``$)$'' to indicate that $x_i$ is inserted in the $f_0,\dots,f_n$ and then also into $f_0',\dots,f_n'$ to distinguish (\ref{eq:doublezeroDfn}) from \Cref{dfn:kreinMatrix} to avoid confusion. Hence, in \Cref{dfn:kreinMatrix} we have
\[\det\begin{pmatrix}
f_0 & \dots & f_{i-1} & f_i & f_{i+1} & f_{i+2} & \dots & f_n\\
x_0 & \dots & x_{i-1} & x_i & x_i & x_{i+2} & \dots & x_n
\end{pmatrix} = 0\]
since in two rows $x_i$ is inserted into $f_0,\dots,f_n$, while in (\ref{eq:doublezeroDfn}) we have that
\[\begin{pmatrix}
f_0 & \dots & f_{i-1} & f_i & f_{i+1} & f_{i+2} & \dots & f_n\\
x_0 & \dots & x_{i-1} & (x_i & x_i) & x_{i+2} & \dots & x_n
\end{pmatrix}\]
indicates that $x_i$ is inserted in $f_0,\dots,f_n$ and then also into $f_0',\dots,f_n'$.

Extending this to zeros of multiplicity $m$ for $\cat^{m-1}$-functions is straight forward and we leave it to the reader to write down the formulas.
Similar to (\ref{eq:doublezeroDfn}) we write for any $a\leq x_0 \leq x_1 \leq \dots \leq x_n\leq b$ the matrix as
\begin{equation}\label{eq:matrixStar}
\begin{pmatrix}
f_0 & f_1 & \dots & f_n\\
x_0 & x_1 & \dots & x_n
\end{pmatrix}^*
\end{equation}
when $f_0,\dots, f_n$ are sufficiently differentiable.

We often want to express polynomials $f\in\lin\cF$ as determinants (\ref{eq:gjdfns}) only by knowing their zeros $x_1,\dots,x_k$.
If arbitrary multiplicities appear we only have $x_1 \leq x_2 \leq \dots \leq x_n$ where we include zeros multiple times according to their algebraic multiplicities.
Hence, for
\[x_0 = \dots = x_{i_1} \;<\; x_{i_1+1} = \dots = x_{i_2} \;<\; \dots \;<\; x_{i_k+1} = \dots = x_n\]
we introduce a simpler notation to write down (\ref{eq:doublezeroDfn}):
\begin{equation}\label{eq:doublezeroDfn2}
\left(\begin{array}{c|cccc}
f_0 &\, f_1 & f_2 & \dots & f_n\\
x &\, x_1 & x_2 & \dots & x_n
\end{array}\right) :=
\begin{pmatrix}
f_0 & f_1\; \dots\; f_{i_1} & f_{i_1+1} \;\dots\; f_{i_2} & \dots & f_{i_k+1} \;\dots\; f_{i_k+1}\\
x & (x_1\; \dots\; x_{i_1}) & (x_{i_1+1} \;\dots\; x_{i_2}) & \dots & (x_{i_k+1} \;\dots\; \;\; x_n)\;\,
\end{pmatrix}.
\end{equation}
Clearly $(\ref{eq:doublezeroDfn2}) \in\lin\cF$. For (\ref{eq:doublezeroDfn2}) to be well-defined we need $\cF\subseteq \cat^{m-1}$ where $m$ is the largest multiplicity of any zero.

We see here why we require in \Cref{dfn:etSystem} $\cF = \{f_i\}_{i=0}^n\subseteq\cat^n([a,b],\rset)$.
In the case of \mbox{$x_0 = x_1 = \dots = x_n$} the functions $f_i$ need to be $\cat^n([a,b],\rset)$, not just $\cat^{n-1}([a,b],\rset)$.

Similar to \Cref{lem:determinant} we have the following.

\begin{thm}[\cite{krein51} or e.g.\ {\cite[p.\ 37, P.1.1]{kreinMarkovMomentProblem}}]\label{thm:etDet}
Let $n\in\nset$ and $\cF=\{f_i\}_{i=0}^n\subseteq \cat^n([a,b],\rset)$ with $a<b$. Then the following are equivalent:
\begin{enumerate}[(i)]
\item $\cF$ is an ET-system.

\item We have
\[\det\begin{pmatrix}
f_0 & f_1 & \dots & f_n\\
x_0 & x_1 & \dots & x_n
\end{pmatrix}^* \neq 0\]
for every $a\leq x_0 \leq x_1 \leq \dots \leq x_n\leq b$.
\end{enumerate}
\end{thm}
\begin{proof}
Let $x_0,\dots,x_n\in [a,b]$ with
\[a\;\leq\; x_0 = \dots = x_{i_1} \;<\; x_{i_1+1} = \dots = x_{i_2} \;<\; \dots \;<\; x_{i_k} = \dots = x_n \;\leq\; b\]
be the zeros of some $f = \sum_{i=0}^n a_n f_i\in\lin\cF$.
We get the coefficients $a_0,\dots,a_n$ from the system
\begin{equation}\label{eq:zerosystemET}
0 = \begin{pmatrix}
f(x_0)\\ f'(x_0)\\ \vdots\\ f^{(i_1)}(x_0)\\ f(x_{i_1+1})\\ \vdots\\ f^{(n-i_k)}(x_{i_k})
\end{pmatrix} = \underbrace{\begin{pmatrix}
f_0 & f_1 & \dots & f_n\\
x_0 & x_1 & \dots & x_n
\end{pmatrix}^*}_{=:M}\cdot \begin{pmatrix}
a_0\\ a_1\\ \vdots\\ a_n\end{pmatrix}.
\end{equation}
Hence, since $x_0,\dots,x_n$ are arbitrary we have (i) $\cF$ is an ET-systems $\Leftrightarrow a_0 = \dots = a_n = 0 \Leftrightarrow$ (\ref{eq:zerosystemET}) has only the trivial solution $\Leftrightarrow$ $M$ has full rank $\Leftrightarrow$ (ii).
\end{proof}

\begin{rem}\label{rem:signDetET}
Similar to \Cref{rem:signDet} for T-systems we can assume after a sign change in $f_n$ that for every ET-system $\cF = \{f_i\}_{i=0}^n$ on $[a,b]$ we have that
\[\det\begin{pmatrix}
f_0 & f_1 & \dots & f_n\\
x_0 & x_1 & \dots & x_n
\end{pmatrix}^* > 0\]
holds for all $a\leq x_0 \leq x_1 \leq \dots \leq x_n\leq b$ since $\cF\subseteq\cat^n([a,b],\rset)$.\exmsymbol
\end{rem}

An even more special case of ET-systems and therefore T-systems are the ECT-systems which we define now.

\begin{dfn}
Let $n\in\nset_0$ and let $f_0,\dots,f_n\in\cat^n([a,b],\rset)$ with $a<b$. The family $\cF = \{f_0\}_{i=0}^n$ is called an \emph{extended complete Tchebycheff system}\index{T-system!extended!complete} (short \emph{ECT-system})\index{ECT-system}\index{system!ECT-} \emph{on $[a,b]$} if $\{f_i\}_{i=0}^k$ is an ET-system on $[a,b]$ for all $k=0,\dots,n$.
\end{dfn}

\section{Wronskian Determinant}

To handle and work with ECT-systems it is useful to introduce the following determinant.

\begin{dfn}\label{dfn:wronski}
Let $n\in\nset_0$ and let $f_0,\dots,f_n\in\cat^n([a,b],\rset)$ be with $a<b$.
For each $k=0,\dots,n$ we define the \emph{Wronskian determinant}\index{determinant!Wronskian}\index{Wronskian determiant|see{Wronskian}} (short \emph{Wronskian})\index{Wronskian} $\cW(f_0,\dots,f_k)$ of $f_0,\dots,f_k$ to be
\begin{equation}\label{eq:wronski}
\cW(f_0,f_1,\dots,f_k) := \det\begin{pmatrix}
f_0 & f_0' & \dots & f_0^{(k)}\\
f_1 & f_1' & \dots & f_1^{(k)}\\
\vdots & \vdots & & \vdots\\
f_k & f_k' & \dots & f_k^{(k)}
\end{pmatrix}.
\end{equation}
\end{dfn}

The Wronskian is a common tool in the theory of ordinary differential equations.

In the previous definition (\ref{eq:wronski}) we could also shortly write
\[\cW(f_0,\dots,f_k)(x) := \det\begin{pmatrix}
f_0 & f_1 & \dots & f_k\\ x & x & \dots & x
\end{pmatrix}^*\]
for all $x\in [a,b]$.

Let $m_1,\dots,m_k\in\nset$ with $m_1+\dots+m_k=n+1$ and $x_1<\dots<x_k$. Then the first $m_j$ columns of $\cW(f_0,\dots,f_n)$ are the $m_j$ columns in
\[\begin{pmatrix}
f_0 & \dots & f_{m_1-1} & f_{m_1} & \dots & f_{m_1+m_2-1} & f_{m_1+m_2} & \dots & f_{n}\\
x_1 & \dots & x_1 & x_2 & \dots & x_2 & x_3 & \dots & x_k
\end{pmatrix}^*\]
involving $x_j$.

\begin{lem}\label{lem:etMultiplication}
Let $n\in\nset_0$, let $\cF = \{f_i\}_{i=0}^n$ be an ET-system on $[a,b]$ with $a<b$, and let $g\in\cat^n([a,b],\rset)$ with $g>0$.
Then
\[\cG := \{g_i\}_{i=0}^n \quad\text{with}\quad  g_i := g\cdot f_i\]
is an ET-system and we have
\[\cW(g_0,\dots,g_n) = g^{n+1}\cdot\cW(f_0,\dots,f_n).\]
\end{lem}
\begin{proof}
See Problem \ref{prob:etMultiplication}.
\end{proof}

\begin{lem}\label{lem:ettrans}
Let $n\in\nset_0$, let $\cF = \{f_i\}_{i=0}^n$ be an ET-system on $[c,d]$, and $g\in\cat^n([a,b],[c,d])$ with $g' >0$ on $[a,b]$.
Then
\[\cG := \{f_i\circ g\}_{i=0}^n \quad\text{with}\quad g_i := f_i\circ g\]
is an ET-system on $[a,b]$ with
\[\cW(g_0,\dots,g_n) = (g')^{\frac{n(n+1)}{2}}\cdot \cW(f_0,\dots,f_n)\circ g.\]
\end{lem}
\begin{proof}
See Problem \ref{prob:wronsTrans}.
\end{proof}

For the Wronskian the following reduction property holds.

\begin{lem}[see e.g.\ {\cite[p.\ 377]{karlinStuddenTSystemsBook}}]\label{lem:wronskiReduction}
Let $n\in\nset_0$ and let $f_0,\dots,f_n\in\cat^n([a,b],\rset)$ be with $a<b$ and $f_0>0$.
Then for the \emph{reduced system}\index{reduced!system}\index{system!reduced} $g_0,\dots,g_{n-1}\in\cat^{n-1}([a,b],\rset)$ defined by
\begin{equation}\label{eq:redSysDfn}
g_i := \left( \frac{f_{i+1}}{f_0} \right)'
\end{equation}
for all $i=0,\dots,n-1$ we have
\begin{equation}\label{eq:wronskiReduction}
\cW(f_0,\dots,f_n) = f_0^{n+1}\cdot\cW(g_0,\dots,g_{n-1}).
\end{equation}
\end{lem}
\begin{proof}
See Problem \ref{prob:wronskiReduction}.
\end{proof}

\begin{rem}\label{rem:signWronski}
Since $f_0,\dots,f_n\in\cat^n([a,b],\rset)$ we have that $\cW(f_0,\dots,f_k)(x)$ is continuous in $x\in [a,b]$ and hence after adjusting the signs of $f_0,\dots,f_n$ we have that (\ref{eq:wronski}) being non-zero on $[a,b]$ is equivalent to $\cW(f_0,\dots,f_k)>0$ on $[a,b]$ for all $k=0,\dots,n$, see also \Cref{rem:signDet} and \Cref{rem:signDetET}.\exmsymbol
\end{rem}

\begin{lem}[see e.g.\ {\cite[pp.\ 242--245, Lem.\ 5.1 - 5.3]{karlinStuddenTSystemsBook}}]\label{lem:wronskiFirstCase}
Let $n\in\nset_0$ and let $f_1,\dots,f_n\in\cat^n([a,b],\rset)$ be such that
\[\cW(f_0)>0,\quad\dots,\quad\cW(f_0,\dots,f_n)>0\]
on $[a,b]$.
Define functions $g_0,\dots,g_n:[a,b]\to\rset$ by
%
%
\begin{align*}
g_0 &:= f_0\\
g_1 &:= D_0 f_1\\
g_2 &:= D_1 D_0 f_2\\
&\ \;\;\vdots\\
g_n &:= D_{n-1}\dots D_1 D_0 f_n
\end{align*}
with
\begin{equation}\label{eq:DjDfn}
D_j f := \left(\frac{f}{g_j}\right)',\quad i.e.,\quad D_j = \frac{\diff}{\diff x} \frac{1}{g_j}.
\end{equation}
Then
\begin{enumerate}[(i)]
\item $g_i\in\cat^{n-i}([a,b],\rset)$ are well defined with
\[g_1 = \frac{\cW(f_0,f_1)}{f_0^2} \qquad\text{and}\qquad g_i = \frac{\cW(f_0,\dots,f_i)\cdot\cW(f_0,\dots,f_{i-2})}{\cW(f_0,\dots,f_{i-1})^2}\]
for all $i=2,\dots,n$,

\item $g_i>0$ on $[a,b]$ for all $i=0,\dots,n$,

\item\label{item:wronski} for any $g_{n+1}\in\cat([a,b],\rset)$ with $g_{n+1}>0$ on $[a,b]$ we define
\[f_{n+1}(x) := g_0(x)\int_a^x g_1(y_1) \int_a^{y_1} g_2(y_2)\dots \int_a^{y_n} g_{n+1}(y_{n+1})~\diff y_{n+1}\dots\diff y_1\]
and we get
\[g_{n+1} = D_n\dots D_1 D_0 f_{n+1},\]

\item for all $k=0,\dots,n+1$ we have
\[\cW(f_0,\dots,f_k) = g_0^{k+1} g_1^k\cdots g_k\]
with $g_{n+1}$ and $f_{n+1}$ from (\ref{item:wronski}),

\item there exists a $f_{n+1}\in\cat^{n+1}([a,b],\rset)$ such that
\[\cW(f_0,\dots,f_n,f_{n+1}) > 0\]
on $[a,b]$, and

\item for all $k=0,\dots,n+1$ the families $\{f_i\}_{i=0}^k$ are T-systems on $[a,b]$.\label{item:wronskiFirstCase}
\end{enumerate}
\end{lem}
\begin{proof}
(i) and (ii): Since $\cW(f_0)>0$ we have $f_0>0$ and hence $g_1 = (f_1/f_0)'$ is well-defined and we have
\[\frac{\cW(f_0,f_1)}{f_0^2} = f_0^{-2}\cdot \det\begin{pmatrix}
f_0 & f_0'\\ f_1 & f_1'\end{pmatrix} = \frac{f_0 f_1' - f_1 f_0'}{f_0^2} = \left(\frac{f_1}{f_0}\right)' = g_1,\]
i.e., $g_1>0$ on $[a,b]$.
The relations for $g_i$ for all $i=2,\dots,n$ follow by induction from Sylvester's identity\index{Sylvester's identity}\index{identity!Sylvester} \cite{sylves51,akritas96}.

(iii): From the definition of $f_{n+1}$ we get immediately $g_{n+1} = D_n \dots D_1 D_0 f_{n+1}$.

(iv): Follows immediately from (i).

(v): Take the $f_{n+1}$ from (iii).

(vi): For $k=0$ it is clear that $\{f_i\}_{i=0}^0$ is a T-system since $f_0>0$ on $[a,b]$.
So assume that for any $f_0,\dots,f_{n-1}$ with $\cW(f_0,\dots,f_k)>0$ on $[a,b]$ for all $k=0,\dots,n-1$ we have that all $\{f_i\}_{i=0}^k$ with $k=0,\dots,n-1$ are T-systems.
We show that $\{f_i\}_{i=0}^n$ is also a T-system.
So let $x_0,\dots,x_n\in [a,b]$ with $x_0 < x_1 < \dots < x_n$.
We then have
\begin{align*}
&\det \begin{pmatrix}
f_0 & \dots & f_n\\ x_0 & \dots & x_n
\end{pmatrix} = \det (f_i(x_j))_{i,j=0}^n
\intertext{and factoring out $f_0(x_j)>0$ in each column gives}
&= \prod_{j=0}^n f_0(x_j)\cdot \det \left( \tilde{f}_i(x_j) \right)_{i,j=0}^n
\intertext{with $\tilde{f}_i := f_i/f_0$ for all $i=0,\dots,n$ and substracting from each row its predecessor (the row above) gives}
&= \prod_{j=0}^n f_0(x_j)\cdot \det \left(\delta_{0,j},\tilde{f}_1(x_j) - \tilde{f}_1(x_{j-1}),\dots, \tilde{f}_n(x_j) - \tilde{f}_n(x_{j-1})\right)_{j=0}^n.
\intertext{Expanding along the first column and applying the theorem of the mean gives}
&= \prod_{j=0}^n f_0(x_j)\cdot \prod_{i=0}^{n-1} (x_{i+1} - x_i)\cdot \det \left( \hat{f}_i(y_j) \right)_{i,j=0}^{n-1}
\end{align*}
for some $y_0,\dots,y_{n-1}$ with $x_0 < y_0 < x_1 < y_1 < \dots < y_{n-1} < x_n$ and $\hat{f}_i := (f_{i+1}/f_0)'$ for all $i=0,\dots,n-1$.
The family $\{\hat{f}_i\}_{i=0}^{n-1}$ is the reduced system\index{system!reduced}\index{reduced!system} from \Cref{lem:wronskiReduction} and hence by (\ref{eq:wronskiReduction}) we have
\[\cW(\hat{f}_0,\dots,\hat{f}_{k-1}) = \frac{\cW(f_0,\dots,f_k)}{f_0^{k+1}} > 0\]
on $[a,b]$ for all $k=1,\dots,n$.
By the induction hypothesis we have that $\{\hat{f}_i\}_{i=0}^{n-1}$ is a T-system, i.e.,
\[\det \left( \hat{f}_i(y_j) \right)_{i,j=0}^{n-1} \neq 0 \quad\Rightarrow\quad  \det \begin{pmatrix} f_0 & \dots & f_n\\ x_0 & \dots & x_n \end{pmatrix} \neq 0\]
and $\{f_i\}_{i=0}^n$ is a T-system which ends the proof.
\end{proof}

The previous lemma is used to characterize all ECT-systems.

\section{Characterizations of ECT-Systems}

We have the following characterization of ECT-systems.

\begin{thm}[see e.g.\ {\cite[p.\ 376, Thm.\ 1.1]{karlinStuddenTSystemsBook}}]\label{thm:ectWronski}
Let $n\in\nset_0$ and let $f_0,\dots,f_n\in\cat^n([a,b],\rset)$ be with $a<b$. The following are equivalent:
\begin{enumerate}[(i)]
\item $\cF = \{f_i\}_{i=0}^n$ is an ECT-system.

\item For all $k=0,\dots,n$ we have that $\cW(f_0,\dots,f_k) \neq 0$ on $[a,b]$.
\end{enumerate}
\end{thm}

After adjusting the signs of $f_0,\dots,f_n$ by \Cref{rem:signWronski} we can in \Cref{thm:ectWronski} (ii) also assume that $\cW(f_0,\dots,f_k) > 0$ on $[a,b]$ for all $k=0,\dots,n$.

The following proof is adapted from \cite[pp.\ 376--379]{karlinStuddenTSystemsBook}.

\begin{proof}
(i) $\Rightarrow$ (ii): Since every ECT-system is also an ET-system the statement is \Cref{thm:etDet} (i) $\Rightarrow$ (ii) because
\[\cW(f_0,\dots,f_k)(x) = \begin{pmatrix}
f_0 & f_1 & \dots & f_k\\ x & x & \dots & x
\end{pmatrix}^*\]
for all $x\in [a,b]$.

(ii) $\Rightarrow$ (i): To show that $\cF$ is an ECT-system we have to show that $\{f_i\}_{i=0}^k$ is an ET-system for all $k=0,\dots,n$.
And to show that $\{f_i\}_{i=0}^k$ is an ET-system it is by \Cref{thm:etDet} sufficient to show
\[\det\begin{pmatrix}
f_0 & f_1 & \dots & f_k\\
x_0 & x_1 & \dots & x_k
\end{pmatrix}^* \neq 0\]
for every $a\leq x_0 \leq x_1 \leq \dots \leq x_k\leq b$.
We make two case distinctions:
\begin{enumerate}[\itshape {Case} I:]
\item All $x_0,\dots,x_k$ are pairwise distinct: $x_0 < x_1 < \dots < x_n$.

\item At least once we have $x_j = x_{j+1}$ for some $j=0,\dots,n-1$.
\end{enumerate}
After renaming $x_0,\dots, x_k$ we can assume $a\leq x_1 < x_2 < \dots < x_l \leq b$ and $m_1,\dots,m_l\in\nset$ are the algebraic multiplicities with $m_1 + \dots + m_l = n+1$ for some $l\in\nset_0$.

\textit{Case I:} We have $m_0 = \dots = m_k = 1$ and that is \Cref{lem:wronskiFirstCase} (\ref{item:wronskiFirstCase}).

\textit{Case II:} We assume $m_j\geq 2$ for some $j$. We show that we can reduce the system.

We show this reduction by induction over $n$.

\textit{Induction beginning} ($n=0$): Since $\cW(f_0)(x) \neq 0$ it is an ET- and an ECT-system.
We can assume by changing the sign of $f_0$ that $f_0>0$ on $[a,b]$.

\textit{Induction step} ($n-1\to n$): By the induction beginning ($n=0$) we can assume $f_0>0$ on $[a,b]$.
Then we have to show that
\begin{equation}\label{eq:ectproof1}
\det\begin{pmatrix}
f_0 & f_1 & \dots & f_{m_1-1} & f_{m_1} & \dots & f_n \\
x_1 & x_1 & \dots & x_1 & x_2 & \dots & x_l
\end{pmatrix}^*
\end{equation}
is non-zero.
To show this we factor $f_0(x_j)>0$ out of the $m_j$ rows containing $x_j$ in (\ref{eq:ectproof1}) for each $j=0,\dots,l$ to get
\[\det\begin{pmatrix}
1 & \frac{f_0'}{f_0}(x_1) & \dots & \frac{f_0^{(m_1-1)}}{f_0}(x_1) & 1 & \dots & \frac{f_0^{(m_l-1)}}{f_0}(x_l)\\
\frac{f_1}{f_0}(x_1) & \frac{f_1'}{f_0}(x_1) & \dots & \frac{f_1^{(m_1-1)}}{f_0}(x_1) & \frac{f_1}{f_0}(x_2) & \dots & \frac{f_1^{(m_l-1)}}{f_0}(x_l)\\
\vdots & \vdots & & \vdots & \vdots & & \vdots\\
\frac{f_n}{f_0}(x_1) & \frac{f_n'}{f_0}(x_1) & \dots & \frac{f_n^{(m_1-1)}}{f_0}(x_1) & \frac{f_n}{f_0}(x_2) & \dots & \frac{f_n^{(m_l-1)}}{f_0}(x_l)
\end{pmatrix}.\]
Then subtract from each of the columns containing $x_j$ a linear combination of its predecessors to obtain for these $m_j$ columns the first $m_j$ columns of $\cW(1,f_1/f_0,\dots,f_n/f_0)$ evaluated at $x_j$:
\begin{multline*}
\det\begin{pmatrix}
1 & \left(\frac{f_0}{f_0}\right)'(x_1) & \dots & \left(\frac{f_0}{f_0}\right)^{(m_1-1)}(x_1) & 1 & \dots & \left(\frac{f_0}{f_0}\right)^{(m_l-1)}(x_l)\\
\frac{f_1}{f_0}(x_1) & \left(\frac{f_1}{f_0}\right)'(x_1) & \dots & \left(\frac{f_1}{f_0}\right)^{(m_1-1)}(x_1) & \frac{f_1}{f_0}(x_2) & \dots & \left(\frac{f_1}{f_0}\right)^{(m_l-1)}(x_l)\\
\vdots & \vdots & & \vdots & \vdots & & \vdots\\
\frac{f_n}{f_0}(x_1) & \left(\frac{f_n}{f_0}\right)'(x_1) & \dots & \left(\frac{f_n}{f_0}\right)^{(m_1-1)}(x_1) & \frac{f_n}{f_0}(x_2) & \dots & \left(\frac{f_n}{f_0}\right)^{(m_l-1)}(x_l)
\end{pmatrix}\\
= \det\begin{pmatrix}
1 & 0 & \dots & 0 & 1 & \dots & 0\\
\frac{f_1}{f_0}(x_1) & \left(\frac{f_1}{f_0}\right)'(x_1) & \dots & \left(\frac{f_1}{f_0}\right)^{(m_1-1)}(x_1) & \frac{f_1}{f_0}(x_2) & \dots & \left(\frac{f_1}{f_0}\right)^{(m_l-1)}(x_l)\\
\vdots & \vdots & & \vdots & \vdots & & \vdots\\
\frac{f_n}{f_0}(x_1) & \left(\frac{f_n}{f_0}\right)'(x_1) & \dots & \left(\frac{f_n}{f_0}\right)^{(m_1-1)}(x_1) & \frac{f_n}{f_0}(x_2) & \dots & \left(\frac{f_n}{f_0}\right)^{(m_l-1)}(x_l)
\end{pmatrix}.
\end{multline*}
The Leibniz rule on differentiation, here for us explicitly
\[\left( \frac{f_i}{f_0} \right)^{(k)} = \sum_{j=0}^k \binom{k}{j}\cdot f_i^{(k-j)}\cdot \left( \frac{1}{f_0}\right)^{(j)},\]
ensures that this is always possible.

We then subtract from each column which starts with a $1$ its predecessor which also starts with a $1$ and apply the mean value theorem to get apart from the positive factor $(x_{j+1}-x_j)$
\[\det\begin{pmatrix}
1 & 0 & \dots & 0 & 0 & 0 & \dots & 0\\
\frac{f_1}{f_0}(x_1) & \left(\frac{f_1}{f_0}\right)'(x_1) & \dots & \left(\frac{f_1}{f_0}\right)^{(m_1-1)}(x_1) & \left(\frac{f_1}{f_0}\right)'(y_2) & \left(\frac{f_1}{f_0}\right)'(x_2) & \dots & \left(\frac{f_1}{f_0}\right)^{(m_l-1)}(x_l)\\
\vdots & \vdots & & \vdots & \vdots & & \vdots\\
\frac{f_n}{f_0}(x_1) & \left(\frac{f_n}{f_0}\right)'(x_1) & \dots & \left(\frac{f_n}{f_0}\right)^{(m_1-1)}(x_1) & \left(\frac{f_n}{f_0}\right)'(y_2) & \left(\frac{f_1}{f_0}\right)'(x_2) & \dots & \left(\frac{f_n}{f_0}\right)^{(m_l-1)}(x_l)
\end{pmatrix}\]
with $x_1 < y_2 < x_2 < \dots < x_l$ and expanding by the first row gives
\begin{equation}\label{eq:ectproof2}
\det\begin{pmatrix}
\left(\frac{f_1}{f_0}\right)'(x_1) & \dots & \left(\frac{f_1}{f_0}\right)^{(m_1-1)}(x_1) & \left(\frac{f_1}{f_0}\right)'(y_2) & \left(\frac{f_1}{f_0}\right)'(x_2) & \dots & \left(\frac{f_1}{f_0}\right)^{(m_l-1)}(x_l)\\
\vdots & & \vdots & \vdots & & \vdots\\
\left(\frac{f_n}{f_0}\right)'(x_1) & \dots & \left(\frac{f_n}{f_0}\right)^{(m_1-1)}(x_1) & \left(\frac{f_n}{f_0}\right)'(y_2) & \left(\frac{f_1}{f_0}\right)'(x_2) & \dots & \left(\frac{f_n}{f_0}\right)^{(m_l-1)}(x_l)
\end{pmatrix}.
\end{equation}

In (\ref{eq:ectproof2}) we now have the reduced system $g_i := (f_{i+1}/f_0)'$ with $i=0,\dots,n-1$ from (\ref{eq:redSysDfn}) in \Cref{lem:wronskiReduction}.
By (\ref{eq:wronskiReduction}) in \Cref{lem:wronskiReduction} and since the reduced systems is of dimension $n-1$ where the inductions hypotheses holds we have that (\ref{eq:ectproof2}) is non-zero and hence also (\ref{eq:ectproof1}) is non-zero which we wanted to prove.
\end{proof}

\begin{rem}[see e.g.\ {\cite[p.\ 379, Rem.\ 1.2]{karlinStuddenTSystemsBook}}]\label{rem:zeroECT}
We find the following complete characterization of ECT-systems which requires the additional property (\ref{eq:ectBegRequirement}).
Fortunately, this seemingly additional property can always be generated by a change of basis vectors, i.e., for any vector space spanned by an ECT-system a suitable basis with (\ref{eq:ectBegRequirement}) can be found.\exmsymbol
\end{rem}

\begin{thm}[see e.g.\ {\cite[p.\ 379, Thm.\ 1.2]{karlinStuddenTSystemsBook}}]\label{thm:generalETsystem}
Let $n\in\nset_0$ and let $f_0,\dots,f_n\in\cat^n([a,b],\rset)$ be such that
\begin{equation}\label{eq:ectBegRequirement}
f_j^{(k)}(a) = 0
\end{equation}
holds for all $k=0,\dots,j-1$ and $j=1,\dots,n$. 
After suitable sign changes in $f_0,\dots,f_n$ the following are equivalent:
\begin{enumerate}[(i)]
\item There exist $g_0,\dots,g_n$ with $g_i\in\cat^{n-i}([a,b],\rset)$ and $g_i>0$ on $[a,b]$ for all $i=0,\dots,n$ such that
\begin{align*}
f_0(x) &= g_0(x)\\
f_1(x) &= g_0(x)\cdot \int_a^x g_1(y_1)~\diff y_1\\
f_2(x) &= g_0(x)\cdot \int_a^x g_1(y_1)\cdot\int_a^{y_1} g_2(y_2)~\diff y_2~\diff y_1\\
&\;\,\vdots \\
f_n(x) &= g_0(x)\cdot \int_a^x g_1(y_1)\cdot \int_a^{y_1} g_2(y_2)~ {\dots} \int_a^{y_{n-1}} g_n(y_n)~\diff y_n~ \dots ~\diff y_2~\diff y_1.
\end{align*}

\item $\{f_i\}_{i=0}^n$ is an ECT-system on $[a,b]$.

\item $\cW(f_0,\dots,f_k)>0$ on $[a,b]$ for all $k=0,\dots,n$.
\end{enumerate}
If one and therefore all of the equivalent conditions (i) -- (iii) hold then the $g_i$ in (i) are given by
\[g_0 := f_0 \quad\text{and}\quad g_i := D_{i-1} \dots D_1 D_0 f_i \quad\text{with}\quad D_i := \frac{\diff}{\diff x} \frac{1}{f_0}\]
for all $i=1,\dots,n$ or equivalently by
\[g_0 := f_0,\quad g_1 := \frac{\cW(f_0,f_1)}{f_0^2}, \quad\text{and}\quad
g_i := \frac{\cW(f_0,\dots,f_i)\cdot\cW(f_0,\dots,f_{i-2})}{\cW(f_0,\dots,f_{i-1})^2}\]
for all $i=2,\dots,n$.
\end{thm}
\begin{proof}
``(ii) $\Leftrightarrow$ (iii)'' is \Cref{thm:ectWronski}, ``(iii) $\Rightarrow$ (i)'' is \Cref{lem:wronskiFirstCase} (i) -- (iii), and \mbox{``(i) $\Rightarrow$ (iii)''} is \Cref{lem:wronskiFirstCase} (iv).
\end{proof}

Condition (ii) in \Cref{thm:generalETsystem} is of course to be understood after  suitable sign changes in $f_0,\dots,f_n$.

The partial statement \Cref{thm:generalETsystem} (i) $\Rightarrow$ (ii) can be found e.g.\ in {\cite[p.\ 19, Exm.\ 12]{karlinStuddenTSystemsBook}} and {\cite[pp.\ 39--40, P.2.4]{kreinMarkovMomentProblem}}.

\section{Examples of ET- and ECT-Systems}

An equivalent result as \Cref{cor:restriction} for T-systems, i.e., restricting the domain $\cX$ of a T-system leads again to a T-system, also holds for ET- and ECT-systems.
We leave that to the reader, see Problem \ref{prob:ecRestriction}.
Hence, it is sufficient to give (examples of) ET- and ECT-systems with the largest possible domain $\cX\subseteq\rset$.

While the condition of being an ET-system or being even an ECT-system seems very restrictive, several examples are known.

\begin{exm}\label{exm:algECTsystem}
Let $n\in\nset_0$ and $\cF = \{x^i\}_{i=0}^n$. Then $\cF$ on $\rset$ is an ECT-system.\exmsymbol
\end{exm}
\begin{proof}
Clearly, $\cF\subset\cat^\infty(\rset,\rset)$ and every non-trivial $f\in\lin\cF = \rset[x]_{\leq n}$ has at most $n$ real zeros counting multiplicities by the fundamental theorem of algebra, i.e., $\cF$ is an ET-systems. Besides that we have that
\[\cW(1,x,x^2,\dots,x^k)(x) 
= \det\begin{pmatrix}
1 & 0 & 0 & \dots & 0\\
x & 1 & 0 & \dots & 0\\
x^2 & 2x & 2 & \dots & 0\\
\vdots & \vdots & \vdots & & \vdots\\
x^k & kx^k & k(k-1)x^{k-1} & \dots & k!
\end{pmatrix} \geq 1\]
holds for all $x\in\rset$ and $k=0,\dots,n$ which shows that $\cF$ is also an ECT-system.
\end{proof}

\begin{exm}\label{exm:nonETalg}
Let $\cF= \{1,x,x^3\}$ on $[0,b]$ with $b>0$. Then $\cF$ is a T-system (see \Cref{exm:tsysAlpha}) but not an ET-system.
To see this let $x_0 = x_1 = x_2 = 0$, then
\[\begin{pmatrix}
f_0 & f_1 & f_2\\
0 & 0 & 0
\end{pmatrix}^* = \begin{pmatrix}
1 & 0 & 0\\
0 & 1 & 0\\
0 & 0 & 0
\end{pmatrix}.\]
This shows that $\cF$ is not an ET-system.\exmsymbol
\end{exm}

In the previous example the position $x=0$ prevents the T-system to be an ET-system.
If $x=0$ is removed then it is even an ECT-system.

\begin{exm}\label{exm:algETsystem}
Let $\alpha_0,\dots,\alpha_n\in\nset_0$ with $\alpha_0 < \alpha_1 < \dots < \alpha_n$.
Then $\cF = \{x^{\alpha_i}\}_{i=0}^n$ on $(0,\infty)$ is an ECT-system.
For $n=2m$ and $0 < x_1 < x_2 < \dots < x_m$ we often encounter a specific polynomial structure and hence we write it down explicitly once:
\begin{align}
&\det\begin{pmatrix}
x^{\alpha_0} & x^{\alpha_1} & x^{\alpha_2} & \dots & x^{\alpha_{2m-1}} & x^{\alpha_{2m}}\\
x & (x_1 & x_1) & \dots & (x_m & x_m)
\end{pmatrix}\notag\\
&= \lim_{\varepsilon\to 0} \varepsilon^{-m}\cdot \det\begin{pmatrix}
x^{\alpha_0} & x^{\alpha_1} & x^{\alpha_2} & \dots & x^{\alpha_{2m-1}} & x^{\alpha_{2m}}\\
x & x_1 & x_1+\varepsilon & \dots & x_m & x_m+\varepsilon
\end{pmatrix}\notag\\
&=\lim_{\varepsilon\to 0} \left[\prod_{i=1}^m (x_i-x)(x_i+\varepsilon-x)\right]\cdot \left[\prod_{1\leq i<j\leq m} (x_j-x_i)^2(x_j-x_i-\varepsilon)(x_j+\varepsilon-x_i) \right]\label{eq:schurPolyRepr}\\
&\qquad\times s_\alpha(x,x_1,x_1+\varepsilon,\dots,x_m,x_m+\varepsilon)\notag\\
&= \prod_{i=1}^m (x_i-x)^2 \cdot \prod_{1\leq i<j\leq m} (x_j-x_i)^4\cdot s_\alpha(x,x_1,x_1,\dots,x_m,x_m)\notag
\end{align}
where $s_\alpha$ is the Schur polynomial\index{polynomial!Schur}\index{Schur polynomial} of $\alpha = (\alpha_0,\dots,\alpha_n)$ \cite{macdonSymFuncHallPoly}.
Hence,
\[s_\alpha(x,x_1,x_1,\dots,x_m,x_m)\]
is not divisible by any $(x_i-x)$.
\exmsymbol
\end{exm}
\begin{proof}
Combine the induction
\[f^{(m+1)}(x) = \lim_{h\to 0} \frac{f^{(m)}(x+h) - f^{(m)}(x)}{h}\]
and
\[\det \begin{pmatrix}
x^{\alpha_0} & \dots & x^{\alpha_n}\\
x_0 & \dots & x_n
\end{pmatrix} =
\prod_{0\leq i<j\leq n} (x_j - x_i)\cdot s_\alpha(x_0,\dots,x_n)\]
where $s_\alpha$ is the Schur polynomial of $\alpha = (\alpha_0,\dots,\alpha_n)$.
\end{proof}

With \Cref{thm:generalETsystem} the previous example can be generalized.

\begin{exms}[Examples \ref{exm:tsysAlphaReal} and \ref{exm:tsysExp} continued]\label{exm:etSystems}
Let $n\in\nset_0$ and let
\[-\infty < \alpha_0 < \alpha_1 < \dots < \alpha_n < \infty\]
be reals.
Then
\begin{enumerate}[\bfseries\; (a)]
\item 
$\cF = \{x^{\alpha_0},\dots,x^{\alpha_n}\}$
on $\cX = (0,\infty)$ (\Cref{exm:tsysAlphaReal}) and

\item 
$\cG = \{e^{\alpha_0 x},\dots,e^{\alpha_n x}\}$
on $\cY = \rset$ (\Cref{exm:tsysExp})
\end{enumerate}
are ECT-systems.\exmsymbol
\end{exms}
\begin{proof}
See Problem \ref{prob:etSystExm}.
\end{proof}

In Problem \ref{prob:xf} we will see that also \Cref{exm:xf} are ET- and ECT-systems.

\section{Representation as a Determinant, Zeros, and Non-Negativity}

Similar to \Cref{thm:detRepr} we have the following for ET-systems, i.e., knowing $n$ zeros of a polynomial $f$ counting multiplicities determines $f$ uniquely up to a scalar.\index{representation!as a determinant}\index{determinant!representation as a}

\begin{thm}
Let $n\in\nset_0$ and let $\cF = \{f_i\}_{i=0}^n\subseteq\cat^n([a,b],\rset)$ be an ET-system.
Let $x_1,\dots,x_n\in [a,b]$ with
\[x_1 = \dots = x_{i_1} \;<\; x_{i_1+1} = \dots = x_{i_1+i_2} \;<\; \dots \;<\; x_{i_1+\dots +i_{k-1}+1} = \dots = x_{i_1+\dots+i_k=n}\]
for some $k,i_1,\dots,i_k\in\nset$ and let $f\in\lin\cF$. The following are equivalent:
\begin{enumerate}[(i)]
\item $f^{(l)}(x_j) = 0$ for all $j=1,\dots,k$ and $l=0,\dots,i_j-1$.

\item There exists a constant $c\in\rset$ such that
\[f(x) = c\cdot \det\left(\begin{array}{c|cccc}
f_0 & \, f_1 & f_2 & \dots & f_n\\
x & \, x_1 & x_2 & \dots & x_n
\end{array}\right).\]
\end{enumerate}
\end{thm}
\begin{proof}
(ii) $\Rightarrow$ (i): Clear.

(i) $\Rightarrow$ (ii):
If $f=0$ then $c=0$ so the assertion holds.
If $f\neq 0$ then there exists a point $x_0\in\cX\setminus\{x_1,\dots,x_n\}$ such that $f(x_0)\neq 0$ since $\cF$ is an ET-system.
Then also the determinant in (ii) is non-zero and we can choose $c$ such that both $f$ and the scaled determinant coincide also in $x_0$.
Since $\cF$ is an ET-system we have by \Cref{thm:etDet} that 
\[\begin{pmatrix}
f_0 & f_1 & \dots & f_n\\
x_0 & x_1 & \dots & x_n
\end{pmatrix}^*\]
has full rank, i.e., the coefficients of $f$ and
\[c\cdot \det\left(\begin{array}{c|cccc}
f_0 & \, f_1 & f_2 & \dots & f_n\\
x & \, x_1 & x_2 & \dots & x_n
\end{array}\right)\]
coincide.
\end{proof}

The following result is a strengthened version of \Cref{thm:zeros2}.
It is a small extension of e.g.\ {\cite[p.\ 28, Thm.\ 5.1]{karlinStuddenTSystemsBook}} with explicit multiplicities of the zeros of a non-negative polynomial.

\begin{thm}\label{thm:zeros4}
Let $n\in\nset_0$ and let $\cF=\{f_i\}_{i=0}^n$ be an ET-system on $[a,b]$ with $a<b$.
Let $x_1<\dots<x_k$ in $[a,b]$ and let $m_1,\dots,m_k\in\nset$ for some $k\in\nset$.
The following hold:
\begin{enumerate}[(a)]
\item If $m_1 + \dots + m_k \leq n$ and $m_i\in 2\nset$ for all $x_i\in (a,b)$ then there exists a $f\in\lin\cF$ such that
\begin{enumerate}[(i)]
\item $f\geq 0$ on $[a,b]$,
\item $f$ has precisely the zeros $x_1,\dots,x_k$,
\item the zeros $x_i\in (a,b)$ of $f$ have multiplicity $m_i$,
\item if $x_1 = a$ then $x_1 = a$ has multiplicity $m_1$ or $m_1+1$, and
\item if $x_k = b$ then $x_k = b$ has multiplicity $m_k$ or $m_k+1$.
\end{enumerate}

\item If $\cF$ is an ECT-system or $m_1 + \dots + m_k = n$ then there exists a $f\in\lin\cF$ such that
\begin{enumerate}[(i)]
\item $f\geq 0$ on $[a,b]$,
\item $f$ has precisely the zeros $x_1,\dots,x_k$, and
\item the zeros $x_i$ of $f$ have multiplicity exactly $m_i$.
\end{enumerate}

\end{enumerate}
\end{thm}
\begin{proof}
(a): Set $m := m_1 + \dots + m_k$.
If all $x_1,\dots,x_k\in (a,b)$ and $n = m + p$ for some $p\in\nset_0$ then the polynomial
\begin{align*}
f(x) &= (-1)^p\cdot \det \left(\begin{array}{c|cccccccc}
f_0 & \, f_1 & \dots & f_p & f_{p+1} & \dots & f_{p+m_1} & \dots & f_n\\
x & \, (a & \dots & a) & (x_1 & \dots & x_1) & \dots & x_k)
\end{array}\right)\\
&\quad + \det \left(\begin{array}{c|cccccccc}
f_0 & \, f_1 & \dots & f_{m_1} & \dots & f_{m} & f_{m+1} & \dots & f_n\\
x & \, (x_1 & \dots & x_1) & \dots & x_k) & (b & \dots & b)
\end{array}\right)
\end{align*}
fulfills the requirements.
If $x_1 = a$ and/or $x_k = b$ then include $x_1 = a$ with multiplicity $m_1$ or $m_1+1$ and $x_k = b$ with multiplicity $m_k$ or $m_k+1$.
Use the choice $m_1$ or $m_1 + 1$ resp.\ $m_k$ or $m_k+1$ to let $p\in 2\nset_0$ and add $y$ and $z$ with $x_{k-1} < y < z < x_k$.
Once construct a polynomial with the zeros $x_1,\dots,x_k,y$ with the corresponding multiplicities and add another polynomial with the zeros $x_1,\dots,x_k,z$ with the corresponding multiplicities to it as above.

(b): Use $\{f_i\}_{i=0}^m$ as the ET-system in (a).
\end{proof}

\section*{Problems}
\addcontentsline{toc}{section}{Problems}

\begin{prob}\label{prob:etMultiplication}
Prove \Cref{lem:etMultiplication}.
\end{prob}

\begin{prob}\label{prob:wronsTrans}
Prove \Cref{lem:ettrans}.
\end{prob}

\begin{prob}\label{prob:wronskiReduction}
Prove \Cref{lem:wronskiReduction}.
\end{prob}

\begin{prob}\label{prob:ecRestriction}
\begin{enumerate}[\bfseries (a)]
\item Let $n\in\nset_0$ and let $\cF = \{f_i\}_{i=0}^n$ be an ET-system on $[a,b]$ for some $a<b$.
Show that $\cF$ on $[a',b']$ with $a<a'<b'<b$ is also an ET-system.
\end{enumerate}
\begin{enumerate}[\quad\;\;\;\bfseries (a)]\setcounter{enumi}{1}
\item Show (a) for ECT-systems.
\end{enumerate}
\end{prob}

\begin{prob}\label{prob:xf}
Prove that \Cref{exm:xf} is an ECT-system.
\end{prob}

\begin{prob}\label{prob:etSystExm}
Prove that the \Cref{exm:etSystems} are ECT-systems.
\end{prob}

\begin{prob}\label{prob:exampleECTpolynomial}
Let \[\cF := \{1,x^2,x^3,x^5,x^8,x^{11},x^{13},x^{42}\}\] on $[0,\infty)$. Give an algebraic polynomial $f\in\lin\cF$ such that
\begin{enumerate}[\quad(a)]
\item $f$ is non-negative on $[0,\infty)$,
\item $f$ has $x_1 = 1$ as a zero with multiplicity $m_1 = 2$,
\item $f$ has $x_2 = 3$ as a zero with multiplicity $m_2 = 4$, and
\item $f$ has no zeros in $[0,\infty)$ other than $x_1$ and $x_2$.
\end{enumerate}
\end{prob}

\motto{Life is a short affair;\\
we should try to make it smooth, and free from strife.\\ \medskip
\ \hspace{1cm} \normalfont{Euripides: The Suppliant Women {\cite[p.\ 175]{euripi13}}}\index{Euripides}}

\chapter{Generating ET-Systems from T-Systems by Using Kernels}
\label{ch:ETfromT}

We have seen that ET- and especially ECT-systems have much nicer properties than T-systems.
Therefore, especially for technical reasons, it is desirable to smoothen a T-system into an ET-system.
Usually, a function is smoothed by convolution with e.g.\ the Gaussian kernel.
This procedure is also used for T-systems.

\section{Kernels}

Let $\cX$ and $\cY$ be sets and
\[K:\cX\times \cY\to\rset\]
be a bivariate function, also called \emph{kernel}.\index{kernel}
A family $\{f_i\}_{i=0}^n$ on $\cY$ can then be seen as a special case of $K$ with $\cX =\{0,1,\dots,n\}$, i.e., $f_i = K(i,\,\cdot\,)$ for all $i\in\cX$.
For a kernel $K$ we define the short hand notation
\begin{equation}\label{eq:kernelDetDef}
K\!\begin{pmatrix} x_0 & x_1 & \dots & x_n\\ y_0 & y_1 & \dots & y_n\end{pmatrix} :=
\det (K(x_i,y_j))_{i,j=0}^n.
\end{equation}

\begin{dfn}
Let $k\in\nset_0$, $\cX$ and $\cY$ be ordered sets, and $K:\cX\times\cY\to\rset$ be a kernel.
The kernel $K$ is called \emph{totally positive}\index{totally!positive}\index{positive!totally} (\emph{of order $k$}), short (TP$_k$) property,\index{TPk@TP$_k$} if for all $i=0,1,\dots,k$ we have
\[K\!\begin{pmatrix} x_1 & x_2 & \dots & x_i\\ y_1 & y_2 & \dots & y_i\end{pmatrix} \geq 0\]
for all $x_1 < x_2 < \dots < x_i$, $y_1 < y_2 < \dots < y_i$, and $(x_l,y_m)\in\cX\times\cY$ for all $l,m=1,\dots,i$.
The kernel $K$ is called \emph{strictly totally positive}\index{strictly!totally!positive}\index{positive!totally!strictly} (\emph{of order $k$}), short (STP$_k$),\index{STPk@STP$_k$} if we always have
\[K\!\begin{pmatrix} x_1 & x_2 & \dots & x_i\\ y_1 & y_2 & \dots & y_i\end{pmatrix} > 0.\]
\end{dfn}

For more on sign regular kernels see e.g.\ \cite{karlin68totalPositivity} and \cite{gasca96totalPositivity}.

\begin{cor}[see e.g.\ {\cite[p.\ 10, Exm.\ 3]{karlinStuddenTSystemsBook}}]
Let $n\in\nset_0$, let $K$ be a STP$_{n+1}$ kernel with $\cX = [a,b]$, $\cY = [c,d]$ and $K(x,\,\cdot\,)\in\cat([c,d],\rset)$ for all $x\in\cX$, and let $x_0 < x_1 < \dots < x_n$ in $\cX$.

Then $\{K(x_i,\,\cdot\,)\}_{i=0}^k$ is a continuous T-system on $\cY=[c,d]$ for all $k=0,\dots,n$.
\end{cor}
\begin{proof}
Follows immediately from \Cref{lem:determinant}.
\end{proof}

\begin{dfn}\label{dfn:etpk}
Let $k\in\nset$, $\cX = [a,b]$, $\cY = [c,d]$, and $K:\cX\times\cY\to\rset$ be a kernel such that $K(x,\,\cdot\,)\in\cat^k(\cY,\rset)$ for all $x\in\cX$. We define
\begin{equation}\label{eq:kstarDet}
K^*\!\begin{pmatrix} x_1 & x_2 & \dots & x_k\\ y_1 & y_2 & \dots & y_k\end{pmatrix} :=
\det\begin{pmatrix}
K(x_1,\,\cdot\,) & K(x_2,\,\cdot\,) & \dots & K(x_k,\,\cdot\,)\\
y_1 & y_2 & \dots & y_k\end{pmatrix}^*
\end{equation}
for all $x_1 < x_2 < \dots < x_k$ in $\cX$ and $y_1 \leq y_2 \leq \dots \leq y_k$ in $\cY$.

We say $K$ is \emph{extended totally positive}\index{extended!totally!positive}\index{positive!totally!extended} (\emph{of order $k$}), short ETP$_k$,\index{ETPk@ETP$_k$} if for all $i=1,2,\dots,k$ we have
\[K^*\!\begin{pmatrix} x_1 & x_2 & \dots & x_i\\ y_1 & y_2 & \dots & y_i\end{pmatrix}>0\]
for all $x_1 < x_2 < \dots < x_i$ in $\cX$ and $y_1 \leq y_2 \leq \dots \leq y_i$ in $\cY$.
\end{dfn}

\begin{cor}[see e.g.\ {\cite[p.\ 10, Exm.\ 3]{karlinStuddenTSystemsBook}}]
Let $n\in\nset_0$, let $K$ be an ETP$_{n+1}$ kernel with $\cX = [a,b]$, $\cY = [c,d]$ and $K(x,\,\cdot\,)\in\cat^n([c,d],\rset)$ for all $x\in\cX$, and let $x_0 < x_1 < \dots < x_n$ in $\cX$.

Then $\{K(x_i,\,\cdot\,)\}_{i=0}^n$ is an ECT-system on $\cY=[c,d]$.
\end{cor}
\begin{proof}
Follows immediately from \Cref{thm:etDet}.
\end{proof}

\begin{exm}
Let $\cX = \rset$, $\cY = [a,b]\subset (0,\infty)$, and $K(x,y) = y^x$. Then $K$ is ETP$_k$ for all $k\in\nset_0$.\exmsymbol
\end{exm}
\begin{proof}
Follows immediately from \Cref{exm:etSystems}.
\end{proof}

\begin{exm}[see e.g.\ {\cite[p.\ 11, Exm.\ 5]{karlinStuddenTSystemsBook}}]\label{exm:gaussiankernel}
For any $\sigma>0$ the Gaussian kernel\index{kernel!Gaussian}\index{Gaussian kernel}
\begin{equation}\label{eq:gaussiankernel}
K_\sigma(x,y) := \frac{1}{\sqrt{2\pi\sigma^2}}\exp\left( -\frac{1}{2}\left(\frac{x-y}{\sigma}\right)^2\right) \qquad\text{on\; $\cX \times \cY = \rset^2$}
\end{equation}
is ETP$_k$ for any $k\in\nset$.
\end{exm}

The proof is adapted from \cite[p.\ 11]{karlinStuddenTSystemsBook}.

\begin{proof}
It is sufficient to show that $K(x,y) = e^{-(x-y)^2}$ is ETP$_k$ for all $k\in\nset_0$.

In Example \ref{exm:etSystems} (b) we have seen that $\{e^{\alpha_i x}\}_{i=0}^n$ is an ECT-system on $\rset$ for all $n\in\nset_0$ and all $\alpha_0 < \alpha_1 < \dots < \alpha_n$ in $\rset$.
Hence, by writing
\[f_n(x) := \sum_{i=0}^n a_n\cdot e^{-(x_i-x)^2}
\qquad\text{as}\qquad
f_n(x) = e^{-x^2}\cdot\sum_{i=0}^n a_i\cdot e^{-x_i^2}\cdot e^{2x_i x}\]
we see that $f_n$ has at most $n$ zeros (counting multiplicities) in $\rset$ if $a_0,\dots,a_n\in\rset$ with $a_0^2 + \dots + a_n^2 > 0$.
\end{proof}

\section{The Basic Composition Formulas}

The following equations (\ref{eq:basic1}) and (\ref{eq:basic2}) are the \emph{basic composition formulas}.\index{basic composition formulas}

\begin{lem}[see e.g.\ {\cite[pp.\ 13--14, Exm.\ 8]{karlinStuddenTSystemsBook}}]\label{lem:basiccomposition}
Let $K:[a,b]\times [c,d]\to\rset$ and $L:[c,d]\times[e,f]\to\rset$ be kernels. Let $\mu$ be a $\sigma$-finite measure such that $M(x,z)$ defined by
\[M:[a,b]\times [e,f]\to\rset,\quad M(x,z) := \int_c^d K(x,y)\cdot L(y,z)~\diff\mu(y)\]
exists for all $(x,z)\in [a,b]\times [e,f]$.
The following hold:
\begin{enumerate}[(i)]
\item $M$ is a kernel.

\item For all $k\in\nset$, $x_1 < \dots < x_k$ in $[a,b]$, and $z_1 < \dots < z_k$ in $[e,f]$ we have
\begin{multline}\label{eq:basic1}
M\!\begin{pmatrix} x_1 & \dots & x_k\\ z_1 & \dots & z_k\end{pmatrix} =\\
\underset{c\leq y_1 < \dots < y_k\leq d}{\int\;\,\dots\;\,\int} K\!\begin{pmatrix}
x_1 & \dots & x_k\\ y_1 & \dots & y_k\end{pmatrix}\cdot L\!\begin{pmatrix}
y_1 & \dots & y_k\\ z_1 & \dots & z_k\end{pmatrix}~\diff\mu(y_1)\dots\diff\mu(y_k).
\end{multline}

\item If $L(y,\,\cdot\,)\in\cat^{k-1}([e,f],\rset)$ for some $k\in\nset$ and
\begin{equation}\label{eq:partialderivintegr}
\partial_z^i M(x,z) := \int_c^d K(x,y)\cdot \partial_z^i L(y,z)~\diff\mu(y)
\end{equation}
holds for all $i=0,\dots,k-1$ then
\begin{multline}\label{eq:basic2}
M^*\!\begin{pmatrix} x_1 & \dots & x_k\\ z_1 & \dots & z_k\end{pmatrix} =\\
\underset{c\leq y_1 < \dots < y_k\leq d}{\int\;\,\dots\;\,\int} K\!\begin{pmatrix}
x_1 & \dots & x_k\\ y_1 & \dots & y_k\end{pmatrix}\cdot L^*\!\begin{pmatrix}
y_1 & \dots & y_k\\ z_1 & \dots & z_k\end{pmatrix}~\diff\mu(y_1)\dots\diff\mu(y_k)
\end{multline}
for all $x_1 < \dots < x_k$ in $[a,b]$, and $z_1 \leq \dots \leq z_k$ in $[e,f]$.
\end{enumerate}
\end{lem}
\begin{proof}
(i) is clear, (ii) follows by straight forward calculations, see e.g.\ \cite[p.\ 48, No.\ 68]{polya70}, and (iii) follows from (ii) with (\ref{eq:partialderivintegr}).
\end{proof}

\section{Smoothing T-Systems into ET-Systems}

With the Gaussian kernel from \Cref{exm:gaussiankernel} we get from \Cref{lem:basiccomposition} the following smoothing result.

\begin{cor}[see e.g.\ {\cite[p.\ 15]{karlinStuddenTSystemsBook}}]\label{thm:ETfromT}
Let $n\in\nset_0$ and $\cF = \{f_i\}_{i=0}^n$ be a continuous T-system on $[a,b]$.
For any $\sigma>0$ let
\[K_\sigma(x) := \frac{1}{\sqrt{2\pi\sigma^2}}\exp\left( -\frac{1}{2}\left(\frac{x}{\sigma}\right)^2\right) \qquad\text{on\; $\cX = \rset$}\]
be the Gaussian kernel and define $f_{i,\sigma} := f_i*K_\sigma$ for all $i=0,\dots,n$.
Then $\cF_\sigma := \{f_{i,\sigma}\}_{i=0}^n$ is an ET-system.
\end{cor}
\begin{proof}
See Problem \ref{prob:smoothing}.
\end{proof}

If $\cF$ is a continuous T-system on $[a,b]$ then
\[\lim_{\sigma\searrow 0} f_{i,\sigma}(x) = f_i(x)\]
for all $x\in (a,b)$ and $i=0,\dots,n$.

\begin{cor}
If $\{f_i\}_{i=0}^k$ in \Cref{thm:ETfromT} is a T-system for all $k=0,\dots,n$ then $\cF_\sigma$ is an ECT-system.
\end{cor}
\begin{proof}
Apply \Cref{thm:ETfromT} for every $k=0,1,\dots,n$.
\end{proof}

Approximating a T-system by ET-systems with the Gaussian kernel is often used \cite{gantma50,schoen53,karlin68totalPositivity}, see also \cite[p.\ 16]{karlinStuddenTSystemsBook}.
We will need it in the proof of \Cref{thm:karlin}.

\section*{Problems}
\addcontentsline{toc}{section}{Problems}

\begin{prob}\label{prob:smoothing}
Prove \Cref{thm:ETfromT} from \Cref{lem:basiccomposition}.
\end{prob}

\part{Karlin's Positivstellensätze and Nichtnegativstellensätze}
\label{part:karlinPosNonNeg}

\motto{Beauty is the first test: there is no permanent place\\ in this world for ugly mathematics.\\ \medskip
\ \hspace{1cm} \normalfont{Godfrey Harold Hardy {\cite[§10, p.\ 85]{hardy67}}}\index{Hardy, G.\ H.}}

\chapter{Karlin's Positivstellensatz and Nichtnegativstellensatz on $[a,b]$}
\label{ch:karlinPosab}

We now come to the main result (\Cref{thm:karlin}) and its variations: \Cref{thm:karlinPosab} for T-systems on $[a,b]$ and \Cref{thm:karlinNonNegab} for ET-systems on $[a,b]$.
Earlier versions were already developed in \cite{karlin53}.
Both results are used in the following chapters to prove \Cref{thm:karlinPos0infty} for T-systems on $[0,\infty)$, \Cref{thm:karlinNonNeg0infty} for ET-systems on $[0,\infty)$, \Cref{thm:karlinPosR} for T-systems on $\rset$, and finally \Cref{thm:karlinNonNegR} for ET-systems on $\rset$.

The main applications and examples will be the various sparse algebraic Positivstellensätze and sparse algebraic Nichtnegativstellensätze in \Cref{part:algPos}.

\section{Karlin's Positivstellensatz for T-Systems on $[a,b]$}

For the following main result we remind the reader what it means that a set has an index, see \Cref{dfn:index}: If $x\in (a,b)$ then its index is $2$ and if $x = a$ or $b$ then its index is $1$.
The following result is due to Karlin and we name it therefore after him.

\begin{karthm}[for {$f>0$ on $[a,b]$}; {\cite[Thm.\ 1]{karlin63}} or e.g.\ {\cite[p.\ 66, Thm.\ 10.1]{karlinStuddenTSystemsBook}}]\label{thm:karlin}\index{Karlin!Theorem!for $f>0$ on $[a,b]$}\index{Theorem!Karlin!for $f>0$ on $[a,b]$}
Let $n\in\nset_0$, $\cF = \{f_i\}_{i=0}^n$ be a continuous T-system of order $n$ on $[a,b]$ with $a<b$, and let $f\in \cat([a,b],\rset)$ with $f>0$ on $[a,b]$ be a strictly positive continuous function. The following hold:
\begin{enumerate}[(i)]
\item There exists a unique polynomial $f_*\in\lin\cF$ such that
\begin{enumerate}[(a)]
\item $f(x) \geq f_*(x) \geq 0$ for all $x\in [a,b]$,

\item $f_*$ vanishes on a set with index $n$,

\item the function $f-f_*$ vanishes at least once between each pair of adjacent zeros of $f_*$,

\item the function $f-f_*$ vanishes at least once between the larges zero of $f_*$ and the end point $b$, and

\item $f_*(b)>0$.
\end{enumerate}

\item There exists a unique polynomial $f^*\in \lin\cF$ which satisfies the conditions \mbox{(a) to (d)} of (i) and
\begin{enumerate}[(a')]\setcounter{enumii}{4}
\item $f^*(b) = 0$.
\end{enumerate}
\end{enumerate}
\end{karthm}

Examples of $f_*$ and $f^*$ are depicted in \Cref{fig:karlinPos} for an odd and an even $n$.
\begin{figure}[t]\centering
\begin{subfigure}{0.9\textwidth}
\includegraphics[width=\textwidth]{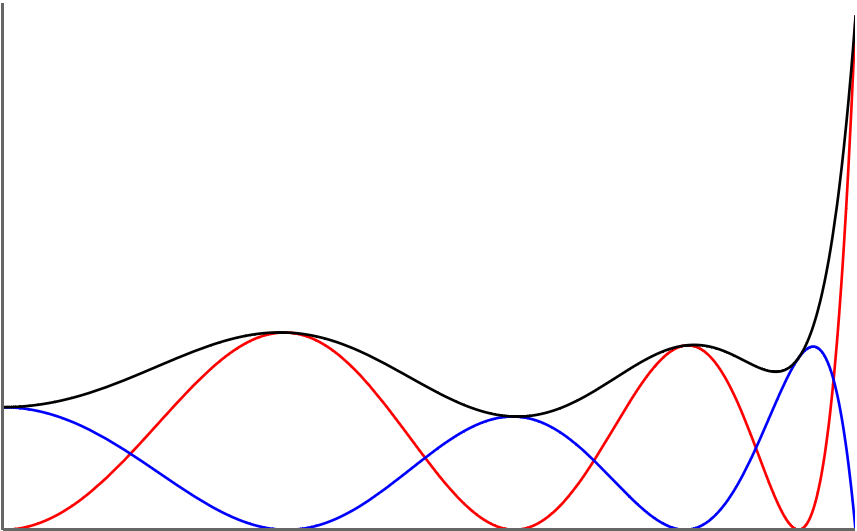}
\caption{$n=5$}
\end{subfigure}

\begin{subfigure}{0.9\textwidth}
\includegraphics[width=\textwidth]{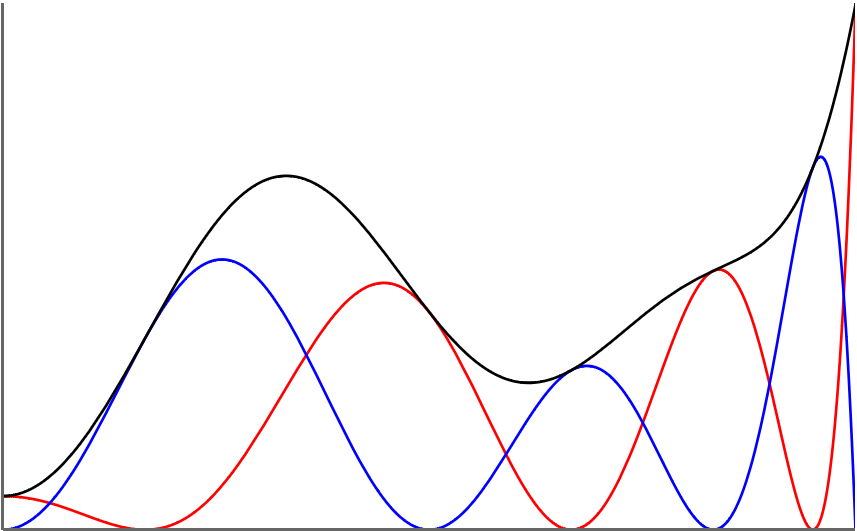}
\caption{$n=6$}
\end{subfigure}
\caption{The functions $f\in\cat([a,b],\rset)$ with $f>0$ (black), $f_*\in\lin\cF$ (red), and $f^*\in\lin\cF$ (blue) from the \Cref{thm:karlinPosab} with $n=5$ and $n=6$.\label{fig:karlinPos}}
\end{figure}

The proof is taken from \cite[pp.\ 68--71]{karlinStuddenTSystemsBook}.
The proof constructs the polynomials $f_*$ and $f^*$ by using the Fixed Point Theorem of Brouwer \index{Theorem!Brouwer Fixed Point}\index{Fixed Point Theorem of Brouwer}\cite[Satz 4]{brouwer11}, see also e.g.\ \cite[Prop.\ 2.6]{zeidlerNonlinFuncAna1}.\footnote{Note that in \cite{zeidlerNonlinFuncAna1} the work \textit{\"Uber Abbildungen von Mannigfaltigkeiten} \cite{brouwer11} is incorrectly dated in the references and Proposition 2.6 on p.\ 52 to the year 1912 while the paper actually appeared in 1911 in the \textit{Mathematische Annalen}. However, we also want to point out that Zeidler gives three proofs of the Fixed Point Theorem of Brouwer, including a constructive one in \cite[pp.\ 254--255, Problem 6.7e]{zeidlerNonlinFuncAna1}.}

\begin{proof}
We distinguish three different cases.

\textit{Case 1:} Let $n=2m$ and let $\cF$ be an ET-system.
We construct $f_*$ in (i) as follows.
For each point $\xi=(\xi_0,\dots,\xi_m)$ in the $m$-dimensional simplex
\begin{equation}\label{eq:simplex}
\Xi^m := \left\{(\xi_0,\dots,\xi_m)\in\rset^{m+1} \,\middle|\, \xi_i\geq 0,\ i=0,1,\dots,m,\ \sum_{i=0}^m \xi_i = b-a\right\}
\end{equation}
set
\[x_i := a + \sum_{k=0}^{i-1} \xi_k\]
for all $i=0,\dots,m$ and define
\begin{equation}\label{eq:fxi}
f_\xi(x) := c_\xi\cdot\det \left(\begin{array}{c|ccccc}
f_0 &\, f_1 & f_2 & \dots & f_{n-1} & f_n\\
x &\, x_1 & x_1 & \dots & x_m & x_m
\end{array}\right)
\end{equation}
with $c_\xi\in\rset$ such that $f_\xi = \sum_{i=0}^n a_i f_i\geq 0$ on $[a,b]$ with $a_0^2 + \dots + a_n^2 = 1$.
If $p$ of the points $x_i$ coincide, this common point is to have multiplicity $2p$.

Define
\begin{equation}\label{eq:deltaiDfn}
\delta_i(\xi):=\min\{\delta\geq 0\,|\,\delta\cdot f\geq u_\xi\ \text{on}\ [x_i,x_{i+1}]\}
\end{equation}
for all $i=0,\dots,m$ with $x_0 = a$ and $x_{m+1} = b$.
The coefficients $a_i(\xi)$ are continuous in $\xi$ and hence the functions $\delta_i(\xi)$ are continuous in $\xi$.

Next, define
\begin{equation}\label{eq:FiDfn}
F_i(\xi) := \delta_i(\xi) - \min_{k} \delta_k(\xi)
\end{equation}
for all $i=0,\dots,m$ and set $F_{m+1}(\xi) := F_0(\xi)$.
If there does not exist a point $\xi$ such that $F_i(\xi) = 0$ for all $i=0,\dots,m$, then $\sum_{i=0}^m F_i(\xi) >0$ for all $\xi\in\Xi^m$.
In this event the continuous mapping
\[\cdot\,':\Xi^m\to\Xi^m,\ \xi\mapsto\xi'\quad \text{with}\quad \xi_i' := \frac{F_{i+1}(\xi)}{\sum_{k=0}^m F_k(\xi)}\cdot (b-a)\]
for all $i=0,\dots,m$ is well-defined.
The Fixed Point Theorem of Brouwer affirms the existence of a point $\xi^*\in\Xi^m$ for which
\begin{equation}\label{eq:fixedPoint}
\xi_i^* := \frac{F_{i+1}(\xi^*)}{\sum_{k=0}^m F_k(\xi^*)}\cdot (b-a)
\end{equation}
for all $i=0,\dots,m$.
By (\ref{eq:FiDfn}) we have that for any $\xi\in\Xi^m$ we have $F_i(\xi)=0$ for some $i$.
Suppose $F_j(\xi^*)=0$ for some fixed $j=0,\dots,m$.
Then (\ref{eq:fixedPoint}) implies $\xi_{j-1}^*=0$.
By (\ref{eq:deltaiDfn}) and (\ref{eq:FiDfn}) imply $F_{j-1}(\xi^*)=0$.
Continuing in this way we get $F_i(\xi^*) = 0$ for all $i=0,\dots,j$ and since $F_{m+1}(\xi)=F_0(\xi)$ we have $F_i(\xi^*) = 0$ for all $i=0,\dots,m$.
But this contradicts our assumption $\sum_{i=0}^m F_i(\xi^*) >0$.
Therefore, there exists at least one point $\xi^*\in\Xi^m$ such that $\delta_i(\xi^*)=\delta$ for all $i=0,\dots,m$.
Since $f_\xi\neq 0$ it follows that $\delta>0$ and hence all $x_i$ are distinct, i.e.,
\[a = x_0 < x_1 < \dots < x_m = b.\]
Hence, $f_* := \delta^{-1}\cdot f_{\xi^*}$ by the nature of its construction fulfills the requirements \mbox{(a) -- (e)} of (i).

For $f^*$ we let $x_0 = a$ and $x_m = b$ and we define similar to (\ref{eq:fxi}) the polynomial
\[g_\xi(x) := d_\xi\cdot\det \left(\begin{array}{c|ccccccc}
f_0 &\, f_1 & f_2 & f_3 & \dots & f_{n-2} & f_{n-1} & f_n\\
x &\, a & x_1 & x_1 & \dots & x_{m-1} & x_{m-1} & b
\end{array}\right).\]
Repeating the arguments from above we get $f^*$ which fulfills (a) -- (d) and (e') in (ii).

\textit{Case 2:} Let $n=2m+1$ and let $\cF$ be an ET-system. Similar to case 1, we define the polynomials
\[f_\xi(x) := d_\xi\cdot\det \left(\begin{array}{c|cccccc}
f_0 &\, f_1 & f_2 & f_3 & \dots & f_{n-1} & f_n\\
x &\, a & x_1 & x_1 & \dots & x_{m} & x_{m}
\end{array}\right).\]
and
\[g_\xi(x) := d_\xi\cdot\det \left(\begin{array}{c|cccccc}
f_0 &\, f_1 & f_2 & \dots & f_{n-2} & f_{n-1} & f_n\\
x &\, x_1 & x_1 & \dots & x_{m} & x_{m} & b
\end{array}\right).\]
Repeating the procedure of case 1 gives the statement.

\textit{Case 3:} Let $n=2m$ and $\cF$ be a T-systems. Then we consider the functions
\[f_i(x;\sigma) := \int_a^b K_\sigma(x,y)\cdot f_i(y)~\diff y\]
where
\[K_\sigma(x,y) := \frac{1}{\sigma\cdot\sqrt{2\pi}}\exp\left[ -\frac{1}{2}\left( \frac{x-y}{\sigma} \right)^2 \right]\]
with $\sigma>0$, see \Cref{ch:ETfromT}.
By \Cref{thm:ETfromT} we have that $\cF_\sigma := \{f_i(\,\cdot\,;\sigma)\}_{i=0}^n$ is an ET-system on $[a,b]$ and hence also on any subinterval $[a',b']$ with $a<a'<b'<b$.
The need to restrict the system $\cF_\sigma$ to the proper interval $[a',b']$ is due to the annoyance that at the end points $x = a$ and $x = b$ we have
\[\lim_{\sigma\searrow 0} f_i(x;\sigma) = \frac{1}{2}f_i(x)\]
for all $i=0,\dots,n$ while for $x\in (a,b)$ we have
\[\lim_{\sigma\searrow 0} f_i(x;\sigma) = f_i(x).\]

From the cases 1 and 2 we find that for any $\sigma>0$ we have a polynomial $f_{*,\sigma}$ satisfying conditions (a) -- (e) of (i) on the interval $[a',b']$.
If
\[f_{*,\sigma} = \sum_{i=0}^n a_i(\sigma)\cdot f_i(\,\cdot\,,\sigma)\]
we can chose a sequence $\sigma_k\searrow 0$ and let $x_1^{(k)},\dots,x_m^{(k)}$ be the zeros of $f_{*,\sigma_k}$.
Additionally, let $y_1^{(k)},\dots,y_{m+1}^{(k)}$ be the points which interlace with $\{x_i^{(k)}\}_{i=0}^m$, i.e., $a' < y_1^{(k)} < x_1^{(k)} < \dots < x_m^{(k)} < y_{m+1}^{(k)}\leq b'$ and satisfying $f(y_i^{(k)}) = f_{*,\sigma_k}(y_i^{(k)})$ for all $i=0,\dots,m+1$.

Since $f(x) \geq f_{*,\sigma}\geq 0$ on $[a',b']$ and solving the system of equations
\[f_{*,\sigma}(x_j) = \sum_{i=0}^n a_i(\sigma)\cdot f_i(x_j;\sigma)\]
for $i=0,\dots,n$ we find that these quantities are uniformly bounded.
We now select a subsequence $\{\sigma_{k'}\}$ from $\{\sigma_k\}$ with the property that as $k'\to\infty$ we obtain
\begin{align*}
a_i(\sigma_{k'}) &\to a_i &\text{for all}\ i=0,\dots,n,\\
y_j^{(k')} &\to y_j &\text{for all}\ j=1,\dots,m+1,\\
x_l^{(k')} &\to x_l &\text{for all}\ l=1,\dots,m
\end{align*}
and
\[a' \leq y_1 \leq x_1 \leq \dots \leq x_m \leq y_{m+1} \leq b'.\]

The function $f_{*,a',b'} := \sum_{i=0}^n a_i\cdot f_i$ vanishes at all $x_l$, $l=1,\dots,m$, and equals $f$ at all $y_j$, $j=1,\dots,m+1$.
Therefore, since $f_{*,a',b'}$ is continuous we see that
\[a' \leq y_1 < x_1 < \dots < x_m < y_{m+1}\leq b'.\]
Hence, $f_{*,a',b'}$ satisfies (a) -- (e) of (i) on the interval $[a',b']$.

Performing a last limiting procedure letting $a'\searrow a$ and $b'\nearrow b$ we obtain a polynomial $f_*$ satisfying (a) -- (e) in (i) on the full interval $[a,b]$.

For $f^*$ the same procedure gives the desired polynomial satisfying the conditions (a) -- (d) and (e').

\textit{Uniqueness of $f_*$ and $f^*$:} Let $n=2m$.
Observe that if another polynomial $\tilde{f}_*$ with properties (a) -- (e) exists then it must have $m$ interior zeros $\tilde{x}_1,\dots,\tilde{x}_m$. Denote by $x_1,\dots,x_m$ the zeros of $f_*$.
Without loss of generality we can assume that either $\tilde{x}_1 < x_1$ or $\tilde{x}_1 = x_1$ and $f_* - \tilde{f}_*$ is non-negative in some interval $(x_1-\varepsilon,x_1)$.
Otherwise we interchange the roles of $f_*$ and $\tilde{f}_*$.
We count the zeros of $g := f_* - \tilde{f}_*$.
We say $g$ has a zero in the closed interval $[c,d]$ if 
\begin{itemize}
\item $g(t_0) = 0$ for $t_0\in (c,d)$,

\item $g(c) = 0$ and $g\geq 0$ on $(c,c+\varepsilon)$, or

\item $g(d) = 0$ and $g\geq 0$ on $(d-\varepsilon,d)$.
\end{itemize}
Counting zeros in this fashion we see that $g$ has at least two zeros in each of the intervals $[x_{i-1},x_i]$ for $i=1,\dots,m$ where $x_0 = a$ and at least one in the interval $[x_m,b]$.
In total $g$ vanishes at least $n+1$ times.
Notice, that certain non-nodal zeros of $g$ have been counted twice and hence by \Cref{thm:zeros1} we have $g = 0$.

In a similar way we get uniqueness of $f^*$ and also in the case $n=2m+1$.
\end{proof}

Note, in the previous result we do not need to have $f\in\lin\cF$. The function $f$ only needs to be continuous and strictly positive on $[a,b]$.

An earlier version of (or at least connected to) \Cref{thm:karlin} combined with \Cref{thm:zeros1} (which was used in the proof of \Cref{thm:karlin}) is a lemma by Markov\index{Markov Lemma}\index{Lemma!Markov} \cite{markov84}, see also \cite[p.\ 80]{shohat43}.

\begin{lem}[\cite{markov84}, see also {\cite[p.\ 80]{shohat43}}]
Let $m\in\nset$ and let $f\in \cat^{n+1}([a,b],\rset)$ be such that $f>0$ and $f^{(k)}\geq 0$ for all $k=1,\dots,m+1$ in $[a,b]$.
Let $p_m\in\rset[x]_{\leq m}$ and $c\in (a,b)$.
Let $m_1\in\nset$ be the number of zeros in $(a,c)$ of the function $f-p_m$ and $m_2$ be the number of zeros of $p_m$ in $(c,b)$, both counted with multiplicity. Then $m_1+m_2\leq m+1$.
\end{lem}

\Cref{thm:karlin} is of course much more general.
As a consequence of \Cref{thm:karlin} we get Karlin's Positivstellensatz for T-systems on $[a,b]$.

\begin{karposthm}[for T-Systems on {$[a,b]$}; see {\cite[Cor.\ 1]{karlin63}} or e.g.\ {\cite[p.\ 71, Cor.\ 10.1(a)]{karlinStuddenTSystemsBook}}]\index{Karlin!Positivstellensatz!on $[a,b]$}\index{Positivstellensatz!Karlin!on $[a,b]$}\index{Theorem!Karlin!Positivstellensatz on $[a,b]$}\label{thm:karlinPosab}
Let $n\in\nset_0$, let $\cF$ be a continuous T-system of order $n$ on $[a,b]$ with $a<b$, and let $f\in\lin\cF$ with $f>0$ on $[a,b]$.
Then there exists a unique representation
\[ f = f_* + f^*\]
with $f_*,f^*\in\lin\cF$ such that
\begin{enumerate}[(i)]
\item $f_*,f^*\geq 0$ on $[a,b]$,

\item the zeros of $f_*$ and $f^*$ each are sets of index $n$,

\item the zeros of $f_*$ and $f^*$ strictly interlace,

\item $f_*(b) = f(b) >0$, and

\item $f^*(b) = 0$.
\end{enumerate}
\end{karposthm}
\begin{proof}
Let $f_*$ be the unique $f_*$ from \Cref{thm:karlin}(i). Then $f - f_*\in\lin\cF$ is a polynomial and fulfills (a) -- (d), and (e') of $f^*$ in \Cref{thm:karlin}.
But since also $f^*$ is unique we have $f - f_* = f^*$.
\end{proof}

\section{The Snake Theorem: An Interlacing Theorem}

In \Cref{thm:karlin} a polynomial $f_*\in\lin\cF$ was found with $0\leq f_*\leq f$ for some given $f\in \cat([a,b],\rset)$ with $f>0$ on $[a,b]$.
This can be extended to find a function $f_*\in\lin\cF$ between some $g_1,g_2\in\cat([a,b],\rset)$ as the following result shows.
In \cite[p.\ 368, Thm.\ 6.1]{karlinStuddenTSystemsBook} M.\ G.\ Krein\index{Krein, M.\ G.} and A.\ A.\ Nudel'man\index{Nudel'man, A.\ A.} called it the \emph{Snake Theorem} which is an accurate description of its graphical representation, see \Cref{fig:g1g2g}.

\begin{snakethm}[{\cite[Thm.\ 2]{karlin63}} or e.g.\ {\cite[p.\ 72, Thm.\ 10.2]{karlinStuddenTSystemsBook}} and {\cite[p.\ 368, Thm.\ 6.1]{kreinMarkovMomentProblem}}]\label{thm:g1g2g}\index{Interlacing Theorem|see{Snake Theorem}}\index{Theorem!Interlacing|see{Snake}}\index{Theorem!Snake}\index{Snake Theorem}
Let $n\in\nset_0$, $\cF = \{f_i\}_{i=0}^n$ be a continuous T-system of order $n$ on $[a,b]$ with $a<b$, and let $g_1,g_2\in\cat([a,b],\rset)$ be two continuous functions on $[a,b]$ such that there exists a function $g\in\lin\cF$ with
\[ g_1 < g < g_2\]
on $[a,b]$. Then the following hold:
\begin{enumerate}[(i)]
\item There exists a unique polynomial $f_*\in\lin\cF$ such that
\begin{enumerate}[(a)]
\item $g_1(x) \leq f_*(x) \leq g_2(x)$ for all $x\in [a,b]$, and

\item there exist $n+1$ points $x_1<\dots<x_{n+1}$ in $[a,b]$ such that
\[ f_*(x_{n+1-i}) = \begin{cases}
g_1(x_{n+1-i}) & \text{for}\ i=1,3,5,\dots,\\
g_2(x_{n+1-i}) & \text{for}\ i=0,2,4,\dots .
\end{cases}\]
\end{enumerate}

\item There exists a unique polynomial $f^*\in\lin\cF$ such that
\begin{enumerate}[(a')]
\item $g_1(x) \leq f^*(x) \leq g_2(x)$ for all $x\in [a,b]$, and

\item there exist $n+1$ points $y_1 < \dots < y_{n+1}$ in $[a,b]$ such that
\[ f^*(y_{n+1-i}) = \begin{cases}
g_2(y_{n+1-i}) & \text{for}\ i=1,3,5,\dots,\\
g_1(y_{n+1-i}) & \text{for}\ i=0,2,4,\dots .
\end{cases}\]
\end{enumerate}
\end{enumerate}
\end{snakethm}

The functions $g_1$, $g_2$, $g$, $f_*$, and $f^*$ of the \Cref{thm:g1g2g} are illustrated in \Cref{fig:g1g2g}.
\begin{figure}[htb]
\begin{center}
\includegraphics[width=0.9\textwidth]{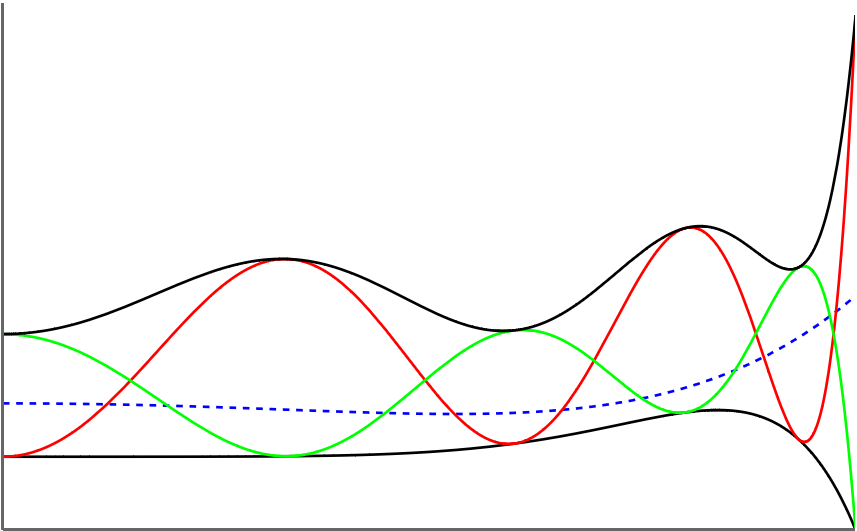}
\end{center}
\caption{The functions $g_1,g_2\in\cat([a,b],\rset)$ (black, $g_1$ bottom, $g_2$ top), $g\in\lin\cF$ (blue, dashed), $f_*\in\lin\cF$ (red), and $f^*\in\lin\cF$ (green) from the \Cref{thm:g1g2g}.\label{fig:g1g2g}}
\end{figure}
The following proof is taken from {\cite[p.\ 73]{karlinStuddenTSystemsBook}}.

\begin{proof}
Let $n=2m$ and $\cF$ be an ET-system.
We proceed as in the proof of \Cref{thm:karlin}.
For each $\xi = (\xi_0,\dots,\xi_n)\in\Xi^n$ and $\sum_{i=0}^n \xi_i = b-a$ we construct the polynomial
\[f_\xi(x) = \sum_{i=0}^n a_i(\xi)\cdot f_i(x) = c_\xi\cdot\det\left(\begin{array}{c|ccc}
f_0 &\, f_1 & \dots & f_n\\
x &\, x_1 & \dots & x_n
\end{array}\right)\]
which vanishes at each of the points
\[x_i := a + \sum_{k=0}^{i-1} \xi_k\]
for all $i=0,\dots,n$ and let $c_\xi\in\rset$ be such that $a_0(\xi)^2 + \dots + a_n(\xi)^2 = 1$ and $f_\xi\geq 0$ on $[x_i,x_{i+1}]$ if $i$ is even.

For $i=0,2,4,\dots,n$ we define
\[\delta_i(\xi) := \min \left\{\delta\geq 0 \,\middle|\, \delta\cdot (g_2-g) \geq f_\xi\ \text{on}\ [x_i,x_{i+1}]\right\},\]
where $x_0=a$ and $x_{n+1} = b$, while for $i=1,3,\dots,n-1$ we define
\[\delta_i(\xi) := \min \left\{ \delta\geq 0 \,\middle|\, f_\xi \geq \delta\cdot (g-g_1)\ \text{on}\ [x_i,x_{i+1}]\right\}.\]
As in \Cref{thm:karlin} we define $F_k(\xi) := \delta_k(\xi) - \min_i \delta_i(\xi)$.
And as before assuming $\sum_{k=0}^n F_k(\xi)>0$ for all $\xi\in\Xi^n$ leads to a contradiction.
Therefore, there exists a $\xi^*\in\Xi^n$ for which $\delta_i(\xi^*) = \delta$ for all $i=0,\dots,n$.
It is clear that $\delta>0$ and that the polynomial $f_* := \delta^{-1}\cdot f_{\xi^*} + g$ satisfies the conditions of the theorem.

The polynomial $f^*$ is constructed employing the same line of arguments.

The extension encompassing the case where $\cF$ is merely a T-system and the proof of the uniqueness proceed as in the proof of \Cref{thm:karlin}.
\end{proof}

\section{Karlin's Nichtnegativstellensatz for ET-Systems on $[a,b]$}

While \Cref{thm:karlin} with $f>0$ can be proved for T-systems, an equivalent version allowing zeros in $f\in\cat([a,b],\rset)$, i.e., $f\geq 0$ but not $f>0$, needs to assume that $\cF$ is an ET-system.

\begin{karthm}[for {$f\geq 0$ on $[a,b]$}; {\cite[Thm.\ 3]{karlin63}} or e.g.\ {\cite[p.\ 74, Thm.\ 10.3]{karlinStuddenTSystemsBook}}]\label{thm:karlinNonNeg}\index{Karlin!Theorem!for $f\geq 0$ on $[a,b]$}\index{Theorem!Karlin!for $f\geq 0$ on $[a,b]$}
Let $n\in\nset_0$, let $\cF = \{f_i\}_{i=0}^n$ be a continuous ET-system of order $n$ on $[a,b]$ with $a<b$, and let $f\in \cat^n([a,b],\rset)$ be such that $f\geq 0$ on $[a,b]$ and $f$ has $r < n$ zeros (counting multiplicities).
The following hold:
\begin{enumerate}[(i)]
\item There exists a unique polynomial $f_*\in\lin\cF$ such that
\begin{enumerate}[(a)]
\item $f(x)\geq f_*(x)\geq 0$ for all $x\in [a,b]$,

\item $f_*$ has $n$ zeros counting multiplicities,

\item if $x_{1} < \dots < x_{{n-r}}$ in $(a,b)$ are the zeros of $f_*$ which remain after removing the $r$ zeros of $f$ then $f - f_*$ vanishes at least twice more (counting multiplicities) in each open interval $(x_i,x_{i+1})$, $i=1,\dots,n-r-1$, and at least once more in each of the intervals $[a,x_1)$ and $(x_{n-r},b]$,

\item the zeros $x_1,\dots,x_{n-r}$ of (c) are a set of index $n-r$, and

\item $x_{n-r} < b$.
\end{enumerate}

\item There exists a unique polynomial $f^*\in\lin\cF$ which satisfies the conditions (a) to (d) and 
\begin{enumerate}[(a')]\setcounter{enumii}{4}
\item $x_{n-r} = b$.
\end{enumerate}
\end{enumerate}
\end{karthm}

The proof is taken from \cite[pp.\ 74--75]{karlinStuddenTSystemsBook}.

\begin{proof}
Let $z_1,\dots,z_p$ be the distinct zeros of $f$ with multiplicities $m_1,\dots,m_p$ where $\sum_{i=1}^p m_i = r \leq n-1$ and set $n':= n-r$.
The proof is now similar to the proof of \Cref{thm:karlin} where $n$ is replaced by $n'$.
Since the odd and the even cases are again somewhat the same and for the sake of some slight variety we treat now the odd case $n' = 2m'+1$.
The construction of $f_*$ in part (i) proceeds as follows.
For each $\xi\in\Xi^{m'}$ we construct the polynomial
\begin{multline}\label{eq:fxiLong}
f_\xi(x) := \sum_{i=0}^n a_i(\xi)\cdot f_i\\
= c_\xi\cdot\det \left(\begin{array}{c|cccccccc}
f_0 &\, f_1 \dots f_{m_1} & \dots & f_{m_1+m_{p-1}+1} & \dots f_{r} & f_{r+1}\; f_{r+2} & \dots & f_{n-2}\; f_{n-1} & f_n\\
x &\, z_1\, \dots\, z_1\;\; & \dots & z_p & \dots\, z_p & x_1\;\;\; x_1 & \dots & x_{n'} \;\;\, x_{n'} & a
\end{array}\right)
\end{multline}
where $c_\xi\in\rset$ is chosen such that $a_1(\xi)^2 + \dots + a_n(\xi)^2 = 1$ and
\[x_i := a + \sum_{k=0}^i \xi_k\]
for all $i=1,\dots,m'$ are the zeros of multiplicity two and $a$ is a zero of multiplicity one.
Now we define
\[\delta_i(\xi) := \min\left\{ \delta\geq 0 \,\middle|\, \delta\geq \frac{f_\xi}{f}\ \text{on}\ [x_i,x_{i+1}]\right\}\]
for $i=1,\dots,m'+1$ with $x_{m'+2}=b$ where the ratio is evaluated by l'Hopital's rule at the zeros $z_1,\dots,z_p$ of $f$.

By examining $\frac{f_\xi}{f}$ first in the neighborhood of each of the points $z_1,\dots,z_p$ and then over the remaining part we find that if $\xi^{(k)}\to\xi$ then
\[ \frac{f_{\xi^{(k)}}}{f} \to \frac{f_\xi}{f}\]
uniformly on $[a,b]$.
Consequently, each of th $\delta_i$ is continuous in $\xi$ and $\delta_i(\xi)=0$ if and only $\xi_i=0$.

The same arguments used in the proof of \Cref{thm:karlin} now show that for some $\xi^*\in\inter\Xi^{m'}$ we have $\delta_i(\xi^*)=\delta>0$ for all $i=1,\dots,m'+1$.
It is simple to see that $f_* := \delta^{-1}\cdot f_{\xi^*}$ possesses the properties (a), (b), (d), and (e) in (i).
To show property (c) observe that if $x_i = z_j$ for some $j$ then $f_{\xi^*}$ has a zero at $z_j$ with multiplicity exceeding that of $f$ so that $\delta$ is strictly greater than $f_{\xi^*}\cdot f^{-1}$ in some neighborhood of $z_j$.
This implies the equality $\delta = f_{\xi^*}(x)\cdot f(x)^{-1}$ for some $x$ in each of the open intervals $(x_1,x_2)$, \dots, $(x_{m'},x_{m'+1})$ and somewhere in $(x_{m'+1},b]$.
Thus, in each $(x_i,x_{i+1})$, either $f(x) - \delta^{-1}\cdot f_{\xi^*}(x)$ vanishes somewhere other than at the zeros of $f$ or the multiplicity of one of the common zeros of $f$ and $\delta^{-1}\cdot f_{\xi^*}$ is increased by two.
In the interval $(x_{m'+1},b]$ the function $f-\delta^{-1}\cdot f_{\xi}$ may vanish at $b$ with multiplicity only one greater than the zero of $f$ at this point.
This concludes that $f_*$ also fulfills (c) in (i).

The polynomial $f^*$ when $n' = 2m'+1$ is constructed in the same manner by replacing $a$ in (\ref{eq:fxiLong}) by $b$.

\textit{Uniqueness:} Assume another polynomial $g$ satisfies the same properties as $f_*$.
Without loss of generality we can assume that the first zero of $f-g$ other than the zeros of $f$ is less than or equal to first zero of $f-f_*$.
Define $h := \frac{f_*-g}{f}$.
A zero of $h$ occurring at one of the values $x_i$, $i=2,\dots,n'+1$ is necessarily at least a double zero.
In this case we assign one zero to each of the intervals $[x_{i-1},x_i]$ and $[x_i,x_{i+1}]$ with $x_{m'+2} = b$.
Under this counting procedure, and taking account of the oscillation properties of $f_*$ and $g$, we deduce that $h$ has at least three zeros in $[a,x_1]$, at least two zeros in each of the intervals $[x_i,x_{i+1}]$, $i=2,\dots,m'$, and at least one zero in $[x_{m'+1},b]$.
Clearly, all of these zeros are other than the $r$ zeros of $f$, so that $f_*-g$ has at least $3+2(m'-1) + 1 + r = n+1$ zeros (counting multiplicities). Hence, $h=0$ and $f_* = g$.
\end{proof}

If $f\in\lin\cF$ in \Cref{thm:karlinNonNeg} we get similar to \Cref{thm:karlinPosab} the following Nichtnegativstellensatz on $[a,b]$ due to Karlin.

\begin{karnegthm}[for ET-Systems on {$[a,b]$}; {\cite[p.\ 603, Cor.\ after Thm.\ 3]{karlin63}} or e.g.\ {\cite[p.\ 76, Cor.\ 10.3]{karlinStuddenTSystemsBook}}]\index{Karlin!Nichtnegativstellensatz!on $[a,b]$}\index{Theorem!Karlin!Nichtnegativstellensatz on $[a,b]$}\index{Nichtnegativstellensatz!Karlin!on $[a,b]$}\label{thm:karlinNonNegab}
Let $n\in\nset_0$, $\cF = \{f_i\}_{i=0}^n$ be an ET-system of order $n$ on $[a,b]$ with $a<b$, and let $f\in\lin\cF$ be such that $f\geq 0$ on $[a,b]$ and $f$ has $r < n$ zeros $a\leq z_1 \leq z_2 \leq \dots \leq z_r\leq b$ (counting multiplicities).
Then there exists a unique representation
\[f = f_* + f^*\]
with $f_*,f^*\in\lin\cF$ such that
\begin{enumerate}[(i)]
\item $f_*, f^*\geq 0$ on $[a,b]$,

\item for $f_*$ and $f^*$ the sets of zeros counting algebraic multiplicities is after removing the zeros of $f$ with algebraic multiplicity a set of index $n-r$ which strictly interlace, and

\item the set of zeros of $f^*$ contains after removing the zeros of $f$ with algebraic multiplicities the point $b$.
\end{enumerate}
\end{karnegthm}
\begin{proof}
Let $f_*$ be the polynomial from \Cref{thm:karlinNonNeg} and set $g := f - f_*$.
Then $g$ fulfills the conditions of $f^*$ in \Cref{thm:karlinNonNeg} and by its uniqueness we have $g = f^*$ which proves the statement.
\end{proof}

\begin{rem}
Since \Cref{thm:karlinNonNegab} (ii) might be a little bit confusing we explain it more detailed.

Let $\cF$ be an ET-system of order $n\in\nset_0$ on $[a,b]$ with $a<b$ and let $f\in\lin\cF$ be such that $f\geq 0$ on $[a,b]$ and $f$ has the zeros $z_1,\dots,z_l$ with algebraic multiplicities $m_1,\dots,m_l$, $m_1 + \dots + m_l =: r < n$.
\begin{enumerate}[\it(i)]
\item If $n-r = 2m$ is \emph{even} then the zeros of $f_*$ from \Cref{thm:karlinNonNegab} are $x_1,\dots,x_m$ all with algebraic multiplicity $2$ and the zeros of $f^*$ are $y_0,y_1,\dots,y_m$ where $y_0$ and $y_m$ have algebraic multiplicity $1$ and otherwise the $y_i$ have algebraic multiplicity $2$.
They interlace, i.e., we have
\[a= y_0 < x_1 < y_1 < \dots < x_m < y_m = b.\]
The $f_*$ and $f^*$ are then given by
\[f_*(x) = c_*\cdot\det \left(\begin{array}{c|cccccccc}
f_0 & \, f_1 & f_2 & \dots & f_{2m-1} & f_{2m} & f_{2m+1} & \dots & f_n\\
x & \, x_1 & x_1 & \dots & x_m & x_m & z_1 & \dots & z_l
\end{array}\right)\]
and
\[f^*(x) = c^*\cdot\det \left(\begin{array}{c|ccccccccc}
f_0 & \, f_1 & f_2 \; f_3 \; \dots & f_{2m-2} & f_{2m-1} & f_{2m} & f_{2m+1} & \dots & f_n\\
x & \, a & y_1 \; y_1 \; \dots & y_{m-1} & y_{m-1} & b & z_1 & \dots & z_l
\end{array}\right)\]
where $c_*,c^*\in\rset\setminus\{0\}$ and the signs are such that $f_*,f^*\geq 0$ on $[a,b]$.
The zeros $z_1,\dots,z_l$ are included with their corresponding algebraic multiplicities $m_1,\dots,m_l$, i.e., $z_1$ is included $m_1$-times, \dots, $z_l$ is included $m_l$-times.

\item If $n-r = 2m+1$ is \emph{odd} then the zeros of $f_*$ from \Cref{thm:karlinNonNegab} are $x_0,\dots,x_m$ where $x_0$ has algebraic multiplicity $1$ and the other algebraic multiplicity $2$. For $f^*$ we have the zeros $y_0,\dots,y_m$ where $y_0,\dots,y_{m-1}$ have algebraic multiplicity $2$ and $y_m$ has algebraic multiplicity $1$.
They interlace, i.e., we have
\[x_0 = a < y_0 < \dots < x_m < y_m = b.\]
The $f_*$ and $f^*$ are then given by
\[f_*(x) = c_*\cdot\det \left(\begin{array}{c|ccccccccc}
f_0 & \, f_1 & f_2 \; f_3 \; \dots & f_{2m} & f_{2m+1} & f_{2m+2} & \dots & f_n\\
x & \, a & x_1 \; x_1 \; \dots & x_m & x_m  & z_1 & \dots & z_l
\end{array}\right)\]
and
\[f^*(x) = c^*\cdot\det \left(\begin{array}{c|cccccccc}
f_0 & \, f_1\; f_2 & \dots & f_{2m-1} & f_{2m} & f_{2m+1} & f_{2m+2} & \dots & f_n\\
x & \, y_0\; y_0 & \dots & y_{m-1} & y_{m-1} & b & z_1 & \dots & z_l
\end{array}\right)\]
where $c_*,c^*\in\rset\setminus\{0\}$ and the signs are such that $f_*,f^*\geq 0$ on $[a,b]$.
The zeros $z_1,\dots,z_l$ are included with their corresponding algebraic multiplicities $m_1,\dots,m_l$, i.e., $z_1$ is included $m_1$-times, \dots, $z_l$ is included $m_l$-times.
\exmsymbol
\end{enumerate}
\end{rem}

With the proof of \Cref{thm:karlinNonNeg} one can prove a similar interlacing theorem as the \Cref{thm:g1g2g} when $g_2-g_1$ has a certain number of zeros, see \cite[p.\ 76, Rem.\ 10.1]{karlinStuddenTSystemsBook}.

We stated here \Cref{thm:karlinPosab} and \Cref{thm:karlinNonNegab} for functions on $[a,b]$.
There are also similar statements for periodic functions, see \cite[Thm.\ 6 and 7]{karlin63}.
The cases on $[0,\infty)$ and $\rset$ are given in the next chapter.

\section*{Problems}
\addcontentsline{toc}{section}{Problems}

\begin{prob}\label{prob:etNeighborhood}
Examine the proof of \Cref{thm:karlinNonNeg} more closely.
In the statement of the theorem it is required that $\cF$ is an ET-system on $[a,b]$.
But for a given $f\geq 0$ where does the family $\cF$ actually only needs to be an ET-system?
\end{prob}

\motto{Look at the conclusion! And try to think of a familiar\\ theorem having the same or a similar conclusion.\\ \medskip
\ \hspace{1cm} \normalfont{George P\'olya {\cite[p.\ 25]{polyaHowToSolveIt}}}\index{P\'olya, G.}}

\chapter{Karlin's Positivstellensätze and Nichtnegativstellensätze on $[0,\infty)$ and $\rset$}
\label{ch:karlinPos0inftyR}

In this chapter we extend the results of the previous chapter, i.e., we extend \Cref{thm:karlinPosab} on $[a,b]$ to $[0,\infty)$ in \Cref{thm:karlinPos0infty} and to $\rset$ in \Cref{thm:karlinPosR} as well as we extend \Cref{thm:karlinNonNegab} on $[a,b]$ to $[0,\infty)$ in \Cref{thm:karlinNonNeg0infty} and to $\rset$ in \Cref{thm:karlinNonNegR}.

\section{Karlin's Positivstellensatz for T-Systems on $[0,\infty)$}

By a transformation of $[a,b]$ to $[0,\infty]$ and then restriction to $[0,\infty)$ we get from \Cref{thm:karlinPosab} the following.

\begin{karposthm}[for T-Systems on {$[0,\infty)$}; see {\cite[Thm.\ 9]{karlin63}} or e.g.\ {\cite[p.\ 169, Thm.\ 8.1]{karlinStuddenTSystemsBook}}]\index{Karlin!Positivstellensatz!on $[0,\infty)$}\index{Positivstellensatz!Karlin!on $[0,\infty)$}\index{Theorem!Karlin!Positivstellensatz on $[0,\infty)$}\label{thm:karlinPos0infty}
Let $n\in\nset_0$ and $\cF = \{f_i\}_{i=0}^n$ be a continuous T-system of order $n$ on $[0,\infty)$ such that
\begin{enumerate}[\quad (a)]
\item there exists a $C>0$ such that $f_n(x)>0$ for all $x\geq C$,

\item $\displaystyle\lim_{x\to\infty} \frac{f_i(x)}{f_n(x)} = 0$ for all $i=0,\dots, n-1$, and

\item $\{f_i\}_{i=0}^{n-1}$ is a continuous T-system on $[0,\infty)$.
\end{enumerate}
Then for any $f = \sum_{i=0}^n a_i f_i\in\lin\cF$ with $f>0$ on $[0,\infty)$ and $a_n>0$ there exists a unique representation
\[f = f_* + f^*\]
with $f_*,f^*\in\lin\cF$ and $f_*,f^*\geq 0$ on $[0,\infty)$ such that the following hold:
\begin{enumerate}[(i)]
\item If $n = 2m$ the polynomials $f_*$ and $f^*$ each possess $m$ distinct zeros $\{x_i\}_{i=1}^m$ and $\{y_i\}_{i=0}^{m-1}$ satisfying
\[0 = y_0 < x_1 < y_1 < \dots < y_{m-1} < x_m <\infty.\]
All zeros except $y_0$ are double zeros.

\item If $n = 2m+1$ the polynomials $f_*$ and $f^*$ each possess the zeros $\{x_i\}_{i=1}^{m+1}$ and $\{y_i\}_{i=1}^m$ satisfying
\[0 = x_1 < y_1 < x_2 < \dots < y_m < x_{m+1} < \infty.\]
All zeros except $x_1$ are double zeros.

\item The coefficient of $f_n$ in $f_*$ is equal to $a_n$.
\end{enumerate}
\end{karposthm}

The proof is adapted from \cite[pp.\ 168]{karlinStuddenTSystemsBook}.

\begin{proof}
By (a) there exists a function $w\in\cat([0,\infty),\rset)$ such that $w>0$ on $[0,\infty)$ and $\lim_{x\to\infty} \frac{f_n(x)}{w(x)} = 1$.
By (b) we define
\[v_i(x) := \begin{cases}
\frac{f_i(x)}{w(x)} & \text{if}\ x\in [0,\infty),\\
\delta_{i,n} & \text{if}\ x=\infty
\end{cases}\]
for all $i=0,1,\dots,n$.
Then by (c) and \Cref{cor:scaling} we have that $\{v_i\}_{i=0}^n$ is a T-system on $[0,\infty]$.
With $t(x) := \tan (\pi x/2)$ we define
$g_i(x) := v_i\circ t$
for all $i=0,1,\dots,n$.
Hence, $\cG = \{g_i\}_{i=0}^n$ is a T-system on $[0,1]$ by \Cref{cor:transf}.
We now apply \Cref{thm:karlinPosab} to $\cG$. Set $g := \big(\frac{f}{w}\big)\circ t$.

(i): Let $n = 2m$. Then by \Cref{thm:karlinPosab} there exist points
\[0=y_0 < x_1 < y_1 < \dots < x_m < y_m = 1\]
and unique functions $g_*$ and $g^*$ such that $g = g_* + g^*$, $g_*,g^*\geq 0$ on $[0,1]$, $x_1,\dots,x_m$ are the zeros of $g_*$, and $y_0,\dots,y_m$ are the zeros of $g^*$.
Then $f_* := (g_*\circ t^{-1})\cdot w$ and $f^* := (g^*\circ t^{-1})\cdot w$ are the unique components in the decomposition $f = f_* + f^*$.

(ii): Similar to (i).

(iii): From (i) (and (ii) in a similar way) we have $g_i(1) = 0$ for $i=0,\dots,n-1$ and $g_n(1) = 1$.
Hence, we get with $g^*(y_m=1) = 0$ that $g_n$ is not contained in $g^*$, i.e., $g_*$ has the only $g_n$ contribution because $\cG$ is linearly independent.
This is inherited by $f_*$ and $f^*$ which proves (iii).
\end{proof}

The transformation $g_i = v_i\circ t$ with $t$ the $\tan$-function is due to Krein \cite{krein51}.

If $\cF$ in \Cref{thm:karlinPos0infty} is an ET-system then the $f_*$ and $f^*$ can be written down explicitly.
For that we only need $\cF$ to be an ET-system on $(0,\infty)$ not on all $[0,\infty)$ since at $x=0$ a possible zero in $f_*$ or $f^*$ only has multiplicity one.

\begin{cor}\label{cor:posrepr0infty}
If in \Cref{thm:karlinPos0infty} we have additionally that $\cF$ is an ET-system on $(0,\infty)$ then the unique $f_*$ and $f^*$ are given
\begin{enumerate}[(i)]
\item for $n= 2m$ by
\[ f_*(x) = c^*\cdot \det \begin{pmatrix}
f_0 & f_1 & f_2 & \dots & f_{2m-1} & f_{2m}\\
x & (x_1 & x_1) & \dots & (x_m & x_m)
\end{pmatrix}\]
and
\[f^*(x) = -c_*\cdot \det \begin{pmatrix}
f_0 & f_1 & f_2 & f_3 & \dots & f_{2m-2} & f_{2m-1}\\
x & y_0 & (y_1 & y_1) & \dots & (y_{m-1} & y_{m-1})
\end{pmatrix},\]

\item and for $n= 2m+1$ by
\[f_*(x) = -c_*\cdot \det \begin{pmatrix}
f_0 & f_1 & f_2 & f_3 & \dots & f_{2m} & f_{2m+1}\\
x & x_1 & (x_2 & x_2) & \dots & (x_{m+1} & x_{m+1})
\end{pmatrix}\]
and
\[f^*(x) = c^*\cdot \det \begin{pmatrix}
f_0 & f_1 & f_2 & \dots & f_{2m-1} & f_{2m}\\
x & (y_1 & y_1) & \dots & (y_m & y_m)
\end{pmatrix}\]
\end{enumerate}
for some $c_*,c^* > 0$.
\end{cor}
\begin{proof}
Combine \Cref{thm:karlinPos0infty} with \Cref{rem:doublezeros} and note that since $0$ is never a multiple zero we only need $\cF$ to be an ET-system on $(0,\infty)$.
\end{proof}

\section{Karlin's Nichtnegativstellensatz for ET-Systems on $[0,\infty)$}

In \Cref{thm:karlinPos0infty} we needed to transform the domain $[a,b]$ into $[0,\infty]$ of a T-system.
For \Cref{thm:karlinNonNeg0infty} we needed an ET-system because of the additional zeros from $f\geq 0$.

With the same technique as in the proof of \Cref{thm:karlinPos0infty} and \Cref{lem:ettrans} we get from \Cref{thm:karlinNonNegab} the following.

\begin{karnegthm}[for ET-Systems on {$[0,\infty)$}]
\index{Karlin!Nichtnegativstellensatz!on $[0,\infty)$}\index{Theorem!Karlin!Nichtnegativstellensatz on $[0,\infty)$}\index{Nichtnegativstellensatz!Karlin!on $[0,\infty)$}
\label{thm:karlinNonNeg0infty}
Let $n\in\nset_0$ and $\cF = \{f_i\}_{i=0}^n$ be an ET-system of order $n$ on $[0,\infty)$ such that
\begin{enumerate}[(a)]
\item there exists a $C>0$ such that $f_n(x)>0$ for all $x\geq 0$,

\item $\displaystyle\lim_{x\to\infty} \frac{f_i(x)}{f_n(x)} = 0$ for all $i=0,\dots,n-1$, and

\item $\{f_i\}_{i=0}^{n-1}$ is an ET-system.
\end{enumerate}
Then for any $f = \sum_{i=0}^n a_i f_i \in\lin\cF$ such that $f\geq 0$ on $[0,\infty)$, $a_n>0$, and $f$ has $r<n$ zeros counting multiplicity there exists a unique representation
\[f = f_* + f^*\]
with $f_*,f^*\in\lin\cF$ such that the following hold:
\begin{enumerate}[(i)]
\item $f_*,\ f^*\geq 0$ on $[0,\infty)$,

\item $f_*$ has $n$ zeros (counting multiplicities),

\item $f^*$ has $n-1$ zeros (counting multiplicities),

\item the zeros of $f_*$ and $f^*$ strictly interlace if the zeros of $f$ are removed, and

\item the coefficient of $f_n$ in $f_*$ is equal to $a_n$.
\end{enumerate}
\end{karnegthm}
\begin{proof}
The conditions (a) -- (c) are such that $\cF$ on $[0,\infty]$, i.e., including $\infty$, is an ET-system.

With the same argument as in the proof of \Cref{thm:karlinPos0infty} we transform $\cF$ on $[0,\infty]$ into $\cG$ on $[0,1]$ with the $\tan$-function.
Here \Cref{lem:ettrans} ensures that also $\cG$ is an ET-system.

Application of \Cref{thm:karlinNonNegab} on $[0,1]$ gives the desired decomposition $g = g_* + g^*$ with the observation that $x=1$ is a zero of at most multiplicity one by (a) and (b).
Backwards transformation into $\cF$ on $[0,\infty]$ resp.\ $[0,\infty)$ then gives the assertion.
\end{proof}


\section{Karlin's Positivstellensatz for T-Systems on $\rset$}

We have seen that from \Cref{thm:karlinPosab} on $[a,b]$ we get \Cref{thm:karlinPos0infty} on $[0,\infty)$ with the transformation $t(x) = \tan(\pi x/2)$ from $[0,1]$ to $[0,\infty]$ and only need to pay attention to the end point $x=1$ resp.\ $x=\infty$.
The same transformation however also applies going from $[-1,1]$ to $[-\infty,\infty]$ now paying attention to both end points.

\begin{karposthm}[for T-Systems on {$\rset$}; see {\cite[Thm.\ 10]{karlin63}} or e.g.\ {\cite[p.\ 198, Thm.\ 8.1]{karlinStuddenTSystemsBook}}]\index{Karlin!Positivstellensatz!on $\rset$}\index{Positivstellensatz!Karlin!on $\rset$}\index{Theorem!Karlin!Positivstellensatz on $\rset$}\label{thm:karlinPosR}
Let $m\in\nset_0$ and $\cF = \{f_i\}_{i=0}^{2m}$ be a continuous T-system of order $2m$ on $\rset$ such that
\begin{enumerate}[\quad (a)]
\item there exists a $C>0$ such that $f_{2m}(x)>0$ for all $x\in (-\infty,-C]\cup [C,\infty)$,

\item $\displaystyle\lim_{|x|\to\infty} \frac{f_i(x)}{f_{2m}(x)} = 0$ for all $i=0,\dots,2m-1$, and

\item $\{f_i\}_{i=0}^{2m-1}$ is a continuous T-system of order $2m-1$ on $\rset$.
\end{enumerate}
Let $f = \sum_{i=0}^{2m} a_i f_i$ be such that $f>0$ on $\rset$ and $a_{2m}>0$.
Then there exists a unique representation
\[f = f_* + f^*\]
with $f_*,f^*\in\lin\cF$ and $f_*,f^*\geq 0$ on $\rset$ such that
\begin{enumerate}[(i)]
\item the coefficient of $f_{2m}$ in $f_*$ is $a_{2m}$, and

\item $f_*$ and $f^*$ are non-negative polynomials having zeros $\{x_i\}_{i=1}^m$ and $\{y_i\}_{i=1}^{m-1}$ with
\[-\infty < x_1 < y_1 < x_2 < \dots < y_{m-1} < x_m < \infty.\]
\end{enumerate}
\end{karposthm}
\begin{proof}
See Problem \ref{prob:posR}.
\end{proof}

\section{Karlin's Nichtnegativstellensatz for ET-Systems on $\rset$}

On $\rset$ we have the following Nichtnegativstellensatz for ET-systems.

\begin{karnegthm}[for ET-Systems on {$\rset$}]
\index{Karlin!Nichtnegativstellensatz!on $\rset$}\index{Theorem!Karlin!Nichtnegativstellensatz on $\rset$}\index{Nichtnegativstellensatz!Karlin!on $\rset$}
\label{thm:karlinNonNegR}
Let $m\in\nset_0$ and $\cF = \{f_i\}_{i=0}^{2m}$ be an ET-system of order $2m$ on $\rset$ such that
\begin{enumerate}[\quad (a)]
\item there exists a $C>0$ such that $f_{2m}>0$  for all $x\in (-\infty,-C]\cup[C,\infty)$,

\item $\displaystyle\lim_{|x|\to\infty} \frac{f_i(x)}{f_{2m}(x)} = 0$ for all $i=0,\dots,2m-1$,

\item $\{f_i\}_{i=0}^{n-1}$ is an ET-system of order $n-1$ on $\rset$.
\end{enumerate}
Let $f = \sum_{i=0}^{2m} a_i f_i\in\lin\cF$ be such that $f\geq 0$, $a_{2m}>0$, and $f$ has $r<n$ zeros counting multiplicities.
Then there exists a unique representation
\[f = f_* + f^*\]
with $f_*, f^*\in\lin\cF$ such that the following hold:
\begin{enumerate}[(i)]
\item $f_*,f^*\geq 0$ on $\rset$,
\item $f_*$ has $2m$ zeros counting multiplicity,
\item $f^*$ has $2m-2$ zeros counting multiplicity,
\item the zeros of $f_*$ and $f^*$ strictly interlace if the zeros of $f$ are removed, and
\item the coefficient of $f_n$ in $f_*$ is equal to $a_n$.
\end{enumerate}
\end{karnegthm}
\begin{proof}
See Problem \ref{prob:nonnegR}.
\end{proof}

\section*{Problems}
\addcontentsline{toc}{section}{Problems}

\begin{prob}\label{prob:posR}
Prove \Cref{thm:karlinPosR}, i.e., adapt the proof of \Cref{thm:karlinPos0infty} such that both interval ends $a$ and $b$ of $[a,b]$ are mapped to $-\infty$ and $+\infty$, respectively.
\end{prob}

\begin{prob}\label{prob:nonnegR}
Prove \Cref{thm:karlinNonNegR}, i.e., adapt the proof of \Cref{thm:karlinNonNeg0infty} such that both interval ends $a$ and $b$ of $[a,b]$ are mapped to $-\infty$ and $+\infty$, respectively.
\end{prob}

\part{Non-Negative Algebraic Polynomials on $[a,b]$, $[0,\infty)$, and $\rset$}
\label{part:algPos}

\motto{I hold that it is only when we can prove everything we assert\\ that we understand perfectly the thing under consideration.\\ \medskip
\ \hspace{1cm} \normalfont{Gottfried Wilhelm Leibniz {\cite{leibniz89FoucherLetter}}}\index{Leibniz, G.\ W.}}

\chapter{Non-Negative Algebraic Polynomials on $[a,b]$}
\label{ch:nonNegAlgPolab}

We developed in the previous chapters the Positiv- and Nichtnegativestellensätze for general T- and ET-systems due to Karlin.
We will now apply these to the algebraic polynomials, i.e., we will plug in \Cref{exm:algECTsystem} and \Cref{exm:algETsystem}.

\section{Sparse Algebraic Positivstellensatz on $[a,b]$}

At first let us have a look how all sparse strictly positive polynomials on some interval $[a,b]\subseteq (0,\infty)$ look like.

\begin{thm}[Sparse Algebraic Positivstellensatz on {$[a,b]$ with $0<a<b$}]\label{thm:algPosSatzab}\index{sparse!algebraic Positivstellensatz!on $[a,b]$}
Let $n\in\nset_0$, $\alpha_0,\dots,\alpha_n\in\rset$ be real numbers with $\alpha_0 < \alpha_1 < \dots < \alpha_n$, and let $\cF=\{x^{\alpha_i}\}_{i=0}^n$.
Then for any $f=\sum_{i=0}^n a_i x^{\alpha_i}\in\lin\cF$ with $f>0$ on $[a,b]$ and $a_n>0$ there exists a unique decomposition
\[f = f_* + f^*\]
with $f_*,f^*\in\lin\cF$ such that
\begin{enumerate}[(i)]
\item for $n = 2m$ there exist points $x_1,\dots,x_m,y_1,\dots,y_{m-1}\in [a,b]$ with
\[a < x_1 < y_1 < \dots < x_m < b\]
and constants $c_*,c^*>0$ with
\begin{equation}\label{eq:fStarStart}
f_*(x) = c_*\cdot\det \begin{pmatrix}
x^{\alpha_0} & x^{\alpha_1} & x^{\alpha_2} & \dots & x^{\alpha_{2m-1}} & x^{\alpha_{2m}}\\
x & (x_1 & x_1) & \dots & (x_m & x_m)
\end{pmatrix}\geq 0
\end{equation}
and
\begin{equation}
f^*(x) = -c^*\cdot\det \begin{pmatrix}
x^{\alpha_0} & x^{\alpha_1} & x^{\alpha_2} & x^{\alpha_3} & \dots & x^{\alpha_{2m-2}} & x^{\alpha_{2m-1}} & x^{\alpha_{2m}}\\
x & a & (y_1 & y_1) & \dots & (y_{m-1} & y_{m-1}) & b
\end{pmatrix}\geq 0
\end{equation}
for all $x\in [a,b]$, or

\item for $n = 2m+1$ there exist points $x_1,\dots,x_m,y_1,\dots,y_m\in [a,b]$ with
\[a < y_1 < x_1 < \dots < y_m < x_m < b\]
and $c_*,c^*>0$ with
\begin{equation}
f_*(x) = -c_*\cdot\det\begin{pmatrix}
x^{\alpha_0} & x^{\alpha_1} & x^{\alpha_2} & x^{\alpha_3} & \dots & x^{\alpha_{2m}} & x^{\alpha_{2m+1}}\\
x & a & (x_1 & x_1) & \dots & (x_m & x_m)
\end{pmatrix}\geq 0
\end{equation}
and
\begin{equation}\label{eq:fStarEnd}
f^*(x) = c^*\cdot\det\begin{pmatrix}
x^{\alpha_0} & x^{\alpha_1} & x^{\alpha_2} & \dots & x^{\alpha_{2m-1}} & x^{\alpha_{2m}} & x^{\alpha_{2m+1}}\\
x & (y_1 & y_1) & \dots & (y_m & y_m) & b
\end{pmatrix}\geq 0
\end{equation}
for all $x\in [a,b]$.
\end{enumerate}
\end{thm}
\begin{proof}
By \Cref{exm:algETsystem} we have that $\cF$ on $[a,b]$ is an ET-system. Hence, \Cref{thm:karlinPosab} applies. We check both cases $n = 2m$ and $n=2m+1$ separately.

$n=2m$:
By \Cref{thm:karlinPosab} we have that the zero set $\cZ(f^*)$ of $f^*$ has index $2m$ and contains $b$ with index $1$, i.e., $a\in\cZ(f^*)$ and all other zeros have index $2$.
Hence, $\cZ(f^*) = \{a=y_0 < y_1 < \dots < y_{m-1} < y_m = b\}$.
By \Cref{thm:karlinPosab} we have that $\cZ(f_*)$ also has index $2m$ and the zeros of $f_*$ and $f^*$ interlace.
Then the determinantal representations of $f_*$ and $f^*$ follow from \Cref{rem:doublezeros}.

$n=2m+1$:
By \Cref{thm:karlinPosab} we have that $b\in\cZ(f^*)$ and since the index of $\cZ(f^*)$ is $2m+1$ we have that there are only double zeros $y_1,\dots,y_m\in (a,b)$ in $\cZ(f^*)$.
Similar we find that $a\in\cZ(f_*)$ since its index is odd and only double zeros $x_1,\dots,x_m\in (a,b)$ in $\cZ(f_*)$ remain.
By \Cref{thm:karlinPosab} (iii) the zeros $x_i$ and $y_i$ strictly interlace and the determinantal representation of $f_*$ and $f^*$ follow again from \Cref{rem:doublezeros}.
\end{proof}

Note, if $\alpha_0,\dots,\alpha_n\in\nset_0$ then by \Cref{exm:algETsystem} equation (\ref{eq:schurPolyRepr}) the algebraic polynomials $f_*$ and $f^*$ in (\ref{eq:fStarStart}) -- (\ref{eq:fStarEnd}) can be written down with Schur polynomials.

\begin{rem}
The condition $a_n>0$ in \Cref{thm:algPosSatzab} is no restriction.
The result also holds for $a_n<0$ as long as $f>0$ on $[a,b]$.
Since $[a,b]$ is compact the polynomials $x^{\alpha_i}$ are bounded.
In the definition of a T-system the order of the functions $f_i$ can be altered since only any linear combination has to have at most $n$ zeros.
Hence, in a $f>0$ at least one coefficient $a_i$ is larger then zero and we interchange $f_i$ with $f_n$.
A possible sign change in the $f_*$ and $f^*$ in (\ref{eq:fStarStart}) -- (\ref{eq:fStarEnd}) might appear.
\exmsymbol
\end{rem}

\Cref{thm:algPosSatzab} does not hold for $a=0$ and $\alpha_0>0$ or $\alpha_0,\dots,\alpha_k<0$.
In case $\alpha_0>0$ the determinantal representations of $f^*$ for $n=2m$ and $f_*$ for $n=2m+1$ are the zero polynomials.
In fact, in this case $\cF$ is not even a T-system since in \Cref{lem:determinant} the determinant contains a zero column if $x_0 = 0$.
We need to have $\alpha_0=0$ ($x^{\alpha_0} = 1$) to let $a=0$.
For $\alpha_0,\dots,\alpha_k<0$ we have singularities at $x=0$ and hence no T-system.

\begin{cor}\label{cor:azero}
If $\alpha_0=0$ in \Cref{thm:algPosSatzab} then \Cref{thm:algPosSatzab} also holds with $a=0$.
\end{cor}
\begin{proof}
The determinantal representations of $f_*$ for $n=2m+1$ and $f^*$ for $n=2m$ in \Cref{thm:algPosSatzab} continuously depend on $a$.
It is sufficient to show that these representations are non-trivial (not the zero polynomial) for $a=0$.
We show this for $f_*$ in case (ii) $n=2m+1$.
The other cases are equivalent.

We have that $\cF$ is a T-system on $[0,b]$ with $b>0$.
For $\varepsilon>0$ small enough we set
\begin{align*}
g_\varepsilon(x) &= -\varepsilon^{-m}\cdot\det\begin{pmatrix}
1 & x^{\alpha_1} & x^{\alpha_2} & x^{\alpha_3} & \dots & x^{\alpha_{2m}} & x^{\alpha_{2m+1}}\\
x & 0 & x_1 & x_1+\varepsilon & \dots & x_m & x_m+\varepsilon
\end{pmatrix}\\
&= -\varepsilon^{-m}\cdot \det\begin{pmatrix}
1 & x^{\alpha_1} & x^{\alpha_2} & \dots & x^{\alpha_{2m+1}}\\
1 & 0 & 0 & \dots & 0\\
1 & x_1^{\alpha_1} & x_1^{\alpha_2} & \dots & x_1^{\alpha_{2m+1}}\\
\vdots & \vdots & \vdots & & \vdots\\
1 & (x_{m}+\varepsilon)^{\alpha_1} & (x_{m}+\varepsilon)^{\alpha_2} & \dots & (x_{m}+\varepsilon)^{\alpha_{2m+1}}
\end{pmatrix}
\intertext{develop with respect to the second row}
&= \varepsilon^{-m}\cdot\det\begin{pmatrix}
x^{\alpha_1} & x^{\alpha_2} & \dots & x^{\alpha_{2m-1}}\\
x_1^{\alpha_1} & x_1^{\alpha_2} & \dots & x_1^{\alpha_{2m-1}}\\
\vdots & \vdots & & \vdots\\
(x_{m}+\varepsilon)^{\alpha_1} & (x_{m}+\varepsilon)^{\alpha_2} & \dots & (x_{m}+\varepsilon)^{\alpha_{2m+1}}
\end{pmatrix}\\
&= \varepsilon^{-m}\cdot\det\begin{pmatrix}
x^{\alpha_1} & x^{\alpha_2} & x^{\alpha_3} & \dots & x^{\alpha_{2m}} & x^{\alpha_{2m+1}}\\
x & x_1 & x_1+\varepsilon & \dots & x_m & x_m+\varepsilon
\end{pmatrix}.
\end{align*}
Then $x_1,x_1+\varepsilon,\dots,x_m,x_m+\varepsilon\in (0,b]$, i.e., $\{x^{\alpha_i}\}_{i=1}^n$ is an ET-system on $[a',b]$ with $0=a < a' < x_1$, see \Cref{exm:algETsystem}.
By \Cref{rem:doublezeros} the limit $\varepsilon\searrow 0$ is not the zero polynomial which ends the proof.
\end{proof}

\begin{rem}
It is clear that if $\alpha_0 > 0$ then we can just factor out $x^{\alpha_0}$
\[f(x) = a_0 x^{\alpha_0} + a_1 x^{\alpha_1} + \dots + a_n x^{\alpha_n} = x^{\alpha_0}\cdot (\underbrace{a_0 + a_1 x^{\alpha_1-\alpha_0} + \dots + a_n x^{\alpha_n - \alpha_0}}_{=:\tilde{f}(x)})\]
and apply \Cref{thm:algPosSatzab} or \Cref{cor:azero} to $\tilde{f}$.
\exmsymbol
\end{rem}

We now prove a stronger version of (\ref{eq:pos-11}).
We only need the sparse algebraic Positivstellensatz on $[a,b]$ (\Cref{thm:algPosSatzab}) but not the sparse algebraic Nichtnegativestellensatz (\Cref{thm:algNichtNegab}) even for $p\geq 0$ on $[a,b]$.
This result was already proved in \cite{karlin53}. Later the T-system approach was developed in \cite{karlin63} and summarized and expanded in \cite{karlinStuddenTSystemsBook}.

We now get the strengthened version of the \Cref{thm:luma}. Earlier versions are due to Markov \cite{markov06} and Luk\'acs \cite{lukacs18}, see the \Cref{thm:luma} in \Cref{sec:classicalResults} and the discussion around it.

\begin{lumathm}[see {\cite[Thm.\ 10.3]{karlin53}} or {\cite[p.\ 373, Thm.\ 6.4]{kreinMarkovMomentProblem}}]\label{cor:nonnegabRx}\label{thm:luma2}
\index{Luk\'acs--Markov Theorem}\index{Theorem!Luk\'acs--Markov}
Let $p\in\rset[x]$ with $p\geq 0$ on $[a,b]$ with $-\infty<a<b<\infty$ and let $z_1,\dots,z_r\in [a,b]$ be the zeros of $p$ in $[a,b]$ with algebraic multiplicities $m_1,\dots,m_r\in\nset$.
\begin{enumerate}[(i)]
\item If $\deg p - m_1 - \dots - m_r = 2m$, $m\in\nset_0$, is even then there exist points $x_1,\dots,x_m$ and $y_1,\dots,y_{m-1}$ with
\[a < x_1 < y_1 < \dots < y_{m-1} < x_m < b\]
and constants $\alpha,\beta>0$ such that
\begin{multline*}
p(x) = (x-z_1)^{m_1}\cdots (x-z_r)^{m_r}\cdot\left(\alpha\cdot\prod_{i=1}^m (x-x_i)^2\right.\\
\left.+ \beta\cdot(x-a)\cdot(b-x)\cdot\prod_{i=1}^{m-1} (x-y_i)^2 \right).
\end{multline*}

\item If $\deg p - m_1 - \dots - m_r = 2m+1$, $m\in\nset_0$, is odd then there exist points $x_1,\dots,x_m$ and $y_0,\dots,y_{m-1}$ with
\[a < y_0 < x_1 < y_1 < \dots < y_{m-1} < x_m < b\]
and constants $\alpha,\beta>0$ such that
\begin{multline*}
p(x) = (x-z_1)^{m_1}\cdots (x-z_r)^{m_r}\cdot\left(\alpha\cdot(x-a)\cdot\prod_{i=1}^m (x-x_i)^2\right.\\
\left.+ \beta\cdot(b-x)\cdot\prod_{i=0}^{m-1} (x-y_i)^2 \right).
\end{multline*}
\end{enumerate}
\end{lumathm}
\begin{proof}
We have $p(x) = (x-z_1)^{m_1}\cdots (x-z_r)^{m_r}\cdot\tilde{p}(x)$ with $\tilde{p}\in\rset[x]$ and $\tilde{p}>0$ on $[a,b]$.
By a translation $p(\,\cdot\,+a)$ we can assume $a=0$ and the assertion follows from \Cref{cor:azero}.
\end{proof}

Note, in \Cref{thm:algPosSatzab} (and \Cref{thm:algNichtNegab}) we need $a\geq 0$.
But in the \Cref{cor:nonnegabRx} we can allow for arbitrary $a\in\rset$ since by $p\in\rset[x]_{\leq \deg p}$ the translation $p(\,\cdot\,+a)$ remains in $\rset[x]_{\leq\deg p}$.
We see here also why in \Cref{thm:algPosSatzab} and \Cref{cor:azero} we have the restriction $a\geq 0$ since a translation can produce monomials which are not in the family $\{x^{\alpha_i}\}_{i=0}^n$.

Additionally, note that in \Cref{cor:nonnegabRx} we can have $z_i = a$ or $b$ for some $i$.

\section{Sparse Hausdorff Moment Problem}
\index{sparse!Hausdorff moment problem}

\Cref{thm:algPosSatzab} is a complete description of $\inter(\lin\cF)_+$.
Since $\cF$ is continuous on the compact interval $[a,b]$ and $x^{\alpha_0}>0$ on $[a,b]$, we have that the truncated moment cone is closed.
Hence, $(\lin\cF)_+$ and the moment cone are dual to each other.
With \Cref{thm:algPosSatzab} we can now write down the conditions for the sparse truncated Hausdorff moment problem on $[a,b]$ with $a>0$.
A first but insufficient attempt was done in \cite{hausdo21a} since Hausdorff\index{Hausdorff, F.} did not have access to the sparse Positivstellensatz by Karlin and therefore \Cref{thm:algPosSatzab}.

\begin{thm}[Sparse Truncated Hausdorff Moment Problem on {$[a,b]$ with $a>0$}]\index{Theorem!Hausdorff!sparse}\index{Hausdorff!truncated moment problem!sparse}\label{thm:sparseTruncHausd}\index{sparse!Hausdorff moment problem!truncated}
Let $n\in\nset_0$, $\alpha_0,\dots,\alpha_n\in[0,\infty)$ with $\alpha_0 < \dots < \alpha_n$, and $a,b$ with $0 < a < b$. Set $\cF = \{x^{\alpha_i}\}_{i=0}^n$. Then the following are equivalent:
\begin{enumerate}[(i)]
\item $L:\lin\cF\to\rset$ is a truncated $[a,b]$-moment functional.

\item $L(p)\geq 0$ holds for all
\begin{align*}
p(x) &:= \begin{cases}
\det \begin{pmatrix}
x^{\alpha_0} & x^{\alpha_1} & x^{\alpha_2} & \dots & x^{\alpha_{2m-1}} & x^{\alpha_{2m}}\\
x & (x_1 & x_1) & \dots & (x_m & x_m)
\end{pmatrix}\\
-\det\begin{pmatrix}
x^{\alpha_0} & x^{\alpha_1} & x^{\alpha_2} & x^{\alpha_3} & \dots & x^{\alpha_{2m-2}} & x^{\alpha_{2m-1}} & x^{\alpha_{2m}}\\
x & a & (x_1 & x_1) & \dots & (x_{m-1} & x_{m-1}) & b
\end{pmatrix}
\end{cases} \tag*{\textit{if} $n = 2m$}
\intertext{and}
p(x) &:=\begin{cases}
-\det\begin{pmatrix}
 x^{\alpha_0} & x^{\alpha_1} & x^{\alpha_2} & x^{\alpha_3} & \dots & x^{\alpha_{2m}} & x^{\alpha_{2m+1}}\\
x & a & (x_1 & x_1) & \dots & (x_m & x_m)
\end{pmatrix}\\
\det\begin{pmatrix}
x^{\alpha_0} & x^{\alpha_1} & x^{\alpha_2} & \dots & x^{\alpha_{2m-1}} & x^{\alpha_{2m}} & x^{\alpha_{2m+1}}\\
x & (x_1 & x_1) & \dots & (x_m & x_m) & b
\end{pmatrix}
\end{cases} \tag*{\textit{if} $n=2m+1$}
\end{align*}
and all $x_1,\dots,x_m$ with $a < x_1 < \dots < x_m < b$.
\end{enumerate}
\end{thm}
\begin{proof}
The implication (i) $\Rightarrow$ (ii) is clear since all given polynomials $p$ are non-negative on $[a,b]$.
It is therefore sufficient to prove (ii) $\Rightarrow$ (i).

Since $a>0$ we have that $x^{\alpha_0} > 0$ on $[a,b]$ and since $[a,b]$ is compact we have that the moment cone $((\lin\cF)_+)^*$ as the dual of the cone of non-negative (sparse) polynomials $(\lin\cF)_+$ is a closed pointed cone.

To establish $L\in ((\lin\cF)_+)^*$ it is sufficient to have $L(f)\geq 0$ for all $f\in(\lin\cF)_+$.
Let $f\in(\lin\cF)_+$.
Then for all $\varepsilon>0$ we have $f_\varepsilon := f+\varepsilon\cdot x^{\alpha_n} > 0$ on $[a,b]$, i.e., by \Cref{thm:algPosSatzab} $f_\varepsilon$ is a conic combination of the polynomials $p$ in (ii) and hence $L(f) + \varepsilon\cdot L(x^{\alpha_n}) = L(f_\varepsilon) \geq 0$ for all $\varepsilon>0$.
Since $x^{\alpha_n}>0$ on $[a,b]$ we also have that $x^{\alpha_n}$ is a conic combination of the polynomials $p$ in (ii) and therefore $L(x^{\alpha_n}) \geq 0$.
Then $L(f)\geq 0$ follows from $\varepsilon\to 0$ which proves (i).
\end{proof}

\begin{cor}
If $\alpha_0 = 0$ in \Cref{thm:sparseTruncHausd} then \Cref{thm:sparseTruncHausd} also holds with $a=0$, i.e., the following are equivalent:
\begin{enumerate}[(i)]
\item $L:\lin\cF\to\rset$ is a truncated $[0,b]$-moment functional.

\item $L(p)\geq 0$ holds for all
\begin{align*}
p(x) &:= \begin{cases}
\det \begin{pmatrix}
1 & x^{\alpha_1} & x^{\alpha_2} & \dots & x^{\alpha_{2m-1}} & x^{\alpha_{2m}}\\
x & (x_1 & x_1) & \dots & (x_m & x_m)
\end{pmatrix}\\
\det\begin{pmatrix}
x^{\alpha_1} & x^{\alpha_2} & x^{\alpha_3} & \dots & x^{\alpha_{2m-2}} & x^{\alpha_{2m-1}} & x^{\alpha_{2m}}\\
x & (x_1 & x_1) & \dots & (x_{m-1} & x_{m-1}) & b
\end{pmatrix}
\end{cases} \tag*{if $n = 2m$}
\intertext{and}
p(x) &:=\begin{cases}
\det\begin{pmatrix}
x^{\alpha_1} & x^{\alpha_2} & x^{\alpha_3} & \dots & x^{\alpha_{2m}} & x^{\alpha_{2m+1}}\\
x & (x_1 & x_1) & \dots & (x_m & x_m)
\end{pmatrix}\\
\det\begin{pmatrix}
1 & x^{\alpha_1} & x^{\alpha_2} & \dots & x^{\alpha_{2m-1}} & x^{\alpha_{2m}} & x^{\alpha_{2m+1}}\\
x & (x_1 & x_1) & \dots & (x_m & x_m) & b
\end{pmatrix}
\end{cases} \tag*{if $n=2m+1$}
\end{align*}
and all $x_1,\dots,x_m$ with $a < x_1 < \dots < x_m < b$.
\end{enumerate}
\end{cor}
\begin{proof}
Follows immediately from \Cref{cor:azero}.
\end{proof}

For the following we want to remind the reader of the \emph{M\"untz--Sz\'asz Theorem} \cite{muntz14,szasz16}.\index{Theorem!Muntz--Szasz@M\"untz--Sz\'asz}
It states that for real exponents $\alpha_0=0 < \alpha_1 < \alpha_2 < \dots$ the vector space $\lin\{x^{\alpha_i}\}_{i\in\nset_0}$ of finite linear combinations is dense in $\cat([0,1],\rset)$ with respect to the $\sup$-norm if and only if $\sum_{i\in\nset} \frac{1}{\alpha_i} = \infty$.

We state the following only for the classical case of the interval $[0,1]$.
Other cases $[a,b]\subseteq [0,\infty)$ are equivalent.
Hausdorff required $\alpha_i\to\infty$.
The M\"untz--Sz\'asz Theorem does not require $\alpha_i\to\infty$.
The conditions $\alpha_0=0$ and $\sum_{i\in\nset} \frac{1}{\alpha_i} = \infty$ already appear in \cite[eq.\ (17)]{hausdo21a}.
We can remove here the use of the M\"untz--Sz\'asz Theorem and therefore the condition $\sum_{i\in\nset} \frac{1}{\alpha_i} = \infty$ for the existence of a representing measure.
We need it only for uniqueness.
Additionally, we allow negative exponents.
The following is an improvement of \cite{hausdo21a} and we are not aware of a reference for this result.

\begin{thm}[General Sparse Hausdorff Moment Problem on {$[a,b]$ with $0\leq a < b$}]
\index{Hausdorff!moment problem!sparse}\label{thm:generalSparseHausd}\index{sparse!Hausdorff moment problem}
Let $I\subseteq\nset_0$ be an index set (finite or infinite), let $\{\alpha_i\}_{i\in I}$ be such that $\alpha_i\neq \alpha_j$ for all $i\neq j$ and
\begin{enumerate}[\quad (a)]
\item if $a=0$ then $\{\alpha_i\}_{i\in I}\subset [0,\infty)$ with $\alpha_i=0$ for an $i\in I$, or

\item if $a>0$ then $\{\alpha_i\}_{i\in I}\subset\rset$.
\end{enumerate}
Let $\cF = \{x^{\alpha_i}\}_{i\in I}$.
Then the following are equivalent:
\begin{enumerate}[(i)]
\item $L:\lin\cF\to\rset$ is a Hausdorff moment functional.

\item $L(p)\geq 0$ holds for all $p\in (\lin\cF)_+$.

\item $L(p)\geq 0$ holds for all $p\in\lin\cF$ with $p>0$.

\item $L(p)\geq 0$ holds for all
\[p(x) = \begin{cases}
\det \begin{pmatrix}
x^{\alpha_{i_0}} & x^{\alpha_{i_1}} & x^{\alpha_{i_2}} & \dots & x^{\alpha_{i_{2m-1}}} & x^{\alpha_{2m}}\\
x & (x_1 & x_1) & \dots & (x_m & x_m)
\end{pmatrix}, & \text{if $|I| = 2m$ or $\infty$,}\\
\det\begin{pmatrix}
x^{\alpha_{i_1}} & x^{\alpha_{i_2}} & x^{\alpha_{i_3}} & \dots & x^{\alpha_{i_{2m-2}}} & x^{\alpha_{i_{2m-1}}} & x^{\alpha_{i_{2m}}}\\
x & (x_1 & x_1) & \dots & (x_{m-1} & x_{m-1}) & b
\end{pmatrix}, & \text{if $|I| = 2m$ or $\infty$,}\\
\det\begin{pmatrix}
x^{\alpha_{i_1}} & x^{\alpha_{i_2}} & x^{\alpha_{i_3}} & \dots & x^{\alpha_{i_{2m}}} & x^{\alpha_{i_{2m+1}}}\\
x & (x_1 & x_1) & \dots & (x_m & x_m)
\end{pmatrix},\  & \text{if $|I| = 2m+1$ or $\infty$,}\\
\det\begin{pmatrix}
x^{\alpha_{i_0}} & x^{\alpha_{i_1}} & x^{\alpha_{i_2}} & \dots & x^{\alpha_{i_{2m-1}}} & x^{\alpha_{i_{2m}}} & x^{\alpha_{i_{2m+1}}}\\
x & (x_1 & x_1) & \dots & (x_m & x_m) & b
\end{pmatrix},  & \text{if $|I| = 2m+1$ or $\infty$,}
\end{cases}\]
for all $m\in\nset$ if $|I|=\infty$, all $0 < x_1 < x_2 < \dots < x_m < b$, and all $\alpha_{i_0} < \alpha_{i_1} < \dots < \alpha_{i_m}$ with $\alpha_{i_0}=0$ if $a=0$.
\end{enumerate}
If additionally $\sum_{i:\alpha_i\neq 0} \frac{1}{|\alpha_i|} = \infty$ then $L$ is determinate.
\end{thm}
\begin{proof}
The case $|I|<\infty$ is \Cref{thm:sparseTruncHausd}.
We therefore prove the case $|I|=\infty$.
The choice $\alpha_{i_0} < \alpha_{i_1} < \dots < \alpha_{i_m}$ with $\alpha_{i_0}=0$ if $a=0$ makes $\{x^{\alpha_{i_j}}\}_{j=0}^m$ a T-system.
The implications ``(i) $\Rightarrow$ (ii) $\Leftrightarrow$ (iii)'' are clear and ``(iii) $\Leftrightarrow$ (iv)'' is \Cref{thm:algPosSatzab}.
It is therefore sufficient to show ``(ii) $\Rightarrow$ (i)''.
But the space $\lin\cF$ is an adapted space and the assertion follows therefore from the \Cref{thm:basicrepresentation}.

For the determinacy of $L$ split $\{\alpha_i\}_{i\in I}$ into positive and negative exponents. If $\sum_{i:\alpha_i\neq 0} \frac{1}{|\alpha_i|} = \infty$ then the corresponding sum over at least one group is infinite. If the sum over the positive exponents is infinite apply the M\"untz--Sz\'asz Theorem. If the sum over the negative exponents is infinite apply the M\"untz--Sz\'asz Theorem to $\{(x^{-1})^{-\alpha_i}\}_{i\in I: \alpha_i < 0}$ since $a>0$.
\end{proof}

Note, since $[a,b]$ is compact the fact that $\{x^{\alpha_i}\}_{i\in I}$ is an adapted space is trivial.

\begin{rem}
If in \Cref{thm:generalSparseHausd} we have $a=0$ and $\alpha_0>0$ then we can of course factor out $x^{\alpha_0}$ and instead of determining $\diff\mu(x)$ of the linear functional $L$ we determine $\diff\tilde{\mu}(x) = x^{\alpha_0}~\diff\mu(x)$.
\exmsymbol
\end{rem}

\section{Sparse Algebraic Nichtnegativstellensatz on $[a,b]$}

The non-negative polynomials are described in the following result.

\begin{thm}[Sparse Algebraic Nichtnegativstellensatz on {$[a,b]$ with $0<a<b$}]\label{thm:algNichtNegab}\index{sparse!algebraic Nichtnegativstellensatz!on $[a,b]$}
Let $n\in\nset_0$, $\alpha_0,\dots,\alpha_n\in\rset$ be real numbers with $\alpha_0 < \alpha_1 < \dots < \alpha_n$, and let $\cF = \{x^{\alpha_i}\}_{i=0}^n$. Let $f\in\lin\cF$ with $f\geq 0$ on $[a,b]$.
Then there exist points $x_1,\dots,x_n,y_1,\dots,y_n\in [a,b]$ (not necessarily distinct) with $y_n=b$ which include the zeros of $f$ with multiplicities such that
\[f = f_* + f^*\]
with $f_*,f^*\in\lin\cF$, $f_*,f^*\geq 0$ on $[a,b]$.
The polynomials $f_*$ and $f^*$ are given by
\[f_*(x) = c_*\cdot\det\left(\begin{array}{c|ccc}
f_0 &\ f_1 & \dots & f_n\\
x &\ x_1 & \dots & x_n
\end{array}\right)
\qquad\text{and}\qquad
f^*(x) = c_*\cdot\det\left(\begin{array}{c|ccc}
f_0 &\ f_1 & \dots & f_n\\
x &\ y_1 & \dots & y_n
\end{array}\right)\]
for all $x\in [a,b]$ and some constants $c_*,c^*\in\rset$

Removing the zeros of $f$ from $x_1,\dots,x_n,y_1,\dots,y_n$ we can assume that the remaining $x_i$ and $y_i$ are disjoint and when grouped by size the groups strictly interlace:
\[a \leq x_{i_1} = \dots = x_{i_k} < y_{j_1} = \dots = y_{j_l} < \dots < x_{i_p} = \dots = x_{i_q} < y_{j_r} = \dots = y_{j_s}=b.\]
Each such group in $(a,b)$ has an even number of members.
\end{thm}
\begin{proof}
By \Cref{exm:algETsystem} we have that $\cF$ on $[a,b]$ is an ET-system. We then apply \Cref{thm:karlinNonNegab} similar to the proof of \Cref{thm:algPosSatzab}.
\end{proof}

\begin{rem}
The signs of $c_*$ and $c^*$ are determined by $x_1$ and $y_1$ and their multiplicity.
If $x_1 = \dots = x_k < x_{k+1} \leq \dots \leq x_n$ then $\sign\, c_* = (-1)^k$.
The same holds for $c^*$ from $y_1$.
\exmsymbol
\end{rem}

\begin{exm}
Let $\alpha\in (0,\infty)$ and let $\cF = \{1,x^{\alpha}\}$ on $[0,1]$. Then we have $1 = 1_* + 1^*$ with $1_* = x^\alpha$ and $1^* = 1 - x^\alpha$.\exmsymbol
\end{exm}

In \Cref{thm:algNichtNegab} we can let $a=0$ if $\alpha_0 = 0$ and $f(0)>0$.

\begin{thm}[Sparse Algebraic Nichtnegativstellensatz on {$[0,b]$ with $0<b$}]\label{thm:algNichtNeg0b}\index{sparse!algebraic Nichtnegativstellensatz!on $[0,b]$}
Let $n\in\nset_0$, $\alpha_0,\dots,\alpha_n\in\rset$ be real numbers with $0=\alpha_0 < \alpha_1 < \dots < \alpha_n$, and let $\cF = \{x^{\alpha_i}\}_{i=0}^n$ on $[0,b]$ with $b>0$.
Let $f\in\lin\cF$ with $f\geq 0$ on $[0,b]$ and $f(0)>0$.
Then there exist points $x_1,\dots,x_n,y_1,\dots,y_n\in [0,b]$ (not necessarily distinct) with $y_n=b$ which include the zeros of $f$ with multiplicities such that
\[f = f_* + f^*\]
with $f_*,f^*\in\lin\cF$, $f_*,f^*\geq 0$ on $[0,b]$ and the points $x_1,\dots,x_n$ are the zeros of $f_*$ and $y_1,\dots,y_n$ are the zeros of $f^*$.
Removing the zeros of $f$ from $x_1,\dots,x_n,y_1,\dots,y_n$ we can assume that the remaining $x_i$ and $y_i$ are disjoint and when grouped by size the groups strictly interlace:
\[0 \leq x_{i_1} = \dots = x_{i_k} < y_{j_1} = \dots = y_{j_l} < \dots < x_{i_p} = \dots = x_{i_q} < y_{j_r} = \dots = y_{j_s}=b.\]
Each such group in $(a,b)$ has an even number of members.
\end{thm}
\begin{proof}
See Problem \ref{prob:a0}.
\end{proof}

\section*{Problems}
\addcontentsline{toc}{section}{Problems}

\begin{prob}\label{prob:a0}
Prove \Cref{thm:algNichtNeg0b}, i.e., show that \Cref{thm:algNichtNegab} can be extended to the case $a=0$, i.e., on $[0,b]$ with $b>0$.
\end{prob}

\motto{Mathematics is the tool specially suited for dealing with\\
abstract concepts of any kind and there is no limit to its\\
power in this field.\\ \medskip
\ \hspace{1cm} \normalfont{Paul Adrien Maurice Dirac {\cite[p.\ viii]{dirac58}}}\index{Dirac, P.\ A.\ M.}}

\chapter{Non-Negative Algebraic Polynomials on $[0,\infty)$ and on $\rset$}
\label{ch:nonNegAlgPol0infty}

We went a long way to arrive here.
But by using \Cref{thm:karlinPos0infty} and \Cref{thm:karlinNonNeg0infty} on the interval $[0,\infty)$ we can now describe all sparse algebraic strictly positive and non-negative polynomials on $[0,\infty)$ and on $\rset$.

\section{Sparse Algebraic Positivstellensatz on $[0,\infty)$}

For the sparse algebraic Positivstellensatz on $[a,b]$ (\Cref{thm:algPosSatzab}) we had a lot of freedom in the exponents $\alpha_i$ for $a>0$.
We no longer have such a large range of freedom on $[0,\infty)$.
If we now plug \Cref{exm:tsysAlphaReal} into \Cref{thm:karlinPos0infty} we get the following.

\begin{thm}[Sparse Algebraic Positivstellensatz on {$[0,\infty)$}]\label{thm:algPosSatz0infty}
\index{sparse!algebraic Positivstellensatz!on $[0,\infty)$}
Let $n\in\nset_0$, $\alpha_0,\dots,\alpha_n\in [0,\infty)$ be real numbers with $\alpha_0 = 0 < \alpha_1 < \dots < \alpha_n$, and let $\cF = \{x^{\alpha_i}\}_{i=0}^n$ on $[0,\infty)$.
Then for any $f = \sum_{i=0}^n a_i f_i\in\lin\cF$ with $f>0$ on $[0,\infty)$ and $a_n>0$ there exists a unique decomposition
\[f = f_* + f^*\]
with $f_*,f^*\in\lin\cF$ and $f_*,f^*\geq 0$ on $[0,\infty)$ such that the following hold:
\begin{enumerate}[(i)]
\item If $n = 2m$ then the polynomials $f_*$ and $f^*$ each possess $m$ distinct zeros $\{x_i\}_{i=1}^m$ and $\{y_i\}_{i=0}^{m-1}$ satisfying
\[0 = y_0 < x_1 < y_1 < \dots < y_{m-1} < x_m < \infty.\]
The polynomials $f_*$ and $f^*$ are given by
\[f_*(x) = c_*\cdot\det\begin{pmatrix}
1 & x^{\alpha_1} & x^{\alpha_2} & \dots & x^{\alpha_{2m-1}} & x^{\alpha_{2m}}\\
x & (x_1 & x_1) & \dots & (x_m & x_m)
\end{pmatrix}\]
and
\begin{align*}
f^*(x) 
%
&= c^*\cdot \det\begin{pmatrix}
x^{\alpha_1} & x^{\alpha_2} & x^{\alpha_3} &\dots& x^{\alpha_{2m-2}} & x^{\alpha_{2m-1}}\\
x & (y_1 & y_1) & \dots & (y_{m-1} & y_{m-1})
\end{pmatrix}
\end{align*}
for some $c_*,c^*>0$.

\item If $n=2m+1$ then $f_*$ and $f^*$ have zeros $\{x_i\}_{i=1}^{m+1}$ and $\{y_i\}_{i=1}^m$ respectively which satisfy
\[0 = x_1 < y_1 < x_2 < \dots < y_m < x_{m+1}<\infty.\]
The polynomials $f_*$ and $f^*$ are given by
\begin{align*}
f_*(x) 
&= c_*\cdot \det\begin{pmatrix}
x^{\alpha_1} & x^{\alpha_2} & x^{\alpha_3} & \dots & x^{\alpha_{2m}} & x^{\alpha_{2m+1}}\\
x & (x_2 & x_2) & \dots & (x_{m+1} & x_{m+1})
\end{pmatrix}
\end{align*}
and
\[f^*(x) = c^*\cdot\det\begin{pmatrix}
1 & x^{\alpha_1} & x^{\alpha_2} & \dots & x^{\alpha_{2m-1}} & x^{\alpha_{2m}}\\
x & (y_1 & y_1) & \dots & (y_m & y_m)
\end{pmatrix}\]
for some $c_*,c^*>0$.
\end{enumerate}
\end{thm}
\begin{proof}
We have that $\cF$ fulfills conditions (a) and (b) of \Cref{thm:karlinPos0infty} and by \Cref{exm:tsysAlpha} we known that $\cF$ on $[0,\infty)$ is also a T-system, i.e., (c) in \Cref{thm:karlinPos0infty} is fulfilled.
We can therefore apply \Cref{thm:karlinPos0infty}.

(i) $n=2m$: By \Cref{thm:karlinPos0infty} (i) the unique $f_*$ and $f^*$ each possess $m$ distinct zeros $\{x_i\}_{i=1}^m$ and $\{y_i\}_{i=0}^{m-1}$ with $0\leq y_0 < x_1 < \dots < y_{m-1} < x_m < \infty$.
Since $x_1,\dots,x_m\in (0,\infty)$ and $\cF$ on $[x_1/2,\infty)$ is an ET-system we immediately get the determinantal representation of $f_*$ by \Cref{cor:posrepr0infty} (combine \Cref{thm:karlinPos0infty} with \Cref{rem:doublezeros}).
For $f^*$ we have $y_0=0$ and by \Cref{exm:nonETalg} this is no ET-system.
Hence, we prove the representation of $f^*$ by hand, similar as in the proof of \Cref{cor:azero}.

Let $\varepsilon>0$ be such that $0=y_0 < y_1 < y_1+\varepsilon < \dots < y_{m-1} < y_{m-1}+\varepsilon$ holds. Then
\begin{align*}
g_\varepsilon(x) &= -\varepsilon^{-m+1}\cdot \det \begin{pmatrix}
1 & x^{\alpha_1} & x^{\alpha_2} & x^{\alpha_3} & \dots & x^{\alpha_{2m-2}} & x^{\alpha_{2m-1}}\\
x & 0 & y_1 & y_1+\varepsilon & \dots & y_{m-1} & y_{m-1}+\varepsilon
\end{pmatrix}\\
&= -\varepsilon^{-m+1}\cdot \det\begin{pmatrix}
1 & x^{\alpha_1} & x^{\alpha_2} & \dots & x^{\alpha_{2m-1}}\\
1 & 0 & 0 & \dots & 0\\
1 & y_1^{\alpha_1} & y_1^{\alpha_2} & \dots & y_1^{\alpha_{2m-1}}\\
\vdots & \vdots & \vdots & & \vdots\\
1 & (y_{m-1}+\varepsilon)^{\alpha_1} & (y_{m-1}+\varepsilon)^{\alpha_2} & \dots & (y_{m-1}+\varepsilon)^{\alpha_{2m-1}}
\end{pmatrix}
\intertext{expand by the second row}
&= \varepsilon^{-m+1}\cdot\det\begin{pmatrix}
x^{\alpha_1} & x^{\alpha_2} & \dots & x^{\alpha_{2m-1}}\\
y_1^{\alpha_1} & y_1^{\alpha_2} & \dots & y_1^{\alpha_{2m-1}}\\
\vdots & \vdots & & \vdots\\
(y_{m-1}+\varepsilon)^{\alpha_1} & (y_{m-1}+\varepsilon)^{\alpha_2} & \dots & (y_{m-1}+\varepsilon)^{\alpha_{2m-1}}
\end{pmatrix}\\
&= \varepsilon^{-m+1}\cdot\det\begin{pmatrix}
x^{\alpha_1} & x^{\alpha_2} & \dots & x^{\alpha_{2m-2}} & x^{\alpha_{2m-1}}\\
x & y_1 & y_1+\varepsilon & \dots & y_{m-1} & y_{m-1}+\varepsilon
\end{pmatrix}
\end{align*}
is non-negative on $[0,y_1]$ and every $[y_i+\varepsilon,y_{i+1}]$.
Now $y_0=0$ is removed and all $y_i,y_i+\varepsilon>0$.
Hence, we can work on $[y_1/2,\infty)$ where $\{x^{\alpha_i}\}_{i=1}^{2m}$ is an ET-system and we can go to the limit $\varepsilon\searrow 0$ as in \Cref{rem:doublezeros}.
Then \Cref{cor:posrepr0infty} proves the representation of $f^*$.

(ii) $n=2m+1$: Similar to the case (i) with $n=2m$.
\end{proof}

If all $\alpha_i\in\nset_0$ then we can express the $f_*$ and $f^*$ in \Cref{thm:algPosSatz0infty} also with Schur polynomials, see (\ref{eq:schurPolyRepr}) in \Cref{exm:algETsystem}.

We now prove a stronger version of (\ref{eq:pos0infty2}), i.e., $p = f^2 + x\cdot g^2$ for any $p\geq 0$ on $[0,\infty)$.
It is sufficient to have only the sparse algebraic Positivstellensatz (\Cref{thm:algPosSatz0infty}).
A previous version already appeared in \cite{karlin53}.

\begin{cor}[see {\cite[p.\ 169, Cor.\ 8.1]{karlinStuddenTSystemsBook}}]\label{cor:nonneg0inftyRx}
Let $p\in\rset[x]$ with $p\geq 0$ on $[0,\infty)$.
Let $z_1,\dots,z_r\in [0,\infty)$ be the zeros of $p$ in $[0,\infty)$ and let $m_1,\dots,m_r\in\nset$ be the corresponding algebraic multiplicities.
\begin{enumerate}[(i)]
\item If $\deg p - m_1 - \dots - m_r = 2m$, $m\in\nset_0$, is even then there exist points $\{x_i\}_{i=1}^m$ and $\{y_i\}_{i=1}^{m-1}\subseteq (0,\infty)$ with
\[0 < x_1 < y_1 < \dots < y_{m-1} < x_m < \infty\]
and constants $a,b>0$ such that
\[p(x) = \prod_{i=1}^r (x-z_i)^{m_i}\cdot \left( a\cdot\prod_{i=1}^m (x-x_i)^2 + b\cdot x\cdot\prod_{i=1}^{m-1} (x-y_i)^2 \right).\]
The constant $a$ is the leading coefficient of $p$.

\item If $\deg p - m_1 - \dots - m_r = 2m+1$, $m\in\nset_0$, is odd then there exist points $\{x_i\}_{i=1}^m$ and $\{y_i\}_{i=1}^m\subset (0,\infty)$ with
\[0 < x_1 < y_1 < \dots < x_m < y_m < \infty\]
and constants $a,b>0$ such that
\[p(x) = \prod_{i=1}^r (x-z_i)^{m_i}\cdot \left( a\cdot\prod_{i=1}^m (x-x_i)^2 + b\cdot x\cdot\prod_{i=1}^m (x-y_i)^2 \right).\]
The constant $b$ is the leading coefficient of $p$.
\end{enumerate}
\end{cor}
\begin{proof}
Since $z_1,\dots,z_r$ are the zeros of $p$ in $[0,\infty)$ with multiplicities $m_1,\dots,m_r$ we have that $p(x) = (x-z_1)^{m_1}\cdots (x-z_r)^{m_r}\cdot \tilde{p}(x)$ with $\tilde{p}\in\rset[x]$ and $\tilde{p}>0$ on $[0,\infty)$. Applying \Cref{thm:algPosSatz0infty} to $\tilde{p}$ gives the assertion.
\end{proof}

Note, in the previous result we were able to factor out the zeros of $p$ and were only left with $\tilde{p}>0$ on $[0,\infty)$ since we are working in $\rset[x]_{\leq\deg p}$ where all monomials $1,x,\dots,x^{\deg p}$ are present.
In sparse systems we are not able to factor out the zeros since we no longer know which monomials in $\tilde{p}$ will appear.

\begin{rem}\label{rem:factorizationfg0infty}
Working in the sparse setting, i.e., in T-systems, gives us an additional information.
In (\ref{eq:pos0infty2}) we only have $p(x) = x\cdot f^2 + g^2$.
But this also includes that $f$ and $g$ might contain factors $((x-y_i)^2 + \delta_i)$ with $\delta_i >0$, i.e., a pair of complex conjugated zeros can be present.
In \Cref{cor:nonneg0inftyRx} we see that this is not necessary.
The polynomials $f$ and $g$ can always be chosen such that they decompose into linear factors having only real zeros.
A similar results holds on $\rset$, see \Cref{cor:nonnegRRx}.\exmsymbol
\end{rem}

\section{Sparse Stieltjes Moment Problem}

In \Cref{sec:earlyGaps} we have seen that Boas\index{Boas, R.\ P.} already investigated the sparse Stieltjes moment problem \cite{boas39}.
However, the description was complicated and is even incomplete since Boas did not had access to \Cref{thm:karlinPos0infty} and therefore \Cref{thm:algPosSatz0infty}.
We get the following complete and simple description.
It fully solves \cite{boas39}.
We are not aware of a reference for the following result.

\begin{thm}[Sparse Stieltjes Moment Problem]\label{thm:sparseStieltjesMP}\index{sparse!Stieltjes moment problem}
Let $\{\alpha_i\}_{i\in\nset_0}\subseteq [0,\infty)$ be such that $\alpha_0 = 0 < \alpha_1 < \alpha_2 < \dots$ and let $\cF = \{x^{\alpha_i}\}_{i\in\nset_0}$.
Then the following are equivalent:
\begin{enumerate}[(i)]
\item $L:\lin\cF\to\rset$ is a $[0,\infty)$-moment functional.

\item $L(p)\geq 0$ for all $p\in\lin\cF$ with $p\geq 0$.

\item $L(p)\geq 0$ for all $p\in\lin\cF$ with $p>0$.

\item $L(p)\geq 0$ for all
\[p(x) = \begin{cases}
\det\begin{pmatrix}
1 & x^{\alpha_1} & x^{\alpha_2} & \dots & x^{\alpha_{2m-1}} & x^{\alpha_{2m}}\\
x & (x_1 & x_1) & \dots & (x_m & x_m)
\end{pmatrix},\\
\det\begin{pmatrix}
x^{\alpha_1} & x^{\alpha_2} & x^{\alpha_3} &\dots& x^{\alpha_{2m-2}} & x^{\alpha_{2m-1}}\\
x & (x_1 & x_1) & \dots & (x_{m-1} & x_{m-1})
\end{pmatrix},\\
\det\begin{pmatrix}
x^{\alpha_1} & x^{\alpha_2} & x^{\alpha_3} & \dots & x^{\alpha_{2m}} & x^{\alpha_{2m+1}}\\
x & (x_2 & x_2) & \dots & (x_{m+1} & x_{m+1})
\end{pmatrix},\ \text{and}\\
\det\begin{pmatrix}
1 & x^{\alpha_1} & x^{\alpha_2} & \dots & x^{\alpha_{2m-1}} & x^{\alpha_{2m}}\\
x & (x_1 & x_1) & \dots & (x_m & x_m)
\end{pmatrix}
\end{cases}\]
for all $m\in\nset_0$ and $0<x_1<\dots<x_m$.
\end{enumerate}
\end{thm}
\begin{proof}
The implications ``(i) $\Rightarrow$ (ii) $\Leftrightarrow$ (iii)'' are clear and ``(iii) $\Leftrightarrow$ (iv)'' is \Cref{thm:algPosSatz0infty}.
It is therefore sufficient to prove ``(ii) $\Rightarrow$ (i)''.

We have $\lin\cF = (\lin\cF)_+ - (\lin\cF)_+$, we have $1 = x^{\alpha_0}\in\lin\cF$, and for any $g = \sum_{i=0}^m a_i\cdot x^{\alpha_i} \in (\lin\cF)_+$ we have $\lim_{x\to\infty} \frac{g(x)}{x^{\alpha_{m+1}}} = 0$, i.e., there exists a $f\in (\lin\cF)_+$ which dominates $g$.
Hence, $\lin\cF$ is an adapted space on $[0,\infty)$ and the assertion follows from the \Cref{thm:basicrepresentation}.
\end{proof}

In the previous result we did needed $0=\alpha_0 < \alpha_1 < \alpha_2 < \dots$. We did not needed $\alpha_i\to\infty$.
Hence, \Cref{thm:sparseStieltjesMP} also includes the case $\sup_{i\in\nset_0}\alpha_i < \infty$.

\Cref{thm:sparseStieltjesMP} also holds with $\alpha_0 > 0$ since we can factor out $x^{\alpha_0}$ and therefore determine $x^{\alpha_0}~\diff\mu(x)$ instead of $\diff\mu(x)$.

\section{Sparse Algebraic Nichtnegativstellensatz on $[0,\infty)$}

For $\{1,x,x^3\}$ we have seen in \Cref{exm:nonETalg} that this is not an ET-systen on $[0,\infty)$, or on any other $[0,b]$.
If we remove the point $x=0$ and work on $(0,\infty)$ then it is an ET-system and even an ECT-system (\Cref{exm:etSystems}).
For a Nichtnegativstellensatz we therefore have to exclude zeros at $x=0$ in a sparse polynomial $p\geq 0$.

\begin{thm}[Sparse Algebraic Nichtnegativstellensatz on $[0,\infty)$]\label{thm:sparseNonneg0infty}\index{sparse!algebraic Nichtnegativstellensatz!on $[0,\infty)$}
Let $n\in\nset_0$, $\alpha_0,\dots,\alpha_n\in [0,\infty)$ be real numbers with $\alpha_0 = 0 < \alpha_1 < \dots < \alpha_n$, and let $\cF = \{x^{\alpha_i}\}_{i=0}^n$.
Let $f = \sum_{i=0}^n a_i x^{\alpha_i} \geq 0$ on $[0,\infty)$ with $a_n>0$ and $f(0) = a_0 > 0$.
Then there exist points $x_1,\dots,x_n,y_1,\dots,y_{n-1}\in [0,\infty)$ (not necessarily distinct) which include the zeros of $f$ with multiplicities and there exist constants $c_*,c^*\in\rset$ such that
\[f = f_* + f^*\]
with $f_*,f^*\in\lin\cF$, $f_*,f^*\geq 0$ on $[0,\infty)$, and the polynomials $f_*$ and $f^*$ are given by
\[f_*(x) = c_*\cdot\det\left(\begin{array}{c|ccc}
1 &\ x^{\alpha_1} & \dots & x^{\alpha_n}\\
x &\ x_1 & \dots & x_n
\end{array}\right)
\qquad\text{and}\qquad
f^*(x) = c_*\cdot\det\left(\begin{array}{c|ccc}
1 &\ x^{\alpha_1} & \dots & x^{\alpha_{n-1}}\\
x &\ y_1 & \dots & y_{n-1}
\end{array}\right)\]
for all $x\in [0,\infty)$.
\end{thm}
\begin{proof}
See Problem \ref{prob:nonneg0infty}.
\end{proof}

\begin{rem}\label{rem:a0infty}
Note, if $f(0) = a_0 = 0$ in \Cref{thm:sparseNonneg0infty} then
\[f(x) = a_ix^{\alpha_i} + a_{i+1}x^{\alpha_{i+1}} + \dots + a_n x^{\alpha_n} = x^{\alpha_i}\cdot (\underbrace{a_i + a_{i+1}x^{\alpha_{i+1}-\alpha_i} + \dots + a_n x^{\alpha_n-\alpha_i}}_{=:\tilde{f}(x)})\]
where $a_i$ is the first non-zero coefficient and it fulfills $a_i>0$ since $f\geq 0$.
Then apply \Cref{thm:sparseNonneg0infty} to $\tilde{f}$ to get $\tilde{f} = \tilde{f}_* + \tilde{f}^*$ and hence $f = x^{\alpha_i}\cdot (\tilde{f}_* + \tilde{f}^*)$.
\exmsymbol
\end{rem}

\section{Algebraic Positiv- and Nichtnegativstellensatz on $\rset$}

Since we treat $\cF = \{x^{i}\}_{i=0}^n$ we need only \Cref{thm:karlinPosR} on $\rset$ but not \Cref{thm:karlinNonNegR} on $\rset$ as we will see in the next result.

\begin{thm}[Algebraic Positiv- and Nichtnegativstellensatz on $\rset$, see {\cite[]{karlin53}} or e.g.\ {\cite[p.\ 198, Cor.\ 8.1]{karlinStuddenTSystemsBook}}]\label{cor:nonnegRRx}\index{sparse!algebraic Positivstellensatz!on $\rset$}\index{sparse!algebraic Nichtnegativstellensatz!on $\rset$}
Let $p\in\rset[x]$ with $p\geq 0$ on $\rset$ and let $z_1,\dots,z_r\in\rset$ be the zeros of $p$ with algebraic multiplicities $m_1,\dots,m_r\in 2\nset$.
Then there exist pairwise distinct points $\{x_i\}_{i=1}^m,\{y_i\}_{i=1}^{m-1}\subseteq\rset$ with $2m =\deg p - m_1 - \dots - m_r$ and
\[ -\infty < x_1 < y_1 < \dots < y_{m-1} < x_m < \infty\]
as well as constants $a,b>0$ such that
\begin{equation}\label{eq:p=f2g2}
p(x) = \prod_{i=1}^r (x-z_i)^{m_i}\cdot \left(a\cdot\prod_{i=1}^m (x-x_i)^2 + b\cdot\prod_{i=1}^{m-1} (x-y_i)^2 \right).
\end{equation}
The constant $a$ is the leading coefficient of $p$.
\end{thm}
\begin{proof}
We have $p(x) = (x-z_1)^{m_1}\cdots (x-z_r)^{m_r}\cdot \tilde{p}(x)$ for some $\tilde{p}\in\rset[x]$ with $\tilde{p}>0$ on $\rset$. Applying \Cref{thm:karlinPosR} to $\tilde{p}$ gives the assertion.
\end{proof}

Like in the case on $[0,\infty)$ in \Cref{cor:nonneg0inftyRx} a factorization
\[p(x) = (x-z_1)^{m_1}\cdots (x-z_r)^{m_r}\cdot \tilde{p}(x)\]
is not possible in T-systems or sparse algebraic systems on $\rset$.
But since we are working in $\rset[x]_{\leq\deg p}$ all monomials $1,x,\dots,x^{\deg p}$ are present.

\begin{rem}
Similar to \Cref{rem:factorizationfg0infty} we see that \Cref{cor:nonnegRRx} gives a stronger version of (\ref{eq:posR}), i.e., $p = f^2 + g^2$.
By applying only the Fundamental Theorem of Algebra\index{Theorem!Algebra!Fundamental}\index{Fundamental Theorem of Algebra} $f$ and $g$ might contain pairs of complex conjugated zeros, see e.g.\ \cite[Prop.\ 1.2.1]{marshallPosPoly}.
But by working in the T-system framework of \Cref{thm:karlinPosR} on $\rset$ we see that $f$ and $g$ can be chosen to have only real zeros.\exmsymbol
\end{rem}

\section*{Problems}
\addcontentsline{toc}{section}{Problems}

\begin{prob}\label{prob:nonneg0infty}
Use \Cref{thm:karlinNonNeg0infty} to prove \Cref{thm:sparseNonneg0infty}.
\end{prob}

\begin{prob}\label{prob:leadingcoeffR}
Show that $a$ in (\ref{eq:p=f2g2}) in \Cref{cor:nonnegRRx} is the leading coefficient of $p$.
\end{prob}

\part{Applications of T-Systems}

\motto{Long is the way and hard, that out of Hell leads up to light.\\ \medskip
\ \hspace{1cm} \normalfont{John Milton: Paradise Lost}\index{Milton, J.}}

\chapter{Moment Problems for continuous T-Systems on $[a,b]$}

In this chapter we demonstrate how e.g.\ \Cref{thm:karlinPosab} for general T-systems on $[a,b]$ can be used to prove moment problems which do not live on the algebraic polynomials $\rset[x]$.

\section{General Moment Problems for continuous T-Systems on $[a,b]$}

For T-system $\cF$ on $[a,b]$ \Cref{thm:karlinPosab} describes all polynomials $f\in\lin\cF$ with $f>0$.

\begin{thm}\label{thm:genMPtsystemAB}
Let $n\in\nset$, let $\cF = \{f_i\}_{i=0}^n$ be a continuous T-system on $[a,b]$ with $a<b$.
The following are equivalent:
\begin{enumerate}[(i)]
\item $L:\lin\cF\to\rset$ is an $[a,b]$-moment functional.

\item $L(f)\geq 0$ for all $f\in\lin\cF$ such that
\begin{enumerate}[(a)]
\item $f\geq 0$ on $[a,b]$ and

\item the zero set of $f$ has index $n$.
\end{enumerate}
\end{enumerate}
\end{thm}
\begin{proof}
The implication (i) $\Rightarrow$ (ii) is clear since $f\geq 0$.
It is therefore sufficient to prove (ii) $\Rightarrow$ (i).

Since $\cF$ is a continuous T-system there exists a polynomial $e\in\lin\cF$ with $e>0$ on $[a,b]$.
Since $[a,b]$ is compact, $\cF$ is continuous and finite dimensional, and there exists a $e>0$ we have that the moment cone $((\lin\cF)_+)^*$ is closed.
Therefore, to show that $L$ is a moment functional it is sufficient to show that $L(f)\geq 0$ for all $f\in(\lin\cF)_+$.

By \Cref{thm:karlinPosab} there are $e_*,e^*\in\lin\cF$ with $e_*,e^*\geq 0$ and the zero sets of $e_*$ and of $e^*$ have index $n$.
Hence, $L(e) = L(e_*) + L(e^*) \geq 0$.

Let $f\in(\lin\cF)_+$ and $\varepsilon>0$.
Then $f_\varepsilon = f + \varepsilon\cdot e >0$ on $[a,b]$, i.e., by \Cref{thm:karlinPosab} there exist $(f_\varepsilon)_*, (f_\varepsilon)^*\in (\lin\cF)_+$ each with zero sets of index $n$.
Assumption (ii) then implies $L(f + \varepsilon\cdot e) = L((f_\varepsilon)_*) + L((f_\varepsilon)^*)\geq 0$ for all $\varepsilon>0$, i.e., $L(f)\geq 0$.
That proves the assertion.
\end{proof}

Note, that a continuous T-system on $[a,b]$ is always an adapted space.
Additionally, the use of \Cref{thm:basicrepresentation} is not necessary since we only need to check in this case $L\in ((\lin\cF)_+)^*$ since the moment cone is $((\lin\cF)_+)^*$ and hence it is closed.

If in the previous theorem we additionally have that $\cF$ is an ET-system then we can write down $f_*$ and $f^*$ explicitly in the similar way as in \Cref{thm:sparseTruncHausd}.

\begin{thm}\label{thm:genMPetSystem}
Let $n\in\nset$, let $\cF = \{f_i\}_{i=0}^n$ be an ET-system on $[a,b]$ with $a<b$.
The following are equivalent:
\begin{enumerate}[(i)]
\item $L:\lin\cF\to\rset$ is a moment functional.

\item $L(f)\geq 0$ holds for all
\begin{align*}
f(x) &:= \begin{cases}
\det \begin{pmatrix}
f_0 & f_1 & f_2 & \dots & f_{2m-1} & f_{2m}\\
x & (x_1 & x_1) & \dots & (x_m & x_m)
\end{pmatrix}\\
-\det\begin{pmatrix}
f_0 & f_1 & f_2 & f_3 & \dots & f_{2m-2} & f_{2m-1} & f_{2m}\\
x & a & (x_1 & x_1) & \dots & (x_{m-1} & x_{m-1}) & b
\end{pmatrix}
\end{cases} \tag*{\textit{if} $n = 2m$}
\intertext{and}
f(x) &:=\begin{cases}
-\det\begin{pmatrix}
f_0 & f_1 & f_2 & f_3 & \dots & f_{2m} & f_{2m+1}\\
x & a & (x_1 & x_1) & \dots & (x_m & x_m)
\end{pmatrix}\\
\det\begin{pmatrix}
f_0 & f_1 & f_2 & \dots & f_{2m-1} & f_{2m} & f_{2m+1}\\
x & (x_1 & x_1) & \dots & (x_m & x_m) & b
\end{pmatrix}
\end{cases} \tag*{\textit{if} $n=2m+1$}
\end{align*}
and all $x_1,\dots,x_m$ with $a < x_1 < \dots < x_m < b$.
\end{enumerate}
\end{thm}
\begin{proof}
Follows from \Cref{thm:genMPtsystemAB} with \Cref{thm:etDet}.
\end{proof}

\section{A Non-Polynomial Example}

In \Cref{exm:fraction} we have seen that
\[\cF=\left\{\frac{1}{x+\alpha_0},\frac{1}{x+\alpha_1},\dots,\frac{1}{x+\alpha_n}\right\}\]
with $n\in\nset$ and $\alpha_0 < \alpha_1 < \dots < \alpha_n$ reals is a continuous T-system on any $[a,b]$ with $-\alpha_0 < a < b$, see Problem \ref{prob:fraction} for the proof.
But in the proof of \Cref{exm:fraction} we actually showed that this $\cF$ is an ET-system since we multiplied with $(x+\alpha_0)\cdots (x+\alpha_n)$ which has no zeros on $[a,b]$ and hence the multiplicities of the zeros do not change.
Multiplicity restriction from the fundamental theorem of algebra then shows that $\cF$ is an ET-system.

\begin{cor}\label{cor:fractionET}
Let $n\in\nset$ and $\alpha_0 < \alpha_1 < \dots < \alpha_n$ be reals. Then
\[\cF=\left\{\frac{1}{x+\alpha_0},\frac{1}{x+\alpha_1},\dots,\frac{1}{x+\alpha_n}\right\}\]
is an ET-system on any $[a,b]$ with $-\alpha_0 < a < b$.
\end{cor}

From \Cref{thm:genMPetSystem} and \Cref{cor:fractionET} we therefore get the following.

\begin{cor}
Let $n\in\nset$, let $\alpha_0 < \alpha_1 < \dots < \alpha_n$ be reals, and let
\[\cF=\left\{f_0(x) = \frac{1}{x+\alpha_0}, f_1(x) = \frac{1}{x+\alpha_1},\dots, f_n(x) = \frac{1}{x+\alpha_n}\right\}\]
on $[a,b]$ with $-\alpha_0 < a < b$.
Then the following are equivalent:
\begin{enumerate}[(i)]
\item $L:\lin\cF\to\rset$ is a $[a,b]$-moment functional.

\item $L(f)\geq 0$ holds for all
\begin{align*}
f(x) &:= \begin{cases}
\det \begin{pmatrix}
f_0 & f_1 & f_2 & \dots & f_{2m-1} & f_{2m}\\
x & (x_1 & x_1) & \dots & (x_m & x_m)
\end{pmatrix}\\
-\det\begin{pmatrix}
f_0 & f_1 & f_2 & f_3 & \dots & f_{2m-2} & f_{2m-1} & f_{2m}\\
x & a & (x_1 & x_1) & \dots & (x_{m-1} & x_{m-1}) & b
\end{pmatrix}
\end{cases} \tag*{\textit{if} $n = 2m$}
\intertext{and}
f(x) &:=\begin{cases}
-\det\begin{pmatrix}
f_0 & f_1 & f_2 & f_3 & \dots & f_{2m} & f_{2m+1}\\
x & a & (x_1 & x_1) & \dots & (x_m & x_m)
\end{pmatrix}\\
\det\begin{pmatrix}
f_0 & f_1 & f_2 & \dots & f_{2m-1} & f_{2m} & f_{2m+1}\\
x & (x_1 & x_1) & \dots & (x_m & x_m) & b
\end{pmatrix}
\end{cases} \tag*{\textit{if} $n=2m+1$}
\end{align*}
and all $x_1,\dots,x_m$ with $a < x_1 < \dots < x_m < b$.
\end{enumerate}
\end{cor}

In a similar way many other T-system moment problems can be proven from \Cref{thm:genMPtsystemAB}.

\motto{The rest is silence.\\ \medskip
\ \hspace{1cm} \normalfont{William Shakespeare: Hamlet (Act 5, Scene 2)}\index{Shakespeare, W.}}

\chapter{Polynomials of Best Approximation and Optimization over Linear Functionals}
\label{ch:best}

This last chapter is devoted to best approximation polynomials and optimization over linear functionals.

We started in \Cref{ch:measures} with moments and moment functionals, went to the theory of T-systems in \Cref{part:tSystems}, proved Karlin's Theorems in \Cref{part:karlinPosNonNeg}, and applied them to algebraic polynomials in \Cref{part:algPos}.
Now we finish our lecture by closing the circle.
We apply the previous results to best approximation in \Cref{sec:bestApprox} and to optimization over linear (moment) functionals in \Cref{sec:momOpt}.

\section{Polynomials of Best Approximation}
\label{sec:bestApprox}

A classical question is:
\begin{quote}
How to approximate a given function $f\in\cat([a,b],\rset)$ in the $\sup$-norm by a finite linear combination $\sum_{i=0}^n a_i f_i$ of some given $f_0,\dots,f_n\in\cat([a,b],\rset)$?
\end{quote}

\begin{dfn}
Let $n\in\nset_0$, let $f,f_0,\dots,f_n\in\cat([a,b],\rset)$, and let $\cF:=\{f_i\}_{i=0}^n$.
The polynomial $\underline{f}\in\lin\cF$ which solves
\begin{equation}\label{eq:tApprox}
\min_{a_0,\dots,a_n} \left\|\; f - \sum_{i=0}^n a_i f_i \;\right\|_\infty
\end{equation}
is called the \emph{polynomial of best approximation}.\index{polynomial!best approximation}\index{best approximation!polynomial}\index{approximation!best!polynomial}
\end{dfn}

Approximations (\ref{eq:tApprox}) with the $\sup$-norm are called \emph{Tchebycheff approximations}.\index{Tchebycheff approximation}\index{approximation!Tchebycheff}

The connection between polynomials of best approximation and T-systems is revealed in the following result.

\begin{thm}[see \cite{haar18}, \cite{bernstein26}; or e.g.\ {\cite[p.\ 74, §48]{achieser56}}, {\cite[p.\ 280, Thm.\ 1.1]{karlinStuddenTSystemsBook}}]\label{thm:bestApprox1}
Let $n\in\nset_0$, let $a,b\in\rset$ with $a<b$, and let $\cF:=\{f_i\}_{i=0}^n\subseteq\cat([a,b],\rset)$ be a family of continuous functions.
The following hold:
\begin{enumerate}[(i)]
\item The following are equivalent:
\begin{enumerate}[(a)]
\item The polynomial minimizing
\begin{equation}\label{eq:minInfimum}
\min_{a_0,\dots,a_n} \left\|\; f - \sum_{i=0}^n a_i f_i \;\right\|_\infty
\end{equation}
is uniquely determined for every $f\in\cat([a,b],\rset)$.

\item The family $\cF$ is a continuous T-system on $[a,b]$.
\end{enumerate}

\item If $\cF$ is a T-system then for each $f\in\cat([a,b],\rset)$ the unique polynomial
\[\underline{f} = \sum_{i=0}^n \underline{a}_i f_i\]
minimizing (\ref{eq:minInfimum}) is characterized by the property that there exist $n+2$ points
\[a\leq\quad x_1 < x_2 < \dots < x_{n+2} \quad\leq b\]
such that
\begin{equation*}
(-1)^i\cdot\delta\cdot (f(x_i) - \underline{f}(x_i)) = \max_{a\leq x\leq b} \left|\, f(x) - \underline{f}(x)\,\right|
\end{equation*}
holds for all $i=1,2,\dots,n+2$ with $\delta = +1$ or $-1$.
\end{enumerate}
\end{thm}

Statement (i) of the previous theorem is essentially due to A.\ Haar\index{Haar, A.} \cite{haar18}.
The following proof significantly differs from Haar's proof and it is more general.
It is taken from \cite[pp.\ 284--286]{karlinStuddenTSystemsBook}, see also \cite[pp.\ 75--76]{achieser56}.

\begin{proof}
(a) $\Rightarrow$ (b): We prove $\neg$(b) $\Rightarrow \neg$(a).

Assume $\cF$ is not a T-system. There exist $n+1$ distinct points $a\leq x_0 < x_1 < \dots < x_n\leq b$ such that
\begin{equation}\label{eq:zeroDetLast}
\det \left(f_i(x_j)\right)_{i,j=0}^{n} = 0.
\end{equation}
Hence, there exist real coefficients $c_0,c_1,\dots,c_n$  with $\sum_{i=0}^n c_i^2 > 0$ with $\sum_{i=0}^n c_i f_j(x_i) = 0$ for all $j=0,\dots,n$.
That implies
\begin{equation}\label{eq:polyZeroLast}
\sum_{i=0}^n c_i p(x_i) = 0
\end{equation}
for all $p\in\lin\cF$.

The relation (\ref{eq:zeroDetLast}) also implies the existence of a non-trivial polynomial $\tilde{p} = \sum_{i=0}^n b_i f_i\in\lin\cF$ which vanishes at the points $x_0,x_1,\dots,x_n$.

Let $g\in\cat([a,b],\rset)$ be such that $\|g\|_\infty\leq 1$ and
\[g(x_i) = \frac{c_i}{|c_i|}\]
for all $i=0,1,\dots,n$ with $c_i\neq 0$.

Let $\lambda>0$ be such that $\|\lambda \tilde{p}\|_\infty<1$ then $f := g\cdot (1 - |\lambda \tilde{p}|)$ has the same signs at the points $x_i$ with $c_i\neq 0$ as $g$.

We will now construct an infinite number of polynomials of the same minimum deviation from $f$.

If
\[\left\|\;f  - \sum_{i=0}^n a_i f_i \;\right\|_\infty < 1\]
for some $a_0,a_1,\dots,a_n$ then
\[-1 < g(x_j)\cdot (1- |\lambda \tilde{p}(x_j)|) - \sum_{i=0}^n a_i f_i(x_j) < 1\]
for all $j=0,1,\dots,n$ which reduces to
\[-1 < g(x_j) - \sum_{i=0}^n a_i f_i(x_j) < 1\]
for all $j=0,1,\dots,n$.
Hence, if $c_j\neq 0$ the value of $\sum_{i=0}^n a_i f_i(x_j)$ has the sign of the $c_j$ so that $\sum_{j=0}^n c_j \sum_{i=0}^n a_i f_i(x_j) \neq 0$ which contradicts (\ref{eq:polyZeroLast}).
Therefore,
\[\left\|\; f - \sum_{i=0}^n a_i f_i \;\right\|_\infty \geq 1.\]
If now $|\delta|<1$ then
\begin{align*}
|f(x) - \delta\lambda\tilde{p}(x)|
&\leq |f(x)| + |\delta\lambda\tilde{p}(x)|\\
&\leq |g(x)|\cdot (1- |\lambda\tilde{p}(x)) + |\delta\lambda\tilde{p}(x)|\\
&\leq 1 - (1-|\delta|)\cdot |\lambda\tilde{p}(x)|\\
&\leq 1
\end{align*}
so that $\delta\lambda\tilde{p}$ minimizes the distance to $f$ independent of $\delta\in (-1,1)$.
Hence, we proved $\neg$(a).

We now prove (ii) which will also establish (b) \folgt\ (a).
Let $\cF$ be a T-system.
At least one minimal polynomial exists since $\lin\cF$ is finite dimensional.
Assume $g = \sum_{i=0}^n b_i f_i$ fulfills
\[\|f - g \|_\infty = m = \min_{a_0,\dots,a_n} \left\|\; f - \sum_{i=0}^n a_i f_i \;\right\|_\infty\]
and $f-g$ takes on the values $\pm m$ alternatively at only $k\leq n+1$ points.
We suppose for definiteness that $f-g$ assumes the values $+m$ before it takes the value $-m$.
In this case there exist $k-1$ points
\[a\leq\quad y_1<\dots<y_{k-1} \quad\leq b\]
such that
\[f(y_i) - g(y_i) = 0\]
for all $i=1,2,\dots,k-1$ and for some $d>0$ we have
\begin{align*}
m &\;\geq\; f-g \;\geq\; -m+d &&\text{on}\ [a,y_1]\cup [y_2,y_3]\cup\dots\\
m-d &\;\geq\; f-g \;\geq\; -m &&\text{on}\ [y_1,y_2]\cup [y_3,y_4]\cup.
\end{align*}
By \Cref{thm:zeros3} and \Cref{rem:kreinError} there exists a polynomial $h$ whose \emph{only} zeros on the \emph{open} interval $(a,b)$ are the nodal zeros $y_1,\dots,y_{k-1}$ and additionally $h\leq 0$ on $[a,y_1]$.
Let $\delta>0$ be such that $|\delta h|\leq d/2$ then
\begin{equation}\label{eq:boundabs}
| f-g + \delta h| < m
\end{equation}
on $(a,b)$.

Equality in (\ref{eq:boundabs}) is possible at the end point $a$ only if $f(a)-g(a) = m$ and $h(a)=0$ and at $b$ only if $|f(a) - g(b)| = m$ and $h(b) = 0$.
To repair the situation at the points $a$ and $b$ let $\tilde{h}$ be such that $\tilde{h}\cdot (f-g)>0$ at $a$ and $b$.
Then for sufficient small $\eta$ we have
\[|f-g+\delta h - \eta\tilde{h}|<m\]
on $[a,b]$.
Hence, by continuity on the compact interval $[a,b]$ we have
\[ \min_{a_0,\dots,a_n} \left\|\; f - \sum_{i=0}^n a_i f_i \;\right\|_\infty < m\]
contradicting the fact that $m$ is the minimum deviation.
That proves (ii) including uniqueness in (i).
\end{proof}

In the previous theorem we have seen the close connection between the best approximation polynomials from the minimum problem (\ref{eq:minInfimum}) and T-systems.
The next result shows that the connection is even closer, i.e., the solution of (\ref{eq:minInfimum}) is connected to the \Cref{thm:g1g2g}.

\begin{thm}[see e.g.\ {\cite[p.\ 283, Thm.\ 2.1]{karlinStuddenTSystemsBook}}]
Let $n\in\nset_0$ and let $f_0,\dots,f_n,f\in\cat([a,b],\rset)$ be such that $\{f_0,\dots,f_n\}$ and $\{f_0,\dots,f_n,f\}$ are continuous T-systems on $[a,b]$ with $a<b$.
Let
\[f^* = c\cdot f + \sum_{i=0}^n c_i\cdot f_i\]
be the $f^*$ from the \Cref{thm:g1g2g} with $g_1 = -1$ and $g_2 = 1$, i.e., $f^*$ is uniquely characterized by the following conditions:
\begin{enumerate}[(a)]
\item $-1 \leq f^* \leq 1$ on $[a,b]$, and

\item there exist $n+2$ points $x_1 < x_2 < \dots < x_{n+2}$ in $[a,b]$ such that
\[f^*(x_i) = (-1)^{n+1-i}\]
for all $i=1,\dots,n+2$.
\end{enumerate}
Then $c\neq 0$ and the polynomial
\[\underline{f} := -\frac{1}{c}\cdot\sum_{i=0}^n c_i f_i\]
is the unique minimizer of
\[d = \min_{a_0,\dots,a_n} \left\|\; f - \sum_{i=0}^n a_i f_i \;\right\|_\infty\]
and the minimum deviation is $d = |c|^{-1}$.
\end{thm}

The proof is taken from \cite[pp.\ 283--284]{karlinStuddenTSystemsBook}.

\begin{proof}
The coefficient $c$ can not be zero.
Otherwise the polynomial $\sum_{i=0}^n c_i f_i$ vanishes at $n+1$ points in the T-system $\{f_0,\dots,f_n\}$ by (b) and would therefore be equal to zero by \Cref{lem:determinant}.

From (a) we get
\[ \left\|\; f - \left( -\frac{1}{d}\sum_{i=0}^n c_i f_i\right) \;\right\|_\infty\leq \frac{1}{|d|}.\]
Since $\underline{f}$ fulfills (b) we get from \Cref{thm:bestApprox1} (ii) uniqueness of $\underline{f}$ and $d = |c|^{-1}$.
\end{proof}

Finding approximations is also done with respect to the $\cL^p$-norms
\begin{equation}\label{eq:bestLp}
\min_{a_0,\dots,a_n}\int \left|\; f(x) - \sum_{i=0}^n a_i f_i(x)  \;\right|^p~\diff\mu(x)
\end{equation}
with a fixed measure $\mu$ and $p\geq 1$.
For $p=2$ this leads to the well-studied \emph{orthogonal polynomials},\index{polynomial!orthogonal}\index{orthogonal!polynomial} a special branch of the theory of moments.

For $p=1$ in (\ref{eq:bestLp}) this also is connected to T-systems.
D.\ Jackson\index{Jackson, D.} \cite{jackson24} showed that if $\cF = \{f_0,\dots,f_n\}$ is a T-system then the best approximation of (\ref{eq:bestLp}) is unique, see also \cite[p.\ 77]{achieser56}.

\section{Optimization over Linear Functionals}
\label{sec:momOpt}

In optimization\index{optimization!over linear functionals} one often encounters the problem of having only a linear functional $L:\cV\to\rset$, e.g.\ a moment functional, and one wants to minimize $L(f)$ over $\cV_+$.
By removing the dependency on the scaling of $f$ we get the following result.

\begin{thm}[see e.g.\ {\cite[p.\ 312, Thm.\ 9.1]{karlinStuddenTSystemsBook}}]
Let $n\in\nset_0$, let $\cF = \{f_i\}_{i=0}^n$ be an ET-system on $[a,b]$ with $a<b$, and let $L,S:\lin\cF\to\rset$ be two linear functionals such that $S$ is strictly positive on $(\lin\cF)_+$, i.e., $S(f)>0$ for all $f\in\lin\cF\setminus\{0\}$ with $f\geq 0$.
Then
\begin{equation}\label{eq:minmax}
\min_{f\in(\lin\cF)_+\setminus\{0\}} \frac{L(f)}{S(f)} \qquad\text{and}\qquad \max_{f\in(\lin\cF)_+\setminus\{0\}} \frac{L(f)}{S(f)}
\end{equation}
are attained at non-negative polynomials possessing $n$ zeros counting multiplicities.
\end{thm}

The proof is taken from \cite[p.\ 312]{karlinStuddenTSystemsBook}.

\begin{proof}
Since $\lin\cF$ is finite dimensional the values in (\ref{eq:minmax}) are attained.

It is sufficient to prove the statement for the maximum.
But maximizing $\frac{L(f)}{S(f)}$ over $(\lin\cF)_+\setminus\{0\}$ is equivalent to maximize $L(f)$ over $f\in (\lin\cF)_+\setminus\{0\}$ with $S(f)=1$.

Let $f\geq 0$ be such that $S(f) = 1$ and suppose $f$ has at most $n-1$ zeros counting multiplicities.
Then by \Cref{thm:karlinNonNegab} there is a unique decomposition $f = f_* + f^*$ where $f_*$ and $f^*$ differ, are non-negative, and both have $n$ zeros counting multiplicities.
Set $\alpha := S(f_*)$ and $\beta := S(f^*)$.
Then $\alpha,\beta>0$ since $S$ is strictly positive and $\alpha + \beta = S(f_*) + S(f^*) = S(f) = 1$.
Then
\[f = \alpha\cdot \frac{f_*}{\alpha} + \beta\cdot \frac{f^*}{\beta}\]
and by linearity
\[L(f) \leq \max\left( \frac{L(f_*)}{\alpha}, \frac{L(f^*)}{\beta}\right)\]
which proves the statement.
\end{proof}

More results on best approximation and optimization over linear functionals can already be found in \cite{bernstein26}, \cite{achieser56}, and \cite{karlinStuddenTSystemsBook}.
Let alone the enormous literature after that.




\appendix
\part*{Appendices}
\addcontentsline{toc}{part}{Appendices}


\backmatter

\Extrachap{Solutions}
\setcounter{chapter}{19}
\renewcommand{\thechapter}{\Alph{chapter}}




\section*{Problems of \Cref{ch:measures}}

\begin{sol}{prob:determinacy}
The \Cref{thm:stoneWeierstrass} states that for a compact set $K\subset\rset^n$ the polynomials $\rset[x_1,\dots,x_n]$ are dense in $\cat(K,\rset)$ with respect to the $\sup$-norm.
Let $A\in\fB(K)$ be a Borel measurable set, let $\varepsilon>0$, and let $\mu_1$ and $\mu_2$ be two representing measures of $L$.
Set $A_\delta := (A + B_\delta(0))\cap K$ for all $\delta>0$.
Then for any $\varepsilon>0$ there exists a $\delta=\delta(\varepsilon) >0$ such that $\mu_1(A_\delta\setminus A),\mu_2(A_\delta\setminus A)<\varepsilon$.

By \Cref{lem:ury} there exists a $\varphi_\varepsilon\in\cat(K,[0,1])$ such that
\[\varphi_\varepsilon(x) = \begin{cases} 1 & \text{for}\ x\in A\\
0 & \text{for}\ x\in K\setminus A_\varepsilon
\end{cases}\]
and since $\rset[x_1,\dots,x_n]$ is dense in $\cat(K,\rset)$ there exists a family of polynomials $(p_i^\varepsilon)_{i\in\nset}\subseteq\rset[x_1,\dots,x_n]$ such that
\[\|p_i^\varepsilon - \varphi_\varepsilon\|_\infty \xrightarrow{i\to\infty} 0
\quad\text{and hence}\quad
\int_K p_i^\varepsilon(x)~\diff\mu_j(x)\xrightarrow{i\to\infty} \int_K \varphi_\varepsilon(x)~\diff\mu_j(x)\]
for $j=1,2$.
Then we have
\begin{align*}
\mu_1(A) &= \lim_{\varepsilon\searrow 0}\int_K \varphi_\varepsilon(x)~\diff\mu_1(x)\\
&= \lim_{\varepsilon\searrow 0} \lim_{i\to\infty} \int_K p_i^{\varepsilon}(x)~\diff\mu_1(x)\\
&= \lim_{\varepsilon\searrow 0} \lim_{i\to\infty} L(p_i^\varepsilon)\\
&= \lim_{\varepsilon\searrow 0} \lim_{i\to\infty} \int_K p_i^{\varepsilon}(x)~\diff\mu_2(x)\\
&= \lim_{\varepsilon\searrow 0}\int_K \varphi_\varepsilon(x)~\diff\mu_2(x)
\qquad = \mu_2(A).
\end{align*}
Since $A\in\fB(K)$ was arbitrary we have $\mu_1=\mu_2$, i.e., $L$ has a unique representing measure and is therefore determinate.
\end{sol}

\begin{sol}{prob:convexML}
\textbf{Proof of \Cref{cor:convexML}}\\
Let $\mu_1,\mu_2\in\cM(L)$ and $\lambda\in [0,1]$. Then
\begin{align*}
\int p(x)~\diff[\lambda\mu_1 + (1-\lambda)\mu_2](x) &= \lambda \int p(x)~\diff\mu_1(x) + (1-\lambda) \int p(x)~\diff\mu_2(x)\\
&= \lambda L(p) + (1-\lambda) L(p)\\
&= L(p)
\end{align*}
and hence $\lambda\mu_1 + (1-\lambda)\mu_2\in\cM(L)$ which proves convexity.
\end{sol}

\begin{sol}{prob:indeter}
\textbf{Proof of \Cref{cor:indeter}}\\
Let $\mu_0,\mu_1\in\cM(L)$ with $\mu_0\neq \mu_1$, i.e., there exists a $A\in\fA$ such that $\mu_0(A)\neq \mu_1(A)$ and without loss of generality we have $\mu_0(A) < \mu_1(A)$.
Hence, for all $\lambda\in [0,1]$ we set $\mu_\lambda := \lambda\mu_1 + (1-\lambda)\mu_0$ and we have
\[\mu_{\lambda_0}(A) < \mu_{\lambda_1}(A)\]
for all $0\leq \lambda_0 < \lambda_1 \leq 1$ which proves that $\mu_{\lambda_0}\neq \mu_{\lambda_1}$ for all $\lambda_0\neq \lambda_1$.
Hence, we have at least $|[0,1]| = |\rset|$ many representing measures for $L$.
\end{sol}

\section*{Problems of \Cref{ch:choquet}}

\begin{sol}{prob:linearCone}
\textbf{Proof of \Cref{lem:linearCone}}\\
The proof is taken from \cite[Vol.\ 2, p.\ 268]{choquet69}.

(i) \folgt\ (ii): If $F+C$ is a vector space then $-(F+C) = (F+C)$ and $-(F+C) = F - C$ since $-F = F$.

(ii) \folgt\ (iii): If $x\in F+C$, i.e.,  $x = y' + z$ for some $y'\in F$ and $z\in C$, then $x\geq y'$.
Similarly, if $x = y-w$ then $y\geq x$.

(iii) \folgt\ (i): First note that $F+C$ is a convex cone. So if suffices to show that $F + C = -(F+C)$, i.e., $F+C = F-C$.
But if $x\in F+ C$ and $x\leq y$ then $x = y - z$ for some $z\in C$, or $x\in F-C$.
Similarly, if $x\in F-C$ and $x = y' + w$ for some $w\in C$ then $x\in F + C$.
\end{sol}

\begin{sol}{prob:adaptedLem}
\textbf{Proof of \Cref{lem:adapted}}

(i) \folgt\ (ii): Set $K_\varepsilon = \supp h_\varepsilon$.

(ii) \folgt\ (iii): Chose by \Cref{lem:ury} a $\eta_\varepsilon\in\cat_c(\cX,\rset)$ with $\eta_\varepsilon|_{K_\varepsilon} = 1$.

(iii) \folgt\ (i): Take $h_\varepsilon = \eta_\varepsilon\cdot g\in\cat_c(\cX,\rset)$.
\end{sol}

\begin{sol}{prob:compactAdapted}
Since $\cX$ is compact for every $f\in E$ we have $m_f:= \min_{x\in\cX} f(x) > -\infty$ and $M_f := \max_{x\in\cX} f(x) < \infty$, especially for $f= e>0$ we have $m_e>0$. Then for every $f$ there exists a $d_f>0$ such that $f = (f+d_f e) - d_f e$ such that $f+d_f e, d_f e\in E_+$ and hence $E = E_+ - E_+$ proving (i) in \Cref{dfn:adaptedSpace}.

Since $e>0$ we also have (ii) in \Cref{dfn:adaptedSpace}.

For (iii) in \Cref{dfn:adaptedSpace} it is sufficient to note that $\cX$ is compact, i.e., for every $g$ there is a $c_g>0$ such that $g\leq c_g e$.
\end{sol}

\begin{sol}{prob:adaptedPolynomials}
Let $E = \rset[x_1,\dots,x_n]$ on $\cX$.
Then (i) $E = E_+ - E_+$ follows immediately from the fact that for every $f\in E$ there is a $g\in E_+$ such that $f = f+g - g$ with $f+g\in E_+$.

For (ii) we take $f = 1 >0$ on $\cX$.

For (iii) take the $g$ from (i).
\end{sol}

\begin{sol}{prob:adaptedPolynomials2}
Since $E$ is finite dimensional we can equip it with a norm, e.g.\ the $l^2$-norm in the coefficients of $f$.
Assume $\cX$ is not compact then there exists an unbounded sequence $(x_i)_{i\in\nset_0}$ and a $f\in E$ with $\|f\|\leq 1$ such that $(f(x_i))_{i\in\nset_0}$ grows faster than any other $(g(x_i))_{i\in\nset_0}$.
Hence, $f$ can not be dominated by any $g$.
\end{sol}

\begin{sol}{prob:adaptedCompact}
\textbf{Proof of \Cref{lem:adaptedCompact}}\\
Since $K=\supp g$ is compact and $E$ is an adapted space, i.e., there exists a $f\in E_+$ with $f>0$ we have that $\min_{x\in K} f(x) >0$ and hence there exists a $c>0$ such that $cf > g$ on $K$ and hence on all $\cX$.
\end{sol}

\section*{Problems of \Cref{ch:classical}}

\begin{sol}{prob:stieltjes}
\textbf{Proof of \Cref{thm:stieltjesMP}}\\
We have (iii) \gdw\ (iv) \gdw\ (v) by the definition of the Hankel matrix and also (i) \folgt\ (ii) \folgt\ (iii).
Additionally, we have (iii) \folgt\ (ii) by (\ref{eq:pos0infty2}) since $L(p) = L(f^2) + L(xg^2) \geq 0$.
At last (ii) \folgt\ (i) holds by the \Cref{thm:basicrepresentation} since $\rset[x]$ on $[0,\infty)$ is an adapted space.
\end{sol}

\begin{sol}{prob:hamburger}
\textbf{Proof of \Cref{thm:hamburgerMP}}\\
We have (i) \folgt\ (ii) \folgt\ (iii) and additionally (iii) \gdw\ (iv) \gdw\ (v) by the definition of the Hankel matrix.
The implication (iii) \folgt\ (ii) follows from \Cref{eq:posR} by $L(p) = L(f^2 + g^2) \geq 0$.
At last (ii) \folgt\ (i) holds by the \Cref{thm:basicrepresentation} since $\rset[x]$ on $\rset$ is an adapted space.
\end{sol}

\begin{sol}{prob:hausdorff}
\textbf{Proof of \Cref{thm:hausdorffMP}}\\
We have (i) \folgt\ (ii) \folgt\ (iii) and additionally (iii) \gdw\ (iv) \gdw\ (v) by the definition of the Hankel matrix.
The implication (iii) \folgt\ (ii) follows from (\ref{eq:posab}) since it is sufficient to look only at $f(x)^2 + xg(x)^2 + (1-x) h(x)^2$.
At last (ii) \folgt\ (i) holds by the \Cref{thm:basicrepresentation} since $\rset[x]$ on $[0,1]$ is an adapted space.
\end{sol}

\begin{sol}{prob:haviland}
\textbf{Proof of \Cref{thm:haviland}}\\
Since (i) \folgt\ (ii) is clear it is sufficient to show (ii) \folgt\ (i).
But since $E=\rset[x_1,\dots,x_n]$ on $K$, is an adapted space (see Problem \ref{prob:adaptedPolynomials}) and since $E_+ = \pos(K)$ by definition the \Cref{thm:basicrepresentation} applies and gives the assertion.
\end{sol}

\begin{sol}{prob:bernstein}
\textbf{Proof of \Cref{cor:bernstein}}\\
We have that (ii) $\Rightarrow$ (i) is clear since $x^k\cdot (1-x)^l >0$ on $(0,1)$ and at least one $c_{k',l'}>0$.
It remains to prove (i) $\Rightarrow$ (ii).

Let $f\in\rset[x]\setminus\{0\}$ with $f>0$ on $(0,1)$ then we can write $f$ as
\[f(x) = x^p\cdot (1-x)^q\cdot \tilde{f}(x)\]
with $\tilde{f}\in\rset[x]$, $\tilde{f}>0$ on $[0,1]$, and $p,q\in\nset_0$, i.e., by the fundamental theorem of algebra we can factor out the zeros at $x=0$ and at $x=1$.
Applying \Cref{thm:bernstein} (ii) to $\tilde{f}$ then gives the assertion.
\end{sol}

\begin{sol}{prob:boundaryCone}
\textbf{Proof of \Cref{lem:boundaryCone}}\\
Since the moment cone $\cS_\cF$ and the hyperplane $H$ are convex we have that $\cS_\cF\cap H$ is a convex cone, i.e., it is a moment cone and there exists a family $\cG\subsetneq\lin\cF$ of $m<n$ elements which spans $\cS_\cF\cap H$.
It is sufficient to show that $\cG$ lives on $(\cY,\fA|_\cY)$ for some $\cY\subseteq\cX$.

For the hyperplane $H$ there exists a function $h\in\lin\cF$ such that $L_s(h)\geq 0$ for all $s\in\cS_\cF$.
Note, that $\cN=\cap_{k\in\nset} \{x\in\cX \,|\, f_1(x)^2 + \dots + f_n(x)^2 \geq k\}$ has measure zero for any representing measure $\mu_s$ on $\cX$ of a moment sequence $s\in\cS_\cF$ since the moments are finite, i.e., the $f_i$ are $\mu_s$-integrable.
Without loss of generality we can therefore work on $\cX\setminus\cN$.
Hence, all $\delta_x$ with $x\in\cX\setminus\cN$ are moment measures and $L_s(h)\geq 0$ implies $h\geq 0$ on $\cX\setminus\cN$.

Then $s\in\cS_\cF\cap H \gdw L_s(h) = 0$ implies that all representing measures $\mu$ of all $s\in\cS_\cF\cap H$ have the support in $\cY := \{x\in\cX\setminus\cN \,|\, h(x) = 0\}$.
\end{sol}

\begin{sol}{prob:richterFromRosen}
Let $\cF = \{f_1,\dots,f_n\}$ be measurable functions on $(\cX,\fA)$ which are not necessarily bounded.
Set
\[I := \bigcap_{k\in\nset} \{x\in\cX \,|\, |f_i(x)|>k\ \text{for all}\ i=1,\dots,n\}.\]
Then $I$ is measurable.
Let $s$ be a moment sequence with representing measure $\mu$.
Since all $f_i$ are $\mu$-measurable we have $\mu(I) = 0$.
Therefore, by working on $\cX\setminus I$ we can assume without loss of generality that $|f_i(x)|<\infty$ for all $x\in\cX$.

Define $\cG = \{g_1,\dots,g_n\}$ with $g_i := \frac{g_i}{f}$ and $f := 1 + \sum_{i=1}^n f_i^2$.

At first we note that from
\begin{equation}\label{eq:richterRosenbloom}
\int_\cX f_i(x)~\diff\mu(x) = \int_\cX g_i(x)\cdot f(x)~\diff\mu = \int_\cX g_i(x)~\diff\nu(x),
\end{equation}
we have that every sequence $s=(s_1,\dots,s_n)$ is a moment sequence with respect to $\cG$ if and only if it is moment sequence with respect to $\cF$.

Since all $g_i$ are bounded we have by Rosenbloom's Theorem that there is a $k$-atomic representing measure $\nu = \sum_{i=1}^k c_i\cdot\delta_{x_i}$ which represents the moment sequence $s$.
Then by (\ref{eq:richterRosenbloom}) we find that $\mu = \sum_{i=1}^k c_i\cdot f(x_i)^{-1}\cdot\delta_{x_i}$ is a representing measure of $s$ with respect to $\cF$ which proves the statement.
\end{sol}

\section*{Problems of \Cref{ch:tsystems}}

\begin{sol}{prob:restriction}
\textbf{Proof of \Cref{cor:restriction}}\\
Let $f\in\lin\cF$. Then $f$ has at most $n$ zeros in $\cX$ and hence $f|_\cY$ has at most $n$ zeros in $\cY\subset\cX$.
Since for any $g\in\lin\cG$ there is a $f\in\lin\cF$ such that $g = f|_\cY$ we have the assertion.
\end{sol}

\begin{sol}{prob:transf}
\textbf{Proof of \Cref{cor:transf}}\\
Let $w_0,\dots,w_n\in\cW$ be pairwise distinct.
Since $g$ is injective we have that also $g(w_0),\dots,g(w_n)\in\cX$ are pairwise distinct.
Hence,
\[\det\begin{pmatrix}
g_0 & g_1 & \dots & g_n\\ w_0 & w_1 & \dots & w_n
\end{pmatrix} =
\det\begin{pmatrix}
f_0 & f_1 & \dots & f_n\\ g(w_0) & g(w_1) & \dots & g(w_n)
\end{pmatrix} \neq 0\]
and the statement follows from \Cref{lem:determinant}.
\end{sol}

\begin{sol}{prob:scaling}
\textbf{Proof of \Cref{cor:scaling}}\\
Let $x_0,\dots,x_n\in\cX$ be pairwise distinct.
Then
\[\det\begin{pmatrix}
g_0 & g_1 & \dots & g_n\\ x_0 & x_1 & \dots & x_n
\end{pmatrix} =
\det\begin{pmatrix}
f_0 & f_1 & \dots & f_n\\ x_0 & x_1 & \dots & x_n
\end{pmatrix}\cdot g(x_1)\cdots g(x_n) \neq 0\]
and the statement follows from \Cref{lem:determinant}.
\end{sol}

\begin{sol}{prob:4.1}
\textbf{Proof of \Cref{cor:uniqueDeter}}
\begin{enumerate}[(i)]
\item Assume $f_0,\dots,f_n$ are linearly dependent, i.e., there are $a_0,\dots,a_n\in\rset$ not all zero such that $a_0 f_0 + \dots + a_n f_n$ is the zero polynomial. Hence, $f$ has at least $n+1$ zeros. But since $\cF$ is a T-system this is a contradiction.

\item Let $x_0,\dots,x_n\in\cX$ be $n+1$ pairwise distinct points. Then by \Cref{dfn:kreinMatrix} we have
\[\begin{pmatrix} f(x_0)\\ \vdots\\ f(x_n)\end{pmatrix} =
\underbrace{\begin{pmatrix} f_0 & \dots & f_n\\ x_0 & \dots & x_n\end{pmatrix}}_{=:M}\cdot
\begin{pmatrix} a_0\\ \vdots\\ a_n\end{pmatrix}\]
and since $\cF$ is a T-system we have that $M$ has full rank by \Cref{lem:determinant}. Hence, the coefficients $a_0,\dots,a_n$ are unique.
\end{enumerate}
\end{sol}

\begin{sol}{prob:fraction}
\textbf{Proof of \Cref{exm:fraction}}\\
Set $f_i(x) := (x+\alpha_i)^{-1}$ and $g(x) = (x+\alpha_0)\cdots (x+\alpha_n)$.
Then $g>0$ on $[a,b]$ since $-\alpha_0 < a < b$.
Hence, $\cF$ is a T-system on $[a,b]$ if and only if $\cG = \{g_i:= g\cdot f_i\}_{i=0}^n$ is a T-system on $[a,b]$ by \Cref{cor:scaling}.

We have $g_i(x) = (x+\alpha_0)\cdots (x+\alpha_{i-1})\cdot(x+\alpha_{i+1})\cdots (x+\alpha_n)$ and $\deg g_i = n$.
It is now sufficient to show that $\cG$ is a T-system on $\rset$ by \Cref{cor:restriction} since then it will also be a T-system on $[a,b]$.

Since $g_i(\alpha_j) = 0$ for all $i\neq j$ we have that the $g_i$ are linearly independent.
Hence, $\lin\cG = \rset[x]_{\leq n}$.
But since $\{x^i\}_{i=0}^n$ is a T-system so is $\cG$ since every non-trivial $f\in\lin\cG = \rset[x]_{\leq n}$ has at most $n$ zeros.

In summary, we have that $\{x^i\}_{i=0}^n$ is a T-system on $\rset$ \folgt\ $\cG$ on $\rset$ is a T-system \folgt\ $\cG$ on $[a,b]$ is a T-system \folgt\ $\cF$ on $[a,b]$ is a T-system.
\end{sol}

\begin{sol}{prob:4.2}
To the points $x_0,\dots,x_{k+l}\in [a,b]$ add pairwise distinct points $x_{k+l+1},\dots,x_n\in [a,b]\setminus\{x_0,\dots,x_{k+l}$. Then the matrix
\begin{equation}\label{eq:solMatrix}
\begin{pmatrix}
f_0(x_0) & \dots & f_n(x_0)\\
\vdots & & \vdots\\
f_0(x_n) & \dots & f_n(x_n)
\end{pmatrix}
\end{equation}
has full rank since $\cF$ is a T-system, i.e., every vector, especially
\[(m,\dots,m,-m,\dots,-m,0,\dots,0, *,\dots,*)^T\in\rset^{n+1}\]
is in its image.
But the matrix
\[\begin{pmatrix}
f_0(x_1) & \dots & f_n(x_1)\\
\vdots & & \vdots\\
f_0(x_{k+l}) & \dots & f_n(x_{k+l})
\end{pmatrix}\]
in (\ref{eq:bigGleichungssystem}) only contains the first $k+l$ rows of (\ref{eq:solMatrix}), i.e., (\ref{eq:bigGleichungssystem}) has at least one solution.
\end{sol}

\begin{sol}{prob:kreinError}
By \Cref{rem:kreinError} only the case $n = 2m + 2p$ and one end point is contained.
But then we can apply \Cref{thm:zeros2} to $\tilde{\cF} = \{f_i\}_{i=0}^{n-1}$ which ensures by the same arguments in \Cref{rem:kreinError} that $x_1,\dots,x_p$ are the only zeros of some $f\geq 0$.
\end{sol}

\section*{Problems of \Cref{ch:etsystems}}

\begin{sol}{prob:etMultiplication}
\textbf{Proof of \Cref{lem:etMultiplication}}\\
Set $g_i := g\cdot f_i$.
Then we have to check that
\[\cW(g_0,\dots,g_k)(x) = \det\begin{pmatrix}
g_0(x) & g_1(x) & \dots & g_n(x)\\
g_0'(x) & g_1'(x) & \dots & g_n'(x)\\
\vdots & \vdots & & \vdots\\
g_0^{(n)}(x) & g_1^{(n)}(x) & \dots & g_n^{(n)}(x)
\end{pmatrix}  \neq 0\]
holds for all $x\in [a,b]$.
Since $g_i = g\cdot f_i$ we apply the product rule and get
\begin{align*}
\cW(g_0,\dots,g_k)(x)
&= g^2\cdot \det\begin{pmatrix}
f_0(x) & f_1(x) & \dots & f_n(x)\\
f_0'(x) & f_1'(x) & \dots & f_n'(x)\\
g_0''(x) & g_1''(x) & \dots & g_n''(x)\\
\vdots & \vdots & & \vdots\\
g_0^{(n)}(x) & g_1^{(n)}(x) & \dots & g_n^{(n)}(x)
\end{pmatrix}
\end{align*}
since in the first line we factored out $g$ and then subtracted $g'$-times the first line from the second, and factored out $g$ from the remaining second line.
For the second derivatives in the third line we have
\[(g\cdot f_i)'' = g''\cdot f_i + 2g'\cdot f_i' + g\cdot f_i''\]
and hence subtracting $g''$-times the first row, $2g'$-times the second row, and finally factoring out $g$ from the remaining third row we get
\begin{align*}
\cW(g_0,\dots,g_k)(x)
&= g^3\cdot \det\begin{pmatrix}
f_0(x) & f_1(x) & \dots & f_n(x)\\
f_0'(x) & f_1'(x) & \dots & f_n'(x)\\
f_0''(x) & f_1''(x) & \dots & f_n''(x)\\
g_0'''(x) & g_1'''(x) & \dots & g_n'''(x)\\
\vdots & \vdots & & \vdots\\
g_0^{(n)}(x) & g_1^{(n)}(x) & \dots & g_n^{(n)}(x)
\end{pmatrix}.
\end{align*}
Proceeding in this manner we arrive at
\begin{align*}
\cW(g_0,\dots,g_k)(x)
&= g^{n+1}\cdot \cW(f_0,\dots,f_n)(x) \neq 0
\end{align*}
for all $x\in [a,b]$ which proves the statement.
\end{sol}

\begin{sol}{prob:wronsTrans}
\textbf{Proof of \Cref{lem:ettrans}}\\
We proceed similar to Problem/Solution \ref{prob:etMultiplication} but now with the rule of differentiation for $f_i\circ g$.
We have
\[(f_i\circ g)' = g'\cdot (f_i'\circ g)\]
and hence
\[\cW(g_0,\dots,g_n) = g'\cdot \det\begin{pmatrix}
f_0\circ g & \dots & f_n\circ g\\
f_0'\circ g & \dots & f_n'\circ g\\
(f_0\circ g)'' & \dots & (f_n\circ g)''\\
\vdots & & \vdots\\
(f_0\cdot g)^{(n)} & \dots & (f_n\circ g)^{(n)}
\end{pmatrix}\]
by factoring out $g'$ from the second row.
Then we have
\begin{align*}
(f_i\circ g)'' &= (g'\cdot (f_i'\circ g))'
= g''\cdot (f_i'\circ g) + (g')^2\cdot (f_i''\circ g),
\end{align*}
i.e., we subtract $g''$-times the second row and factor out $(g')^2$ to get
\[\cW(g_0,\dots,g_n) = (g')^3\cdot \det\begin{pmatrix}
f_0\circ g & \dots & f_n\circ g\\
f_0'\circ g & \dots & f_n'\circ g\\
f_0''\circ g & \dots & f_n''\circ g\\
(f_0\circ g)''' & \dots & (f_n\circ g)'''\\
\vdots & & \vdots\\
(f_0\cdot g)^{(n)} & \dots & (f_n\circ g)^{(n)}
\end{pmatrix}.\]
Proceeding in this manner with
\[(f_i\circ g)^{(k)}=(g')^{(k)}\cdot (f_i^{(k)}\circ g)+{\dots}+ g^{(k)}\cdot (f_i'\circ g)\]
we get
\[\cW(g_0,\dots,g_n) = (g')^{\frac{n(n+1)}{2}}\cdot \cW(f_0,\dots,f_n)\circ g\]
with proves the assertion.
\end{sol}

\begin{sol}{prob:wronskiReduction}
\textbf{Proof of \Cref{lem:wronskiReduction}}\\
Set $\cH = \{h_i\}_{i=0}^n$ with $h_i := \frac{f_i}{f_0}$.
Then by \Cref{lem:etMultiplication} we have
\begin{align*}
\cW(f_0,\dots,f_n) &= f_0^{n+1}\cdot\cW(h_0,\dots,h_n)
\intertext{and since $h_0=1$ we have $h_0' = h_0'' = \dots = 0$ and}
&=f_0^{n+1}\cdot \det \begin{pmatrix}
1 & h_1 & \dots & h_n\\
0 & h_1' & \dots & h_n'\\
\vdots & \vdots & & \vdots\\
0 & h_1^{(n)} & \dots & h_n^{(n)}
\end{pmatrix}
\intertext{which gives by expanding along the first column}
&= f_0^{n+1}\cdot \det \begin{pmatrix}
h_1' & \dots & h_n'\\
\vdots & & \vdots\\
h_1^{(n)} & \dots & h_n^{(n)}
\end{pmatrix}\\
&= f_0^{n+1}\cdot \cW(h_1',\dots,h_n')
\intertext{and with $g_i = h_{i+1}'$ for $i=0,\dots,n-1$ we get}
&= f_0^{n+1}\cdot \cW(g_0,\dots,g_{n-1})
\end{align*}
which proves the statement.
\end{sol}

\begin{sol}{prob:ecRestriction}
\textbf{(a)} Since $\cF$ is an ET-system on $[a,b]$ we have
\[\cW(f_0,\dots,f_n)(x) \neq 0\]
for all $x\in [a,b]$, i.e., also for all $x\in [a',b']\subseteq [a,b]$ and hence it is an ET-system on $[a',b']$.

\noindent
\textbf{(b)} Since $\cF$ is an ECT-system on $[a,b]$ we have
\[\cW(f_0,\dots,f_k)(x) \neq 0\]
for all $x\in [a,b]$ and $k=0,\dots,n$, i.e., also for all $x\in [a',b']\subseteq [a,b]$ and $k=0,\dots,n$ and hence it is an ECT-system on $[a',b']$.
\end{sol}

\begin{sol}{prob:xf}
\textbf{Proof of \Cref{exm:xf}}\\
We already know that $\{1,x,x^2,\dots,x^k\}$ is an ET-system for any $k=0,1,\dots,n$ since
\[cW(1,x,\dots,x^k)(x) = 1\cdot 1!\cdot{\cdots}\cdot k! > 0.\]
From the Wronskian determinant
\[\cW(1,x,\dots,x^n,f)(x) = 1\cdot 1!\cdot 2!\cdot{\dots}\cdot n!\cdot f^{(n)}(x) >0\]
we then get that $\cF$ is an ECT-system on $[a,b]$ by \Cref{thm:ectWronski}.
\end{sol}

\begin{sol}{prob:etSystExm}
\textbf{Proof of \Cref{exm:etSystems}}\\
By \Cref{lem:ettrans} we only need to prove the statement for one case, say case (b) $\cG = \{e^{\alpha_i x}\}_{i=0}^n$.
Let $k\in \{0,1,\dots,n\}$.
Then
\begin{align*}
\cW(g_0,\dots,g_k) &= \det\begin{pmatrix}
g_0 & g_1 & \dots & g_k\\
g_0' & g_1' & \dots & g_k'\\
\vdots & \vdots & & \vdots\\
g_0^{(k)} & g_1^{(k)} & \dots & g_k^{(k)}
\end{pmatrix}
\intertext{and with $g_i^{(j)} = \alpha_i^j\cdot g_i$ we get}
&= \deg\begin{pmatrix}
g_0 & g_1 & \dots & g_k\\
\alpha_0 g_0 & \alpha_1 g_1 & \dots & \alpha_k g_k\\
\vdots & \vdots & & \vdots\\
\alpha_0^k g_0 & \alpha_1^k g_1 & \dots & \alpha_k^k g_k
\end{pmatrix}
= g_0\cdot g_1 \cdots g_n\cdot\det\begin{pmatrix}
1 & 1 & \dots & 1\\
\alpha_0 & \alpha_1 & \dots & \alpha_k\\
\vdots & \vdots & & \vdots\\
\alpha_0^k & \alpha_1^k & \dots & \alpha_k^k
\end{pmatrix}\\
&= g_0\cdot g_1 \cdots g_k\cdot\prod_{0\leq i<j\leq k} (\alpha_j-\alpha_i) \neq 0
\end{align*}
which proves the statement.
\end{sol}

\begin{sol}{prob:exampleECTpolynomial}
To construct the non-negative polynomial on $[0,\infty)$ with the double zero $x_1 = 1$ and the zero $x_2=2$ with algebraic multiplicity $m_2=4$ we need $7$ monomials.
We chose $f_0(x)=1,f_1(x)=x^2, f_2(x)=x^3,f_3(x)=x^5,f_4(x)=x^8,f_5(x)=x^{11},f_6(x)=x^{13}$ and leave out $x^{42}$.
With (\ref{eq:doublezeroDfn2}) we get
\begin{align*}
f(x) &= \det\left(\begin{array}{c|cccccc}
f_0 &\, f_1 & f_2 & f_3 & f_4 & f_5 & f_6\\
x &\, 1 & 1 & 2 & 2 & 2 & 2
\end{array}\right)\\
&= \det \begin{pmatrix}
f_0(x) & f_1(x) & f_2(x) & f_3(x) & f_4(x) & f_5(x) & f_6(x)\\
f_0(x_1) & f_1(x_1) & f_2(x_1) & f_3(x_1) & f_4(x_1) & f_5(x_1) & f_6(x_1)\\
f_0'(x_1) & f_1'(x_1) & f_2'(x_1) & f_3'(x_1) & f_4'(x_1) & f_5'(x_1) & f_6'(x_1)\\
f_0(x_2) & f_1(x_2) & f_2(x_2) & f_3(x_2) & f_4(x_2) & f_5(x_2) & f_6(x_2)\\
f_0'(x_2) & f_1'(x_2) & f_2'(x_2) & f_3'(x_2) & f_4'(x_2) & f_5'(x_2) & f_6'(x_2)\\
f_0''(x_2) & f_1''(x_2) & f_2''(x_2) & f_3''(x_2) & f_4''(x_2) & f_5''(x_2) & f_6''(x_2)\\
f_0'''(x_2) & f_1'''(x_2) & f_2'''(x_2) & f_3'''(x_2) & f_4'''(x_2) & f_5'''(x_2) & f_6'''(x_2)\\
f_0^{(4)}(x_2) & f_1^{(4)}(x_2) & f_2^{(4)}(x_2) & f_3^{(4)}(x_2) & f_4^{(4)}(x_2) & f_5^{(4)}(x_2) & f_6^{(4)}(x_2)
\end{pmatrix}\\
&= \det \begin{pmatrix}
 1 & x^2 & x^3 & x^5 & x^8 & x^{11} & x^{13} \\
 1 & 1 & 1 & 1 & 1 & 1 & 1 \\
 0 & 2 & 3 & 5 & 8 & 11 & 13 \\
 1 & 4 & 8 & 32 & 256 & 2\,048 & 8\,192 \\
 0 & 4 & 12 & 80 & 1\,024 & 11\,264 & 53\,248 \\
 0 & 2 & 12 & 160 & 3\,584 & 56\,320 & 319\,488 \\
 0 & 0 & 6 & 240 & 10\,752 & 253\,440 & 1\,757\,184
\end{pmatrix}\\
f(x) &= 48\cdot \left(14\,980\,788 x^{13} - 184\,325\,420 x^{11} + 2\,421\,354\,616 x^8 - 26\,336\,028\,160 x^5\right.\\
&\qquad\quad\, \left. + 112\,945\,898\,496 x^3 - 112\,347\,781\,120 x^2 + 23\,485\,900\,800\right).
\end{align*}
The function $f$ is shown in \Cref{fig:S54}.

\begin{figure}[htb]\centering
\includegraphics[width=0.9\textwidth]{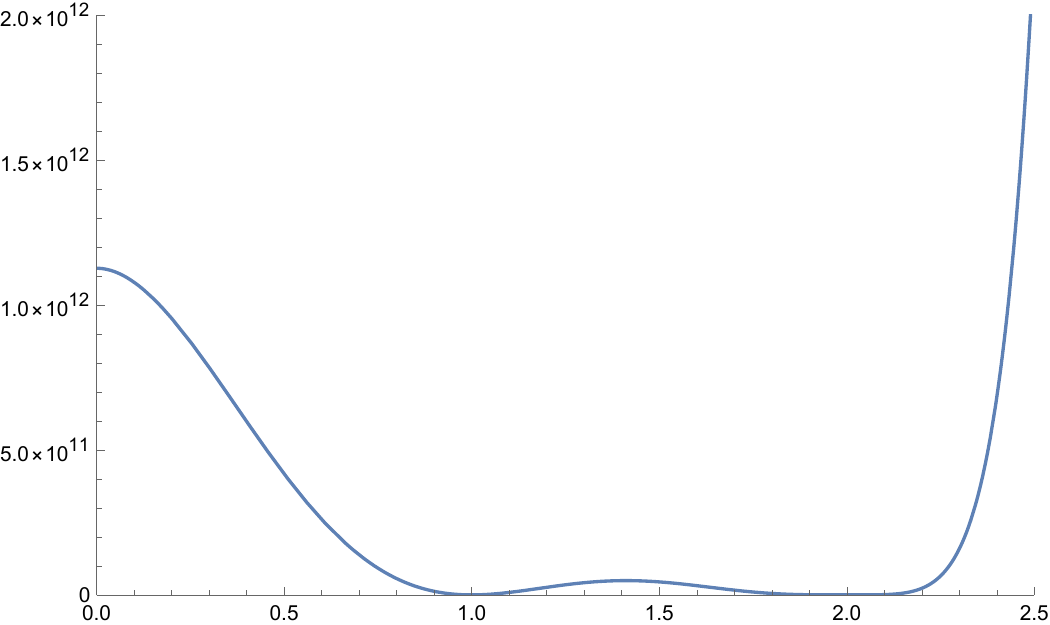}
\caption{The function $f$ from the solution of Problem \ref{prob:exampleECTpolynomial}.\label{fig:S54}}
\end{figure}
This function $f$ we gave here is not unique.
Of course every multiple of $f$ also fulfills the requirements but we also made the restrictions to use all monomials except $x^{42}$.
We get another polynomial when we e.g.\ leave out $x^{13}$ (or any other monomial except $1$) instead of $x^{42}$.
Then any conic linear combination of these functions also fulfills the requirements.

We can not leave out $1$ since any linear combination has the additional zero $x=0$.
\end{sol}

\section*{Problems of \Cref{ch:ETfromT}}

\begin{sol}{prob:smoothing}
\textbf{Proof of \Cref{thm:ETfromT}}\\
Since $\cF$ is a continuous T-system we can assume that
\[\det\begin{pmatrix}
f_0 & f_1 & \dots & f_n\\
x_0 & x_1 & \dots & x_n
\end{pmatrix}>0\]
for all $a\leq x_0 < x_1 < \dots < x_n \leq b$.
Since the Gaussian kernel is ETP$_k$ for every $k\in\nset_0$, see \Cref{exm:gaussiankernel},
we have
\[K_\sigma^*\!\begin{pmatrix} x_1 & x_2 & \dots & x_n\\ y_1 & y_2 & \dots & y_n\end{pmatrix}>0\]
for all $x_1 < x_2 < \dots < x_n$ and $y_1 \leq y_2 \leq \dots \leq y_n$ in $\rset$ as well as $\sigma>0$.
Hence, in $\cW(f_{\sigma,0},f_{\sigma,1},\dots,f_{\sigma,n})(x)=$ (\ref{eq:basic2}) in \Cref{lem:basiccomposition} we are integrating over a non-negative functions with respect to the Lebesgue measure $\mu=\lambda$, i.e., $\cW(f_{\sigma,0},f_{\sigma,1},\dots,f_{\sigma,n})(x)>0$ for all $x\in[a,b]$ which proves the statement.
\end{sol}

\section*{Problems of \Cref{ch:karlinPosab}}

\begin{sol}{prob:etNeighborhood}
The family $\cF$ on $[a,b]$ needs for a fixed $f\geq 0$ only be an ET-system around the zeros of $f$ but otherwise the proof of \Cref{thm:karlin} is employed, i.e., there we only need $\cF$ to be a T-system.
\end{sol}

\section*{Problems of \Cref{ch:karlinPos0inftyR}}

\begin{sol}{prob:posR}
\textbf{Proof of \Cref{thm:karlinPosR} on $\rset$}\\
By (a) there exists a function $w\in\cat(\rset,\rset)$ such that $w>0$ on $\rset$ and
\[\lim_{x\to\infty} \frac{f_n(x)}{w(x)} = 1.\]
By (b) we define
\[v_i(x) := \begin{cases}
\frac{f_i(x)}{w(x)} & \text{if}\ x\in \rset,\\
\delta_{i,n} & \text{if}\ x=\pm\infty
\end{cases}\]
for all $i=0,1,\dots,n$.
Then by (c) and \Cref{cor:scaling} we have that $\{v_i\}_{i=0}^n$ is a T-system on $[0,\infty]$.
With $t(x) := \tan (\pi x/2)$ we define
\[g_i(x) := v_i\circ t\]
for all $i=0,1,\dots,n$.
Hence, $\cG = \{g_i\}_{i=0}^n$ is a T-system on $[-1,1]$ by \Cref{cor:transf}.
We now apply \Cref{thm:karlinPosab} to $\cG$. Set $g := (\frac{f}{w})\circ t$.

(ii): By \Cref{thm:karlinPosab} on $[a,b]$ there exist points
\[-1=y_0 < x_1 < y_1 < \dots < x_m < y_m = 1\]
and unique functions $g_*$ and $g^*$ such that $g = g_* + g^*$, $g_*,g^*\geq 0$ on $[-1,1]$, $x_1,\dots,x_m$ are the zeros of $g_*$, and $y_0,\dots,y_m$ are the zeros of $g^*$.
Then $f_* := (g_*\circ t^{-1})\cdot w$ and $f^* := (g^*\circ t^{-1})\cdot w$ are the unique components in the decomposition $f = f_* + f^*$.

(i): Since $g^*(y_0) = g^*(y_m) = 0$ we have that $g^*$ contains no $g_{2m}$ and hence the coefficient of $g_{2m}$ in $g_*$ is $a_{2m}$.
\end{sol}

\begin{sol}{prob:nonnegR}
\textbf{Proof of \Cref{thm:karlinNonNegR} on $\rset$}\\
Similar to the proof of \Cref{thm:karlinNonNeg0infty} on $[0,\infty)$ and hence Problem/Solution \ref{prob:posR}.

The conditions (a) -- (c) are such that $\cF$ on $[-\infty,\infty]$, i.e., including $\pm\infty$, is an ET-system.

With the same argument as in the proof of \Cref{thm:karlinPos0infty} we transform $\cF$ on $[-\infty,\infty]$ into $\cG$ on $[-1,1]$ with the $\tan$-function.
Here \Cref{lem:ettrans} ensures that also $\cG$ is an ET-system.

Application of \Cref{thm:karlinNonNegab} on $[-1,1]$ gives the desired decomposition $g = g_* + g^*$ with the observation that $x=\pm 1$ is a zero of at most multiplicity one by (a) and (b).
Backwards transformation into $\cF$ on $[-\infty,\infty]$ resp.\ $[-\infty,\infty)$ then gives the assertion.
\end{sol}

\section*{Problems of \Cref{ch:nonNegAlgPolab}}

\begin{sol}{prob:a0}
\textbf{Proof of \Cref{thm:algNichtNeg0b}}\\
\Cref{thm:algNichtNegab} can in general not be extended to $[0,b]$ since $\{x^{\alpha_0},\dots,x^{\alpha_n}\}$ is not an ET-system.
This fails at $x=0$.
But on $(0,b]$ it is an ET-system.
We can therefore factor out the zeros of $f\geq 0$ at $x=0$
\[f(x) = a_i x^{\alpha_i} + a_{i+1} x^{\alpha_{i+1}} + \dots + a_n x^{\alpha_n}
= x^{\alpha_i}\cdot (\underbrace{a_i + a_{i+1} x^{\alpha_{i+1}-\alpha_i} + \dots + a_n x^{\alpha_n-\alpha_i}}_{=:\tilde{f}(x)})\]
to get some $\tilde{f}$ with $\tilde{f}\geq 0$ on $[0,b]$ and $\tilde{f}(0)>0$.
To $\tilde{f}$ we can then apply \Cref{thm:algNichtNegab} with $a=0$.

In summary, \Cref{thm:algNichtNegab} on $[0,b]$ holds if $f(0)>0$, see also \Cref{thm:sparseNonneg0infty} and \Cref{rem:a0infty} for the corresponding version on $[0,\infty)$.
\end{sol}

\section*{Problems of \Cref{ch:nonNegAlgPol0infty}}

\begin{sol}{prob:nonneg0infty}
\textbf{Proof of \Cref{thm:sparseNonneg0infty}}\\
To prove \Cref{thm:sparseNonneg0infty} we have to note that $\cF = \{x^{\alpha_i}\}_{i=0}^n$ with $\alpha_0 = 1$ is an ET-system on $(0,\infty)$.
The only difficulty is $x=0$ where $\cF$ fails to be a ET-system.

But looking closely at the proof of \Cref{thm:karlinNonNeg} (see Problem/Solution \ref{prob:etNeighborhood}) the ET-system property is only required in a neighborhood of the zeros of $f$ and otherwise it is the proof of \Cref{thm:karlin} for T-systems.
Since $f(0)>0$ we have no zero at $x=0$ where $\cF$ fails to be a T-system.
In fact, we have $f(x) > 0$ for all $x\in [0,\varepsilon)$ for some $\varepsilon>0$.
Hence, we can apply \Cref{thm:karlinNonNeg0infty} since its proof requires for our $f$ with $f(0)>0$ only that $\cF$ to be an ET-system on $(0,\infty)$ which is fulfilled.
\end{sol}

\begin{sol}{prob:leadingcoeffR}
By expanding
\[\prod_{i=1}^r (x-z_i)^{m_i}\cdot \left(a\cdot\prod_{i=1}^m (x-x_i)^2 + b\cdot\prod_{i=1}^{m-1} (x-y_i)^2 \right)\]
we see that $a\cdot x^{m_1 + \dots + m_r + 2m}$ is the monomial with the highest degree $m_1 + \dots + m_r + 2m = \deg p$ and the coefficient is $a$.
\end{sol}




\newpage

\providecommand{\bysame}{\leavevmode\hbox to3em{\hrulefill}\thinspace}
\providecommand{\MR}{\relax\ifhmode\unskip\space\fi MR }
\providecommand{\MRhref}[2]{%
  \href{http://www.ams.org/mathscinet-getitem?mr=#1}{#2}
}
\providecommand{\href}[2]{#2}

\Extrachap{List of Symbols}

\section*{Matrices}

$\begin{pmatrix}
f_0 & f_1 & \dots & f_n\\
x_1 & x_2 & \dots & x_n
\end{pmatrix}$: \Cref{dfn:kreinMatrix}, eq.\ (\ref{eq:kreinMatrix})\dotfill \pageref{dfn:kreinMatrix}\medskip

\noindent
$\begin{pmatrix}
f_0 & \dots & f_{i-1} & f_i & \dots & f_{i+p} & f_{i+p+1} & \dots & f_n\\
x_0 & \dots & x_{i-1} & (x_i & \dots & x_i) & x_{i+p+1} & \dots & x_n
\end{pmatrix}$: Eq.\ (\ref{eq:doublezeroDfn}) \dotfill \pageref{eq:doublezeroDfn}\medskip

\noindent
$\begin{pmatrix}
f_0 & f_1 & \dots & f_n\\
x_1 & x_2 & \dots & x_n
\end{pmatrix}^*$: Eq.\ (\ref{eq:matrixStar}) \dotfill \pageref{eq:matrixStar}\medskip

\noindent
$\left(\begin{array}{c|cccc}
f_0 & \, f_1 & f_2 & \dots & f_n\\
x & \, x_1 & x_2 & \dots & x_n
\end{array}\right)$: Eq.\ (\ref{eq:doublezeroDfn2}) \dotfill \pageref{eq:doublezeroDfn2}

\section*{Determinants}

\noindent
$K\!\begin{pmatrix} x_0 & x_1 & \dots & x_n\\ y_0 & y_1 & \dots & y_n\end{pmatrix}$: Eq.\ (\ref{eq:kernelDetDef}) \dotfill \pageref{eq:kernelDetDef}\medskip

\noindent
$K^*\!\begin{pmatrix} x_1 & x_2 & \dots & x_i\\ y_1 & y_2 & \dots & y_i\end{pmatrix}$: \Cref{dfn:etpk}, eq.\ (\ref{eq:kstarDet}) \dotfill \pageref{eq:kstarDet}\medskip

\noindent
$\cW(f_0,\dots,f_k)$: \Cref{dfn:wronski}, eq.\ (\ref{eq:wronski}) \dotfill \pageref{eq:wronski}

\section*{Further Mathematical Symbols}

\noindent
$\leq$ \dotfill \pageref{eq:leq}

\noindent
$B_{f,d}$: Eq.\ (\ref{eq:bernPoly}) \dotfill \pageref{eq:bernPoly}

\noindent
$h_n\nearrow g$ \dotfill \pageref{nearrow}

\noindent
$\fB(\rset^n)$ \dotfill \pageref{eq:borelSigmaAlg}

\noindent
$\cat(\cX,\cY)$ \dotfill \pageref{eq:catXY}

$\cat_c(\cX,\rset)$ \hfill \pageref{eq:catcXR}

\noindent
$\conv A$ \dotfill \pageref{eq:convA}

\noindent
$E_+$ \dotfill \pageref{eq:E+}

\textit{see also} $(\lin\cF)_+$ \hfill \pageref{dfn:linFsubsets}

\noindent
$E^*$ \dotfill \pageref{eq:Estar}

\noindent
$f_+$, $f_-$, $|f|$ \dotfill \pageref{eq:f+-||}

\noindent
$\cH(s)$: Eq.\ (\ref{eq:hankel}) \dotfill \pageref{eq:hankel}

\noindent
$\inter A$ \dotfill \pageref{eq:intA}

\noindent
$L_\mu$ \dotfill \pageref{eq:Lmu}

\noindent
$\cL^p(\cX,\mu)$ \dotfill \pageref{eq:Lp}

\noindent
$\varepsilon(x)$: \Cref{dfn:index}, eq.\ (\ref{eq:index}) \dotfill \pageref{dfn:index}

$\varepsilon(\cX)$: \Cref{dfn:index}, eq.\ (\ref{eq:indexSet}) \hfill \pageref{eq:indexSet}

\noindent
$K_\sigma$: Eq.\ (\ref{eq:gaussiankernel}) \dotfill \pageref{eq:gaussiankernel}

\noindent
$\lin\cF$: Eq.\ (\ref{eq:linF}) \dotfill \pageref{eq:linF}

$(\lin\cF)_+$: \Cref{dfn:linFsubsets} \hfill \pageref{dfn:linFsubsets}

$(\lin\cF)^e$: \Cref{dfn:linFsubsets} \hfill \pageref{dfn:linFsubsets}

$(\lin\cF)_+^e$: \Cref{dfn:linFsubsets} \hfill \pageref{dfn:linFsubsets}

\noindent
$L_\mu$: \Cref{dfn:Lmu} \dotfill \pageref{dfn:Lmu}

\noindent
$L_s$: \Cref{dfn:rieszFunctional} \dotfill \pageref{dfn:rieszFunctional}

\noindent
$\cM(L)$: \Cref{dfn:momentFunctional}\dotfill \pageref{dfn:momentFunctional}

\noindent
$\cM(\cX)_+$ \dotfill \pageref{eq:MX+}

\noindent
$\nset$ \dotfill \pageref{numbers}

$\nset_0$ \hfill \pageref{numbers}

\noindent
$\pos(K)$: Eq.\ (\ref{eq:posKdfn}) \dotfill \pageref{eq:posKdfn}

$\pos(\rset)$: Eq.\ (\ref{eq:posR}) \hfill \pageref{eq:posR}

$\pos([0,\infty))$: Eqs.\ (\ref{eq:pos0infty1}) and (\ref{eq:pos0infty2})  \hfill\pageref{eq:pos0infty1}

$\pos([-1,1])$: Eq.\ (\ref{eq:pos-11}) \hfill \pageref{eq:pos-11}

$\pos([a,b])$: Eq.\ (\ref{eq:posab}) \hfill \pageref{eq:posab}


\noindent
$\cP(\cX)$ \dotfill \pageref{eq:PX}

\noindent
$\qset$ \dotfill \pageref{numbers}

\noindent
$\rset$ \dotfill \pageref{numbers}

\noindent
$\cS_\cF$ \dotfill \pageref{dfn:SF}

\noindent
$\sigma(A)$ \dotfill \pageref{sigmaA}

\noindent
$\tset$ \dotfill \pageref{numbers}

\noindent
$\Xi^m$: Eq.\ (\ref{eq:simplex}) \dotfill \pageref{eq:simplex}

\noindent
$\zset$ \dotfill \pageref{numbers}


\begin{theindex}
{\bf A}\nopagebreak%
 \indexspace\nopagebreak%
  \item adapted
    \subitem cones\idxquad \hyperpage{25}
    \subitem space\idxquad \hyperpage{23}
  \item algebra\idxquad \hyperpage{4},\,\hyperpage{27}
    \subitem $*$-\idxquad \hyperpage{27}
      \subsubitem Fr\'echet topological\idxquad \hyperpage{27}
  \item approximation
    \subitem best
      \subsubitem polynomial\idxquad \hyperpage{117}
    \subitem Tchebycheff\idxquad \hyperpage{117}

  \indexspace
{\bf B}\nopagebreak%
 \indexspace\nopagebreak%
  \item Banach, S.\idxquad \hyperpage{2}
  \item basic composition formulas\idxquad \hyperpage{73}
  \item Basic Representation Theorem\idxquad \hyperpage{24}
  \item Bernstein
    \subitem polynomial\idxquad \hyperpage{32}
    \subitem Theorem\idxquad \hyperpage{32}
  \item Bernstein, S.\ N.\idxquad \hyperpage{32},\,\hyperpage{44}
  \item best approximation
    \subitem polynomial\idxquad \hyperpage{117}
  \item Boas' Theorem\idxquad \hyperpage{38}
  \item Boas, R.\ P.\idxquad \hyperpage{35},\,\hyperpage{38},\,
		\hyperpage{108}
  \item Borel
    \subitem measure\idxquad \hyperpage{4}
    \subitem $\sigma$-algebra\idxquad \hyperpage{4}
  \item Brickman, L.\idxquad \hyperpage{33}

  \indexspace
{\bf C}\nopagebreak%
 \indexspace\nopagebreak%
  \item Carath\'eodory
    \subitem Theorem\idxquad \hyperpage{3}
  \item Carroll, L.\idxquad \hyperpage{57}
  \item compact
    \subitem locally
      \subsubitem Hausdorff space\idxquad \hyperpage{2}
    \subitem set\idxquad \hyperpage{2}
  \item cone\idxquad \hyperpage{3}
    \subitem adapted\idxquad \hyperpage{25}
    \subitem negative\idxquad \hyperpage{1}
    \subitem positive\idxquad \hyperpage{1}
      \subsubitem linear functional\idxquad \hyperpage{21}
  \item Conic Extension Theorem\idxquad \hyperpage{26}
  \item continuous\idxquad \hyperpage{2}
  \item convex\idxquad \hyperpage{3}
  \item cubature formula
    \subitem Gaussian\idxquad \hyperpage{37}
  \item Curtis--Mairhuber--Sieklucki Theorem\idxquad \hyperpage{47}

  \indexspace
{\bf D}\nopagebreak%
 \indexspace\nopagebreak%
  \item Daniell's Representation Theorem\idxquad \hyperpage{6}
  \item Daniell's Signed Representation Theorem\idxquad \hyperpage{9}
  \item Daniell, P.\ J.\idxquad \hyperpage{6}
  \item decomposition
    \subitem Riesz property\idxquad \hyperpage{12}
  \item determinant
    \subitem representation as a\idxquad \hyperpage{49},\,
		\hyperpage{68}
    \subitem Vandermonde\idxquad \hyperpage{44},\,\hyperpage{48}
    \subitem Wronskian\idxquad \hyperpage{60}
  \item determinate
    \subitem moment functional\idxquad \hyperpage{19}
  \item Dirac, P.\ A.\ M.\idxquad \hyperpage{105}
  \item dominate\idxquad \hyperpage{23}
    \subitem cone\idxquad \hyperpage{25}
  \item dual\idxquad \hyperpage{2}

  \indexspace
{\bf E}\nopagebreak%
 \indexspace\nopagebreak%
  \item ECT-system\idxquad \hyperpage{60}
  \item Einstein, A.\idxquad \hyperpage{1}
  \item ET-system\idxquad \hyperpage{57}
  \item ETP$_k$\idxquad \hyperpage{72}
  \item Euripides\idxquad \hyperpage{71}
  \item extended
    \subitem totally
      \subsubitem positive\idxquad \hyperpage{72}

  \indexspace
{\bf F}\nopagebreak%
 \indexspace\nopagebreak%
  \item Fil'\v{s}tinski\v{\i}, V.\ A.\idxquad \hyperpage{34}
  \item Fixed Point Theorem of Brouwer\idxquad \hyperpage{79}
  \item Fr\'echet space\idxquad \hyperpage{27}
  \item function
    \subitem measurable\idxquad \hyperpage{4}
    \subitem $\mu$-integrable\idxquad \hyperpage{4}
  \item functional
    \subitem $K$-moment\idxquad \hyperpage{18}
    \subitem linear\idxquad \hyperpage{27}
      \subsubitem cone positive\idxquad \hyperpage{21}
      \subsubitem non-negative\idxquad \hyperpage{27}
    \subitem moment\idxquad \hyperpage{17}
    \subitem Riesz\idxquad \hyperpage{18}
    \subitem sublinear\idxquad \hyperpage{2}
    \subitem superlinear\idxquad \hyperpage{2}
  \item Fundamental Theorem of Algebra\idxquad \hyperpage{110}

  \indexspace
{\bf G}\nopagebreak%
 \indexspace\nopagebreak%
  \item Gau\ss, C.\ F.\idxquad \hyperpage{37}
  \item Gaussian kernel\idxquad \hyperpage{72}

  \indexspace
{\bf H}\nopagebreak%
 \indexspace\nopagebreak%
  \item Haar, A.\idxquad \hyperpage{118}
  \item Hahn, H.\idxquad \hyperpage{2}
  \item Hahn--Banach Theorem\idxquad \hyperpage{2}
  \item Hamburger moment problem\idxquad 
		\hyperindexformat{\see{Theorem, Hamburger}}{31}
  \item Hamburger's Theorem\idxquad \hyperpage{30}
  \item Hamburger, H.\ L.\idxquad \hyperpage{30},\,\hyperpage{39}
  \item Hankel matrix\idxquad \hyperpage{3}
    \subitem of a sequence\idxquad \hyperpage{3},\,\hyperpage{18}
  \item Hardy, G.\ H.\idxquad \hyperpage{77}
  \item Hausdorff
    \subitem moment problem
      \subsubitem sparse\idxquad \hyperpage{102}
    \subitem space\idxquad \hyperpage{2}
    \subitem truncated moment problem
      \subsubitem sparse\idxquad \hyperpage{101}
  \item Hausdorff moment problem\idxquad 
		\hyperindexformat{\see{Theorem, Hausdorff}}{32}
  \item Hausdorff's Theorem\idxquad \hyperpage{30}
  \item Hausdorff, F.\idxquad \hyperpage{30},\,\hyperpage{35},\,
		\hyperpage{101}
  \item Haviland's Theorem\idxquad \hyperpage{31}
  \item Haviland, E.\ K.\idxquad \hyperpage{31}
  \item Haviland--Hildebrandt--Schoenberg--Wintner Theorem\idxquad 
		\hyperpage{34}
  \item Haviland--Wintner Theorem\idxquad \hyperpage{34}
  \item Helly, E.\idxquad \hyperpage{2}
  \item Herglotz
    \subitem moment problem\idxquad \hyperpage{31}
  \item Hilbert, D.\idxquad \hyperpage{32}
  \item Hildebrandt, T.\ H.\idxquad \hyperpage{34}
  \item hull
    \subitem convex\idxquad \hyperpage{3}

  \indexspace
{\bf I}\nopagebreak%
 \indexspace\nopagebreak%
  \item identity
    \subitem Sylvester\idxquad \hyperpage{62}
  \item indeterminacy
    \subitem Stieltjes example\idxquad \hyperpage{19}
  \item indeterminate
    \subitem moment
      \subsubitem functional\idxquad \hyperpage{19}
  \item index\idxquad \hyperpage{52}
    \subitem of a set\idxquad \hyperpage{52}
  \item interior
    \subitem of a set\idxquad \hyperpage{2}
  \item Interlacing Theorem\idxquad 
		\hyperindexformat{\see{Snake Theorem}}{83}

  \indexspace
{\bf J}\nopagebreak%
 \indexspace\nopagebreak%
  \item Jackson, D.\idxquad \hyperpage{121}

  \indexspace
{\bf K}\nopagebreak%
 \indexspace\nopagebreak%
  \item $K$-moment functional\idxquad \hyperpage{18}
  \item Kakutani, S.\idxquad \hyperpage{11}
  \item Karlin
    \subitem Nichtnegativstellensatz
      \subsubitem on $[0,\infty)$\idxquad \hyperpage{91}
      \subsubitem on $[a,b]$\idxquad \hyperpage{86}
      \subsubitem on $\rset$\idxquad \hyperpage{92}
    \subitem Positivstellensatz
      \subsubitem on $[0,\infty)$\idxquad \hyperpage{89}
      \subsubitem on $[a,b]$\idxquad \hyperpage{82}
      \subsubitem on $\rset$\idxquad \hyperpage{92}
    \subitem Theorem
      \subsubitem for $f>0$ on $[a,b]$\idxquad \hyperpage{77}
      \subsubitem for $f\geq 0$ on $[a,b]$\idxquad \hyperpage{84}
  \item Karlin, S.\idxquad \hyperpage{vii}
  \item kernel\idxquad \hyperpage{71}
    \subitem Gaussian\idxquad \hyperpage{72}
  \item Krein, M.\ G.\idxquad \hyperpage{33},\,\hyperpage{83}

  \indexspace
{\bf L}\nopagebreak%
 \indexspace\nopagebreak%
  \item lattice
    \subitem of functions\idxquad \hyperpage{5}
    \subitem space\idxquad \hyperpage{5}
  \item Leibniz, G.\ W.\idxquad \hyperpage{97}
  \item Lemma
    \subitem Markov\idxquad \hyperpage{82}
  \item Lewin, K.\idxquad \hyperpage{43}
  \item locally compact\idxquad \hyperpage{2}
  \item Luk\'acs Theorem\idxquad 
		\hyperindexformat{\see{Luk\'acs--Markov Theorem}}{33}
  \item Luk\'acs, F.\idxquad \hyperpage{33}
  \item Luk\'acs--Markov Theorem\idxquad \hyperpage{33},\,
		\hyperpage{99}

  \indexspace
{\bf M}\nopagebreak%
 \indexspace\nopagebreak%
  \item majorized\idxquad \hyperpage{21}
  \item Markov Lemma\idxquad \hyperpage{82}
  \item Markov's Theorem\idxquad 
		\hyperindexformat{\see{Luk\'acs--Markov Theorem}}{33}
  \item Markov, A.\ A.\idxquad \hyperpage{11},\,\hyperpage{33}
  \item Markov--Luk\'acs Theorem\idxquad 
		\hyperindexformat{\see{Luk\'acs--Markov Theorem}}{34}
  \item measurable
    \subitem function\idxquad \hyperpage{4}
    \subitem space\idxquad \hyperpage{4}
  \item measure\idxquad \hyperpage{4}
    \subitem Borel\idxquad \hyperpage{4}
    \subitem Carath\'eodory outer\idxquad \hyperpage{4}
    \subitem outer\idxquad \hyperpage{4}
    \subitem Radon\idxquad \hyperpage{4}
    \subitem regular\idxquad \hyperpage{5}
    \subitem representing\idxquad \hyperpage{17}
    \subitem space\idxquad \hyperpage{4}
  \item Milton, J.\idxquad \hyperpage{113}
  \item minorized\idxquad \hyperpage{21}
  \item moment
    \subitem $f$-moment of $\mu$\idxquad \hyperpage{17}
    \subitem functional\idxquad \hyperpage{17}
      \subsubitem determinate\idxquad \hyperpage{19}
      \subsubitem generated by $\mu$\idxquad \hyperpage{18}
      \subsubitem indeterminate\idxquad \hyperpage{19}
      \subsubitem truncated\idxquad \hyperpage{19}
    \subitem problem
      \subsubitem Hamburger\idxquad 
		\hyperindexformat{\see{Theorem, Hamburger}}{30}
      \subsubitem Hausdorff\idxquad 
		\hyperindexformat{\see{Theorem, Hausdorff}}{30}
      \subsubitem Herglotz\idxquad \hyperpage{31}
      \subsubitem Stieltjes\idxquad 
		\hyperindexformat{\see{Theorem, Stieltjes}}{29}
      \subsubitem {\v{S}}venco\idxquad 
		\hyperindexformat{\see{Theorem, {\v{S}}venco}}{34}
      \subsubitem trigonometric\idxquad \hyperpage{31}
    \subitem sequence\idxquad \hyperpage{18}
  \item $\mu$-integrable\idxquad \hyperpage{4}
  \item $\mu$-measurable\idxquad \hyperpage{4}

  \indexspace
{\bf N}\nopagebreak%
 \indexspace\nopagebreak%
  \item negative cone\idxquad \hyperpage{1}
  \item neighborhood\idxquad \hyperpage{2}
  \item Nichtnegativstellensatz
    \subitem Karlin
      \subsubitem on $[0,\infty)$\idxquad \hyperpage{91}
      \subsubitem on $[a,b]$\idxquad \hyperpage{86}
      \subsubitem on $\rset$\idxquad \hyperpage{92}
  \item Nudel'man, A.\ A.\idxquad \hyperpage{33},\,\hyperpage{83}

  \indexspace
{\bf O}\nopagebreak%
 \indexspace\nopagebreak%
  \item optimization
    \subitem over linear functionals\idxquad \hyperpage{121}
  \item order
    \subitem partial\idxquad \hyperpage{1}
    \subitem total\idxquad \hyperpage{1}
    \subitem vector space\idxquad \hyperpage{1}
  \item orthogonal
    \subitem polynomial\idxquad \hyperpage{121}

  \indexspace
{\bf P}\nopagebreak%
 \indexspace\nopagebreak%
  \item P\'olya, G.\idxquad \hyperpage{32},\,\hyperpage{39},\,
		\hyperpage{89}
  \item partial order\idxquad \hyperpage{1}
  \item polynomial\idxquad \hyperpage{44}
    \subitem Bernstein\idxquad \hyperpage{32}
    \subitem best approximation\idxquad \hyperpage{117}
    \subitem orthogonal\idxquad \hyperpage{121}
    \subitem Schur\idxquad \hyperpage{68}
  \item Pope, A.\idxquad \hyperpage{17}
  \item positive
    \subitem totally\idxquad \hyperpage{71}
      \subsubitem extended\idxquad \hyperpage{72}
      \subsubitem strictly\idxquad \hyperpage{71}
  \item positive cone\idxquad \hyperpage{1}
  \item Positivstellensatz
    \subitem Karlin
      \subsubitem on $[0,\infty)$\idxquad \hyperpage{89}
      \subsubitem on $[a,b]$\idxquad \hyperpage{82}
      \subsubitem on $\rset$\idxquad \hyperpage{92}

  \indexspace
{\bf R}\nopagebreak%
 \indexspace\nopagebreak%
  \item Radau, M.\ R.\idxquad \hyperpage{33}
  \item Radon
    \subitem measure\idxquad \hyperpage{4}
  \item reduced
    \subitem system\idxquad \hyperpage{61},\,\hyperpage{63}
  \item regular
    \subitem measure\idxquad \hyperpage{5}
  \item representation
    \subitem as a determinant\idxquad \hyperpage{49},\,\hyperpage{68}
    \subitem Theorem
      \subsubitem Basic\idxquad \hyperpage{24}
      \subsubitem Daniell\idxquad \hyperpage{6}
      \subsubitem Daniell, signed\idxquad \hyperpage{9}
      \subsubitem Riesz\idxquad \hyperpage{12}
      \subsubitem Riesz, signed\idxquad \hyperpage{11}
      \subsubitem Riesz--Markov--Kakutani\idxquad \hyperpage{11}
  \item representing
    \subitem measure\idxquad \hyperpage{17}
  \item Richter's Theorem\idxquad \hyperpage{36}
  \item Richter, H.\idxquad \hyperpage{35}
  \item Richter--Rogosinski--Rosenbloom Theorem\idxquad \hyperpage{38}
  \item Riesz
    \subitem decomposition property\idxquad \hyperpage{12}
  \item Riesz functional\idxquad \hyperpage{18}
  \item Riesz' Representation Theorem\idxquad \hyperpage{12}
    \subitem Signed\idxquad \hyperpage{11}
  \item Riesz, F.\idxquad \hyperpage{11}
  \item Riesz, M.\idxquad \hyperpage{31}
  \item Riesz--Markov--Kakutani Representation Theorem\idxquad 
		\hyperpage{11}
  \item Rogosinski, W.\ W.\idxquad \hyperpage{38}
  \item Rosenbloom, P.\ C.\idxquad \hyperpage{37}

  \indexspace
{\bf S}\nopagebreak%
 \indexspace\nopagebreak%
  \item Santayana, G.\idxquad \hyperpage{29}
  \item Schmüdgen, K.\idxquad \hyperpage{39}
  \item Schoenberg, I.\ J.\idxquad \hyperpage{34}
  \item Schur polynomial\idxquad \hyperpage{68}
  \item sequence
    \subitem moment\idxquad \hyperpage{18}
  \item Shakespeare, W.\idxquad \hyperpage{117}
  \item Sherman, T.\idxquad \hyperpage{39}
  \item $\sigma$-algebra\idxquad \hyperpage{4}
    \subitem Borel\idxquad \hyperpage{4}
  \item Snake Theorem\idxquad \hyperpage{83}
  \item space
    \subitem adapted\idxquad \hyperpage{23}
    \subitem Hausdorff\idxquad \hyperpage{2}
      \subsubitem locally compact\idxquad \hyperpage{2}
    \subitem lattice\idxquad \hyperpage{5}
    \subitem measurable\idxquad \hyperpage{4}
    \subitem measure\idxquad \hyperpage{4}
    \subitem topological\idxquad \hyperpage{2}
  \item sparse
    \subitem algebraic Nichtnegativstellensatz
      \subsubitem on $[0,\infty)$\idxquad \hyperpage{109}
      \subsubitem on $[0,b]$\idxquad \hyperpage{104}
      \subsubitem on $[a,b]$\idxquad \hyperpage{103}
      \subsubitem on $\rset$\idxquad \hyperpage{110}
    \subitem algebraic Positivstellensatz
      \subsubitem on $[0,\infty)$\idxquad \hyperpage{105}
      \subsubitem on $[a,b]$\idxquad \hyperpage{97}
      \subsubitem on $\rset$\idxquad \hyperpage{110}
    \subitem Hausdorff moment problem\idxquad \hyperpage{101,\,102}
      \subsubitem truncated\idxquad \hyperpage{101}
    \subitem Stieltjes moment problem\idxquad \hyperpage{108}
  \item Stieltjes
    \subitem example
      \subsubitem indeterminacy\idxquad \hyperpage{19}
  \item Stieltjes moment problem\idxquad 
		\hyperindexformat{\see{Theorem, Stieltjes}}{32}
  \item Stieltjes' Theorem\idxquad \hyperpage{29}
  \item Stieltjes, T.\ J.\idxquad \hyperpage{19},\,\hyperpage{29}
  \item Stochel, J.\idxquad \hyperpage{39}
  \item Stone--Weierstrass Theorem\idxquad \hyperpage{3}
  \item STP$_k$\idxquad \hyperpage{71}
  \item Strassen's Theorem\idxquad \hyperpage{22}
  \item strictly
    \subitem totally
      \subsubitem positive\idxquad \hyperpage{71}
  \item sublinear
    \subitem functional\idxquad \hyperpage{2}
  \item suits\idxquad \hyperpage{6}
  \item superlinear
    \subitem functional\idxquad \hyperpage{2}
  \item {\v{S}}venco's Theorem\idxquad \hyperpage{34}
  \item {\v{S}}venco, K.\ I.\idxquad \hyperpage{34}
  \item Sylvester's identity\idxquad \hyperpage{62}
  \item system
    \subitem ECT-\idxquad \hyperpage{60}
    \subitem ET-\idxquad \hyperpage{57}
    \subitem reduced\idxquad \hyperpage{61},\,\hyperpage{63}
    \subitem T-\idxquad \hyperpage{44}
    \subitem Tchebycheff\idxquad \hyperindexformat{\see{T-}}{44}
  \item Szeg\"o, G.\idxquad \hyperpage{32}

  \indexspace
{\bf T}\nopagebreak%
 \indexspace\nopagebreak%
  \item T-system\idxquad \hyperpage{44}
    \subitem continuous\idxquad \hyperpage{45}
    \subitem extended\idxquad \hyperpage{57}
      \subsubitem complete\idxquad \hyperpage{60}
    \subitem periodic\idxquad \hyperpage{45}
  \item Tchakaloff, M.\ V.\idxquad \hyperpage{37}
  \item Tchebycheff approximation\idxquad \hyperpage{117}
  \item Tchebycheff system\idxquad 
		\hyperindexformat{\see{T-system}}{44}
  \item Tchebycheff, P.\ L.\idxquad \hyperpage{43}
  \item Theorem
    \subitem Algebra
      \subsubitem Fundamental\idxquad \hyperpage{110}
    \subitem Basic Representation\idxquad \hyperpage{24}
    \subitem Bernstein\idxquad \hyperpage{32}
    \subitem Boas\idxquad \hyperpage{38}
    \subitem Brouwer Fixed Point\idxquad \hyperpage{79}
    \subitem Carath\'eodory\idxquad \hyperpage{3}
    \subitem Conic Extension\idxquad \hyperpage{26}
    \subitem Curtis--Mairhuber--Sieklucki\idxquad \hyperpage{47}
    \subitem Daniell's Representation\idxquad \hyperpage{6}
      \subsubitem Signed\idxquad \hyperpage{9}
    \subitem Hahn--Banach\idxquad \hyperpage{2}
    \subitem Hamburger\idxquad \hyperpage{30}
    \subitem Hausdorff\idxquad \hyperpage{30}
      \subsubitem sparse\idxquad \hyperpage{101}
    \subitem Haviland\idxquad \hyperpage{31}
    \subitem Haviland--Hildebrandt--Schoenberg--Wintner\idxquad 
		\hyperpage{34}
    \subitem Haviland--Wintner\idxquad \hyperpage{34}
    \subitem Interlacing\idxquad \hyperindexformat{\see{Snake}}{83}
    \subitem Karlin
      \subsubitem for $f>0$ on $[a,b]$\idxquad \hyperpage{77}
      \subsubitem for $f\geq 0$ on $[a,b]$\idxquad \hyperpage{84}
      \subsubitem Nichtnegativstellensatz on $[0,\infty)$\idxquad 
		\hyperpage{91}
      \subsubitem Nichtnegativstellensatz on $[a,b]$\idxquad 
		\hyperpage{86}
      \subsubitem Nichtnegativstellensatz on $\rset$\idxquad 
		\hyperpage{92}
      \subsubitem Positivstellensatz on $[0,\infty)$\idxquad 
		\hyperpage{89}
      \subsubitem Positivstellensatz on $[a,b]$\idxquad \hyperpage{82}
      \subsubitem Positivstellensatz on $\rset$\idxquad \hyperpage{92}
    \subitem Luk\'acs\idxquad 
		\hyperindexformat{\see{Luk\'acs--Markov}}{33}
    \subitem Luk\'acs--Markov\idxquad \hyperpage{33},\,\hyperpage{99}
    \subitem Markov\idxquad 
		\hyperindexformat{\see{Luk\'acs--Markov}}{33}
    \subitem Markov--Luk\'acs\idxquad 
		\hyperindexformat{\see{Luk\'acs--Markov}}{34}
    \subitem M\"untz--Sz\'asz\idxquad \hyperpage{102}
    \subitem Richter\idxquad \hyperpage{36}
    \subitem Richter--Rogosinski--Rosenbloom\idxquad \hyperpage{38}
    \subitem Riesz' Representation\idxquad \hyperpage{12}
      \subsubitem Signed\idxquad \hyperpage{11}
    \subitem Riesz--Markov--Kakutani Representation\idxquad 
		\hyperpage{11}
    \subitem Rogosinski\idxquad \hyperpage{38}
    \subitem Rosenbloom\idxquad \hyperpage{37}
    \subitem Snake\idxquad \hyperpage{83}
    \subitem Stieltjes\idxquad \hyperpage{29}
    \subitem Stone--Weierstrass\idxquad \hyperpage{3}
    \subitem Strassen\idxquad \hyperpage{22}
    \subitem {\v{S}}venco\idxquad \hyperpage{34}
    \subitem Tchakaloff\idxquad \hyperpage{37}
    \subitem Wald\idxquad \hyperpage{37}
  \item topology\idxquad \hyperpage{2}
  \item total order\idxquad \hyperpage{1}
  \item totally
    \subitem positive\idxquad \hyperpage{71}
  \item TP$_k$\idxquad \hyperpage{71}
  \item trigonometric
    \subitem moment problem\idxquad \hyperpage{31}
  \item truncated
    \subitem moment functional\idxquad \hyperpage{19}

  \indexspace
{\bf U}\nopagebreak%
 \indexspace\nopagebreak%
  \item unit element\idxquad \hyperpage{27}

  \indexspace
{\bf V}\nopagebreak%
 \indexspace\nopagebreak%
  \item Vandermonde
    \subitem determinant\idxquad \hyperpage{44},\,\hyperpage{48}
  \item vector space
    \subitem ordered\idxquad \hyperpage{1}

  \indexspace
{\bf W}\nopagebreak%
 \indexspace\nopagebreak%
  \item Wald, A.\idxquad \hyperpage{37}
  \item Wiener, N.\idxquad \hyperpage{21}
  \item Wintner, A.\idxquad \hyperpage{31}
  \item Wronskian\idxquad \hyperpage{60}
  \item Wronskian determiant\idxquad 
		\hyperindexformat{\see{Wronskian}}{60}

  \indexspace
{\bf Z}\nopagebreak%
 \indexspace\nopagebreak%
  \item zero
    \subitem multiplicity\idxquad \hyperpage{57}
    \subitem nodal\idxquad \hyperpage{50}
    \subitem non-nodal\idxquad \hyperpage{50}

\end{theindex}


\end{document}